\documentclass[11pt]{article}

\usepackage{amsmath, amsthm}
\usepackage{amsmath, amsfonts}
\usepackage{amsmath, amssymb}
\usepackage{amsmath}
\usepackage{graphics}
\usepackage{graphicx}
\usepackage{color}
\usepackage{hyperref}
\usepackage[all]{xy}

\newtheorem{theorem}{Theorem}[subsection]
\newtheorem{definition}[theorem]{Definition}
\newtheorem{definition-lemma}[theorem]{Definition/Lemma}
\newtheorem{definition-explanation}[theorem]{Definition/Explanation}
\newtheorem{explanation-definition}[theorem]{Explanation/Definition}
\newtheorem{definition-fact}[theorem]{Definition/Fact}
\newtheorem{definition-notation}[theorem]{Definition/Notation}
\newtheorem{definition-conjecture}[theorem]{Definition/Conjecture}
\newtheorem{lemma}[theorem]{Lemma}
\newtheorem{lemma-definition}[theorem]{Lemma/Definition}
\newtheorem{proposition}[theorem]{Proposition}
\newtheorem{corollary}[theorem]{Corollary}
\newtheorem{remark}[theorem]{\it Remark}
\newtheorem{remark-notation}[theorem]{\it Remark/Notation}

\newtheorem{application-lemma}[theorem]{Application/Lemma}

\newtheorem{convention}[theorem]{\it Convention}

\newtheorem{example}[theorem]{Example}
\newtheorem{example-definition}[theorem]{Example/Definition}
\newtheorem{explanation}[theorem]{Explanation}

\newtheorem{notation}[theorem]{Notation}

\newtheorem{definition-prototype}[theorem]{Definition-Prototype}

\numberwithin{equation}{subsection}

\newtheorem{stheorem}{Theorem}[section]
\newtheorem{sdefinition}[stheorem]{Definition}
\newtheorem{sdefinition-lemma}[stheorem]{Definition/Lemma}
\newtheorem{sdefinition-explanation}[stheorem]{Definition/Explanation}
\newtheorem{sexplanation-definition}[stheorem]{Explanation/Definition}
\newtheorem{sdefinition-fact}[stheorem]{Definition/Fact}
\newtheorem{sdefinition-notation}[stheorem]{Definition/Notation}
\newtheorem{sdefinition-conjecture}[stheorem]{Definition/Conjecture}
\newtheorem{slemma}[stheorem]{Lemma}
\newtheorem{slemma-definition}[stheorem]{Lemma/Definition}

\newtheorem{sremark}[stheorem]{\it Remark}
\newtheorem{sremark-notation}[stheorem]{\it Remark/Notation}

\newtheorem{sapplication-lemma}[stheorem]{Application/Lemma}

\newtheorem{sexample}[stheorem]{Example}
\newtheorem{sexample-definition}[stheorem]{Example/Definition}

\newtheorem{sdefinition-prototype}[stheorem]{Definition-Prototype}

\newtheorem{ssdefinition-lemma}[sstheorem]{Definition/Lemma}
\newtheorem{ssdefinition-explanation}[sstheorem]{Definition/Explanation}
\newtheorem{ssexplanation-definition}[sstheorem]{Explanation/Definition}
\newtheorem{ssdefinition-fact}[sstheorem]{Definition/Fact}
\newtheorem{ssdefinition-notation}[sstheorem]{Definition/Notation}
\newtheorem{ssdefinition-conjecture}[sstheorem]{Definition/Conjecture}

\newtheorem{sslemma-definition}[sstheorem]{Lemma/Definition}

\newtheorem{ssremark-notation}[sstheorem]{\it Remark/Notation}

\newtheorem{ssapplication-lemma}[sstheorem]{Application/Lemma}

\newtheorem{ssexample-definition}[sstheorem]{Example/Definition}

\newtheorem{ssdefinition-prototype}[sstheorem]{Definition-Prototype}

\newcommand{\Aut}{\mbox{\it Aut}\,}
\newcommand{\Autsheaf}{\mbox{\it ${\cal A}$ut}\,}

\newcommand{\Der}{\mbox{\it Der}\,}

\newcommand{\End}{\mbox{\it End}\,}

\newcommand{\Endsheaf}{\mbox{\it ${\cal E}\!$nd}\,}
 \newcommand{\Endsheaffootnotesize}{\mbox{\footnotesize\it ${\cal E}\!$nd}\,}

\newcommand{\Homsheaf}{\mbox{\it ${\cal H}$om}\,}

\newcommand{\Id}{\mbox{\it Id}\,}

\newcommand{\Left}{\mbox{\it\scriptsize Left-}\,}

\newcommand{\Real}{\mbox{\it Re}\,}

\newcommand{\Right}{\mbox{\it\scriptsize Right-}\,}

\newcommand{\SO}{\mbox{\it SO}\,}

\newcommand{\Sym}{\mbox{\it Sym}}
 \newcommand{\scriptsizeSym}{\mbox{\scriptsize\it Sym}}
 \newcommand{\tinySym}{\mbox{\tiny\it Sym}}

\newcommand{\Tr}{\mbox{\it Tr}\,}

\newcommand{\scriptsizeach}{\mbox{\scriptsize\it ach}} 
  \newcommand{\scriptsizedach}{\mbox{\scriptsize\it $d$-ach}} 
  \newcommand{\scriptsizewidehatDach}{\mbox{\scriptsize\it $\widehat{D}$-ach}}

\newcommand{\anticommuting}{\mbox{\scriptsize\it anti-c}}

\newcommand{\scriptsizecan}{\mbox{\scriptsize\it can}\,}
  \newcommand{\tinycan}{\mbox{\tiny\it can}\,}

\newcommand{\chd}{\mbox{\it c.h.d}\,}
  \newcommand{\footnotesizechd}{\mbox{\it\footnotesize c.h.d}\,}
\newcommand{\scriptsizech}{\mbox{\scriptsize\it ch}}
  \newcommand{\scriptsizedch}{\mbox{\scriptsize\it $d$-ch}}
  \newcommand{\scriptsizewidehatDch}{\mbox{\scriptsize\it $\widehat{D}$-ch}}

\newcommand{\even}{\mbox{\scriptsize\rm even}}
  \newcommand{\tinyeven}{\mbox{\tiny\rm even}}

\newcommand{\hol}{\mbox{\it\scriptsize hol}}
  \newcommand{\ahol}{\mbox{\it\scriptsize ahol}}

\newcommand{\leftscriptsize}{\mbox{\it\scriptsize left}}

\newcommand{\odd}{\mbox{\scriptsize\rm odd}}		
  \newcommand{\tinyodd}{\mbox{\tiny\rm odd}}

\newcommand{\bolda}{\mbox{\boldmath $a$}}
\newcommand{\boldd}{\mbox{\boldmath $d$}}
  \newcommand{\scriptsizeboldd}{\mbox{\scriptsize\boldmath $d$}}
    
\newcommand{\boldm}{\mbox{\boldmath $m$}}  
  \newcommand{\scriptsizeboldm}{\mbox{\scriptsize\boldmath $m$}}
\newcommand{\boldn}{\mbox{\boldmath $n$}}    

\newcommand{\boldsigma}{\mbox{\boldmath $\sigma$}}
\newcommand{\boldt}{\mbox{\boldmath $t$}}
  \newcommand{\scriptsizeboldt}{\mbox{\scriptsize\boldmath $t$}}
  
\newcommand{\boldu}{\mbox{\boldmath $u$}}

\newcommand{\boldy}{\mbox{\boldmath $y$}}

\newcommand{\boldz}{\mbox{\boldmath $z$}}

\newcommand{\longrightaarrow}{\longrightarrow\hspace{-3ex}\longrightarrow}
\newcommand{\rightaarrow}{\rightarrow\hspace{-2ex}\rightarrow}

\newcommand{\tinybullet}{\raisebox{.2ex}{\tiny $\bullet$}}							
																				
\begin{document}

\enlargethispage{24cm}

\begin{titlepage}

$ $

\vspace{-1.5cm} 

\noindent\hspace{-1cm}
\parbox{6cm}{\small May 2018}\
   \hspace{7cm}\
   \parbox[t]{6cm}{\small
                arXiv:yymm.nnnnn [math.DG] \\
                D(14.1): $N=1$ sD3, foundations\\
				}

\vspace{2em}

\centerline{\large\bf
$N=1$  fermionic D3-branes in RNS formulation}
\vspace{1ex}
\centerline{\large\bf 
 I.  $C^\infty$-Algebrogeometric foundations of $d=4$, $N=1$ supersymmetry,} 
\vspace{1ex}
\centerline{\large\bf 
 SUSY-rep compatible hybrid connections, and} 
\vspace{1ex}
\centerline{\large\bf   
 $\widehat{D}$-chiral maps from a $d=4$ $N=1$ Azumaya/matrix superspace}

\bigskip

\bigskip\bigskip

\centerline{\large
  Chien-Hao Liu   
            \hspace{1ex} and \hspace{1ex}
  Shing-Tung Yau
}

\vspace{2em}

\begin{quotation}
\centerline{\bf Abstract}
\vspace{0.3cm}

\baselineskip 12pt  
{\small
 To construct a supersymmetric theory for a fermionic D-brane moving in a space-time $Y$, 
   modeled by smooth maps 
   $\widehat{\varphi}: (\widehat{X}^{\!A\!z}\!, \widehat{\cal E}; \widehat{\nabla})
       \rightarrow Y$
   from an Azumaya/matrix supermanifold $X^{\!A\!z}$
     with a fundamental module $\widehat{\cal E}$ with a connection $\widehat{\nabla}$ 
   to $Y$,    
   one needs to impose some SUSY-Rep Compatible Conditions 
   on $(\widehat{\nabla},\widehat{\varphi})$.
 In this work, we begin the task to address such issues.
 To begin, we focus on the theory for dynamical fermionic D3-branes, 
   in which the $d=4$, $N=1$ supersymmetry algebra is involved.

 As the necessary background to construct from the aspect of Grothendieck's Algebraic Geometry
    dynamical fermionic D3-branes along the line of Ramond-Neveu-Schwarz superstrings in string theory,
 three pieces of the building blocks are given in the current notes:
  \begin{itemize}
   \item[(1)]
    basic $C^\infty$-algebrogeometric foundations of
               $d=4$, $N=1$ supersymmetry and $d=4$, $N=1$ superspace in physics,
           with emphasis on the partial $C^\infty$-ring structure on the function ring of the superspace,
	
   \item[(2)]	
    the notion of SUSY-rep compatible hybrid connections on bundles over the superspace
    to address connections on the Chan-Paton bundle on the world-volume of a fermionic D3-brane,
	
  \item[(3)]	
   the notion of $\widehat{D}$-chiral maps $\widehat{\varphi}$
     from a $d=4$ $N=1$ Azumaya/matrix superspace with a fundamental module
	   with a SUSY-rep compatible hybrid connection $\widehat{\nabla}$
	 to a complex manifold $Y$
	as a model for a dynamical D3-branes moving in a target space(-time).
 \end{itemize}
 Some test computations related to the construction of a supersymmetric action functional for
   SUSY-rep compatible $(\widehat{\nabla}, \widehat{\varphi})$
  are given in the end, whose further study is the focus of a separate work.
 The current work is a sequel to
    D(11.4.1) (arXiv:1709.08927 [math.DG]) and
	D(11.2) (arXiv:1412.0771 [hep-th])
   and is the first step in the supersymmetric generalization, in the case of D3-branes,
    of the standard action functional for D-branes constructed in
    D(13.3) (arXiv:1704.03237 [hep-th]).
 } 
\end{quotation}

\smallskip

\baselineskip 12pt
{\footnotesize
\noindent
{\bf Key words:} \parbox[t]{14cm}{fermionic D3-brane, supersymmetry;
   $C^\infty$ hull, partial $C^\infty$-ring structure; SUSY-rep compatible\\ hybrid connection,
   Azumaya/matrix superspace, chiral/antichiral structure, $\widehat{D}$-chiral map
 }} 

\medskip

\noindent {\small MSC number 2010:  81T30, 46L87, 14A22; 81Q60, 58C25, 16S50
} 

\medskip

\baselineskip 10pt
{\scriptsize
\noindent{\bf Acknowledgements.}
We thank
 Andrew Strominger and Cumrun Vafa
   for influence to our understanding of strings, branes, and gravity.
C.-H.L.\ thanks in addition
  Rafael Nepomechie,
  Stephen Weinberg,
  John Terning,
  Shiraz Minwalla,
  Girma Hailu
    for topic courses (time-ordered) on supersymmetry at various institutes during the brewing years;
 Peng Gao, Shinobu Hosono
    for discussions on a technical issue;
 Cumrun Vafa
    for consultations;	
 Man-Wai Cheung, Karsten Gimre, Peter Kronheimer, Jacob Lurie, H\'{e}ctor Past\'{e}n V\'{a}squez
   for topic courses  and
 Ivan Losev, Artan Sheshmani
   for special lecture series,
   spring 2018;
 Xue Su and Rhett Lei for their performance of Carl Reinecke's `Undine' that accompanies the typing of the work;
 Francesca Pei-Jung Chen for the year-long monthly flute lessons (August 2017--August 2018)
   that provide an additional dimension to mind, a getaway from the project, and for her motivating life story;
 Ling-Miao Chou
   for comments that improve the illustrations and moral support.
S.-T.Y.\ thanks in addition
 Mathematical Science Center and Department of Mathematical Sciences, Tsinghua University, Beijing, China,
     for hospitality in fall 2017. 	
The project is supported by NSF grants DMS-9803347 and DMS-0074329.
} 

\end{titlepage}

\newpage

\enlargethispage{24cm}
\begin{titlepage}

$ $

\vspace{6em}

\centerline{\small\it
 Chien-Hao Liu dedicates this work to his mother,
  Mrs.~Pin-Fen Lin Liu,}
\centerline{\small\it
 and siblings $($resp.\ sibling-in-laws$)$}
\centerline{\small\it
 Hsiu-Chuan Liu $($Yi-Shou Chao$)$, Chien-Ying Liu $($Tracy Tse Chui$)$,}
\centerline{\small\it
 Alice Hsiu-Hsiang Liu $($Tom Cheng-Wei Chang$)$, Hsiu-Luan Liu,}
\centerline{\small\it
 and to the memory of his father, Mr.~Han-Tu Liu$^{\dag}$
                                  $[\,\!^{\dag}\!$deceased 1986\,$]$.}
 
\vspace{36em}

\baselineskip 10pt
 {\scriptsize
\noindent$^{\ast}${\it From C.-H.L.}$\,$:
The profound influence to me from my parents
  and what I have owed them are beyond words.
Without their so much patience, understanding,
  and quiet love and trust on me
  and my older sister Hsiu-Chuan's taking care of part of
  the financial responsibility for the family during my teenage,
I would be just a rebellious high school dropout without a future,
  let alone being accidentally exposed to
  the beauty and the mystery of superstring/M/F theory later.
This work means special in the D-project as it is the first time since the start a decade ago
  that the core of supersymmetry is brought~in in the way like physicists do supersymmetry,
  with a polish by $C^\infty$-Algebraic Geometry and some recast to serve the goal of constructing
  dynamical fermionic D-branes.	
It, indeed the entire project, is thus their making as well.
Special thanks to my brother-in-law, sister-in-law, little sister, and niece {\it Yen-Ching Chao}
  who filled the needs of my mother during the last three years when many things happened
  while I was brewing on subjects related to the current work and beyond.

\smallskip
\noindent
\includegraphics[width=1\textwidth]{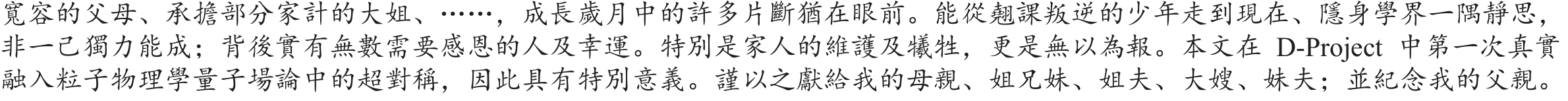}
 
 } 

\end{titlepage}


\newpage
$ $

\vspace{-3em}

\centerline{\sc
 $N=1$ Fermionic D3-Branes in RNS Formulation I. Foundations
 } %

\vspace{2em}

\baselineskip 14pt  

\begin{flushleft}
{\Large\bf 0. Introduction and outline}
\end{flushleft}
From the aspect of Grothendieck's Algebraic Geometry ([Har]; [Ei], [E-H]),
 a dynamical D-brane ([Po1], [Po2]), a notion developed first by Polchinski and his group,
 moving in a target space(-time) $Y$ can be described as
 a `{\it morphism $\varphi$}'  from
 an `{\it Azumaya locally-ringed space}'
 with a fundamental module with a connection to $Y$
 ([L-Y1] (D(1)), [Liu]; see also [H-W], [Wi5]).
Mathematically, depending on the context,
 this `Azumaya locally-ringed space' can
 an Azumaya scheme ([L-Y1] (D(1)), [L-L-S-Y] (D(2)), [L-Y2] (D(6))) or
 an Azumaya $C^\infty$-scheme ([L-Y3] (D(11.1)))\, $X^{\!A\!z}$,
   if one considers only bosonic D-branes,
 or an Azumaya super $C^\infty$-scheme $\widehat{X}^{\!A\!z}$ ([L-Y4] (D(11.2))),
   if one also takes into account fermionic D-branes,
 while the `morphism' turns out to have to be defined cotravariantly
 as an equivalence class of gluing systems of ring-homomorphisms
 $$
  \varphi^\sharp\;:\; {\cal O}_Y\;\longrightarrow\; {\cal O}_X^{A\!z}
  \hspace{2em}\mbox{(resp.\;
    $\widehat{\varphi}^\sharp\;:\; {\cal O}_Y\;\longrightarrow\;  \widehat{\cal O}_X^{A\!z})$}
 $$
 from the equivalence class of gluing systems of local function rings on $Y$ to that on $X^{\!A\!z}$
 (resp.\ $\widehat{X}^{\!A\!z}$) in each context
 in order to match basic Higgsing-un-Higgsing physical behaviors of D-branes,
 ([L-Y1] (D(1)), [L-Y5] (D(11.3.1)), [L-Y9] (D(11.4.1))).
There is something still missing in this picture in the case of dynamical fermionic D-branes:
 \begin{itemize}
  \item[\LARGE $\cdot$] {\it
   To construct a supersymmetric theory for a fermionic D-brane moving in a space-time
   $$
     \widehat{\varphi}\;:\; (\widehat{X}^{\!A\!z}\!, \widehat{\cal E}; \widehat{\nabla})\;
       \longrightarrow\; Y\,,
   $$
   one needs to impose some constraints on
     the connection $\widehat{\nabla}$ on the fundamental module $\widehat{\cal E}$
	    of $\widehat{X}^{\!A\!z}$ and
     $\widehat{\varphi}^\sharp$ that defines the morphism $\widehat{\varphi}$
    so that the true physical degrees of freedom (including both the dynamical and the non-dynamical ones)
	 in the problem match with those dictated
	by the related representations of the supersymmetry algebra involved,
    i.e.\ `{\sl SUSY-Rep Compatible Conditions}' on $(\widehat{\nabla}, \widehat{\varphi}^\sharp)$.}
 \end{itemize}
([L-Y9: Sec.\ 3.1] (D(11.4.1))).
In this work, we begin the task to address such issues.
As the supersymmetry algebras in different dimensions behave slightly differently, to begin,
 we focus on the theory for dynamical fermionic D3-branes,
 in which the $d=4$, $N=1$ supersymmetry algebra is involved.
 
As the necessary background to construct from the aspect of Grothendieck's Algebraic Geometry
    dynamical fermionic D3-branes along the line of Ramond-Neveu-Schwarz superstrings in string theory,
 three pieces of the building blocks are given in the current notes:
  \begin{itemize}
   \item[(1)]
    basic $C^\infty$-algebrogeometric foundations of
               $d=4$, $N=1$ supersymmetry and $d=4$, $N=1$ superspace in physics,
           with emphasis on the partial $C^\infty$-ring structure on the function ring of the superspace,
		   (Sec.\ 1);
	
   \item[(2)]	
    the notion of SUSY-rep compatible hybrid connections on bundles over the superspace
    to address connections on the Chan-Paton bundle on the world-volume of a fermionic D3-brane,
	(Sec.\ 2);
	
  \item[(3)]	
   the notion of $\widehat{D}$-chiral maps $\widehat{\varphi}$
     from a $d=4$ $N=1$ Azumaya/matrix superspace with a fundamental module
	   with a SUSY-rep compatible hybrid connection $\widehat{\nabla}$
	 to a complex manifold $Y$
	as a model for a dynamical D3-branes moving in a target space(-time),
 (Sec.\ 3 and  Sec.\ 4).	
 \end{itemize}
Some test computations related to the construction of a supersymmetric action functional for
  SUSY-rep compatible $(\widehat{\nabla}, \widehat{\varphi})$
  are given in the end (Sec.\ 5), whose further study is the focus of a separate work.
The current work is a sequel to [L-Y9] (D(11.4.1)) and [L-Y4] (D(11.2))
  and is the first step in the supersymmetric generalization, in the case of D3-branes,
    of the standard action functional for D-branes constructed in [L-Y8] (D(13.3)).

\bigskip

\noindent
{\it Remark on footnotes}\;\;
Various footnotes are added to illuminate further the main text without causing distractions.
They can be skipped at the first reading.
Footnotes that are familiar to physicists but unfamiliar to mathematicians are started with a
  `{\it Note for Mathematicians}'.
Similarly, footnotes that are familiar to mathematicians but unfamiliar to physicists are started with a
 `{\it Note for Physicists}'.

\bigskip

\noindent
{\it Remark on sign-related elementary details}\;\;
 Many computations in the work involve {\it sign-factors} $(-1)^{\mbox{\tiny $\bullet$}}$
   or {\it parity conjugations}
  that arise from passing (cohomological degree, parity) bi-graded objects and the conventions chosen.
 Their details are provided whenever appropriate not because they are difficult
    but, rather, because they are so simple yet tedious that they become a source of errors.
 More often than not, such sign-factors or parity-conjugation influence the result and its applications very much
   and hence deserve a special effort to record them accurately.

\bigskip
\bigskip

\noindent
{\bf Convention.}
 References for standard notations, terminology, operations and facts are\footnote{{\it Note
                                                                              for mathematicians}\;
																		     There are different sets of standard conventions and notations
																			   in the physicists' superworld.
																		     One first has to be familiar with all of them and
																			   then choose or make her/his own to proceed.																			
																			 All the grading-or-supersymmetry-related conventions or notations
																			   used in the current work are stated explicitly along our way
																			   to avoid confusions.
                                                                             The following six works influence our setup very strongly:
																			   [D-F2]; [Man], [S-W]; [We-B], [G-G-R-S], and [West1].
																			  }\\    
  (1) differential geometry: [Hi], [K-N];\;\; \\
  (2) algebraic geometry: [Har];\; $C^{\infty}$-algebraic geometry: [Joy];\;\; \\
  (3) graded bundles and supermanifolds: [Man], [S-W]; also [CB], [C-C-F], [DeW];\;\;  \\
  (4) superstring theory: [Gr-S-W], [Po2], [B-B-S];\;  D-branes: [Po1], [Po2]; also [Joh], [Sz2];\;\; \\
  (5) supersymmetry (mathematical aspect):  [D-F1], [D-F2], [D-M], [Freed];\;\;   \\
  (6) supersymmetry (physical aspect, especially $d=4$, $N=1$ case): [We-B], [G-G-R-S], [West1];\\
             \mbox{\hspace{1.6em}}also [Ar], [Bi], [Freund], [Gi], [St], [Te], [Wei].
 \begin{itemize}
  \item[$\cdot$]
   For clarity, the {\it real line} as a real $1$-dimensional manifold is denoted by ${\Bbb R}^1$,
    while the {\it field of real numbers} is denoted by ${\Bbb R}$.
   Similarly, the {\it complex line} as a complex $1$-dimensional manifold is denoted by ${\Bbb C}^1$,
    while the {\it field of complex numbers} is denoted by ${\Bbb C}$.
	
  \item[$\cdot$]	
  The inclusion `${\Bbb R}\subset{\Bbb C}$' is referred to the {\it field extension
   of ${\Bbb R}$ to ${\Bbb C}$} by adding $\sqrt{-1}$, unless otherwise noted.

   
  \item[$\cdot$]
   $\widehat{\Bbb R}$\,:\;
      the ${\Bbb Z}/2$-graded ${\Bbb R}$-algebra of real Grassmann numbers in question;\\	
   $\widehat{\Bbb C}$\,:\;
      the ${\Bbb Z}/2$-graded ${\Bbb C}$-algebra of complex Grassmann numbers in question.
  
  \item[$\cdot$]
  All manifolds are paracompact, Hausdorff, and admitting a (locally finite) partition of unity.
  We adopt the {\it index convention for tensors} from differential geometry.
   In particular, the tuple coordinate functions on an $n$-manifold is denoted by, for example,
   $(y^1,\,\cdots\,y^n)$.
  However, no up-low index summation convention is used.
   
  
  \item[$\cdot$]
  For this note, `{\it differentiable}', `{\it smooth}', and $C^{\infty}$ are taken as synonyms.
  
  %
  %
  %
  %
  %
  %
  %
  %
  %
  
  \item[$\cdot$]
   group {\it action} vs.\  {\it action} functional for D-branes.
   
  \item[$\cdot$]
   {\it dimensions} $d=4$ vs.\ the {\it exterior differential operator} $d$.

  %
  %
  %
 
  \item[$\cdot$]
   For a sheaf ${\cal F}$ on a topological space $X$,
   the notation `$s\in{\cal F}$' means a local section $s\in {\cal F}(U)$
      for some open set $U\subset X$.

  \item[$\cdot$]	
   For an ${\cal O}_X$-module ${\cal F}$,
    the {\it fiber} of ${\cal F}$ at $x\in X$	 is denoted ${\cal F}|_x$
	while the {\it stalk} of ${\cal F}$ at $x\in X$ is denoted ${\cal F}_x$.
  
  \item[$\cdot$]
   {\it coordinate-function index}, e.g.\ $(y^1,\,\cdots\,,\, y^n)$ for a real manifold
      vs.\  the {\it exponent of a power},
	  e.g.\  $a_0y^r+a_1y^{r-1}+\,\cdots\,+a_{r-1}y+a_r\in {\Bbb R}[y]$.
	
 


  
  \item[$\cdot$]
   the {\it collective fermionic coordinate-functions}
    $(\theta,\bar{\theta})
	 := (\theta^1, \theta^2, \bar{\theta}^{\dot{1}}, \bar{\theta}^{\dot{2}})$\\
   vs.\ the {\it ideal}
    $\widehat{\boldm}:= (\theta^1,\theta^2,\bar{\theta}^{\dot{1}},\bar{\theta}^{\dot{2}})$
	generated by $\theta^1, \theta^2, \bar{\theta}^{\dot{1}}, \bar{\theta}^{\dot{2}}$.
  
  \item[$\cdot$]
   {\it derivations} $\xi,\,\eta$
     vs.\ {\it sections} $(\xi,\bar{\eta})$ of the Dirac spinor bundle $S^\prime\oplus S^{\prime\prime}$.
  
  \item[$\cdot$]
   {\it principal bundle} $P$ vs.\ {\it operator} $P$.
  
  \item[$\cdot$]
   {\it Various brackets}\,:\;\;
     $[A,B]:= AB-BA$,\; $\{A,B\}:= AB+BA$,\;\\
	 $[A,B\}:= AB-(-1)^{p(A)p(B)}BA$,
	   where $p(\,\mbox{\tiny $\bullet$}\,)$ is the parity of $\mbox{\tiny $\,\bullet\,$}$\,.
 
  \item[$\cdot$]
   The switching between a {\it (locally free) sheaf of modules} and a {\it vector bundle}
     is taken freely, whichever is conceptually or notationwise more convenient.
  
  \item[$\cdot$]
   Convention~1.3.5 on passings of
                          ({\it cohomological degree}, {\it parity}) {\it bi-graded objects};\\
   Convention~2.1.5 on
        {\it $\Endsheaf_{\widehat{\cal O}_X}\!(\widehat{\cal E})$-valued $1$-forms}
		  on the $d=4$, $N=1$ superspace $\widehat{X}$;\\
   Convention~2.2.1 on {\it left operators, right operators, and their compositions}.		
  
  \item[$\cdot$]
   The current Note D(14.1) continues the study in
	  {\small
	  \begin{itemize}	
	   \item[]  \hspace{-2em} [L-Y4]\hspace{1em}\parbox[t]{37em}{{\it
	    D-branes and Azumaya/matrix noncommutative differential geometry,
         II: Azumaya/ matrix supermanifolds and differentiable maps therefrom
          - with a view toward dynamical fermionic D-branes in string theory},
          arXiv:1412.0771 [hep-th]. (D(11.2))	
		 }\\[1ex] 
		
       \item[]  \hspace{-2em} [L-Y9]\hspace{1em}\parbox[t]{37em}{{\it
	    Further studies of the notion of differentiable maps from Azumaya/matrix supermanifolds,
          I. The smooth case: Ramond-Neveu-Schwarz and Green-Schwarz meeting Grothendieck},
         arXiv:1709.08927 [math.DG]. (D(11.4.1))	
		 }\\[1ex] 

       \item[]  \hspace{-2em} [L-Y8]\hspace{1em}\parbox[t]{37em}{{\it
	    Dynamics of D-branes II. The standard action
         - an analogue of the Polyakov action for (fundamental, stacked) D-branes},
         arXiv:1704.03237 [hep-th]. (D(13.3))			  	
		 }
	  \end{itemize}}Other  	
   notations and conventions follow ibidem when applicable.
 \end{itemize}

\newpage
   
\begin{flushleft}
{\bf Outline}
\end{flushleft}
\nopagebreak
{\small
\baselineskip 12pt  
\begin{itemize}
 \item[0.]
  Introduction
  
 \item[1.]
  The $d=4$, $N=1$ superspace $\widehat{X}$ from the aspect of super $C^{\infty}$-Algebraic Geometry
   \vspace{-.6ex}
   \begin{itemize}
	 \item[1.1]
      A word on $C^{\infty}$-Algebraic Geometry and super $C^{\infty}$-Algebraic Geometry

     \item[1.2]
      The $d = 4$, $N = 1$ superspace as a super $C^{\infty}$-scheme with complexification
			
	 \item[1.3]	
	  Calculus on the $d=4$, $N=1$ superspace $\widehat{X}$
		
	 \item[1.4]
	  Supersymmetry transformations of and related objects on $\widehat{X}$			
   \end{itemize}
 		
 \item[2.]
  $d=4$, $N=1$ Azumaya/matrix superspaces $\widehat{X}^{\!A\!z}$
   with a fundamental module with a connection
   \vspace{-.6ex}
   \begin{itemize}     		
	 \item[2.1]
	  Lessons from left connections on the Chan-Paton bundle $\widehat{E}$ over $\widehat{X}$
		
	 \item[2.2]
	 Hybrid connections on the Chan-Paton bundle $\widehat{E}$ over $\widehat{X}$
	
	 \item[2.3]
	 Simple hybrid connections on $\widehat{\cal E}$
      that ring with the vector representation of the $d=4$, $N=1$\\ supersymmetry
	
	 \item[2.4]
	 $d=4$, $N=1$ Azumaya/matrix superspaces $\widehat{X}^{\!A\!z}$
   \end{itemize}
	
  \item[3.]
   The $\widehat{D}$-chiral and the $\widehat{D}$-antichiral structure sheaf
    of $\widehat{X}^{\!A\!z}$
   \vspace{-.6ex}
   \begin{itemize}     		
     \item[3.1]	
	 The $\widehat{D}$-chiral structure sheaf and the $\widehat{D}$-antichiral structure sheaf
     of $\widehat{X}^{\!A\!z}$
	
	 \item[3.2]
	 Normal form of $\widehat{D}$-chiral sections and $\widehat{D}$-antichiral sections
     of $\widehat{\cal O}_X^{A\!z}$
   \end{itemize}
   
 \item[4.]
  $\widehat{D}$-chiral map from $(\widehat{X}^{\!A\!z}, \widehat{\cal E}; \widehat{\nabla})$
  to a complex manifold
   \vspace{-.6ex}
   \begin{itemize}     		
     \item[\LARGE $\cdot$]	
	  Step 1: Smooth map from $(\widehat{X}^{\!A\!z}, \widehat{\cal E})$  to a real manifold
	
     \item[\LARGE $\cdot$]	
	  Step 2: Smooth map from $(\widehat{X}^{\!A\!z}, \widehat{\cal E})$  to a complex manifold
		
     \item[\Large $\cdot$]	
	  Step 3: $\widehat{D}$-chiral/$\widehat{D}$-antichiral map
	  from $(\widehat{X}^{\!A\!z}, \widehat{\cal E}; \widehat{\nabla})$ to a complex manifold
   \end{itemize} 	
  
 \item[5]
   $\widehat{D}$-chiral maps from
   $(\widehat{X}^{\!A\!z},\widehat{\cal E}; \widehat{\nabla})$ to a K\"{a}hler manifold
   and the $N=1$ Super D3-Brane Theory\\ in Ramond-Neveu-Schwarz formulation
   \vspace{-.6ex}
   \begin{itemize}     		
     \item[\LARGE $\cdot$]
	  Fermionic D3-branes and $\widehat{D}$-chiral maps $\widehat{\varphi}$ from
	   $(\widehat{X}^{\!A\!z},\widehat{\cal E}; \widehat{\nabla})$ to a K\"{a}hler manifold
	
	 \item[\LARGE $\cdot$]
	  The standard supersymmetry-invariant action functional for $\widehat{\nabla}$
	
     \item[\LARGE $\cdot$]	 	
	  Given $\widehat{\nabla}$, the standard supersymmetry-invariant action functional for $\widehat{\varphi}$:\\
	  Zumino meeting Polchinski \& Grothendieck
	
	 \item[\LARGE $\cdot$]
	  (Fundamental) $N=1$  Super (Stacked) D3-Brane Theory in the RNS formulation
   \end{itemize}
      
  \item[] \hspace{-1.3em}
   Appendix\;\; Basic moves for the multiplication of two superfields
   
 %
 %
 %
\end{itemize}
} 

\newpage
	
\section{The $d=4$, $N=1$ superspace $\widehat{X}$
         from the aspect of super $C^{\infty}$-Algebraic Geometry}

\subsection{A word
        on $C^{\infty}$-Algebraic Geometry and super $C^{\infty}$-Algebraic Geometry}
					
$C^{\infty}$-Algebraic Geometry (resp.\ super $C^{\infty}$-Algebraic Geometry)
  is the study of smooth manifolds (resp.\ superspaces, originated from physicists' study of supersymmetry),
  sheaves thereover, and morphisms in-between
  all from the aspect of Grothendieck's modern Algebraic Geometry.\footnote{These
                            mild words are meant only to reduce the psychological barrier to access
				            $C^{\infty}$-Algebraic Geometry and {\it Super $C^{\infty}$-Algebraic Geometry}.
						 There is unfortunately no short-cut here.	
						 Readers are referred to [Joy]
						    for a review, upgrade and references of $C^{\infty}$-Algebraic Geometry  and
						 [L-Y9: Sec.\ 1.3 \& Sec.\ 1.4] (D(11.4.1) for a natural super extension
						   of a minimal set of basic notions in $C^{\infty}$-Algebraic Geometry
						   that are used in the current notes.
						 They are all motivated by one belief (1) and two considerations (2) \& (3):
						    \begin{itemize}
							 \item[(1)]
							  There should be a {\it final mathematical language} on Super Geometry
							     that follows Grothendieck's construction of modern Algebraic Geometry.
                              Furthermore, since a super smooth manifold is meant to be a generalization a smooth manifold
							    and the function-ring of the latter is a $C^{\infty}$-ring,
							    a notion of super $C^{\infty}$-ring, related modules, and morphisms
								 should be the starting building blocks of this final language.
						
						     \item[(2)]
							  This final language should already take physicists' setting of {\it supersymmetry} into account.
							  In particular, most materials in the standard textbooks (e.g.\ [We-B]; also [G-G-R-S], [West1])
							    should be all rephrasable immediately in this final language.
								
                             \item[(3)]							
							  It should has a feedback to physics, where its origin lies. 							  
						    \end{itemize}
						 This ``final mathematical language" is not yet in existence in a completed format as in Algebraic Geometry
						   but several algebraic-geometry-based mathematical studies of superspace have been around
						    (e.g.\ [Man] (1988), [S-W] (1989),  [C-C-F] (2011)).
						Besides the algebrogeometric setting of
						   a {\it superspace/supermanifold as a ${\Bbb Z}/2$-graded-locally-ringed topological space,
						        i.e.\ superscheme} and
						   the {\it role of spinors to odd functions on the superspace/supermanifold},
						 the setting in [L-Y4] (D(11.2)),  [L-Y9] (D(11.4.1)), and continued in the current notes,
						   is guided further by the requirement that it needs to answer the following three questions:
						   \begin{itemize}
						     \item[\bf Q1]
							 {\it Can it be used to easily describe superspaces and supersymmetry in
							    [We-B], [G-G-R-S], [West1]?}
							
						     \item[\bf Q2]
							 {\it Can it be used to easily describe fermionic strings as in [Gr-S-W], [Po2]?}
							
							 \item[\bf Q3]
							   {\it Can it be used to easily describe fundamental fermionic D-branes in a way
							            that takes key physical features of D-branes into account? }						
						   \end{itemize}
                         For that, one has to be able to address in particular
						    the notion of {\it smooth map from a superspace/supermanifold to a smooth manifold}.
						 Since the function ring of a smooth manifold is a $C^\infty$-ring,
						   this naturally leads us to bring out the notion of
						  {\it partial $C^\infty$-ring structure} and the {\it $C^\infty$-hull}
						   of the ${\Bbb Z}/2$-graded function ring of a superspace.						   
                         Such notion/structure can then be generalized to Azumaya/matrix supermanifold and
						   becomes part of the structure on the function ring of
						    a fermionic D-brane or the world-volume of a fermionic D-brane.
						     }
                        %
A beginning central object/notion is the notion of a {\it $C^{\infty}$-ring}
  ([Joy] and more references therein).
This is a ring $R$ that admits operations on its elements more than just the two binary ones,
  addition $+$ and multiplication $\cdot$\,:
  \begin{itemize}
   \item[\LARGE $\cdot$]
   For any $f\in C^{\infty}({\Bbb R}^k)$,
     $f(r_1,\,\cdots\,,\, r_k)\in R$ is defined  for all $r_1,\,\,\cdots\,,\, r_k\in R$.
  \end{itemize}
And this collection of {\it $C^{\infty}$-operations} has to satisfy
  a few compatibility conditions and normalization conditions.
This is the additional structure that characterizes the function-ring $C^{\infty}(X)$ of a smooth manifold $X$
 than just a ring $(R,+,\cdot)$ in Algebra (e.g.\ [Ja]).
  
Since the binary operation $+$ is the operation associated to the $f\in C^{\infty}({\Bbb R^2})$,
  $f(x^1,x^2)= x^1+x^2$, which is the same as $x^2+x^1$,
 a $C^{\infty}$-ring is always commutative.
Thus,
 when one wants to generalize such a notion to the super case,
 one has to decide what to choose in a noncommutative ring to make a $C^{\infty}$-ring.
For a {\it superring}
  (i.e.\ a ring $R$ that is equipped with a ${\Bbb Z}/2$-grading $R=R_{\even} \oplus R_{\odd}$)
 that is ${\Bbb Z}$-commutative,
 the choice is immediate: Its {\it even part} $R_{\even}$ is a commutative ring,
  and one should require $R_{\even}$ be a $C^{\infty}$-ring.

\bigskip

\begin{definition} {\bf [super $C^{\infty}$-ring, $C^\infty$-hull, partial $C^\infty$-ring structure]}\;
{\rm
 Let\\
   $R=R_{\even}\oplus R_{\odd}$ be a ${\Bbb Z}/2$-graded, ${\Bbb Z}/2$-commutative ring.
 We call $R$ a {\it super $C^{\infty}$-ring}
  if its even part $R_{\even}$ is endowed with a $C^{\infty}$-ring structure.
 $R_{\even}$ with the $C^\infty$-ring structure  is then called the {\it $C^\infty$-hull} of $R$.
 We say also that $R$ is equipped with a {\it partial $C^\infty$-ring structure}.
}\end{definition}

\medskip

\begin{example} {\bf  [function-ring of super real line]}\;
{\rm
 (Cf.\ $d=1$, $N=2$ superspace.)\;
 The function-ring $C^{\infty}({\Bbb R}^{1|2})$
   of the $1$-dimensional {\it super real line} ${\Bbb R}^{1|2}$
   with two fermionic generators $\theta^1,\theta^2$
  is the ${\Bbb Z}/2$-graded, ${\Bbb Z}/2$-commutative {\it superpolynomial ring}
   $$
    C^{\infty}({\Bbb R}^1)[\theta^1, \theta^2]^{\anticommuting}\;
	 :=\;   \frac{C^{\infty}({\Bbb R}^1)\langle \theta^1,\,\theta^2 \rangle}
	                   {( f\theta^{\mu}-\theta^{\mu}f\,,\;
			                    \theta^{\mu}\theta^{\nu}+\theta^{\nu}\theta^{\mu}\,
				                |\; f \in C^{\infty}({\Bbb R}^1)\,;\;    \mu, \nu =1,\, 2)}
   $$
   with coefficients in $C^{\infty}({\Bbb R}^1)$.
 (Cf.\ [L-Y9: Example/Definition 1.3.2] (D(11.4.1)).)
 Its even part is given by
  the polynomial rings in $\theta^1\theta^2$ with coefficients in $C^{\infty}({\Bbb R}^1)$
   $$
     C^{\infty}({\Bbb R}^{1|2})_{\even}\;
	  =\;  C^{\infty}({\Bbb R}^1)[\theta^1\theta^2]\;
	         =\; \{f+g\theta^1\theta^2| f, g \in C^{\infty}({\Bbb R}^1)\}\,,
   $$
 whose $C^{\infty}$-ring structure is given by
   $$
     h(f_1+g_1\theta^1\theta^2, \,\cdots\,,\, f_k+g_k\theta^1\theta^2)\;
	  =\;  h(f_1, \,\cdots\,,\, f_k)\,
	         +\,  \sum_{i=1}^k   (\partial_i h)(f_1, \,\cdots\,,\, f_k)(g_i\theta^1\theta^2)\,,			 
   $$
  for any
    $f_1+g_1\theta^1\theta^2, \,\cdots\,,\, f_k+g_k\theta^1\theta^2
	   \in C^{\infty}({\Bbb R}^1)[\theta^1, \theta^2]^{\anticommuting}$,
    $h:{\Bbb R}^k\rightarrow {\Bbb R}$ smooth,    and
	$k\in {\Bbb Z}_{\ge 1}$.
 Here, $\partial_ih$ is the partial derivative of $h$ with respect to its $i$-th argument.	
}\end{example}

\bigskip
 
With these mild words and an example as an introduction,
readers are referred to the preparatory notes [L-Y9: Sec.\ 1.3 \& Sec.\ 1.4] (D(11.4.1))
 for the basic notions, objects, and terminologies in super $C^{\infty}$-Algebraic Geometry
 that will be freely used in the current notes.
Two  lemmas\footnote{For
                                 the case of superrings $\widehat{R}$
								  constructed from an extension of an ordinary ring $R$ by fermionic variables,
                                 it is also natural to just require that $R$ be a $C^{\infty}$-ring.
                                 Lemma~1.1.3 and Lemma ~1.1.4
								  say that actually the $C^{\infty}$-ring structure on $R$
                                   extends to a $C^{\infty}$-ring structure on $\widehat{R}_{\,\mbox{\tiny even}}$
								     for free
								   and that the extension is canonical.
                                          }
	from ibidem are quoted below;  they will play a key role in our construction later:

\bigskip

\begin{lemma} {\bf [$C^{\infty}$ evaluation after nilpotent perturbation]}$\;$
 {\rm [L-Y9: Lemma 1.2.1] (D(11.4.1).)}
 Given a $C^{\infty}$-ring $R$, let
   $r_1,\,\cdots\,,\, r_k \in R$   and
   $n_1,\,\cdots\,,\, n_k$ be nilpotent elements in $R$
      with $n_1^{l+1}=\,\cdots\,= n_k^{l+1}=0$.
 Then,
      for any $h\in C^{\infty}({\Bbb R}^k)$,
   the element $h(r_1+n_1,\,\cdots\,,\,r_k+n_k)\in R$
      from the $C^{\infty}$-ring structure of $R$ is given explicitly by
   $$
     h(r_1+n_1,\,\cdots\,,\,r_k+n_k)\;
	   =\;
	    \sum_{d=0}^{kl}\, \frac{1}{d_1!\cdots d_k!}\,
	       \sum_{d_1+\,\cdots\,+d_k=d}
	       (\partial_1^{\,d_1}\,\cdots\,\partial_k^{\,d_k}  h)(r_1,\,\cdots\,,\, r_k)
		     \cdot n_1^{d_1}\,\cdots\,n_k^{d_k}\,,	
   $$
   where
     $\partial_1^{\,d_1}\,\cdots\,\partial_n^{\,d_n}  h \in C^{\infty}({\Bbb R}^k)$
      is the partial derivative of $h$ with respect to the first variable $d_1$-times, the second variable $d_2$-times,
	  ..., and the $k$-th variable $d_k$-times.
\end{lemma}

\medskip

\begin{lemma} {\bf [extension of $C^{\infty}$-ring structure]}$\;$
 {\rm ([L-Y9: Lemma 1.2.2] (D(11.4.1).))}
 Let
  $R$ be a $C^{\infty}$-ring and
  $S= R\oplus N$ be a commutative ${\Bbb R}$-algebra
     with $N^{l+1}=0$ for some $l\in {\Bbb Z}_{\ge 1}$.
 Then, $S$ admits a unique $C^{\infty}$-ring structure
   such that
    both the built-in ring-monomorphism $R\hookrightarrow S$ and the built-in ring-epimorphism $S\rightaarrow R$
    are $C^{\infty}$-ring-homomorphisms.
\end{lemma}

\bigskip

\subsection{The $d=4$, $N=1$ superspace as a super $C^{\infty}$-scheme with complexification}

Physicists' ``superspace" in the study of Supersymmetry \& Supergravity is meant to be
  a ``space" with not only the ordinary commuting coordinates as for the charts of an ordinary manifold
      but also anticommuting ``fermionic coordinates" so that a supersymmetry can act on this space.
Mathematically in terms of Algebraic Geometry,
 ``coordinates" are nothing but the generating elements of the function-ring of that space.
Before the polishment by further details,
 this gives one the first reason why a physicists' superspace ``should" be naturally described as
 a super $C^{\infty}$-scheme in Super $C^{\infty}$ Algebraic Geometry.
We explain in this subsubsection in four steps
  how this is realized and how a spinor bundle comes into play and influences the $C^{\infty}$-scheme structure
 for the case $d=4$, $N=1$ superspace.

\bigskip

\noindent
$(a)$\;{\it Basic setup: The $4$-dimensional Minkowski space-time $X$ and spinor bundles thereupon}

\medskip

\noindent
Let $X$ be the {\it Minkowski space-time} ${\Bbb R}^{3+1}$, with global coordinates $(x^0,x^1,x^2,x^3)$,
  the Minkowski metric $ds^2= -(dx^0)^2+(dx^1)^2+ (dx^2)^2 + (dx^3)^2$,  and
  the orientation given by $dx^0\wedge dx^1\wedge dx^2\wedge dx^3$.	
Let $C^{\infty}(X)$ be the ring of smooth functions on $X$, it is a $C^{\infty}$-ring.
The corresponding structure sheaf on $X$ is denoted by ${\cal O}_X$;	
 the locally-ringed space $(X,{\cal O}_X)$ is a $C^{\infty}$-scheme, denoted also in the short-hand $X$.

Proper orthochronous frames on $X$ gives rise to the principal $\SO^{\uparrow}(3,1)$-bundle $P$ over $X$,
 where $\SO^{\uparrow}(3,1)$ is the connected component of the Lorentz group $O(3,1)$
 that contains the identity element.
$P$ is canonically trivialized by the flat Levi-Civita connection from the Minkowski metric $ds^2$ on $X$.
Associated to the irreducible real spinor representation of $\SO^{\uparrow}(3,1)$
    and the two irreducible complex spinor representations
     $(\frac{1}{2},0)$ and $(0,\frac{1}{2})$ of the analytic complexification
	   $\SO^{\uparrow}(3,1)^{\Bbb C}$ of $\SO^{\uparrow}(3,1)$
	 are the real spinor bundle $S$ and two complex spinor bundles $S^{\prime}$ and $S^{\prime\prime}$ on $X$.
The former has real rank $4$
      while the latter have complex rank $2$ each and are complex conjugate to each other.	
The three are related by
     $$
	    S^{\,\Bbb C}\,:=\, S\otimes_{\Bbb R}{\Bbb C}\;
		 \simeq\; S^{\prime}\oplus S^{\prime\prime}\,.
	 $$
Sections of $S$ (resp.\ $S^{\prime}\oplus S^{\prime\prime}$, $S^{\prime}$, $S^{\prime\prime}$)
     are called {\it Majorana spinors}
	 (resp.\ {\it Dirac spinors}, {\it chiral Weyl spinors}, {\it antichiral Weyl spinors})  on $X$.
As associated bundles to $P$
       or its fiberwise-complexified $\SO^{\uparrow}(1,3)^{\Bbb C}$-bundle $P^{\,\Bbb C}$,
    the induced flat connection on $S$, $S^{\prime}$, $S^{\prime\prime}$
	  gives a built-in trivialization of these spinor bundles.
   The term `{\it constant sections}' of any of these spinor bundles are referred to
      `sections that are constant with respect to this trivialization'.
They are the covariantly constant sections with respect to the induced connection from that on $P$ or $P^{\,\Bbb C}$.
    
By the identification of the real representation underlying $(\frac{1}{2},0)$ with the real spinor representation,
     one can fix a quadruple of constant generating sections
	   $(\vartheta^1,\vartheta^2, \vartheta^3,\vartheta^4)$ of $S$
    and a pair of constant generating sections
      $\theta^1$, $\theta^2$
	  (resp.\ $\bar{\theta}^{\dot{1}}$, $\bar{\theta}^{\dot{2}}$)
	  of $S^{\prime}$ (resp.\ $S^{\prime\prime}$) such that
     $$
	   \theta^1\;=\; \vartheta^1+\sqrt{-1}\vartheta^2\,,\hspace{1em}
	   \theta^2\;=\; \vartheta^3+\sqrt{-1}\vartheta^4\,,\hspace{1em}
	   \bar{\theta}^{\dot{1}}\;=\; \vartheta^1-\sqrt{-1}\vartheta^2\,,\hspace{1em}
	   \bar{\theta}^{\dot{2}}\;=\; \vartheta^3-\sqrt{-1}\vartheta^4\,.
	 $$
	 under the isomorphism $S^{\,\Bbb C}\simeq S^{\prime}\oplus S^{\prime\prime}$.
Note that the above relation of constant generating sections of spinor bundles
  is a realization of the statement that, in four dimensions,
	a Majorana spinor is a Dirac spinor whose antichiral component is the complex conjugate
	of its chiral component;
		(e.g.\ [G-G-R-S: Sec.\ 3.1.a]).

\bigskip

\noindent
$(b)$\;{\it From a spinor bundle to its associated super $C^{\infty}$-scheme structure on $X$}

\medskip

\noindent
Each of the spinor bundles $S$, $S^{\prime}$, $S^{\prime\prime}$, $S^{\prime}\oplus S^{\prime}$
 defines a super $C^{\infty}$-scheme structure on the $4$-dimensional Minkowski space-time $X$
 as follows; cf.\ [L-Y9: Example 1.4.5] (D(11.4.1)).\footnote{Note
                                       that the convention here is slightly different from that of
                                         [L-Y9: Example 1.4.5] (D(11.4.1)),
										   where it is the dual bundle $S^{\vee}$ of a spinor bundle $S$
										    that is used in the construction.
                                         Indeed, algebrogeometrically it is more natural and functorially correct
										    to think of a fermionic coordinate
											as a map from the spinor bundle $S$ and hence a section of the dual bundle $S^{\vee}$,
                                            rather than the bundle $S$ itself.	
										 (Fortunately, as a module of the principal {\it SO}$^{\,\uparrow}(3,1)$-bundle $P$, 										
										      the two different choices of conventions $S$ vs.\ $S^{\vee}$ keep the same
										      chirality $(\frac{1}{2},0)$ or $(0,\frac{1}{2})$
										      when the spinor bundle involved is $S^{\prime}$ or $S^{\prime\prime}$.)
                                         However, for the purpose of the current notes it is the $C^{\infty}$-ring structure
     										that matters in the end.
										 We thus adopt the physicists' convention here,
										   which directly takes a section of a spinor bundle as giving a fermionic coordinate,
										 to avoid
										    the unnecessary distraction
										      from the burden of notations with $(\,\cdot\,)^{\vee}$ everywhere.
										 A setup, construction, or statement using one convention can always be converted into one
										  using the other convention.}
 
\bigskip

\noindent
$(b.1)$\; {\it The super $C^{\infty}$-scheme structure on $X$ associated to
           the Majorana spinor bundle $S$}

\medskip

\noindent
In this case, local sections of the {\it Grassmann-algebra bundle associated to $S$}
  $$
    \mbox{$\bigwedge$}_{\,\Bbb R}^{\bullet}S\;
	 :=\;  \mbox{$\bigoplus_{l=0}^4\,\bigwedge$}_{\,\Bbb R}^lS
	 \hspace{2em}
	 \mbox{(with $\bigwedge_{\,\Bbb R}^0S:=$
	                     the constant ${\Bbb R}$-line bundle $\underline{\Bbb R}$ over $X$)}
  $$
  over $X$
 defines a sheaf $^{\Bbb R}\widehat{\cal O}_X$ of
  ${\Bbb Z}/2$-graded, ${\Bbb Z}/2$-commutative ${\cal O}_X$-algebras on $X$,
  with
   $$
      ^{\Bbb R}\widehat{\cal O}_X(U)\;
	   =\;  C^{\infty}( \mbox{$\bigwedge$}_{\,\Bbb R}^{\bullet}S|_U)
   $$
    for an open set $U\subset X$   and
   the ${\Bbb Z}/2$-grading given by
   $$
     \begin{array}{cl}
	 & ^{\Bbb R}\widehat{\cal O}_{X,\,\even}\;\;
         :=\;\; \mbox{sheaf of sections of
				   $\;\;\bigwedge_{\,\Bbb R}^{\even}\!S\,
				           :=\,  \underline{\Bbb R}\oplus \bigwedge_{\,\Bbb R}^2 S
						           \oplus  \bigwedge_{\,\Bbb R}^4 S\,$}
           \hspace{3em}			   \\[1.2ex]
	 \mbox{and}\hspace{2em}	
	 & ^{\Bbb R}\widehat{\cal O}_{X,\,\odd}\;\;\;
            :=\;\; \mbox{sheaf of sections of
			         $\;\;\bigwedge_{\,\Bbb R}^{\odd}\!S\,
					        :=\,  S \oplus \bigwedge_{\,\Bbb R}^3 S$}\,.  	
	 \end{array}
    $$
 By Lemma~1.1.3 and Lemma~1.1.4,
  the $C^{\infty}$-ring structure on $C^{\infty}(U)$ extends uniquely and canonically
   to a $C^{\infty}$-ring structure on
	\begin{eqnarray*}
	  C^{\infty}(\mbox{$\bigwedge_{\,\Bbb R}^{\even}\!S|_U$})
	  & = &  C^\infty(U)[\vartheta^1|_U,\, \vartheta^2|_U,\,
	                                              \vartheta^3|_U, \vartheta^4|_U]^{\anticommuting}_{\even}\\[.6ex]
	  & = & C^\infty(U)[(\vartheta^1\vartheta^2)|_U, (\vartheta^1\vartheta^3)|_U,
	                                   (\vartheta^1\vartheta^4)|_U, (\vartheta^2\vartheta^3)|_U,
							           (\vartheta^2\vartheta^4)|_U, (\vartheta^3\vartheta^4)|_U]\,.
	\end{eqnarray*}
 This renders $^{\Bbb R}\widehat{\cal O}_X$ a sheaf of super $C^{\infty}$-rings on $X$.
 
 \bigskip

\begin{definition} {\bf [super $C^{\infty}$-scheme $^{\Bbb R}\!\widehat{X}$ associated to $S$]}\;
{\rm
 We call the locally-ringed space
   $^{\Bbb R}\!\widehat{X}:=(X,\,^{\Bbb R}\!\widehat{\cal O}_X)$
    the {\it super $C^{\infty}$-scheme associated to the Majorana spinor bundle $S$} on $X$.
}\end{definition}

\bigskip

\noindent
$(b.2)$\; {\it The super $C^{\infty}$-scheme structure on $X$ associated to
          the chiral Weyl spinor bundle $S^{\prime}$}

\medskip

\noindent
In this case, local sections of the {\it real-complex-mixed Grassmann-algebra bundle associated to $S^{\prime}$}
  $$
    \mbox{$\bigwedge$}_{\,\Bbb R, \Bbb C}^{\bullet}S^{\prime}\;
	 :=\;  \underline{\Bbb R}\oplus \mbox{$\bigoplus_{l=1}^2\,\bigwedge$}_{\,\Bbb C}^lS^{\prime}
  $$
  over $X$
 defines a sheaf $\widehat{\cal O}^{\,\prime}_X$ of
  ${\Bbb Z}/2$-graded, ${\Bbb Z}/2$-commutative ${\cal O}_X$-algebras on $X$,
  with
   $$
      \widehat{\cal O}^{\,\prime}_X(U)\;
	   =\;  C^{\infty}( \mbox{$\bigwedge$}_{\,\Bbb R,\Bbb C}^{\bullet}S^{\prime}|_U)
   $$
    for an open set $U\subset X$   and
   the ${\Bbb Z}/2$-grading given by
   $$
     \begin{array}{cl}
	 & \widehat{\cal O}^{\,\prime}_{X,\,\even}\;\;
         :=\;\; \mbox{sheaf of sections of
				   $\;\;\bigwedge_{\,\Bbb R,\Bbb C}^{\even}\!S^{\prime}\,
				           :=\,  \underline{\Bbb R}\oplus \bigwedge_{\,\Bbb C}^2 S^{\prime}\,$}
           \hspace{3em}			   \\[1.2ex]
	 \mbox{and}\hspace{2em}	
	 & \widehat{\cal O}^{\,\prime}_{X,\,\odd}\;\;\;
            :=\;\; \mbox{sheaf of sections of
			         $\;\;\bigwedge_{\,\Bbb C}^{\odd}\!S^{\prime}\,
					        :=\,  S^{\prime}$}\,.  	
	 \end{array}
    $$
 By Lemma~1.1.3 and Lemma~1.1.4,
  the $C^{\infty}$-ring structure on $C^{\infty}(U)$ extends uniquely and canonically
   to a $C^{\infty}$-ring structure on
	\begin{eqnarray*}
	  C^{\infty}(\mbox{$\bigwedge_{\,\Bbb R,\Bbb C}^{\even}\!S^{\prime}|_U$})
	   & \:= &  C^{\infty}(U)^{\Bbb R, \Bbb C}[\theta^1|_U,\, \theta^2|_U]^{\anticommuting}_{\even}\\
	   & :=   &  C^{\infty}(U)
						\oplus C^{\infty}(U)^{\Bbb C}\cdot (\theta^1\theta^2)|_U\,. 						 
	\end{eqnarray*}
 This renders $\widehat{\cal O}^{\,\prime}_X$ a sheaf of super $C^{\infty}$-rings on $X$.
 
 \bigskip

\begin{definition} {\bf [super $C^{\infty}$-scheme
               $\widehat{X}^{\prime}$ associated to $S^{\prime}$]}\; {\rm
 We call the locally-ringed space
    $\widehat{X}^{\prime}:=(X,\, \widehat{\cal O}^{\,\prime}_X)$
  the {\it super $C^{\infty}$-scheme associated to the chiral Weyl spinor bundle $S^{\prime}$} on $X$.
}\end{definition}

\bigskip

\noindent
$(b.3)$\; {\it The super $C^{\infty}$-scheme structure on $X$ associated to
           the antichiral Weyl spinor bundle $S^{\prime}$}

\medskip

\noindent
In this case,
local sections of the {\it real-complex-mixed Grassmann-algebra bundle associated to $S^{\prime\prime}$}
  $$
    \mbox{$\bigwedge$}_{\,\Bbb R, \Bbb C}^{\bullet}S^{\prime\prime}\;
	 :=\;  \underline{\Bbb R}\oplus
	          \mbox{$\bigoplus_{l=1}^2\,\bigwedge$}_{\,\Bbb C}^lS^{\prime\prime}
  $$
  over $X$
 defines a sheaf $\widehat{\cal O}^{\,\prime\prime}_X$ of
  ${\Bbb Z}/2$-graded, ${\Bbb Z}/2$-commutative ${\cal O}_X$-algebras on $X$,
  with
   $$
      \widehat{\cal O}^{\,\prime\prime}_X(U)\;
	   =\;  C^{\infty}( \mbox{$\bigwedge$}_{\,\Bbb R,\Bbb C}^{\bullet}S^{\prime\prime}|_U)
   $$
    for an open set $U\subset X$   and
   the ${\Bbb Z}/2$-grading given by
   $$
     \begin{array}{cl}
	 & \widehat{\cal O}^{\,\prime\prime}_{X,\,\even}\;\;
         :=\;\; \mbox{sheaf of sections of
				   $\;\;\bigwedge_{\,\Bbb R,\Bbb C}^{\even}\!S^{\prime\prime}\,
				           :=\,  \underline{\Bbb R}\oplus \bigwedge_{\,\Bbb C}^2 S^{\prime\prime}\,$}
           \hspace{3em}			   \\[1.2ex]
	 \mbox{and}\hspace{2em}	
	 & \widehat{\cal O}^{\,\prime\prime}_{X,\,\odd}\;\;\;
            :=\;\; \mbox{sheaf of sections of
			         $\;\;\bigwedge_{\,\Bbb C}^{\odd}\!S^{\prime\prime}\,
					        :=\,  S^{\prime\prime}$}\,.  	
	 \end{array}
    $$
 By Lemma~1.1.3 and Lemma~1.1.4,
  the $C^{\infty}$-ring structure on $C^{\infty}(U)$ extends uniquely and canonically
   to a $C^{\infty}$-ring structure on
	\begin{eqnarray*}
	  C^{\infty}(\mbox{$\bigwedge_{\,\Bbb R,\Bbb C}^{\even}\!S^{\prime\prime}|_U$})
	   & \:= &  C^{\infty}(U)^{\Bbb R, \Bbb C}
	                  [\bar{\theta}^{\dot{1}}|_U,\, \bar{\theta}^{\dot{2}}|_U]^{\anticommuting}_{\even}\\
	   & :=   &  C^{\infty}(U)
						\oplus C^{\infty}(U)^{\Bbb C}
						         \cdot (\bar{\theta}^{\dot{1}}\bar{\theta}^{\dot{2}})|_U\,. 						 
	\end{eqnarray*}
 This renders $\widehat{\cal O}^{\,\prime\prime}_X$ a sheaf of super $C^{\infty}$-rings on $X$.
 
 \bigskip

\begin{definition} {\bf [super $C^{\infty}$-scheme
               $\widehat{X}^{\prime\prime}$ associated to $S^{\prime\prime}$]}\; {\rm
 We call the locally-ringed space
     $\widehat{X}^{\prime\prime}:=(X,\, \widehat{\cal O}^{\,\prime\prime}_X)$
  the {\it super $C^{\infty}$-scheme
    associated to the antichiral Weyl spinor bundle $S^{\prime\prime}$} on $X$.
}\end{definition}

\bigskip	

\noindent
$(b.4)$\; {\it The super $C^{\infty}$-scheme structure on $X$ associated to
          the Dirac spinor bundle $S^{\prime}\oplus S^{\prime\prime}$}
		
\medskip

\noindent
In this case, local sections of the
{\it real-complex-mixed Grassmann-algebra bundle associated to $S^{\prime}\oplus S^{\prime\prime}$}
  $$
    \mbox{$\bigwedge$}_{\,\Bbb R, \Bbb C}^{\bullet}
	        (S^{\prime}\oplus S^{\prime\prime})\;
	 :=\;  \underline{\Bbb R}\oplus \mbox{$\bigoplus_{l=1}^4\,\bigwedge$}_{\,\Bbb C}^l
	          (S^{\prime}\oplus S^{\prime\prime})
  $$
  over $X$
 defines a sheaf $^{\Bbb C}\widehat{\cal O}_X$ of
  ${\Bbb Z}/2$-graded, ${\Bbb Z}/2$-commutative ${\cal O}_X$-algebras on $X$,
  with
   $$
      ^{\Bbb C}\widehat{\cal O}_X(U)\;
	   =\;  C^{\infty}( \mbox{$\bigwedge$}_{\,\Bbb R,\Bbb C}^{\bullet}
	           (S^{\prime}\oplus S^{\prime\prime})|_U)
   $$
    for an open set $U\subset X$   and
   the ${\Bbb Z}/2$-grading given by
   $$
     \begin{array}{cl}
	 & ^{\Bbb C}\widehat{\cal O}_{X,\,\even}\;\;
         :=\;\; \mbox{sheaf of sections of
				   $\;\;\bigwedge_{\,\Bbb R,\Bbb C}^{\even}\!(S^{\prime}\oplus S^{\prime\prime})\,
				           :=\,  \underline{\Bbb R}
						           \oplus \bigwedge_{\,\Bbb C}^2 (S^{\prime}\oplus S^{\prime\prime})
								   \oplus \bigwedge_{\,\Bbb C}^4 (S^{\prime}\oplus S^{\prime\prime})\,$}
         			   \\[1.2ex]
	 \mbox{and}
	 & ^{\Bbb C}\widehat{\cal O}_{X,\,\odd}\;\;\;
            :=\;\; \mbox{sheaf of sections of
			         $\;\;\bigwedge_{\,\Bbb C}^{\odd}\!(S^{\prime}\oplus S^{\prime\prime})\,
					        :=\,  (S^{\prime}\oplus S^{\prime\prime})
							          \oplus \bigwedge_{\,\Bbb C}^3 (S^{\prime}\oplus S^{\prime\prime})$}\,.  	
	 \end{array}
    $$
 By Lemma~1.1.3 and Lemma~1.1.4,
  the $C^{\infty}$-ring structure on $C^{\infty}(U)$ extends uniquely and canonically
   to a $C^{\infty}$-ring structure on
	\begin{eqnarray*}
	  \lefteqn{C^{\infty}(\mbox{$\bigwedge_{\,\Bbb R,\Bbb C}^{\even}\!
	                                                                (S^{\prime}\oplus S^{\prime\prime})|_U$})
	   \; =\;   C^{\infty}(U)^{\Bbb R, \Bbb C}[\theta^1|_U,\, \theta^2|_U,\,
	                          \bar{\theta}^{\dot{1}}|_U,\, \bar{\theta}^{\dot{2}}|_U]^{\anticommuting}
							                                                                                                                                 _{\even}} \\
	   &&  :=   \mbox{\small the even polynomial ring in anticommuting variables $\theta^1|_U,\, \theta^2|_U,\,
	                          \bar{\theta}^{\dot{1}}|_U,\, \bar{\theta}^{\dot{2}}|_U$ with}\\[-.6ex]
		  && \hspace{1.33em}
		      \mbox{\small the degree-$0$ coefficient in $C^{\infty}(U)$  and
							          all higher-degree coefficients in $C^{\infty}(U)^{\Bbb C}$}\,. 						
	\end{eqnarray*}
 This renders $^{\Bbb C}\widehat{\cal O}_X$ a sheaf of super $C^{\infty}$-rings on $X$.
 
\bigskip

\begin{definition} {\bf [super $C^{\infty}$-scheme $^{\Bbb C}\!\widehat{X}$
              associated to $S^{\prime}\oplus S^{\prime\prime}$]}\; {\rm
 The locally-ringed space
    $^{\Bbb C}\!\widehat{X}:= (X,\,^{\Bbb C}\widehat{\cal O}_X)$
    is called the {\it super $C^{\infty}$-scheme
    associated to the Dirac spinor bundle $S^{\prime}\oplus S^{\prime\prime}$} on $X$.
}\end{definition}

\bigskip

\noindent
$(c)$\;{\it Relation among different super $C^{\infty}$-scheme structures on $X$}

\medskip

\noindent
Since all the structure sheaves
     $^{\Bbb R}\widehat{\cal O}_X$,
     $\widehat{\cal O}^{\,\prime}_X$,
     $\widehat{\cal O}^{\,\prime\prime}_X$,
     $^{\Bbb C}\widehat{\cal O}_X$
   are extensions of ${\cal O}_X$ by nilpotent elements,
   \begin{itemize}
     \item[\LARGE $\cdot$]
     {\it the underlying topology of  the super $C^{\infty}$-schemes
         $\,^{\Bbb R}\!\widehat{X}$,
		 $\widehat{X}^{\prime}$,
         $\widehat{X}^{\prime\prime}$,
		 $^{\Bbb C}\!\widehat{X}$
       are all canonically identical to the topology of $X$, i.e.\ ${\Bbb R}^4$.}
   \end{itemize}
Note that though $\bigwedge_{\Bbb C}$ is used in the construction of
    $\widehat{\cal O}^{\,\prime}_X$,
    $\widehat{\cal O}^{\,\prime\prime}_X$,
    $^{\Bbb C}\widehat{\cal O}_X$
   that involves complex spinor bundles,
 all the four sheaves
    $^{\Bbb R}\widehat{\cal O}_X$,
    $\widehat{\cal O}^{\,\prime}_X$,
    $\widehat{\cal O}^{\,\prime\prime}_X$,
    $^{\Bbb C}\widehat{\cal O}_X$
 are by default sheaves of ${\cal O}_X$-algebras.
With $X$ taken also as a super $C^{\infty}$-scheme with the odd part of the structure sheaf identically zero,
then, by construction,
 all the super $C^{\infty}$ schemes $\widehat{(\,\cdot\,)}$ constructed in Item (b) fit into the
 sequence of morphisms of super $C^{\infty}$-schemes
 $$
  \xymatrix{
   X\;  \ar@{^{(}->}[r]
    & \;\;\mbox{any of $\,^{\Bbb R}\!\widehat{X}$, $\widehat{X}^{\prime}$,
	                      $\widehat{X}^{\prime\prime}$, $^{\Bbb C}\!\widehat{X}$}\;\;
	   \ar@{->>}[r]
	& \;X\,,
  }
 $$
 with the composition the identity map $\Id_X:X\rightarrow X$ of the $C^{\infty}$-scheme $X$.
Furthermore, one has
 the following commutative inclusion-quotient diagram of sheaves of super $C^{\infty}$ ${\cal O}_X$-algebras
 $$
  \xymatrix@R=4ex{
    &   ^{\Bbb C}\widehat{\cal O}_X
	        \ar@<.3ex> @{->>} [dl]  \ar@<.3ex>@{->>}[d]  \ar@<.3ex>@{->>}[dr]\\
   ^{\Bbb R}\widehat{\cal O}_X\;\rule{0ex}{1.2em}
              \ar @<.3ex>@{^{(}->}[ur]   \ar@<.3ex>@{->>}[dr]
       & \widehat{\cal O}^{\,\prime}_X\rule{0ex}{1.2em}
	           \ar@<.3ex> @{^{(}->}[u]   \ar@<.3ex>@{->>}[d]
       & \;\widehat{\cal O}^{\,\prime\prime}_X\,,\rule{0ex}{1.2em}
	           \ar@<.3ex> @{^{(}->}[ul]    \ar@<.3ex>@{->>}[dl] \\
	& \;{\cal O}_X\;\rule{0ex}{1.2em}
        	\ar@<.3ex> @{^{(}->}[ul] \ar@<.3ex> @{^{(}->}[u]  \ar@<.3ex>@{^{(}->}[ur]
   }
 $$
 which gives rise contravariantly
  to a commutative diagram of dominant morphisms and inclusions between super $C^{\infty}$-schemes
  $$
  \xymatrix@R=4ex{
    &   ^{\Bbb C}\!\widehat{X}
	       \ar@<-.3ex>@{->>}[dl]  \ar@<-.3ex>@{->>}[d]  \ar@<-.3ex>@{->>}[dr]\\
   ^{\Bbb R}\!\widehat{X}\;\rule{0ex}{1.2em}
           \ar@<-.3ex>@{_{(}->}[ur] \ar@<-.3ex>@{->>}[dr]
       & \widehat{X}^{\prime}\rule{0ex}{1.2em}
	        \ar@<-.3ex>@{_{(}->}[u] \ar@<-.3ex>@{->>}[d]
       & \;\widehat{X}^{\prime\prime}\,,\rule{0ex}{1.2em}
	         \ar@<-.3ex>@{_{(}->}[ul] \ar@<-.3ex>@{->>}[dl]  \\
	& \; X\;\rule{0ex}{1.2em}
	         \ar@<-.3ex>@{_{(}->}[ul]  \ar@<-.3ex>@{_{(}->}[u] \ar@<-.3ex>@{_{(}->}[ur]
   }
 $$
 that all restrict to the identity map $\Id_X$ on $X$.

\bigskip

\noindent
$(d)$\;{\it The $d=4$, $N=1$ superspace}

\medskip

\noindent
The relations in Item (c) suggest the super $C^{\infty}$-scheme
  $^{\Bbb C}\!\widehat{X} :=(X, \,\!^{\Bbb C}\widehat{\cal O}_X)$
 as the $d=4$, $N=1$ superspace.
However, one soon finds that in practice
  it is not very convenient to work on $C^{\infty}(^{\Bbb C}\!\widehat{X})$.
For example, from the presentation of $C^{\infty}(^{\Bbb C}\!\widehat{X})$
  as a polynomial ring in anti-commuting variables
  $\theta^1,\,\theta^2,\,\bar{\theta}^{\dot{1}},\,\bar{\theta}^{\dot{2}}$,
  one would expect derivations on $C^{\infty}(^{\Bbb C}\!\widehat{X})$ of the form
  $\partial/\partial\theta^1,\, \partial/\partial\theta^2,\,
    \partial/\bar{\theta}^{\dot{1}},\, \partial/\bar{\theta}^{\dot{2}}$.
Yet, the natural definition for these differential operators on $C^{\infty}(^{\Bbb C}\!\widehat{X})$
 does not give operations on $C^{\infty}(^{\Bbb C}\!\widehat{X})$ at all
 since new elements with degree-$0$ coefficients in $C^{\infty}(X)^{\Bbb C}$ can occur.
This consideration leads us to a final revision to the structure sheaf for the $d=4$, $N=1$ superspace.
 
\bigskip

\begin{definition}{\bf [$d=4$, $N=1$ superspace]}\; {\rm
 Continuing the notations in Item (b.4).
 Let $^{\Bbb C}\widehat{\cal O}_X\subset \widehat{\cal O}_X$
  be the extension of $^{\Bbb C}\widehat{\cal O}_X$ as ${\cal O}_X$-algebras with
  $$
     \widehat{\cal O}_X(U)\;
	   =\;  C^{\infty}( \mbox{$\bigwedge$}_{\,\Bbb C}^{\bullet}
	           (S^{\prime}\oplus S^{\prime\prime})|_U)
   $$
   for an open set $U\subset X$,
   where
    $$
     \mbox{$\bigwedge$}_{\,\Bbb C}^{\bullet}(S^{\prime}\oplus S^{\prime\prime})\;
	  :=\;  \mbox{$\bigoplus_{l=0}^4\,\bigwedge$}_{\,\Bbb C}^l
	          (S^{\prime}\oplus S^{\prime\prime})
    $$
	with $\bigwedge^0_{\,\Bbb C}(S^{\prime}\oplus S^{\prime\prime})
	           := \underline{\Bbb C}$ by convention.
 The super $C^{\infty}$-ring structure on $^{\Bbb C}\widehat{\cal O}_X$ is intact.  													
 The ${\Bbb Z}/2$-grading of $\widehat{\cal O}_X$ extends that of $^{\Bbb C}\widehat{\cal O}_X$:
   $$
     \begin{array}{cl}
	 & \widehat{\cal O}_{X,\,\even}\;\;
         :=\;\; \mbox{sheaf of sections of
				   $\;\;\bigwedge_{\,\Bbb C}^{\even}\!(S^{\prime}\oplus S^{\prime\prime})\,
				           :=\,  \underline{\Bbb C}
						           \oplus \bigwedge_{\,\Bbb C}^2 (S^{\prime}\oplus S^{\prime\prime})
								   \oplus \bigwedge_{\,\Bbb C}^4 (S^{\prime}\oplus S^{\prime\prime})\,$}
         			   \\[1.2ex]
	 \mbox{and}
	 & \widehat{\cal O}_{X,\,\odd}\;\;\;
            :=\;\; \mbox{sheaf of sections of
			         $\;\;\bigwedge_{\,\Bbb C}^{\odd}\!(S^{\prime}\oplus S^{\prime\prime})\,
					        :=\,  (S^{\prime}\oplus S^{\prime\prime})
							          \oplus \bigwedge_{\,\Bbb C}^3 (S^{\prime}\oplus S^{\prime\prime})$}\,.  	
	 \end{array}
   $$
 The locally-ringed space  $\widehat{X}:= (X,\widehat{\cal O}_X)$ is called
  the {\it $d=4$, $N=1$ superspace}.
 Explicitly,
	\begin{eqnarray*}	
	  \widehat{\cal O}_X (U)
	   & \:=  &  C^{\infty}(U)^{\Bbb C}[\theta^1|_U,\, \theta^2|_U,\,
	                  \bar{\theta}^{\dot{1}}|_U,\, \bar{\theta}^{\dot{2}}|_U]^{\anticommuting} \\
	   & :=    & \mbox{\small the polynomial ring in anticommuting variables $\theta^1|_U,\, \theta^2|_U,\,
	                          \bar{\theta}^{\dot{1}}|_U,\, \bar{\theta}^{\dot{2}}|_U$}\\[-.6ex]
			   && \mbox{\small with coefficients in $C^{\infty}(U)^{\Bbb C}$}						 
	\end{eqnarray*}
  and the {\it function-ring} of $\widehat{X }$ is given by
    $$
	  C^{\infty}(\widehat{X})\;
	    :=\;  \widehat{\cal O}_X (X)\;
		 =\;  C^{\infty}(X)^{\Bbb C}[\theta^1,\, \theta^2,\,
	                          \bar{\theta}^{\dot{1}},\, \bar{\theta}^{\dot{2}}]^{\anticommuting}\,,
	$$
 which contains the super $C^{\infty}$-ring $C^{\infty}(^{\Bbb C}\!\widehat{X})$
   as a super $C^{\infty}(X)$-subalgebra.
 By construction, for $U\subset X$ open,
   the {\it $C^{\infty}$-hull} of $\widehat{\cal O}_X(U)$
   is given by the $C^\infty$-hull of $^{\Bbb C}\widehat{\cal O}_X(U)$.
 Recall
   the coordinate functions $x^0,\,x^1,\, x^2,\, x^3$ on $X$,
    which generate the $C^{\infty}$-ring $C^{\infty}(X)$.
  We will call the collection
    $x^0,\,x^1,\, x^2,\, x^3,\, \theta^1,\,\theta^2,\,\bar{\theta}^{\dot{1}},\, \bar{\theta}^{\dot{2}}$
  the {\it standard generators} of $C^{\infty}(\widehat{X})$
  (under $C^{\infty}$-operations,
      anticommutativity of $\theta^1,\theta^2,\bar{\theta}^{\dot{1}}, \bar{\theta}^{\dot{2}}$,
	  and with $\sqrt{-1}$\,).
 An element in $C^{\infty}(\widehat{X})$ is called a {\it superfield}
     or a {\it scalar superfield}\footnote{{\it Note for Mathematicians}\;
	                                         There
	                                           are two naming systems or senses in physics literature for fields over a superspace
								               and they are usually used in a mixed way.
											 The first one follows the nature of $(\theta,\bar{\theta})$-degree-$0$ component.
											  If it is a section of the associated bundle of $P$
											     from the trivial (resp.\ spinor, vector, $\cdots$) representation
												 of {\it SO}$^\uparrow(3,1)$,
                                               then it is called {\it scalar superfield}
											     (resp.\ {\it spinor superfield}, {\it vector superfield},  $\cdots$).
                                             The second follows the representation of the supersymmetry algebra in question.
											 If it corresponds to the chiral multiplet (resp.\ vector multiplet, $\cdots$)
											  then it is called {\it chiral superfield}, (resp.\ {\it vector superfield}, $\cdots$).
											 In the second case, the superfields are usually defined along with (a system of)
                                               constraints (i.e.\ SUSY-Rep Compatible Conditions)
											   to remove the surplus degrees of freedom not included
											   in the corresponding representation of the supersymmetry algebra.
                                             In the current work, we call  a general $f\in \widehat{\cal O}_X$ a scalar superfield
											   according the first sense
											   and call a $f\in \widehat{\cal O}_X$ with SUSY-Rep Compatible Conditions
                                               a chiral superfield, a vector superfield, $\cdots$, according to the second sense.
											 See, e.g., [We-B: Chap.s.\ IV, V, VI] and
											         also [G-G-R-S: Sec.\ 3.3.b.3 \& 3.3.b.4], [West1: Sec.\ 11.1].
											  } 
     on the superspace in physics literature (e.g.\ [F-W-Z]; cf.\ [We-B: Chap.\ IV]).
 Recall the principal $\SO^\uparrow(3,1)$-bundle $P$ over $X$.
 Then, $\widehat{\cal O}_X$ can be decomposed into a direct sum of
  complex irreducible (left) $P$-modules
   {\small
   \begin{eqnarray*}
    \widehat{\cal O}_X
	  &  = &  {\cal O}_X^{\,\Bbb C}\,
	    \oplus\, {\cal S}^{\prime}\oplus {\cal S}^{\prime\prime}\,
		\oplus\, \mbox{$\bigwedge$}^2{\cal S}^{\prime}\,
		\oplus\, {\cal S}^{\prime}\otimes_{{\cal O}_X^{\,\Bbb C}}\!{\cal S}^{\prime\prime}\,
		\oplus\, \mbox{$\bigwedge$}^2{\cal S}^{\prime}\, \\
     && \hspace{6em}		
		\oplus\, {\cal S}^{\prime}
		                    \otimes_{{\cal O}_X^{\,\Bbb C}}
						    \!\mbox{$\bigwedge$}^2{\cal S}^{\prime\prime}
        \oplus\, \mbox{\Large $($}\mbox{$\bigwedge$}^2{\cal S}^\prime\mbox{\Large $)$}
		                  \otimes_{{\cal O}_X^{\,\Bbb C}}{\cal S}^{\prime\prime}\,
		\oplus \mbox{$\bigwedge$}^4({\cal S}^\prime\oplus {\cal S}^{\prime\prime})\,,
   \end{eqnarray*}}in 
   which
   $\bigwedge^2{\cal S}^\prime
	   \simeq \mbox{$\bigwedge$}^2{\cal S}^{\prime\prime}
	   \simeq \mbox{$\bigwedge$}^4({\cal S}^\prime\oplus {\cal S}^{\prime\prime})
	   \simeq {\cal O}_X^{\,\Bbb C}$,\;
   ${\cal S}^\prime\otimes_{}\!{\cal S}^{\prime\prime}
       \simeq {\cal T}_{\ast}X^{\Bbb C}$,\;
   ${\cal S}^\prime
	       \otimes_{{\cal O}_X^{\,\Bbb C}}\!
		   \mbox{\large $($}\mbox{$\bigwedge$}^2{\cal S}^\prime\mbox{\large $)$}
	   \simeq {\cal S}^\prime$,\;  and
	$\mbox{\large $($}\bigwedge^2{\cal S}^\prime\mbox{\large $)$}
	       \otimes_{{\cal O}_X^{\,\Bbb C}}\!{\cal S}^{\prime\prime}
	   \simeq {\cal S}^{\prime\prime}$\;
   as complex $P$-modules.
}\end{definition}

\medskip

\begin{notation}		   {\bf [collective standard coordinates, components]}\; {\rm
 For convenience, we denote collectively	
    $$
	  x\;=\;(x^0,x^1,x^2,x^3)=(x^{\mu})_{\mu}\,,\;\;\;
	  \theta\;=\;(\theta^1,\theta^2)=(\theta^{\alpha})_{\alpha}\,,\;\;\;
	 \bar{\theta}\;=\;(\bar{\theta}^{\dot{1}},\bar{\theta}^{\dot{2}})\;
	         =\;(\bar{\theta}^{\dot{\beta}})_{\dot{\beta}}
	$$
    when in need.
 In terms of this,
  we write $f \in C^\infty(\widehat{X}) $ as a polynomial in $(\theta,\bar{\theta})$
   with coefficients in $C^\infty(X)^{\Bbb C}$:
   \begin{eqnarray*}
     f & = & f(x,\theta,\bar{\theta})\; \\
	   & = &  f_{(0)}(x)\,
	        +\, \sum_{\alpha}f_{(\alpha)}(x)\theta^\alpha\,
			+\, \sum_{\dot{\beta}}f_{(\dot{\beta})}(x)\bar{\theta}^{\dot{\beta}}\,
			+\, f_{(12)}(x) \theta^1\theta^2\,
			+\, \sum_{\alpha,\dot{\beta}}f_{(\alpha\dot{\beta})}(x)
			       \theta^\alpha\bar{\theta}^{\dot{\beta}}\,  \\
        && 				
            +\, f_{(\dot{1}\dot{2})}(x)\bar{\theta}^{\dot{1}}\bar{\theta}^{\dot{2}}\,
			+\, \sum_{\dot{\beta}} f_{(12\dot{\beta})}(x)
			        \theta^1\theta^2\bar{\theta}^{\dot{\beta}}\,
			+\, \sum_\alpha f_{(\alpha\dot{1}\dot{2})}(x)
			        \theta^\alpha\theta^{\dot{1}}\theta^{\dot{2}}\,
		    +\, f_{(12\dot{1}\dot{2})}(x)
			        \theta^1\theta^2\theta^{\dot{1}}\theta^{\dot{2}}\,;			
   \end{eqnarray*}
   and call $f_{(\mbox{\tiny $\bullet$})}$ of the {\it components} of $f$.\footnote{{\it On
	                               the index-hidden notation for a superfield.}\;                      
                                 With the notations in [We-B: Appendix A] of Wess \& Bagger,
								  one may express a superfield $f\in C^\infty(\widehat{X})$ as
									\begin{eqnarray*}
									  f(x,\theta, \bar{\theta})
									    & =  &  f(x)\,
										     +\, \theta \phi(x)\, +\, \bar{\theta}\bar{\chi}(x)\,
                                             +\, \theta\theta m(x)\, +\, \bar{\theta}\bar{\theta}n(x)\, \\
									  && \hspace{6em}
                                             +\, \sum_{\mu}\theta\sigma^\mu\bar{\theta}v_\mu(x)\,
                                             +\, \theta\theta\bar{\theta}\bar{\lambda}(x)\,
                                             +\, \bar{\theta}\bar{\theta}\theta\psi(x)\,
                                             +\, \theta\theta\bar{\theta}\bar{\theta}d(x)											 
									\end{eqnarray*}
                                   ([We-B: Eq.\ (4.9)])
								   in accordance with the $P$-module decomposition of $\widehat{\cal O}_X$.
                                  Such notations are used prevailingly in [We-B] and other physicists' supersymmetry literatures.
                                 {\it However}, for the computations in the current work
  								     it is more convenient to keep the fermionic indices explicit and treat
									 a superfield as a polynomial of odd, anticommuting variables explicitly.
                                  For that reason, such by-now-standard index-hidden notations in physics literature
								   are not adopted here.}
}\end{notation}

\medskip

\begin{remark} {\it $[\widehat{\cal O}_X$
                vs.\ $^{\Bbb R}\widehat{\cal O}_X\otimes_{\Bbb R}{\Bbb C}]$}\; {\rm
 As a sheaf of ${\Bbb Z}/2$-graded rings,
  $\widehat{\cal O}_X$ and $^{\Bbb R}\widehat{\cal O}_X\otimes_{\Bbb R}{\Bbb C}$
   are canonically isomorphic.
 However, by standard convention,
    the $C^\infty$-hull of $^{\Bbb R}\widehat{\cal O}_X\otimes_{\Bbb R}{\Bbb C}$
    is defined to be the $C^\infty$-hull of $^{\Bbb R}\widehat{\cal O}_X$
  while the $C^\infty$-hull of $\widehat{\cal O}_X$ is much larger.
 Thus, they are not isomorphic as sheaves of rings with partial $C^\infty$-ring structure.
}\end{remark}

\bigskip

\subsection{Calculus on the $d=4$, $N=1$ superspace $\widehat{X}$}

Calculus on the $d=4$, $N=1$ superspace $\widehat{X}$
 comes from the extension of the calculus on $X$.
Essential basics of the extension we need are collected in this subsection
 for introducing terminology and notations and also for fixing the conventions we will adopt.
Cf.\ [We-B: Chap.'s\ IV  \& XII]; also [G-G-R-S: Sec.'s 3.3.b \& 3.7], [West1: Chap.\ 14].

\bigskip

\begin{definition} {\bf [parity-conjugation]}\; {\rm
 (1)
  For $a=a_{\even}+ a_{\odd}$ a ${\Bbb Z}/2$-graded object,
  define the {\it parity conjugation} of $a$ to be $^{\varsigma}\!a:= a_{\even}-a_{odd}$.
 
 (2)
  Let $\beta$ be another ${\Bbb Z}/2$-graded object (not necessarily of the same kind as $a$)
    of pure parity (i.e.\ $\beta$ is either even or odd).
  Define the {\it $\beta$-induced parity-conjugation} of $a$ to be
    $^{\varsigma\!_\beta}a = a$ if $\beta$ is even, or $^{\varsigma}\!a$ if $\beta$ is odd.
}\end{definition}

\bigskip

Such parity conjugation often occurs in a passing of a ${\Bbb Z}/2$-graded object of pure parity.
E.g.\ for $f_1, f_2\in C^{\infty}(\widehat{X})$ with $f_1$ of pure parity,
 $f_1f_2 =\, ^{\varsigma\!_{f_1}}\!(f_2)f_1$, in which $f_1$ passes over $f_2$.

\bigskip

\begin{flushleft}
{\bf Vector fields on $\widehat{X}\;:=\;$ Derivations of $C^{\infty}(\widehat{X})$}
\end{flushleft}
\begin{definition} {\bf [derivation of $C^{\infty}(\widehat{X})$]}\; {\rm
 A {\it derivation}\footnote{{\it Left vs.\ right derivation}\hspace{1em}
                                      As defined, this is indeed a {\it left derivation},
									   which acts on $C^{\infty}(\widehat{X})$ from the left
									   (of $C^\infty(\widehat{X})$).
									  One can also define the notion of {\it right derivations},
									   which are ${\Bbb C}$-linear and satisfy the ${\Bbb Z}/2$-graded right Leibniz rule
                                      $$
                                         (f_1f_2)\,\!^{\leftarrow}\!\!\!\xi\;
										    =\; (-1)^{p(f_2)p(\xi)}\, (f_1\,\!^{\leftarrow}\!\!\!\xi)f_2\,
											          +\,f_1(f_2\,\!^{\leftarrow}\!\!\!\xi)
                                      $$
									  for $f_2$ and $\xi$ parity-homogeneous.
									 In this work, all derivations of $C^{\infty}(\widehat{X})$ are by default left derivations.
									 }  
   of $C^{\infty}(\widehat{X})$ over ${\Bbb C}$
  is a ${\Bbb Z}/2$-graded ${\Bbb C}$-linear operation
  $\xi: C^{\infty}(\widehat{X})\rightarrow C^{\infty}(\widehat{X})$
  on $C^{\infty}(\widehat{X})$
  that satisfies the ${\Bbb Z}/2$-graded Leibniz rule
  $$
    \xi(fg)\;=\; (\xi f)g\,+\, (-1)^{p(\xi)p(f)}f(\xi g)
  $$
  when in parity-homogeneous situations.
 The set $\Der_{\Bbb C}(\widehat{X}):=\Der_{\Bbb C}(C^{\infty}(\widehat{X}))$
     of derivations of $C^{\infty}(\widehat{X})$
	 is a (left) $C^{\infty}(\widehat{X})$-module,
	with $(a\xi)(\,{\LARGE \cdot}\,):= a (\xi(\,\cdot\,))$ and $p(a\xi):= p(a)+p(\xi)$
	for $a\in C^{\infty}(\widehat{X})$ and $\xi\in \Der_{\Bbb C}(\widehat{X})$.
}\end{definition}

\bigskip

Geometrically, one should think of a derivation of $C^{\infty}(\widehat{X})$
 as a {\it vector field} on $\widehat{X}$.

\bigskip

\begin{example} {\bf [derivations associated to the standard coordinates $(x,\theta,\bar{\theta})$]}\;
{\rm
  Associated to the standard coordinates $(x,\theta,\bar{\theta})$ on $\widehat{X}$ are the following
   basic derivations on $C^{\infty}(\widehat{X})$:
  $$
   \mbox{\Large $\frac{\partial}{\partial x^\mu}$}\,,\;\;\;
   \mbox{\Large $\frac{\partial}{\partial \theta^\alpha}$}\,,\;\;\;
   \mbox{\Large $\frac{\partial}{\rule{0ex}{1.6ex}\partial \bar{\theta}^{\dot{\beta}}}$}\,,
  $$
  for $\mu=0,1,2,3$, $\alpha=1,2$, and $\dot{\beta}=\dot{1}, \dot{2}$.
 They are characterized by\footnote{This
                                            illustrates also that $C^{\infty}(\widehat{X})$
											is more a ${\Bbb Z}/2$-commutative ring than a noncommutative ring.
										 The Leibniz rule for a derivation $\xi$ of a general noncommutative ring $R$ is given by 	
										   $\xi(r_1r_2)=(\xi r_1)r_2 + r_1 (\xi r_2)$.
                                         If taking $C^{\infty}(\widehat{X})$ just as a noncommutative ring
										   with the ${\Bbb Z}/2$-grading suppressed,
										 then one would have, for example,
                                           $$										
                                             \frac{\partial}{\partial\theta^2}(\theta^1\theta^2)
											   = \left(\rule{0ex}{1em}\right.\!\!
											       \frac{\partial}{\partial\theta^2}\theta^1\!\!
												     \left.\rule{0ex}{1em}\right)\theta^2
											         + \theta^1 \left(\rule{0ex}{1em}\right.\!\!
													       \frac{\partial}{\partial\theta^2}\theta^2\!\!
														   \left.\rule{0ex}{1em}\right)
											   = \theta^1\,,
                                           $$
                                           which equals $- \frac{\partial}{\partial\theta^2}(\theta^2\theta^1)= -\theta_1$
										   since $\theta^1\theta^2=-\theta^2\theta^1$.
										 This implies $2\theta^1=0$, a contradiction.			
                                         Such contradictions are resolved exactly by correctly assigning odd-parity to
										   $\theta^{\alpha}$'s, $\bar{\theta}^{\dot{\beta}}$'s,
                                           $\partial/\partial\theta^{\alpha}$'s,
										   and $\partial/\partial\bar{\theta}^{\dot{\beta}}$'s										   
										   and imposing the ${\Bbb Z}/2$-graded Leizniz rule.
                                           }
  $$
    \begin{array}{c}
      \mbox{\Large $\frac{\partial}{\partial x^\mu}$}(x^\nu)
         = \delta_{\mu\nu}\,,\;\;\;
      \mbox{\Large $\frac{\partial}{\partial \theta^\alpha}$}(\theta^\beta)
         = \delta_{\alpha\beta}\,,\;\;\;
      \mbox{\Large $\frac{\partial}{\rule{0ex}{1.6ex}\partial \bar{\theta}^{\dot{\alpha}}}$}
	    (\bar{\theta}^{\dot{\beta}})
	     = \delta_{\dot{\alpha}\dot{\beta}}\,,          \\[2ex]
      \mbox{\Large $\frac{\partial}{\partial x^\mu}$}(\theta^\alpha)
	     =\mbox{\Large $\frac{\partial}{\partial x^\mu}$}(\bar{\theta}^{\dot{\beta}})
         = 0\,,\;\;\;
      \mbox{\Large $\frac{\partial}{\partial \theta^\alpha}$}(x^\mu)
	     =  \mbox{\Large $\frac{\partial}{\partial \theta^\alpha}$}
		       (\bar{\theta}^{\dot{\beta}})
         = 0 \,,\;\;\;
      \mbox{\Large $\frac{\partial}{\rule{0ex}{1.6ex}\partial \bar{\theta}^{\dot{\alpha}}}$}
	    (x^\mu)
		=  \mbox{\Large $\frac{\partial}{\rule{0ex}{1.6ex}\partial \bar{\theta}^{\dot{\alpha}}}$}
	    (\theta^\alpha)
	     =  0\,.
   \end{array}
  $$
  Here, $\delta_{\mu\nu}=1$ if $\mu=\nu$; and $0$ if $\mu\ne \nu$.
  Similarly for $\delta_{\alpha\beta}$ and $\delta_{\dot{\alpha}\dot{\beta}}$.
  
  Since
   $\delta_{\mu\mu}=\delta_{\alpha\alpha}=\delta_{\dot{\beta}\dot{\beta}}=1$ is even,
   we assign the parity
  $$
    p(\partial/\partial x^\mu) = p(x^{\mu}) = 0\,,\;\;\;
    p(\partial/\partial \theta^\alpha) = p(\theta^{\alpha})= 1\,,\;\;\;
    p(\partial/\partial \bar{\theta}^{\dot{\beta}}) = p(\bar{\theta}^{\dot{\beta}}) = 1\,.
 $$	
 
 The collection
   $\;\{\partial/\partial x^\mu,\, \partial/\partial\theta^\alpha,\,
          \partial/\partial\bar{\theta}^{\dot{\beta}} \}
		  _{\mu=0,1,2,3;\,\alpha=1,2;\,\dot{\beta}=\dot{1},\dot{2}}\;$
   forms a basis of the (left) $C^{\infty}(\widehat{X})$-module
   $\Der_{\Bbb C}(\widehat{X})$.
}\end{example}

\bigskip

\begin{lemma} {\bf [chain rule]}\;
 Let
   $\xi\in \Der_{\Bbb C}(\widehat{X})$,
   $h\in C^{\infty}({\Bbb R}^l)$,   and\\
   $f_1,\,\cdots,\, f_l
	  \in  C^{\infty}(\mbox{$\bigwedge_{\,\Bbb R,\Bbb C}^{\even}\!
	                                                                (S^{\prime}\oplus S^{\prime\prime})$})
	  \subset C^{\infty}(\widehat{X})$.
 Then
  $$
    \xi(h(f_1,\,\cdots\,, f_l))\;
		 =\;  \sum_{k=1}^l
		             \left(\rule{0ex}{.8em}\right.\!\!
							        (\partial_k h)(f_1,\,\cdots\,, f_l)
									\!\!\left.\rule{0ex}{.8em}\right)\cdot \xi f_k\;
			\in\; C^\infty(\widehat{X})\,,
  $$							
  where
   $\partial_k h \in C^{\infty}({\Bbb R}^l)$ is the partial derivative of $h$
     with respect to the $k$-th argument.
\end{lemma}

\medskip

\begin{proof}
 Since $f_i$'s are even,
    $h(f_1,\,\cdots\,, f_l)$ and $(\partial_k h)(f_1,\,\cdots\,, f_l)$'s are all even as well.
 This implies that both the left-hand side and the right-hand side of the chain-rule identity to be proved
    are $C^{\infty}(\widehat{X})$-linear.
 Thus, we only need to show that the identity holds for
   $\xi$ being one of basic derivations
     $\partial/\partial x^{\mu}$, $\partial/\partial \theta^{\alpha}$,
	    and  $\partial/\partial\bar{\theta}^{\dot{\beta}}$,
        for $\mu=0,1,2,3$, $\alpha=1, 2$, and $\dot{\beta}=\dot{1}, \dot{2}$.	

 Let $f_i=a_i+n_i$, where $a_i\in C^{\infty}(X)$ and $n_i$ is even and nilpotent, $i=1,\,\ldots\,, l$.
Then,
 \begin{eqnarray*}
  \lefteqn{
     h(f_1,\,\cdots\,, f_l)\; =\; h(a_1+ n_1,\,\cdots\,, a_l+n_l)   }\\
   && =\; h(a_1,\,\cdots\,, a_l)\,
                +\, \sum_{k=1}^l ((\partial_kh)(a_1,\,\cdots\,, a_l))\cdot n_k\,
				+\, \mbox{\Large $\frac{1}{2}$}\,
				       \sum_{k, k^{\prime}=1}^l
					     ((\partial_k\partial_{k^{\prime}}h)(a_1,\,\cdots\,, a_l))
						   \cdot n_k n_{k^{\prime}}\,.
 \end{eqnarray*}
 It follows that,  for $\xi$ one of the basic derivations,
 the left-hand side of the identity is
 {\small
  \begin{eqnarray*}
   \lefteqn{
    \xi(h(f_1,\,\cdots\,, f_l))        }\\
    && =\; \xi( h(a_1,\,\cdots\,, a_l))\,
                +\, \sum_{k=1}^l   \xi((\partial_kh)(a_1,\,\cdots\,, a_l))\cdot n_k\,
				+\, \sum_{k=1}^l ((\partial_kh)(a_1,\,\cdots\,, a_l))\cdot   \xi n_k  \\
    &&  \hspace{2em}				
		        +\, \mbox{\large $\frac{1}{2}$}\,
				       \sum_{k, k^{\prime}=1}^l
					     \xi((\partial_k\partial_{k^{\prime}}h)(a_1,\,\cdots\,, a_l))
						   \cdot n_k n_{k^{\prime}}\,
	            +\, \mbox{\large $\frac{1}{2}$}\,
				       \sum_{k, k^{\prime}=1}^l
					     ((\partial_k\partial_{k^{\prime}}h)(a_1,\,\cdots\,, a_l))
						   \cdot \xi(n_k n_{k^{\prime}})  \\
    &&  =\; \sum_{k=1}^l (\partial_kh)(a_1,\,\cdots\,, a_l))\cdot \xi a_k\,
	              +\, \sum_{k^{\prime}, k=1}^l
				        ((\partial_{k^{\prime}}\partial_k h)(a_1,\,\cdots\,, a_l))
						       \cdot  (\xi a_{k^{\prime}})n_k\,
                  +\, \sum_{k=1}^l ((\partial_kh)(a_1,\,\cdots\,, a_l))\cdot   \xi n_k  \\							   
    &&	\hspace{2em}
		        +\, \mbox{\large $\frac{1}{2}$}\,
				      \sum_{k^{\prime\prime}, k, k^{\prime}=1}^l
					     ((\partial_{k^{\prime\prime}}\partial_k\partial_{k^{\prime}}h)
						            (a_1,\,\cdots\,, a_l))
						   \cdot (\xi a_{k^{\prime\prime}}) n_k n_{k^{\prime}}\,
	            +\, \sum_{k, k^{\prime}=1}^l
					     ((\partial_k\partial_{k^{\prime}}h)(a_1,\,\cdots\,, a_l))
						   \cdot (\xi n_k) n_{k^{\prime}} 	
  \end{eqnarray*}}while	
 the right-hand side of the identity is
  {\small
   \begin{eqnarray*}
   \lefteqn{
	 \sum_{k=1}^l
		((\partial_k h)(f_1,\,\cdots\,, f_l))\cdot \xi f_k\;
     =\;  \sum_{k=1}^l
		     ((\partial_k h)(a_1+n_1,\,\cdots\,, a_l+n_l))\cdot \xi (a_k+n_k)		
    }\\
    &&  =\;
	   \sum_{k=1}^l
	         \left(\rule{0ex}{1em}\right.\!\!
			  (\partial_k h)(a_1,\,\cdots\,, a_l)\,
			    +\, \sum_{k^{\prime}=1}^l
				        ((\partial_{k^{\prime}}\partial_kh)(a_1,\,\cdots\,, a_l)) n_{k^{\prime}} \\
    && \hspace{8em}
				+\, \mbox{\large $\frac{1}{2}$}\,
                       \sum_{k^{\prime}, k^{\prime\prime}=1}^l
					    ((\partial_{k^{\prime}}\partial_{k^{\prime\prime}}\partial_kh)(a_1,\,\cdots\,, a_l))
					      \cdot n_{k^{\prime}}n_{k^{\prime\prime}}
			   \!\!\left.\rule{0ex}{1em}\right)\cdot
			 (\xi a_k +\xi n_k)\,,
  \end{eqnarray*}}which	
  equals the left-hand side of the identity
    after an expansion and some relabeling of $k$, $k^{\prime}$, $k^{\prime\prime}$.
 Here, we used (wherever applicable), in addition to the ${\Bbb Z}/2$-graded Leibniz rule,
    \begin{itemize}
    \item[(1)]
     $n_kn_{k^{\prime}}n_{k^{\prime\prime}}
       = n_k n_{k^{\prime}}(\xi n_{k^{\prime\prime}})=0$
      for $k,k^{\prime}, k^{\prime\prime}=1,\,\ldots\,, l$
	  since the minimal total-$(\theta,\bar{\theta})$-degree of terms in these expressions $>4$;

    \item[(2)]
     the existing chain rule for
	   $\frac{\partial}{\partial x^{\mu}} ( h(a_1,\,\cdots\,, a_l))$,
	   $\frac{\partial}{\partial x^{\mu}} ( (\partial_kh)(a_1,\,\cdots\,, a_l))$,
	      and\\
	   $\frac{\partial}{\partial x^{\mu}}
	       ( (\partial_k\partial_{k^{\prime}}h)(a_1,\,\cdots\,, a_l))$,
       following the classical Calculus for real variables;		

    \item[(3)]	
	  $n_k$ and  $\xi n_{k^{\prime}}$ commute, for all $k$, $k^{\prime}$, since $n_k$ is even;
	
    \item[(4)]
     $\xi( h(a_1,\,\cdots\,, a_l))
	     = \xi( (\partial_kh)(a_1,\,\cdots\,, a_l))
	     = \xi( (\partial_k\partial_{k^{\prime}}h)(a_1,\,\cdots\,, a_l))=0$
	  for $\xi=\partial/\partial\theta^{\alpha}$, $\partial/\partial\bar{\theta}^{\dot{\beta}}$,
	  since $ h(a_1,\,\cdots\,, a_l)$, $(\partial_kh)(a_1,\,\cdots\,, a_l)$,   and
	            $(\partial_k\partial_{k^{\prime}}h)(a_1,\,\cdots\,, a_l)$
				have no $\theta^{\alpha}$- nor $\bar{\theta}^{\dot{\beta}}$-dependence.
    \end{itemize}
	
  This proves the Lemma.
  
\end{proof}

\bigskip

The ${\Bbb Z}/2$-graded Lie bracket
 $[\,\mbox{\LARGE $\cdot$}\,,\,\mbox{\LARGE $\cdot$}\,\}$, defined by
   $(\xi_1,\xi_2)\mapsto
       [\xi_1,\xi_2\}:= \xi_1\xi_2-(-1)^{p(\xi_1)p(\xi_2)}\xi_2\xi_1$
  for $\xi_1$, $\xi_2$ parity-homogeneous,
 gives $\Der_{\Bbb C}(\widehat{X})$
 a super Lie algebra structure
 that satisfies the super Jacobi identity (in the form of ${\Bbb Z}/2$-graded Leibniz rule)
 $$
  [\xi_1, [\xi_2, \xi_3\} \}\;=\;
   [[\xi_1,\xi_2\}, \xi_3\} + (-1)^{p(\xi_1)p(\xi_2)}[\xi_2, [\xi_1,\xi_3\}\}
 $$
 for $\xi_1, \xi_2, \xi_3\in \Der_{\Bbb C}(\widehat{X})$ parity-homogeneous.

\bigskip

\begin{flushleft}
{\bf 1-forms on $\widehat{X}\;:=\;$ differentials of $C^{\infty}(\widehat{X})$}
\end{flushleft}
To begin, we adopt the following convention as in [D-F2: \S 6].									   
 
\bigskip

\begin{convention} $[$cohomological degree vs.\ parity\,$]$\; {\rm
 We treat elements $f$ of $C^{\infty}(\widehat{X})$ as of cohomological degree $0$
  and the exterior differential operator $d$ as of cohomological degree $1$ and {\it even}.
 In notation, $\chd(f)=0$ and $\chd(d)=1$, $p(d)=0$.
 Under such $({\Bbb Z}\times ({\Bbb Z}/2))$-bi-grading,
   $$
      ab = (-1)^{c.h.d(a)\,c.h.d(b)}(-1)^{p(a)p(b)}ba
   $$
   for objects $a, b$ homogeneous with respect to the bi-grading.\footnote{This
                                                         is the convention that matches with the sign rules in
                                                         [We-B: Chap.\ XII, Eq.'s (12.2), (12.3)]	of Wess \& Bagger.
														In many mathematical literatures
														   motivated by the study of supersymmetry in physics,
														 the operator $d$ is taken as odd.
														In that case, to make the parity rule right in many situations
														 one has to introduce the {\it parity-swap operator} $\Pi$:
														 $(\Pi a):= a$ with the parity odd, if $a$ is even, or parity even, if $a$ is odd.
                                                        In physics, even parity corresponds to bosons
														 while odd parity corresponds to fermions. 
                                                        Bosons and fermions are distinguished
														  by the Pauli Exclusion Principle,
														  which is a fundamental nature of a particle, not by assignment.
                                                        Thus, the operator $\Pi$ almost never occurred
														  in physics literature involving supersymmetry.
                                                         }													   
  Here, $a$ and $b$ are not necessarily of the same type.
}\end{convention}

\bigskip

\begin{definition} {\bf [differential of $C^{\infty}(\widehat{X})$]}\; {\rm
 The bi-$C^{\infty}(\widehat{X})$-module
   $\Omega_{\widehat{X}}:= \Omega_{C^{\infty}(\widehat{X})}$
   of {\it differentials} of $C^{\infty}(\widehat{X})$ over ${\Bbb C}$
  is the quotient of
    the free bi-$C^{\infty}(\widehat{X})$-module
      	          generated by $d(f)$, $f\in C^{\infty}(\widehat{X})$,
   by the bi-$C^{\infty}(\widehat{X})$-submodule of relators generated by
   \begin{itemize}
    \item[(1)]
	 [{\it ${\Bbb C}$-linearity}]\hspace{1em}
	 $d(c_1f_1 + c_2f_2)-c_1 d(f_1) - c_2 d(f_2)$\,,\;
	  for $c_1, c_2\in {\Bbb C}$, $f_1, f_2\in C^{\infty}(\widehat{X})$;
	
    \item[(2)]
	 [{\it Leibniz rule}]\hspace{1em}
	  $d(f_1f_2)- (d(f_1))f_2 -f_1 d(f_2)$\,,\; for $f_1, f_2\in C^{\infty}(\widehat{X})$;
	
	\item[(3)]
	 [{\it chain-rule identities from the $C^{\infty}$-hull structure}]\hspace{1em}
	 $$
	    d(h(f_1,\,\cdots\,, f_l))- \sum_{k=1}^l (\partial_kh)(f_1,\,\cdots\,, f_l)\,d(f_k)
	 $$
	 for
	   $h\in C^{\infty}({\Bbb R}^l)$,
	   $f_1,\,\cdots\,, f_l
	          \in   C^{\infty}(\mbox{$\bigwedge_{\,\Bbb R,\Bbb C}^{\even}\!
	                                                                (S^{\prime}\oplus S^{\prime\prime})$})
		      \subset C^{\infty}(\widehat{X})$;
	 here, $\partial_kh$ is the partial derivative of $h\in C^{\infty}({\Bbb R}^l)$
	           with respect to the $k$-th argument.
   \end{itemize}
 The element of $\Omega_{\widehat{X}}$  associated to $d(f)$,
    $f\in C^{\infty}(\widehat{X})$, is denoted by $df$.
 Using Relators (2), one can convert $\Omega_{\widehat{X}}$ to
   either solely a left $C^{\infty}(\widehat{X})$-module
         or solely a right $C^{\infty}(\widehat{X})$-module.\footnote{Indeed,
		                            the sign rule in Convention~1.3.5 already gives
					                $(df_1)f_2 = (-1)^{p(f_1)p(f_2)}f_2 df_1$. 		
		                           Compatibility of the two different ways to convert
								     the bi-$C^{\infty}(\widehat{X})$-module $\Omega_{\widehat{X}}$
									 to a left $C^{\infty}(\widehat{X})$-module
		                             enforces upon us a second form of the Leibniz rule for differentials:
		                            $$
		                               d(f_1f_2)\;=\; (-1)^{p(f_1)p(f_2)}f_2 df_1 + f_1 df_2\,.
		                            $$
								  This second form makes the effect of ${\Bbb Z}/2$-grading to the Leibniz rule manifest.
                                  One should compare this with the Leibniz rule for differentials of a commutative ring:
                                    $d(r_1r_2)= r_2 dr_1 + r_1 dr_2$, for $r_1, r_2$ in a commutative ring $R$.
		 }
    
 A differential of $C^{\infty}(\widehat{X})$  is also called synonymously
  a {\it $1$-form} on $\widehat{X}$.
  
 By construction, there is a built-in map
  $d:C^{\infty}(\widehat{X})\rightarrow \Omega_{\widehat{X}}$ defined by
  $f\mapsto df$.
}\end{definition}

\bigskip

Note that, by Convention~1.3.5, $d$ is even, i.e.\ $p(d)=0$.
Thus, the Leibniz-rule relators in the above definition are indeed ${\Bbb Z}/2$-graded-Leibniz-rule relators,\\
 i.e.\ $\;d(f_1f_2)-(df_1)f_2-(-1)^{p(d)p(f_1)}f_1 (df_2)\;$
 for $f_1, f_2\in C^{\infty}(\widehat{X})$ with $f_1$ parity-homogeneous.

\bigskip

The additivity rule for cohomological degree and parity gives
  $\chd(df)=1$ and $p(df)=p(f)$.
 
\bigskip

\begin{lemma} {\bf [evaluation of $\Omega_{\widehat{X}}$ on
       $\Der_{\Bbb C}(\widehat{X})$ from the right]}\;
 The specification\footnote{Using
                                                      the formula $(df)(\xi):= \xi f$ to define
													  the evaluation of a differential $df$ on a derivation $\xi$ forces us to regard this
													  as an evaluation of $df$ from the right of $\xi$, i.e.\
													  $(\xi)^{\leftarrow}\!\!\!(df)$.
													Admittedly, then one should leave the notation $df(\xi)$  for the evaluation of
													  $df$ on $\xi$ from its left.
													Unfortunately, one soon realizes the burden of notations.
                                                    Since we consider in this work
													   only evaluations of differentials from the right of derivations
													   except in some supplementary remarks,
                                                     we reserve $df(\xi)$ to mean $(\xi)^{\leftarrow}\!\!\!(df)$.
				                                     } 
   $$
  (df)(\xi) \; :=\;  (\xi)^{\leftarrow}\!\!\!(df)\;   :=\; \xi(f)
   $$
   for $f\in C^{\infty}(\widehat{X})$ and $\xi\in \Der_{\Bbb C}(\widehat{X})$,
  defines an {\it evaluation} of $\Omega_{\widehat{X}}$ on $\Der_{\Bbb C}(\widehat{X})$
  from the right: \\
  for
    $\varpi=\sum_{i=1}^k a_i\,df_i\in \Omega_{\widehat{X}}$, with $a_i$ parity-homogeneous, and
	$\xi\in \Der_{\Bbb C}(\widehat{X})$ parity-homogeneous,
  $$
    \varpi(\xi)\;
	:=\;  (\xi)^{\leftarrow}\!\!\!\varpi\;
	:=\;  \sum_{i=1}^k (-1)^{p(\xi)p(a_i)}  a_i\, \xi(f_i)\,.
  $$
 This evaluation is (left) $C^{\infty}(\widehat{X})$-linear:
  $\varpi(a\xi)= a \varpi(\xi)$, for $a\in C^{\infty}(\widehat{X})$.
\end{lemma}

\medskip

\begin{proof}
 To show that the evaluation of $\Omega_{\widehat{X}}$ on $\Der_{\Bbb C}(\widehat{X})$
    as defined is well-defined,
  we need to show that
    the evaluation of the relators in Definition~1.3.6 on any $\xi\in \Der_{\Bbb C}(\widehat{X})$ vanish.
 The vanishing of [{\it ${\Bbb C}$-linearity}]-{\it relator}\,$(\xi)$
    follows from the ${\Bbb C}$-linearity of $\xi$.
 The vanishing of [{\it Leibniz rule}]-{\it relator}\,$(\xi)$ follows from
     the ${\Bbb Z}/2$-graded Leibniz rule of $\xi$, that the evaluation is {\it from the right} of $\xi$,
    and the sign rule for passing ${\Bbb Z}/2$-graded objects:
   $$
    \begin{array}{rcl}
     (d(f_1f_2))(\xi)
	   & :=\;
	     & \xi(f_1f_2)\;\;
	         =\;\; (\xi f_1)f_2 +(-1)^{p(\xi)p(f_1)}f_1 (\xi f_2)\\[1.2ex]
	   & \;=:
	     & (\xi)^{\leftarrow}\!\!\!((df_1)f_2) +   (\xi)^{\leftarrow}\!\!\!(f_1 df_2)\;\;
		     =:\;\; ((df_1)f_2)(\xi) +  (f_1 df_2)(\xi)\,.
	\end{array}
   $$
 The vanishing of  [{\it chain-rule identities from the $C^{\infty}$-hull structure}]-{\it relator}\,$(\xi)$
   follows from Lemma 1.3.4
   on the chain rule for $\xi$ applied to the pre-composition of
    $h\in C^{\infty}({\Bbb R}^l)$ with $(f_1,\,\cdots\,, f_l)$,
	where
	  $f_1,\,\cdots,\, f_l
	            \in  C^{\infty}(\mbox{$\bigwedge_{\,\Bbb R,\Bbb C}^{\even}\!
	                                                                (S^{\prime}\oplus S^{\prime\prime})$})$:
   (Recall that since $f_i$'s are even,
       $h(f_1,\,\cdots\,, f_l)$ and $(\partial_k h)(f_1,\,\cdots\,, f_l)$'s are all even as well.)
  $$
    \begin{array}{l}
    (d(h(f_1,\,\cdots\,, f_l)))(\xi)\;\;
	    :=\;\;   \xi(h(f_1,\,\cdots\,, f_l))\;\;
		 =\;\;  \sum_{k=1}^l((\partial_k h)(f_1,\,\cdots\,, f_l))\,\xi f_k   \\[1.2ex]
		
      \hspace{6em}
	    =:\;\; \sum_{k=1}^l (\xi)^{\leftarrow}\!\!\!((\partial_k h)(f_1,\,\cdots\,, f_l)\,df_k)\;\;
	    =:\; \sum_{k=1}^l ((\partial_k h)(f_1,\,\cdots\,, f_l)\,df_k)(\xi)\,.	
   \end{array}		
  $$
 
 Finally, for $\varpi=df$, $(df)(a\xi):= (a \xi)(f)= a\, (\xi f)=: a ((df)(\xi))$.
 It follows that for general $\varpi\in \Omega_{\widehat{X}}$,
  $\varpi(a\xi) := (a\xi)^{\leftarrow}\!\!\!\varpi =  a\, ((\xi)^{\leftarrow}\!\!\!\varpi)
      =: a\,\varpi(\xi)$.
	
 This completes the proof.	
 
\end{proof}

\bigskip

\begin{remark} $[$evaluation of $\Omega_{\widehat{X}}$ on
       $\Der_{\Bbb C}(\widehat{X})$ from the right vs.\ from the left$\,]$\; {\rm
 The evaluation of $\Omega_{\widehat{X}}$ on $\Der_{\Bbb C}(\widehat{X})$
   {\it from the right} corresponds to the pairing
    $$
	  \begin{array}{ccc}
	   \Der_{\Bbb C}(\widehat{X})\times \Omega_{\widehat{X}}
	      & \longrightarrow & C^{\infty}(\widehat{X})  \\[1.2ex]
       (\xi, df)   & \longmapsto &  \xi f
	  \end{array}
	$$
  while
  the evaluation of $\Omega_{\widehat{X}}$ on $\Der_{\Bbb C}(\widehat{X})$
   {\it from the left} corresponds to the pairing
    $$
	 \begin{array}{ccc}
        \Omega_{\widehat{X}}\times\Der_{\Bbb C}(\widehat{X})
		    & \longrightarrow    & C^{\infty}(\widehat{X})             \\[1.2ex]
	  (df, \xi)    & \longmapsto     & \;(-1)^{p(\xi)p(f)}\,\xi f\:.
	\end{array}
	$$	
 Due to the sign-factor $(-1)^{\mbox{\tiny $\bullet$}}$ when passing ${\Bbb Z}/2$-graded objects,
   the former is
     left $C^{\infty}(\widehat{X})$-linear for the $\Der_{\Bbb C}(\widehat{X})$-component
       and
	 right $C^{\infty}(\widehat{X})$-linear for the $\Omega_{\widehat{X}}$-component
 while the latter is
     left $C^{\infty}(\widehat{X})$-linear for the $\Omega_{\widehat{X}}$-component
       and
	 right $C^{\infty}(\widehat{X})$-linear for the $\Der_{\Bbb C}(\widehat{X})$-component.
 Such a distinction looks conceptually minor but becomes important in exact computations.
 In this work we only use evaluations from the right.
}\end{remark}

\bigskip

\begin{definition} {\bf [tangent sheaf ${\cal T}_{\widehat{X}}$ \&
              cotangent sheaf ${\cal T}^{\ast}_{\widehat{X}}$  of $\widehat{X}$]}\;
{\rm
 Replacing $X$ by open sets $U\subset X$ in the above construction
  gives the {\it tangent sheaf} ${\cal T}_{\widehat{X}}$ of $\widehat{X}$,
    with ${\cal T}_{\widehat{X}}(U):= \Der_{\Bbb C}(\widehat{U}) $,     and
  the {\it cotangent sheaf} ${\cal T}^{\ast}_{\widehat{X}}$ of $\widehat{X}$,
     with ${\cal T}^{\ast}_{\widehat{X}}(U):= \Omega_{\widehat{U}}$,
	 for $U\subset X$ open.
}\end{definition}

\bigskip

\begin{flushleft}
{\bf Differential forms on $\widehat{X}$}
\end{flushleft}
\begin{definition} {\bf [sign to ${\Bbb Z}/2$-permutation]}\; {\rm
 Given
   a $k$-tuple $a=(a_1,\,\cdots\,, a_k)$  of ${\Bbb Z}/2$-graded objects
      with each $a_i$ parity-homogeneous   and
   a permutation $\tau\in \Sym_k$ on the set of $k$ letters $\{1,\,\cdots\,, k\}$.
 Let $\tau= (i_1 i_2)\circ \cdots \circ (i_{2l-1}i_{2l})$
    be a decomposition of $\tau$ into a composition of transitions,
	where a transition $(ij)$ exchanges the label $i$ and the label $j$ and leaves all other labels of the letters intact.
 Define the sign associated to
  $\tau(a):= (a_{\tau(1)},\,\cdots\,, a_{\tau(k)})$ 	
  to be
  $$
   (-1)\,\!^{^\varsigma\!\tau}_a\;
	:=\; \mbox{$\prod_{j=1}^l  \left(\rule{0ex}{.9em}\right.\!\!
	      -(-1)^{p(a_{i_{2j-1}})\, p(a_{i_{2j}})}
		  \!\!\left.\rule{0ex}{.9em}\right)$}\,.
  $$
}\end{definition}

\medskip

\begin{lemma} {\bf [well-definedness of $(-1)\,\!^{^{\varsigma}\!\tau}_a$]}\\
 The sign $(-1)\,\!^{^\varsigma\!\tau}_a$ in Definition~1.3.10 is well-defined.
\end{lemma}

\medskip

\begin{proof}
 If there is at most one $a_i$ that is odd, then the lemma is a classical result from Algebra,
  e.g.\ [Ja].
 Thus, up to a relabelling, one may assume that
  $k\ge 2$ and that $p(a_i)=1$ exactly when $1\le i \le k_1\le k$ for some $k_1\ge 2$.
 Let $\tau=\gamma_l\cdots \gamma_1$ be the decomposition into disjoint cycles.
 Such decomposition is unique up to cyclic permutation within each cycle and re-orderings of the cycles.
 Since in such decomposition $\gamma_1,\,\cdots\,, \gamma_l$ commute with each other,
  up to a relabelling one may assume that exactly
    $\gamma_1, \,\cdots\,, \gamma_{l^{\prime}}$ have the property that
	all its entries are odd.
 In terms of this
  $$
    (-1)\,\!^{^\varsigma\!\tau}_a\;=\;
	(-1)^{\tau}(-1)^{\gamma_1}\,\cdots\,(-1)^{\gamma_{l^{\prime}}}\,,
  $$	
  where
    $(-1)^{\tau}, (-1)^{\gamma_1}, \cdots, (-1)^{\gamma_{l^{\prime}}}$
    take value $1$ (resp.\ $-1$) if the respective permutation in the exponent is even (resp.\ odd).
 This proves the lemma.	

\end{proof}

\medskip

\begin{definition} {\bf [$k$-form on $\widehat{X}$]}\; {\rm
 The bi-$C^{\infty}(\widehat{X})$-module of {\it $k$-forms} on $\widehat{X}$ is
  the bi-sub-$C^{\infty}(\widehat{X})$-module
   $\Omega^k_{\widehat{X}}
      := \bigwedge^k_{C^{\infty}(\widehat{X})}\Omega_{\widehat{X}}$
   of  the $k$-th tensor product
    $$
	  \otimes_{C^{\infty}(\widehat{X})}\Omega_{\widehat{X}} \;
	   :=\; \underbrace{\Omega_{\widehat{X}}
	            \otimes_{C^{\infty}(\widehat{X})}\, \cdots\cdots\,
        	    \otimes_{C^{\infty}(\widehat{X})}
		            \Omega_{\widehat{X}}}_{\mbox{\scriptsize $k$ times}}
	$$
  of $\Omega_{\widehat{X}}$	
  generated by
  $$
   \omega_1\wedge \,\cdots\, \wedge \omega_k\;
   :=\;  \sum_{\tau\in \scriptsizeSym_k}(-1)\,\!^{^\varsigma\!\tau}_{\omega}\,
              \omega_{\tau(1)}\otimes\,\cdots\,\otimes \omega_{\tau(k)}\,.
  $$
 Here,
   $\omega_1, \cdots, \omega_k\in \Omega_{\widehat{X}}$ parity-homogeneous,
   $\omega := \omega_1\otimes \cdots \otimes \omega_k$,   and
   $(-1)\,\!^{^\varsigma\!\tau}_{\omega}$  is the sign associated to
      $\tau(\omega):= \omega_{\tau(1)}\otimes\cdots\otimes \omega_{\tau(k)}$.
 In situations where $\omega$ can be understood from the context
    and it becomes cumbersome to carry the subscript
	$\omega$ in $(-1)\,\!^{^\varsigma\!\tau}_{\omega}$,
  we will denote it by $(-1)\,\!^{^\varsigma\!\tau}_{\mbox{\tiny $\bullet$}}$.
	
 An element of $\Omega^k_{\widehat{X}}$  is also called a {\it $k$-form} on $\widehat{X}$.
 It has {\it cohomological degree} ($\chd$) $k$.
}\end{definition}

\medskip

\begin{definition} {\bf [wedge product]}\; {\rm
 Define the {\it wedge product}
   $\wedge: \Omega^k_{\widehat{X}}\times  \Omega^l_{\widehat{X}}
      \rightarrow \Omega^{k+l}_{\widehat{X}}$ by setting
   $$
    (f_0\, df_1\wedge \,\cdots\, df_k )  \wedge (g_0\,dg_1\wedge \,\cdots\, dg_l)\;
	  :=\;  f_0\, df_1\wedge\,\cdots\,\wedge df_k\wedge (g_0\,dg_1)\wedge\,\cdots\,\wedge dg_l\,.   
   $$
}\end{definition}

\medskip

\begin{lemma} {\bf [basic property of wedge product]}\;
 Let
   $\alpha\in \Omega^k_{\widehat{X}}$ and
   $\beta\in \Omega^l_{\widehat{X}}$, both parity-homogeneous.
 Then
   $$
     \alpha\wedge \beta \;
	   =\; (-1)^{kl}(-1)^{p(\alpha)\,p(\beta)}\,
	       \beta\wedge \alpha\,.
   $$
\end{lemma}

\medskip

\begin{proof}
 With the notation from Definition~1.3.12,
 this follows from the identity
  \begin{eqnarray*}
   \lefteqn{
    \omega_1\wedge\,\cdots\,
	   \wedge \omega_{i-1}\wedge \omega_i\wedge \omega_{i+1}\wedge  \omega_{i+2}
	   \wedge \,\cdots\,\wedge \omega_{k^{\prime}} }\\[.8ex]
    && 	
	 =\; -(-1)^{p(\omega_i)p(\omega_{i+1})}\,
	       \omega_1\wedge\,\cdots\,
		     \wedge \omega_{i-1}\wedge \omega_{i+1}\wedge \omega_i\wedge  \omega_{i+2}
	              \wedge \,\cdots\,\wedge \omega_{k^{\prime}} 	
  \end{eqnarray*}
     in $\Omega^{k^{\prime}}_{\widehat{X}}$ for any $k^{\prime}$   and
  that $p(\omega_1\wedge \,\cdots\,\wedge \omega_l )
                =p(\omega_1)+\,\cdots\,+p(\omega_{l^{\prime}})$ for any $l^{\prime}$.
 Here,
   $\omega_1,\,\cdots\,,\, \mbox{$\omega_{k^{\prime}}$ or $\omega_{l^{\prime}}$}$
       $\in \Omega_{\widehat{X}} = \Omega^1_{\widehat{X}}$.
 Collectively in the end,
   the ``$-$"-signs contribute to the $(-1)^{kl}$-factor
   while the $(-1)^{p(\omega_i)p(\omega_{i+1})}$-factors contribute to
      the $(-1)^{p(\alpha)p(\beta)}$-factor.
     
\end{proof}

\medskip

\begin{definition} {\bf [exterior differential $d$]}\; {\rm
  The operation
    $d: C^{\infty}(\widehat{X})
	       \rightarrow \Omega^1_{\widehat{X}}:= \Omega_{\widehat{X}}$
   extends to the {\it exterior differential operator}
    $d:\Omega^k\rightarrow \Omega^{k+1}$ defined by
	$$
	   f_0\, df_1\wedge \,\cdots\, df_k\; \longmapsto\; df_0\wedge df_1\wedge\,\cdots\,\wedge df_k\,.
	$$
}\end{definition}

\medskip

\begin{lemma} {\bf [basic property of $d$]}\;
 $(1)$ $d\circ d=0$.
 
 $(2)$
 Let
   $\alpha\in \Omega^k_{\widehat{X}}$ and $\beta\in \Omega^l_{\widehat{X}}$.
 Then
   $$
     d(\alpha\wedge \beta) \;
	   =\; d\alpha\wedge \beta\,+\,(-1)^k \alpha\wedge d\beta\,.
   $$
\end{lemma}

\medskip

\begin{proof}
 Statement (1) follows from $d(df)=0$ for $f\in C^{\infty}(\widehat{X})$.
 
 Statement (2) follows from the special case
  $\alpha = f_0\,df_1\wedge \,\cdots\,\wedge df_k$  and
  $\beta=g_0\,dg_1\wedge \,\cdots\,\wedge dg_l$
    with $f_1, \,\cdots\,, f_k, g_0,\, \cdots,\, g_l$ parity-homogeneous,
  in which
  {\small
   \begin{eqnarray*}
    \lefteqn{d(\alpha\wedge \beta)\;
	 =\; d((f_0\, df_1\wedge \,\cdots\,\wedge df_k)
	   \wedge (g_0\, dg_1\wedge \,\cdots\,\wedge dg_l))   }\\[.8ex]
    &&
	 =\;  (-1)^{(p(f_1)+\,\cdots\,+p(f_k))\,p(g_0)}
  	           d(f_0g_0)
			    \wedge df_1\wedge\,\cdots\,\wedge df_k
			    \wedge dg_1\wedge \,\cdots\,\wedge dg_l 	      \\[.8ex]
    &&
	 =\; (-1)^{(p(f_1)+\,\cdots\,+p(f_k))\,p(g_0)}
  	          ((df_0)g_0 + f_0dg_0))
			    \wedge df_1\wedge\,\cdots\,\wedge df_k
			    \wedge dg_1\wedge \,\cdots\,\wedge dg_l 	  \\[.8ex]		
	&&
	 =\; (df_0\wedge df_1\wedge \,\cdots\,\wedge df_k)
	             \wedge (g_0\,dg_1\wedge \,\cdots\,\wedge dg_l)\,
			+\, (-1)^k\,(f_0\,df_1\wedge\,\cdots\,\wedge df_k )
			    \wedge (dg_0\wedge \,\cdots\,\wedge dg_l)   \\[.8ex]
    &&
     =\;  d\alpha\wedge \beta\,+\, (-1)^k\,\alpha\wedge d\beta\,.    				
   \end{eqnarray*}} 
\end{proof}

\bigskip

Note that the identity in Lemma~1.3.16, Statement (2),
  is nothing but the Leibnitz rule for $({\Bbb Z}\times ({\Bbb Z}/2))$-graded objects
  with the bi-grading of $d$ being $(1,0)$;
cf.\ Convention~1.3.5.

\bigskip

\begin{definition} {\bf [evaluation of $\Omega^k_{\widehat{X}}$
       on $\times_k \Der_{\Bbb C}(\widehat{X})$ from the right]}\; {\rm
 The {\it evaluation} of $\Omega^k_{\widehat{X}}$ on
       $\times_k \Der_{\Bbb C}(\widehat{X})$ {\it from the right}
	   is defined by setting\footnote{{\it Convention on the cohomological degree of derivation}\;\;	
	                                                          The defining equation $df(\xi):= \xi f$ leads to an ambiguity
															     as to how one should assign the cohomological degree to a derivation.
                                                               Since both $f$ and $\xi f \in C^{\infty}(\widehat{X})$
															       have cohomological degree $0$,
															     it is natural to set $\footnotesizechd(\xi)=0$.
	                                                          This convention is chosen throughout the work
															     so that the evaluation resumes to the standard form
	                                                            in the commutative case
											                  (i.e.\ when $\xi_1,\,\cdots\,,\xi_k$ and $f_0,\,\cdots\,, f_k$ are all even)
										                        without an additional $(-1)^{\mbox{\tiny $\bullet$}}$-factor.
                                                              On the other hand, since $\footnotesizechd(df)=1$ while
															     $\footnotesizechd(df(\xi))=0$, it is not-too-unnatural to set
																 $\chd(\xi)=-1$.
															  In which case
                                                                \begin{eqnarray*}
                                                                  \lefteqn{
                                                                  (f_0\, df_1\wedge \,\cdots\,\wedge df_k) (\xi_1,\,\cdots\,, \xi_k)\;
                                                                      :=\;   (\xi_1,\,\cdots\,, \xi_k)^{\leftarrow}\!\!\!
																					 (f_0\,df_1\wedge\,\cdots\,\wedge df_k)}\\[.2ex]
                                                                  &&	
	                                                               :=\;  (-1)^{k(k-1)/2}\,
																            \sum_{\tau\in \tinySym_k}
	                                                                       (-1)\,\!^{^{\varsigma}\!\tau}_{\mbox{\tiny $\bullet$}}\,
																           (-1)
																              ^{p(\xi_1)(p(f_0))
	                                                                                 +\,\cdots\,
																					 + p(\xi_k)
																					    (p(f_0)+ p(f_{\tau(1)})
																						           +\,\cdots\,+p(f_{\tau(k-1)}))}\,
	                                                                         f_0\, ((\xi_1)^{\leftarrow}\!\!\!df_{\tau(1)})\,
																			  \cdots\,
																			((\xi_k)^{\leftarrow}\!\!\!df_{\tau(k)})             \\[.2ex]
                                                                  &&								
	                                                               =:\; (-1)^{k(k-1)/2}\,
																            \sum_{\tau\in \scriptsizeSym_k}
	                                                                      (-1)\,\!^{^{\varsigma}\!\tau}_{\mbox{\tiny $\bullet$}}\,
																          (-1)
																            ^{p(\xi_1)(p(f_0))
	                                                                                 +\,\cdots\,
																					 + p(\xi_k)
																					     (p(f_0)+p(f_{\tau(1)})
																						   +\,\cdots\,+p(f_{\tau(k-1)}))}\,
	                                                                        f_0\,(df_{\tau(1)}(\xi_1))\,
																			\cdots\, (df_{\tau(k)}(\xi_k))\,.
                                                                \end{eqnarray*}  	   															  															  								  
                                                            }
  {\small
  \begin{eqnarray*}
   \lefteqn{
   (f_0\, df_1\wedge \,\cdots\,\wedge df_k) (\xi_1,\,\cdots\,, \xi_k)\;
     :=\;   (\xi_1,\,\cdots\,, \xi_k)^{\leftarrow}\!\!\!(f_0\,df_1\wedge\,\cdots\,\wedge df_k)}\\[.8ex]
    &&
	 :=\;   \sum_{\tau\in \scriptsizeSym_k}
	          (-1)\,\!^{^{\varsigma}\!\tau}_{\mbox{\tiny $\bullet$}}\,
   	          (-1)^{p(\xi_1)(p(f_0))
	                          +\,\cdots\,
							  + p(\xi_k)(p(f_0)+p(f_{\tau(1)})+\,\cdots\,+p(f_{\tau(k-1)}))}\,
	           f_0\, ((\xi_1)^{\leftarrow}\!\!\!df_{\tau(1)})\,
			                       \cdots\,  ((\xi_k)^{\leftarrow}\!\!\!df_{\tau(k)})
			                        \\[.8ex]
    &&								
	 =:\;  \sum_{\tau\in \scriptsizeSym_k}
	         (-1)\,\!^{^{\varsigma}\!\tau}_{\mbox{\tiny $\bullet$}}\,	
	         (-1)^{p(\xi_1)(p(f_0))
	                          +\,\cdots\,
							  + p(\xi_k)(p(f_0)+p(f_{\tau(1)})+\,\cdots\,+p(f_{\tau(k-1)}))}\,
	        f_0\,(df_{\tau(1)}(\xi_1))\,\cdots\, (df_{\tau(k)}(\xi_k))\,.
  \end{eqnarray*}}
}\end{definition}

\medskip

\begin{remark} $[$evaluation of $\Omega^{\,k}_{\widehat{X}}$ on
       $\times_k\Der_{\Bbb C}(\widehat{X})$ from the right vs.\ from the left$\,]$\; {\rm
 Similar to the evaluation of $\Omega_{\widehat{X}}$ on $\Der_{\Bbb C}(\widehat{X})$
  (cf.\ Remark~1.3.8),
 the evaluation of $\Omega^{\,k}_{\widehat{X}}$ on $\times_k\Der_{\Bbb C}(\widehat{X})$
   {\it from the right} corresponds to the pairing
    $$
	  \begin{array}{ccc}
	 (\times_k \Der_{\Bbb C}(\widehat{X}))\times  \Omega^{\,k}_{\widehat{X}}
	      & \longrightarrow & C^{\infty}(\widehat{X})  \\[1.2ex]
     (\xi_1,\,\cdots\,,\xi_k; f_0\,df_1\wedge \,\cdots\, \wedge df_k)   & \longmapsto
	      &   \sum_{\scriptsizeSym_k}
		        (-1)^{\mbox{\LARGE $\cdot$}_R}_{\mbox{\tiny $\bullet$},\tau}   \,
				  f_0\, (\xi_1 f_{\tau(1)})  \,\cdots\, (\xi_k f_{\tau(k)} )
	  \end{array}
	$$
  while
  the evaluation of $\Omega^{\,k}_{\widehat{X}}$
   on $\times_k\Der_{\Bbb C}(\widehat{X})$ {\it from the left} corresponds to the pairing
    $$
	 \begin{array}{ccc}
        \Omega^{\,k}_{\widehat{X}}\times (\times_k\Der_{\Bbb C}(\widehat{X}))
		    & \longrightarrow    & C^{\infty}(\widehat{X})             \\[1.2ex]
	  (f_0\,df_1\wedge\,\cdots\,\wedge df_k; \xi_1,\,\cdots\,,\, \xi_k)
	      & \longmapsto
		  & \;  \sum_{\scriptsizeSym_k}
		        (-1)^{\mbox{\LARGE $\cdot$}_L}_{\mbox{\tiny $\bullet$},\tau}   \,
				  f_0\, (\xi_1 f_{\tau(1)})  \,\cdots\, (\xi_k f_{\tau(k)} )   \:,
	\end{array}
	$$	
   for $\xi_1,\,\cdots\,, \xi_k$ and $f_0, f_1,\,\cdots\,, f_k$	parity-homogeneous.
 Here,
  $$
   \begin{array}{lcl}
    (-1)^{\mbox{\LARGE $\cdot$}_R}_{\mbox{\tiny $\bullet$},\tau}         & =
	   &  (-1)\,\!^{^{\varsigma}\!\tau}_{\mbox{\tiny $\bullet$}}\,
   	        (-1)^{p(\xi_1)(p(f_0))
	                          +\,\cdots\,
							  + p(\xi_k)(p(f_0)+p(f_{\tau(1)})+\,\cdots\,+p(f_{\tau(k-1)}))}
							  \\[1.2ex]
    (-1)^{\mbox{\LARGE $\cdot$}_L}_{\mbox{\tiny $\bullet$}, \tau} 	   & =
	   &(-1)\,\!^{^{\varsigma}\!\tau}_{\mbox{\tiny $\bullet$}}\,
	     (-1)^{p(\xi_1)(p(f_{\tau(k)})+ \,\cdots\,+ p(f_{\tau(1)}))
		                  +\,\cdots\,
						  + p(\xi_k)p(f_{\tau(k)})}\,.      	      	
   \end{array}
  $$
 In this work we only use evaluations from the right.
}\end{remark}

\medskip

\begin{definition}
{\bf [sheaf $\bigwedge^k{\cal T}^{\ast} _{\widehat{X}}$ of {\boldmath $k$}-forms]}\; {\rm
 Replacing $X$ by open sets $U\subset X$ in the construction of $\Omega^{\,k}_{\widehat{X}}$
  gives the {\it bi-$\widehat{\cal O}_X$-module
    $\bigwedge^k {\cal T}^{\ast}_{\widehat{X}}$  of  $k$-forms} on $\widehat{X}$,
     with $(\bigwedge^k{\cal T}^{\ast}_{\widehat{X}})(U):= \Omega^{\,k}_{\widehat{U}}$.
 The global construction of the exterior differential operator $d$ and the wedge product $\wedge$ pass to local constructions
  to give the chain complex
  $$
   \xymatrix{
    \widehat{\cal O}_X\; \ar[r]^-{d}
	  & {\cal T}^{\ast}_{\widehat{X}} \ar[r]^-{d}
	  & \; \bigwedge^2{\cal T}^{\ast}_{\widehat{X}}\; \ar[r]^-{d}
	  & \; \bigwedge^3{\cal T}^{\ast}_{\widehat{X}}\; \ar[r]^-{d}   & \;\cdots  	}
  $$
  and the wedge product of bi-$\widehat{\cal O}_X$-modules
  $$
    \wedge\; :\; \mbox{$\bigwedge$}^k{\cal T}^{\ast}_{\widehat{X}}
	                         \times \mbox{$\bigwedge$}^l{\cal T}^{\ast}_{\widehat{X}}\;
		\longrightarrow\; \mbox{$\bigwedge$}^{k+l}{\cal T}^{\ast}_{\widehat{X}}\,.
  $$
}\end{definition}

\bigskip

\begin{flushleft}
{\bf Fermionic integrations over $\widehat{X}$}
\end{flushleft}
For the purpose of the current work,
  a {\it fermionic integration} of a function $f$ on the superspace $\widehat{X}$
    is in effect the same as {\it taking some specified component of $f$
    with respect to the standard coordinate-functions $(x,\theta,\bar{\theta})$
     on $\widehat{X}$}.\footnote{In this work
                                                                    	 this is the only coordinate-functions for which we will consider
																		 fermionic integrations.
																	 In general, one needs to understand how the integration behaves
                                                                         when using different coordinate-functions.
																     } 
There are three kinds of integrations over fermionic variables we will take in this work:
Let
 \begin{eqnarray*}
   \lefteqn{
    f(x,\theta,\bar{\theta})\; =\;
      f_{(0)}(x)\,
	  +\, \sum_{\alpha=1}^2 f_{(\alpha)}(x) \theta^{\alpha}\,
      +\, \sum_{\dot{\alpha}=\dot{1}}^{\dot{2}}
	             f_{(\dot{\alpha})}(x)\bar{\theta}^{\dot{\alpha}} }\\[-.6ex]
   && \hspace{6em}				 				
      +\, f_{(12)}(x)\theta^1\theta^2\, 	
	  +\, f_{(\dot{1}\dot{2})}(x)\bar{\theta}^{\dot{1}}\bar{\theta}^{\dot{2}}\,
      +\, \sum_{\alpha=1}^2\sum_{\dot{\beta}=\dot{1}}^{\dot{2}}
                 f_{(\alpha\dot{\beta})}(x)\theta^{\alpha}\bar{\theta}^{\dot{\beta}}	
				                                                                                                                               \\[-.6ex]	
    && \hspace{8em}
      +\, \sum_{\dot{\alpha}=\dot{1}}^{\dot{2}}
	         f_{(12\dot{\alpha})}(x) \theta^1\theta^2\bar{\theta}^{\dot{\alpha}}\,
      +\, \sum_{\alpha=1}^2
	         f_{(\alpha \dot{1}\dot{2})}(x)
			     \theta^{\alpha}\bar{\theta}^{\dot{1}}\bar{\theta}^{\dot{2}}\,			       				
      +\, f_{(12\dot{1}\dot{2})}(x)
	             \theta^1\theta^2\bar{\theta}^{\dot{1}}\bar{\theta}^{\dot{2}}
 \end{eqnarray*}
 be a function on $\widehat{X}$.
Then
 \begin{eqnarray*}
   \int_{\widehat{X}} f(x,\theta,\bar{\theta})\,  d\theta^2d\theta^1 d^4\!x
        & := & \int_X f_{(12)}(x)\, d^4\!x \,, \\
   \int_{\widehat{X}} f(x,\theta,\bar{\theta})\,
               d\bar{\theta}^{\dot{2}}d\bar{\theta}^{\dot{1}} d^4\!x
      & := & \int_X f_{(\dot{1}\dot{2})}(x)\,d^4\!x \,, \\
   \int_{\widehat{X}} f(x,\theta,\bar{\theta})\,
               d\bar{\theta}^{\dot{2}}d\bar{\theta}^{\dot{1}}  d\theta^2d\theta^1 d^4\!x
	  & := & \int_X f_{(12\dot{1}\dot{2})}(x)\,d^4\!x \,. \\
 \end{eqnarray*}

We leave the general theory of integrations of fermionic variables
   from the aspect of super-$C^{\infty}$-algebraic geometry to future work.
Interested readers are referred to the more recent [Wi6] (2012) and key-word search for existing studies
 both in mathematics and physics.

\bigskip

\subsection{Supersymmetry transformations of and related objects on $\widehat{X}$}

The whole $N=1$ super Poincar\'{e} group acts on the $d=4$, $N=1$ superspace $\widehat{X}$
 with its Poicar\'{e} subgroup acts on the underlying Minkowski space-time $X$ as the group of isometries.
The supersymmetry-transformation part of the super Poincar\'{e} group is parameterized by
  constant sections\footnote{{\it Remark on Notations}:\;
                                                              This is denoted by $(\xi,\bar{\xi})$
															     in [We-B: Chap.\ IV, Eqs.\ (4.2), (4.3), (4.4)].
															  Here, we follow the convention of [We-B] that
															    sections $\xi$ of the chiral Weyl spinor bundle $S^{\prime}$
																  are denoted without a $\,\bar{}\,$
																while sections $\bar{\eta}$ of the antichiral Weyl spinor bundle
																    $S^{\prime\prime}$
																  are denoted with a $\,\bar{}\,$.
															  Since $S^{\prime\prime}$ is the complex conjugate of $S^{\prime}$,
															     this is consistent with `taking complex conjugation'.
															  In other words, $\bar{\eta}$ may be interpreted
															    as the complex conjugate of a section $\eta$ of $S^{\prime}$ as well.
															 {\it However},
															     we reserve the notation $(\xi, \bar{\xi})$ for a Majorana spinor,
																  i.e.\ a Dirac spinor whose antichiral component $\bar{\xi}$
																  is indeed the complex conjugate of its chiral component $\xi$.
															  }
  $(\xi,\bar{\eta})$ of $S^{\prime}\oplus S^{\prime\prime}$.
Regarded as elements in $\widehat{\cal O}_X$, these parameters are fermionic (i.e.\ odd).
While supersymmetry transformations act on the underlying topology $X$ of $\widehat{X}$ trivially,
 they act nontrivially on the super $C^{\infty}$-scheme $\widehat{X}$
   since they transform elements in the function-ring $C^{\infty}(\widehat{X})$ of $\widehat{X}$.
We take this action as our starting point that relates
  Supersymmetry on the physics side and super $C^{\infty}$-Algebraic Geometry on the mathematics side.

In this subsection, we recast
  a minimal subset of objects and notions related to the $d=4$, $N=1$ supersymmetry
  from the standard textbook\footnote{All
                                     the three by-now-classical physics textbooks
									  [G-G-R-S]:
  										 {\sl Superspace -- one thousand and one lessons in supersymmetry},
                                              by S.\ James Gates, Jr., Marcus T.\ Grisaru, Martin Roc\u{c}ek, and Warren Siegel,
									   [West1]:
										 {\sl Introduction to supersymmetry and supergravity},
									          by Peter West, and 									   
									   [We-B]:
										 {\sl Supersymmetry and supergravity},
										      by Julius Wess and Jonathan Bagger
									 on Supersymmetry \& Supergravity have influenced our understanding tremendously.	
                                     The setting and notations of [We-B] were later used in many string-theorists' work.
                                     As part of the preparations for sequels to the current notes,
									   we follow [We-B] as much as we can and as long as there are no severe notational conflicts
									     with the existing notation system in works in the the D-project.
									 }
   [We-B], {\sl Supersymmetry and Supergravity}, by Julius Wess and Jonathan Bagger
  that we need into the super $C^{\infty}$-Algebraic Geometry setting.
See also [G-G-R-S] and [West1].

\bigskip

\begin{flushleft}
{\bf The supersymmetry transformations of the $d=4$, $N=1$ superspace $\widehat{X}$}
\end{flushleft}
Recall the standard generators in collective notation
 $$
   x\;=\;(x^0,\,x^1,\, x^2,\, x^3)\,,\;\;\;
   \theta\; =\;(\theta^1,\,\theta^2)\,,\;\;\;
   \bar{\theta}\; =\; ( \bar{\theta}^{\dot{1}},\, \bar{\theta}^{\dot{2}})
 $$
 of the super $C^{\infty}$-ring $C^{\infty}(\widehat{X})$.
For a constant section
   $(\xi,\bar{\eta})= (\xi^1,\xi^2; \bar{\eta}^{\dot{1}}, \bar{\eta}^{\dot{2}})$
 of $S^{\prime}\oplus S^{\prime\prime}$,
 let $g(\xi,\bar{\eta})$ be the ${\Bbb Z}/2$-grading-preserving automorphism
  of $C^{\infty}(\widehat{X}) $
  defined by the correspondence
 $$
  \begin{array}{lll}
    x^{\mu}   & \longmapsto
	  & x^{\mu}
	     + \sqrt{-1}
		      (\theta \sigma^{\mu}\bar{\eta}^t  -  \xi \sigma^{\mu}\bar{\theta}^t)\\
	\theta^{\alpha}    & \longmapsto   &  \theta^{\alpha} + \xi^{\alpha} \\
	\bar{\theta}^{\dot{\beta}}
	& \longmapsto    & \bar{\theta}^{\dot{\beta}}+ \bar{\eta}^{\dot{\beta}}
	   \hspace{10em},
  \end{array}
 $$
    $\mu=0,1,2,3, \alpha=1,2, \dot{\beta}=\dot{1}, \dot{2}$,
 on the standard generators of $C^{\infty}(\widehat{X})$.
Here,
 $$ \small
   \sigma^0\;=\; \left(\!  \begin{array}{rr}-1 & 0  \\ 0 & -1 \end{array}\!\right)\,,\;\;\;
   \sigma^1\;=\; \left(\!  \begin{array}{rr} 0 & 1  \\ 1 &   0 \end{array}\!\right)\,,\;\;\;
   \sigma^2\;=\; \left(\!  \begin{array}{rr} 0 & -\sqrt{-1}  \\  \sqrt{-1} & 0 \end{array}\!\right)\,,\;\;\;
   \sigma^3\;=\; \left(\!  \begin{array}{rr} 1 & 0  \\ 0 & -1 \end{array}\!\right)
 $$
 are the {\it Pauli matrices} (whose entries are indexed as $\sigma^{\mu}_{\alpha\dot{\beta}}$)\,,
  denoted collectively as\\ $\mbox{\boldmath $\sigma$}=(\sigma^0,\sigma^1,\sigma^2,\sigma^3)$,  and
 $$ \small
   \bar{\eta}^t\;
     =\; \left(\!
	             \begin{array}{l} \bar{\eta}^{\dot{1}} \\ \bar{\eta}^{\dot{2}} \end{array}
		   \!\right)\,,\;\;\;
  \bar{\theta}^t\;
     =\; \left(\!
	             \begin{array}{l} \bar{\theta}^{\dot{1}} \\ \bar{\theta}^{\dot{2}} \end{array}
		   \!\right)\,.
 $$
This gives a ${\Bbb C}$-linear representation of
 the Abelian group of constant sections of $S^{\prime}\oplus S^{\prime\prime}$
 on the ${\Bbb Z}/2$-graded ${\Bbb C}$-algebra $C^{\infty}(\widehat{X})$:
 $$
  g_{(\xi_1+\xi_2 , \bar{\eta}_1+ \bar{\eta}_2)}\;
    =\; g_{(\xi_1, \bar{\eta}_1)}\circ g_{(\xi_2,\bar{\eta}_2)}\;
	=\; g_{(\xi_2, \bar{\eta}_2)}\circ g_{(\xi_1,\bar{\eta}_1)}\;
   \hspace{2em}\mbox{and}\hspace{2em}
  g_{(0,0)}\;=\; \Id_{C^{\infty}(\widehat{X})}\,.
 $$
Explicitly\footnote{See Second Proof of  Lemma~1.4.14,
                                   where a block-matrix form for superfields on $\widehat{X}$ is introduced
								   as a book-keeping tool for explicit computations:
								    $$
									 f\;=\;
     \left[
	  \begin{array} {c|cc|c}
	   f_{(0)}  &  f_{(\dot{1})} & f_{(\dot{2})}  & f_{(\dot{1}\dot{2})}
	                                                                 \\[.6ex] \hline  \rule{0ex}{1em}
	   f_{(1)}  &  f_{(1\dot{1})}& f_{(1\dot{2})}& f_{(1\dot{1}\dot{2})}
	                                                                 \\[.6ex]
	   f_{(2)}  & f_{(2\dot{1})} & f_{(2\dot{2})}& f_{(2\dot{1}\dot{2})}
	                                                                 \\[.6ex] \hline  \rule{0ex}{1em}
	   f_{(12)}& f_{(12\dot{1})} & f_{(12\dot{2})}	& f_{(12\dot{1}\dot{2})}
	  \end{array}
     \right]\,.	
									$$
                                  Such block-matrix form is also generalized to the case for elements
							       $\widehat{m}\in C^{\infty}(\widehat{X}^{\!A\!z})$ of the function ring
                                   of an Azumaya/matrix superspace (Sec.\ 2.3 \& 2.4).
                                  Tedious computations in the current work are computed in this format,
								   from which patterns of the results manifest themselves easily,
								   when there is no immediate conceptual shortcut to begin with.
								  These patterns arising from computations serve also a self-error-detecting tool
								    when there is a term, sign, or labelling index that is off-pattern.
									
                                  Cf.\ Appendix `Basic moves for the multiplication of two superfields'.									
                                   }, 
 for
 {\small
 \begin{eqnarray*}
    f &  =  & f(x,\theta,\bar{\theta})\;  \\
	  & =   &  f_{(0)}(x)\,
	    +\, \sum_{\alpha=1}^2 f_{(\alpha)}(x) \theta^{\alpha}\,
        +\, \sum_{\dot{\beta}=\dot{1}}^{\dot{2}}
	             f_{(\dot{\beta})}(x)\bar{\theta}^{\dot{\beta}}
      +\, f_{(12)}(x)\theta^1\theta^2\, 	
	  +\, f_{(\dot{1}\dot{2})}(x)\bar{\theta}^{\dot{1}}\bar{\theta}^{\dot{2}}\,\\
  &&
      +\, \sum_{\alpha=1}^2\sum_{\dot{\beta}=\dot{1}}^{\dot{2}}
                 f_{(\alpha\dot{\beta})}(x)\theta^{\alpha}\bar{\theta}^{\dot{\beta}}	
      +\, \sum_{\dot{\beta}=\dot{1}}^{\dot{2}}
	         f_{(12\dot{\beta})}(x) \theta^1\theta^2\bar{\theta}^{\dot{\beta}}\,
      +\, \sum_{\alpha=1}^2
	         f_{(\alpha \dot{1}\dot{2})}(x)
			     \theta^{\alpha}\bar{\theta}^{\dot{1}}\bar{\theta}^{\dot{2}}\,			       				
      +\, f_{(12\dot{1}\dot{2})}(x)
	             \theta^1\theta^2\bar{\theta}^{\dot{1}}\bar{\theta}^{\dot{2}}  \;\;\; \\
   &\in & C^{\infty}(\widehat{X})\,,  				
 \end{eqnarray*}}
 {\small
 \begin{eqnarray*}
  g_{(\xi,\bar{\eta})}(f)
  & =  &   f \mbox{\large $($}
                      x + \sqrt{-1}(\theta \mbox{\boldmath $\sigma$} \bar{\eta}^t
	                                         - \xi \mbox{\boldmath $\sigma$} \bar{\theta}^t),
	                       \theta+\xi , \bar{\theta}+\bar{\eta}
				  \mbox{\large $)$}\;     \\[-.6ex]
  & =   &  f_{(0)}\mbox{\large $($}
                                           x+ \sqrt{-1}(\theta \mbox{\boldmath $\sigma$} \bar{\eta}^t
	                                         - \xi \mbox{\boldmath $\sigma$} \bar{\theta}^t)
									   \mbox{\large $)$}\,
	    +\, \sum_{\alpha=1}^2
		        f_{(\alpha)} \mbox{\large $($}
				                          x+ \sqrt{-1}(\theta \mbox{\boldmath $\sigma$} \bar{\eta}^t
	                                               - \xi \mbox{\boldmath $\sigma$} \bar{\theta}^t)
												  \mbox{\large $)$}\, (\theta^{\alpha}+\xi^\alpha)\,  \\
  &&												
        +\, \sum_{\dot{\beta}=\dot{1}}^{\dot{2}}
	             f_{(\dot{\beta})}\mbox{\large $($}
				                    x+ \sqrt{-1}(\theta \mbox{\boldmath $\sigma$} \bar{\eta}^t
	                                         - \xi \mbox{\boldmath $\sigma$} \bar{\theta}^t)
											                     \mbox{\large $)$}\,
											   (\bar{\theta}^{\dot{\beta}}+\bar{\eta}^{\dot{\beta}})\,  \\
  &&											
      +\, f_{(12)}\mbox{\large $($}
	                                x+ \sqrt{-1}(\theta \mbox{\boldmath $\sigma$} \bar{\eta}^t
	                                         - \xi \mbox{\boldmath $\sigma$} \bar{\theta}^t)
									\mbox{\large $)$}\,
										(\theta^1+\xi^1)(\theta^2+\xi^2)\, 	 \\
  &&										
	  +\, f_{(\dot{1}\dot{2})}\mbox{\large $($}
	                                  x+ \sqrt{-1}(\theta \mbox{\boldmath $\sigma$} \bar{\eta}^t
	                                         - \xi \mbox{\boldmath $\sigma$} \bar{\theta}^t)
											                       \mbox{\large $)$}\,
											 (\bar{\theta}^{\dot{1}}+\bar{\eta}^{\dot{1}})
											 (\bar{\theta}^{\dot{2}}+\bar{\eta}^{\dot{2}})\,\\
  &&
      +\, \sum_{\alpha=1}^2\sum_{\dot{\beta}=\dot{1}}^{\dot{2}}
                 f_{(\alpha\dot{\beta})}\mbox{\large $($}
				                          x+ \sqrt{-1}(\theta \mbox{\boldmath $\sigma$} \bar{\eta}^t
	                                         - \xi \mbox{\boldmath $\sigma$} \bar{\theta}^t)
											                               \mbox{\large $)$}\,
											 (\theta^{\alpha}+\xi^\alpha)
											 (\bar{\theta}^{\dot{\beta}}+\bar{\eta}^{\dot{\beta}})\, \\
  &&											
      +\, \sum_{\dot{\beta}=\dot{1}}^{\dot{2}}
	         f_{(12\dot{\beta})}\mbox{\large $($}
			                            x+ \sqrt{-1}(\theta \mbox{\boldmath $\sigma$} \bar{\eta}^t
	                                         - \xi \mbox{\boldmath $\sigma$} \bar{\theta}^t)
											                      \mbox{\large $)$}\,
											 (\theta^1+\xi^1)(\theta^2+\xi^2)
											 (\bar{\theta}^{\dot{\beta}}\bar{\eta}^{\dot{\beta}})\, \\
  &&											
      +\, \sum_{\alpha=1}^2
	         f_{(\alpha \dot{1}\dot{2})}
			   \mbox{\large $($}
			     x+ \sqrt{-1}(\theta \mbox{\boldmath $\sigma$} \bar{\eta}^t
	                                         - \xi \mbox{\boldmath $\sigma$} \bar{\theta}^t)
			   \mbox{\large $)$}\,
			     (\theta^{\alpha}+\xi^\alpha)
				 (\bar{\theta}^{\dot{1}}+\bar{\eta}^{\dot{1}})
				 (\bar{\theta}^{\dot{2}}+\bar{\eta}^{\dot{2}})\, \\
  &&				
      +\, f_{(12\dot{1}\dot{2})}
	               \mbox{\large $($}
				     x+ \sqrt{-1}(\theta \mbox{\boldmath $\sigma$} \bar{\eta}^t
	                                         - \xi \mbox{\boldmath $\sigma$} \bar{\theta}^t)
				   \mbox{\large $)$}\,
	             (\theta^1+\xi^1)(\theta^2+\xi^2)
				 (\bar{\theta}^{\dot{1}}+\bar{\eta}^{\dot{1}})
				 (\bar{\theta}^{\dot{2}}+\bar{\eta}^{\dot{2}})\,,  \;\;\; \\
 \end{eqnarray*}}with
 {\small
 \begin{eqnarray*}
  \lefteqn{f_{(\mbox{\tiny $\bullet$})}
     (x+ \sqrt{-1}(\theta \mbox{\boldmath $\sigma$} \bar{\eta}^t
	                                         - \xi \mbox{\boldmath $\sigma$} \bar{\theta}^t))\; }\\
  &&											
   =\; f_{(\mbox{\tiny $\bullet$})}(x)\,
         +\, \sum_\mu \partial_\mu f_{(\mbox{\tiny $\bullet$})}(x)
		         \cdot \sqrt{-1}(\theta \sigma^\mu \bar{\eta}^t - \xi \sigma^\mu \bar{\theta}^t))\, \\[-1ex]
    && \hspace{1.6em}				
		 -\, \mbox{\Large $\frac{1}{2}$}\,
		       \sum_{\mu, \nu} \partial_\mu\partial_\nu f_{(\mbox{\tiny $\bullet$})}(x)
		         \cdot (\theta \sigma^\mu \bar{\eta}^t - \xi \sigma^\mu \bar{\theta}^t))\,
				            (\theta \sigma^\nu \bar{\eta}^t - \xi \sigma^\nu \bar{\theta}^t)) \\[1ex]
  && 							
   =\;  f_{(\mbox{\tiny $\bullet$})}(x)\,
         -\,\sqrt{-1}\, \sum_\alpha  \xi^\alpha
		       \cdot \mbox{\normalsize $\sum$}_{\mu,\dot{\delta}}
			          \sigma^\mu_{\alpha\dot{\delta}}\bar{\theta}^{\dot{\delta}}
					   \partial_\mu  f_{(\mbox{\tiny $\bullet$})}(x)\,		
         -\,\sqrt{-1}\, \sum_{\dot{\beta}}  \bar{\eta}^{\dot{\beta}}
		       \cdot \mbox{\normalsize $\sum$}_{\mu,\gamma}
			      \theta^\gamma \sigma^\mu_{\gamma\dot{\beta}}
					   \partial_\mu  f_{(\mbox{\tiny $\bullet$})}(x)\,  \\
   && \hspace{1.6em}					
      +\, \xi^1\xi^2\cdot
             \mbox{\normalsize $\sum$}_{\mu,\nu}	
			   (\sigma^\mu_{1\dot{2}}\sigma^\nu_{2\dot{1}}
			       - \sigma^\mu_{1\dot{1}}\sigma^\nu_{2\dot{2}})
				  \bar{\theta}^{\dot{1}}\bar{\theta}^{\dot{2}}
				  \partial_\mu\partial_\nu f_{(\mbox{\tiny $\bullet$})}(x)\,
      -\, \sum_{\alpha,\dot{\beta}} \xi^\alpha \bar{\eta}^{\dot{\beta}} \cdot
		     \mbox{\normalsize $\sum$}_{\mu,\nu, \gamma,\dot{\delta}}
			  \theta^\gamma   \sigma^\mu_{\alpha \dot{\delta}}
			  \sigma^\mu_{\gamma\dot{\beta}}  \bar{\theta}^{\dot{\delta}}
			  \partial_\mu\partial_\nu
			f_{(\mbox{\tiny $\bullet$})}(x)\,   \\
   && \hspace{1.6em}       		
      +\, \bar{\eta}^{\dot{1}}\bar{\eta}^{\dot{2}}\cdot
             \mbox{\normalsize $\sum$}_{\mu,\nu}	
			     \theta^1\theta^2
			   (\sigma^\mu_{1\dot{1}}\sigma^\nu_{2\dot{2}}
			       - \sigma^\mu_{2\dot{1}}\sigma^\nu_{1\dot{2}})
				  \partial_\mu\partial_\nu f_{(\mbox{\tiny $\bullet$})}(x)
 \end{eqnarray*}}by 
 the $C^{\infty}$-hull structure of $C^{\infty}(\widehat{X})$  and
 the fact that
  $\mbox{\large $($}\sqrt{-1}(\theta\sigma \bar{\eta}^t -\xi\sigma \bar{\theta}^t)
     \mbox{\large $)$}^3=0$
  since
    $\xi^{\alpha}$'s (resp.\ $\bar{\eta}^{\dot{\beta}}$'s)
	  are ${\Bbb C}$-linear combinations of $\theta^1$ and $\theta^2$
	  (resp.\ $\bar{\theta}^{\dot{1}}$ and $\bar{\theta}^{\dot{2}}$).
 After the expansion
  
 {\small
  \begin{eqnarray*}
   \lefteqn{
     f(\,\mbox{\tiny $\bullet$}\,, \theta+\xi, \bar{\theta}+\bar{\eta}) }\\	
	 &&
	  =\;  f_{(0)}(\mbox{\tiny $\bullet$})\,
	    +\, \sum_{\alpha=1}^2
		        f_{(\alpha)}(\mbox{\tiny $\bullet$}) (\theta^{\alpha}+\xi^\alpha)\,
        +\, \sum_{\dot{\beta}=\dot{1}}^{\dot{2}}
	             f_{(\dot{\beta})}(\mbox{\tiny $\bullet$})
				(\bar{\theta}^{\dot{\beta}}+\bar{\eta}^{\dot{\beta}})\,
        +\, f_{(12)}(\mbox{\tiny $\bullet$})	(\theta^1+\xi^1)(\theta^2+\xi^2)\, 	  \\
     && \hspace{1.6em}		
	  +\, f_{(\dot{1}\dot{2})}(\mbox{\tiny $\bullet$})
	                                         (\bar{\theta}^{\dot{1}}+\bar{\eta}^{\dot{1}})
											 (\bar{\theta}^{\dot{2}}+\bar{\eta}^{\dot{2}})\,
      +\, \sum_{\alpha=1}^2\sum_{\dot{\beta}=\dot{1}}^{\dot{2}}
                 f_{(\alpha\dot{\beta})}(\mbox{\tiny $\bullet$})																		   
											 (\theta^{\alpha}+\xi^\alpha)
											 (\bar{\theta}^{\dot{\beta}}+\bar{\eta}^{\dot{\beta}})\, \\
     && \hspace{1.6em}
      +\, \sum_{\dot{\beta}=\dot{1}}^{\dot{2}}
	         f_{(12\dot{\beta})}(\mbox{\tiny $\bullet$})																  
											 (\theta^1+\xi^1)(\theta^2+\xi^2)
											 (\bar{\theta}^{\dot{\beta}}\bar{\eta}^{\dot{\beta}})\,											 
      +\, \sum_{\alpha=1}^2
	         f_{(\alpha \dot{1}\dot{2})}(\mbox{\tiny $\bullet$})			
			     (\theta^{\alpha}+\xi^\alpha)
				 (\bar{\theta}^{\dot{1}}+\bar{\eta}^{\dot{1}})
				 (\bar{\theta}^{\dot{2}}+\bar{\eta}^{\dot{2}})\, \\
     && \hspace{1.6em}				
      +\, f_{(12\dot{1}\dot{2})}(\mbox{\tiny $\bullet$})				
	             (\theta^1+\xi^1)(\theta^2+\xi^2)
				 (\bar{\theta}^{\dot{1}}+\bar{\eta}^{\dot{1}})
				 (\bar{\theta}^{\dot{2}}+\bar{\eta}^{\dot{2}})	 	 \\[1ex]
    && =\;
      f(\,\mbox{\tiny $\bullet$}\,,\theta,\bar{\theta})\,
	  +\, \sum_\alpha \xi^\alpha
	          \cdot \frac{\partial}{\partial\theta^\alpha}
	                       f(\,\mbox{\tiny $\bullet$}\,,\theta,\bar{\theta})
	  +\, \sum_{\dot{\beta}} \bar{\eta}^{\dot{\beta}}
	       \cdot \frac{\partial}{\partial\bar{\theta}^{\dot{\beta}}\rule{0ex}{.8em}}
		             f(\,\mbox{\tiny $\bullet$}\,,\theta,\bar{\theta})\,
      +\, \xi^1\xi^2
            \cdot \frac{\partial^2}{\partial\theta^2  \partial\theta^1}	
                       f(\,\mbox{\tiny $\bullet$}\,,\theta,\bar{\theta})\,	\\
     && \hspace{1.6em}
      +\, \sum_{\alpha,\dot{\beta}}	 \xi^\alpha \bar{\eta}^{\dot{\beta}}
	         \cdot \frac{\partial^2}
			                    {\partial\bar{\theta}^{\dot{\beta}}\partial\theta^\alpha\rule{0ex}{.8em}}
								  f(\,\mbox{\tiny $\bullet$}\,,\theta,\bar{\theta})\,								
     +\, \bar{\eta}^{\dot{1}}\bar{\eta}^{\dot{2}}
           \cdot \frac{\partial^2}
			                  {\partial\bar{\theta}^{\dot{2}}\partial\bar{\theta}^{\dot{1}}
							      \rule{0ex}{.8em}}
								  f(\,\mbox{\tiny $\bullet$}\,,\theta,\bar{\theta})\,		
     +\, \sum_{\dot{\beta}}	 \xi^1\xi^2 \bar{\eta}^{\dot{\beta}}
	         \cdot \frac{\partial^3}
			                    {\partial\bar{\theta}^{\dot{\beta}}
								    \partial\theta^2\partial\theta^1\rule{0ex}{.8em}}
								  f(\,\mbox{\tiny $\bullet$}\,,\theta,\bar{\theta})\,  \\
    && \hspace{1.6em}
	 +\, \sum_\alpha	 \xi^\alpha \bar{\eta}^{\dot{1}}\bar{\eta}^{\dot{2}}
	         \cdot \frac{\partial^3}
			                    {\partial\bar{\theta}^{\dot{2}}\partial\bar{\theta}^{\dot{1}}
								    \partial\theta^\alpha\rule{0ex}{.8em}}
								  f(\,\mbox{\tiny $\bullet$}\,,\theta,\bar{\theta})\,			
     +\, \xi^1\xi^2 \bar{\eta}^{\dot{1}}\bar{\eta}^{\dot{2}}
	         \cdot \frac{\partial^4}
			                    {\partial\bar{\theta}^{\dot{2}}\partial\bar{\theta}^{\dot{1}}
								    \partial\theta^2\partial\theta^1\rule{0ex}{.8em}}
								  f(\,\mbox{\tiny $\bullet$}\,,\theta,\bar{\theta})
  \end{eqnarray*}}
   
  \noindent
  in $(\xi, \bar{\eta})$ and
  the substitution of
    $f_{(\mbox{\tiny $\bullet$})}
    (x+ \sqrt{-1}(\theta \mbox{\boldmath $\sigma$} \bar{\eta}^t
	                                         - \xi \mbox{\boldmath $\sigma$} \bar{\theta}^t))$,
  one obtains
 
 {\small
  \begin{eqnarray*}
    g_{(\xi,\bar{\eta})}(f)
	 & = & f\,
	   +\, \sum_{\alpha}\xi^{\alpha}
	         \cdot \mbox{\Large $($}
			   \mbox{\large $\frac{\partial}{\partial\theta^\alpha}$}
                  -\sqrt{-1}\,\mbox{\normalsize $\sum$}_{\mu,\dot{\beta}}
                         \sigma^\mu_{\alpha\dot{\beta}} \bar{\theta}^{\dot{\beta}}
						 \mbox{\large $\frac{\partial}{\partial x^\mu}$}
			           \mbox{\Large $)$}\,f\,
     +\, \sum_{\dot{\beta}}\bar{\eta}^{\dot{\beta}}
	         \cdot \mbox{\Large $($}
			   \mbox{\large $\frac{\partial}{\partial\theta^{\dot{\beta}}\rule{0ex}{.8em}}$}
                  -\sqrt{-1}\,\mbox{\normalsize $\sum$}_{\mu,\alpha}
                         \theta^\alpha \sigma^\mu_{\alpha\dot{\beta}}
						 \mbox{\large $\frac{\partial}{\partial x^\mu}$}
			           \mbox{\Large $)$}\,f\,  \\
	&& +\; \xi^1\xi^2
	            \cdot \mbox{\Large $($}
				        \mbox{\large $\frac{\partial^2}{\partial\theta^2\partial\theta^1}$}
						 - \sqrt{-1}\,\mbox{\normalsize $\sum$}_{\mu,\dot{\beta}}\,
						    \sigma^\mu_{1\dot{\beta}}\bar{\theta}^{\dot{\beta}}
							\mbox{\large $\frac{\partial^2}{\partial x^\mu\partial\theta^2}$}
				         +\sqrt{-1}\,\mbox{\normalsize $\sum$}_{\mu,\dot{\beta}}\,
						    \sigma^\mu_{2\dot{\beta}}\bar{\theta}^{\dot{\beta}}
							\mbox{\large $\frac{\partial^2}{\partial x^\mu\partial\theta^1}$}		 \\
        && \hspace{8em}							
						 + \mbox{\normalsize $\sum$}_{\mu,\nu}\,
						    (\sigma^\mu_{1\dot{2}}\sigma^\nu_{2\dot{1}}
							    - \sigma^\mu_{1\dot{1}}\sigma^\nu_{2\dot{2}})
						    \bar{\theta}^{\dot{1}}\bar{\theta}^{\dot{2}}
							\mbox{\large $\frac{\partial^2}{\partial x^\mu \partial x^\nu}$}
							     \mbox{\Large $)$}f    \\[1ex]
    && +\, \sum_{\alpha\dot{\beta}}	 \xi^\alpha \bar{\eta}^{\dot{\beta}}
	              \cdot \mbox{\Large $($}
				    \mbox{\large $\frac{\partial^2}
					                                        {\partial\bar{\theta}^{\dot{\beta}}
															    \partial\theta^\alpha\rule{0ex}{.8em}}$}
				   - \sqrt{-1}\, \mbox{\normalsize $\sum$}_{\mu, \dot{\delta}}\,
				          \sigma^\mu_{\alpha\dot{\delta}}\bar{\theta}^{\dot{\delta}}
						   \mbox{\large $\frac{\partial^2}
						                                          {\partial x^\mu\rule{0ex}{.8em}
				     											      \partial\bar{\theta}^{\dot{\beta}}}$}
                   	+\sqrt{-1}\,\mbox{\normalsize $\sum$}_{\mu, \gamma}\,
					        \theta^\gamma\sigma^\mu_{\gamma\dot{\beta}}
							\mbox{\large $\frac{\partial^2}
							                                       {\partial x^\mu \partial\theta^\alpha}$} \\[-1ex]
        && \hspace{8em} 																
					 - \mbox{\normalsize $\sum$}_{\mu,\nu,\gamma,\dot{\delta}}\,
					   \theta^\gamma \sigma^\mu_{\alpha\dot{\delta}}
					                                 \sigma^\nu_{\gamma\dot{\beta}}\bar{\theta}^{\dot{\delta}}
							\mbox{\large $\frac{\partial^2}{\partial x^\mu\partial x^\nu}$}
				             \mbox{\Large $)$} f \\[1ex]                	
	&& +\; \bar{\eta}^{\dot{1}}\bar{\eta}^{\dot{2}}
	            \cdot \mbox{\Large $($}
				        \mbox{\large $\frac{\partial^2}
						                                        {\partial\bar{\theta}^{\dot{2}}\rule{0ex}{.8em}
																    \partial\bar{\theta}^{\dot{1}}}$}
						 - \sqrt{-1}\,\mbox{\normalsize $\sum$}_{\mu,\alpha}\,
						    \theta^\alpha \sigma^\mu_{\alpha\dot{1}}
							\mbox{\large $\frac{\partial^2}
							                                        {\partial x^\mu \rule{0ex}{.8em}
																	    \partial\bar{\theta}^{\dot{2}}}$}
				         +\sqrt{-1}\,\mbox{\normalsize $\sum$}_{\mu,\alpha}\,
						    \theta^\alpha \sigma^\mu_{\alpha\dot{2}}
							\mbox{\large $\frac{\partial^2}{\partial x^\mu \rule{0ex}{.8em}
							                                                                     \partial\bar{\theta}^{\dot{1}}}$}		 \\
        && \hspace{8em}							
						 + \mbox{\normalsize $\sum$}_{\mu,\nu}\,
						      \theta^1\theta^2
						    (\sigma^\mu_{1\dot{1}}\sigma^\nu_{2\dot{2}}
							    - \sigma^\mu_{2\dot{1}}\sigma^\nu_{1\dot{2}})
							\mbox{\large $\frac{\partial^2}{\partial x^\mu \partial x^\nu}$}
							     \mbox{\Large $)$}f\,.     \\[1ex]						
    && +\, \sum_{\dot{\beta}}\xi^1\xi^2\bar{\eta}^{\dot{\beta}}	
	         \cdot \mbox{\Large $($}
			   \mbox{\large $\frac{\partial^3}
			                                          {\partial\theta^{\dot{\beta}}\rule{0ex}{.8em}
													      \partial\theta^2 \partial\theta^1}$}
                  -\sqrt{-1}\,\mbox{\normalsize $\sum$}_{\mu,\dot{\delta}}
                         \sigma^\mu_{1\dot{\delta}}\bar{\theta}^{\dot{\delta}}
						 \mbox{\large $\frac{\partial^3}
						                                        {\partial x^\mu
																    \partial\bar{\theta}^{\dot{\beta}}
																	\partial\theta^2}$}
                  +\sqrt{-1}\,\mbox{\normalsize $\sum$}_{\mu,\dot{\delta}}
                         \sigma^\mu_{2\dot{\delta}}\bar{\theta}^{\dot{\delta}}
						 \mbox{\large $\frac{\partial^3}
						                                        {\partial x^\mu
																    \partial\bar{\theta}^{\dot{\beta}}
																	\partial\theta^1}$}   \\[-1ex]
       && \hspace{8em}																	
                  - \sqrt{-1}\,\mbox{\normalsize $\sum$}_{\mu,\alpha}
                         \theta^\alpha\sigma^\mu_{\alpha\dot{\beta}}
						 \mbox{\large $\frac{\partial^3}
						                                        {\partial x^\mu
																    \partial\theta^2
																	\partial\theta^1}$}
			      + \mbox{\normalsize $\sum$}_{\mu,\nu}\,
						    (\sigma^\mu_{1\dot{2}}\sigma^\nu_{2\dot{1}}
							    - \sigma^\mu_{1\dot{1}}\sigma^\nu_{2\dot{2}})
						    \bar{\theta}^{\dot{1}}\bar{\theta}^{\dot{2}}
							\mbox{\large $\frac{\partial^3}
							                                       {\partial x^\mu \partial x^\nu
																       \partial\bar{\theta}^{\dot{\beta}}}$}\\
       && \hspace{8em}																
				  +\mbox{\normalsize $\sum$}_{\mu,\nu,\gamma,\dot{\delta}}\,
					   \theta^\gamma \sigma^\mu_{1\dot{\delta}}
					                                 \sigma^\nu_{\gamma\dot{\beta}}\bar{\theta}^{\dot{\delta}}
							\mbox{\large $\frac{\partial^3}{\partial x^\mu\partial x^\nu \partial\theta^2}$}
				  - \mbox{\normalsize $\sum$}_{\mu,\nu,\gamma,\dot{\delta}}\,
					   \theta^\gamma \sigma^\mu_{2\dot{\delta}}
					                                 \sigma^\nu_{\gamma\dot{\beta}}\bar{\theta}^{\dot{\delta}}
							\mbox{\large $\frac{\partial^3}{\partial x^\mu\partial x^\nu \partial\theta^1}$}
			           \mbox{\Large $)$}\,f\,  \\[1ex]	
   && +\, \sum_\alpha \xi^\alpha\bar{\eta}^{\dot{1}}\bar{\eta}^{\dot{2}}	
	         \cdot \mbox{\Large $($}
			   \mbox{\large $\frac{\partial^3}
			                                          {\partial\bar{\theta}^{\dot{2}}\rule{0ex}{.8em}
													      \partial\bar{\theta}^{\dot{1}} \partial\theta^\alpha}$}
                  +\sqrt{-1}\,\mbox{\normalsize $\sum$}_{\mu,\gamma}
                         \theta^\gamma \sigma^\mu_{\gamma\dot{1}}
						 \mbox{\large $\frac{\partial^3}
						                                        {\partial x^\mu
																    \partial\bar{\theta}^{\dot{2}}
																	\partial\theta^\alpha}$}
                  - \sqrt{-1}\,\mbox{\normalsize $\sum$}_{\mu,\gamma}
                         \theta^\gamma \sigma^\mu_{\gamma\dot{2}}
						 \mbox{\large $\frac{\partial^3}
						                                        {\partial x^\mu
																    \partial\bar{\theta}^{\dot{1}}
																	\partial\theta^\alpha}$}   \\[-1ex]
       && \hspace{8em}																	
                  - \sqrt{-1}\,\mbox{\normalsize $\sum$}_{\mu,\dot{\beta}}
                         \sigma^\mu_{\alpha\dot{\beta}}\bar{\theta}^{\dot{\beta}}
						 \mbox{\large $\frac{\partial^3}
						                                        {\partial x^\mu
																    \partial\bar{\theta}^{\dot{2}}
																	\partial\bar{\theta}^{\dot{1}}}$}
			      + \mbox{\normalsize $\sum$}_{\mu,\nu}\,
				            \theta^1\theta^2
						    (\sigma^\mu_{1\dot{1}}\sigma^\nu_{2\dot{2}}
							    - \sigma^\mu_{2\dot{1}}\sigma^\nu_{1\dot{2}})
							\mbox{\large $\frac{\partial^3}
							                                       {\partial x^\mu \partial x^\nu
																       \partial\theta^\alpha}$}\\
       && \hspace{8em}																
				  - \mbox{\normalsize $\sum$}_{\mu,\nu,\gamma,\dot{\beta}}\,
					   \theta^\gamma \sigma^\mu_{\alpha\dot{\beta}}
					                                 \sigma^\nu_{\gamma\dot{1}}\bar{\theta}^{\dot{\beta}}
							\mbox{\large $\frac{\partial^3}{\partial x^\mu\partial x^\nu
							                                            \partial\bar{\theta}^{\dot{2}}\rule{0ex}{.8em}}$}
				  + \mbox{\normalsize $\sum$}_{\mu,\nu,\gamma,\dot{\beta}}\,
					   \theta^\gamma \sigma^\mu_{\alpha\dot{\beta}}
					                                 \sigma^\nu_{\gamma\dot{2}}\bar{\theta}^{\dot{\beta}}
							\mbox{\large $\frac{\partial^3}{\partial x^\mu\partial x^\nu
							                                            \partial\bar{\theta}^{\dot{1}}\rule{0ex}{.8em}}$}
			           \mbox{\Large $)$}\,f\,  \\[1ex]	        							
   && +\; \xi^1\xi^2\bar{\eta}^{\dot{1}}\bar{\eta}^{\dot{2}}	
	         \cdot \mbox{\Large $($}
			   \mbox{\large $\frac{\partial^4}
			                                          {\partial\bar{\theta}^{\dot{2}}\rule{0ex}{.8em}
													      \partial\bar{\theta}^{\dot{1}}
														  \partial\theta^2 \partial\theta^1}$}
                  - \sqrt{-1}\,\mbox{\normalsize $\sum$}_{\mu,\alpha}
                         \theta^\alpha\sigma^\mu_{\alpha\dot{1}}
						 \mbox{\large $\frac{\partial^4}
						                                        {\partial x^\mu
																    \partial\bar{\theta}^{\dot{2}}
																	\partial\theta^2 \partial\theta^1}$}
                  +\sqrt{-1}\,\mbox{\normalsize $\sum$}_{\mu,\alpha}
                         \theta^\alpha \sigma^\mu_{\alpha\dot{2}}
						 \mbox{\large $\frac{\partial^4}
						                                        {\partial x^\mu
																    \partial\bar{\theta}^{\dot{1}}
																	\partial\theta^2 \partial\theta^1}$}   \\
       && \hspace{8em}																	
                  +\sqrt{-1}\,\mbox{\normalsize $\sum$}_{\mu,\dot{\beta}}
                         \sigma^\mu_{2\dot{\beta}}\bar{\theta}^{\dot{\beta}}
						 \mbox{\large $\frac{\partial^4}
						                                        {\partial x^\mu\rule{0ex}{.8em}
																    \partial\bar{\theta}^{\dot{2}}
																	\partial\bar{\theta}^{\dot{1}} \partial\theta^1}$}
                  - \sqrt{-1}\,\mbox{\normalsize $\sum$}_{\mu,\dot{\beta}}
                         \sigma^\mu_{1\dot{\beta}}\bar{\theta}^{\dot{\beta}}
						 \mbox{\large $\frac{\partial^4}
						                                        {\partial x^\mu\rule{0ex}{.8em}
																    \partial\bar{\theta}^{\dot{2}}
																	\partial\bar{\theta}^{\dot{1}} \partial\theta^2}$} \\
       && \hspace{2em}																	
			      + \mbox{\normalsize $\sum$}_{\mu,\nu}\,
						    (\sigma^\mu_{1\dot{2}}\sigma^\nu_{2\dot{1}}
							    - \sigma^\mu_{1\dot{1}}\sigma^\nu_{2\dot{2}})
								\bar{\theta}^{\dot{1}} \bar{\theta}^{\dot{2}}
							\mbox{\large $\frac{\partial^4}
							                                       {\partial x^\mu \partial x^\nu\rule{0ex}{.8em}
																       \partial\bar{\theta}^{\dot{2}}
																	   \partial\bar{\theta}^{\dot{1}}}$}
				  + \mbox{\normalsize $\sum$}_{\mu,\nu,\alpha,\dot{\beta}}\,
					   \theta^\alpha \sigma^\mu_{2\dot{\beta}}
					                                 \sigma^\nu_{\alpha\dot{2}}\bar{\theta}^{\dot{\beta}}
							\mbox{\large $\frac{\partial^4}{\partial x^\mu\partial x^\nu
							                                            \partial\bar{\theta}^{\dot{1}} \partial\theta^1}$} \\
       && \hspace{2em}																
				  - \mbox{\normalsize $\sum$}_{\mu,\nu,\alpha,\dot{\beta}}\,
					   \theta^\alpha \sigma^\mu_{2\dot{\beta}}
					                                 \sigma^\nu_{\alpha\dot{1}}\bar{\theta}^{\dot{\beta}}
							\mbox{\large $\frac{\partial^4}{\partial x^\mu\partial x^\nu
							                                            \partial\bar{\theta}^{\dot{2}}\partial\theta^1}$}
				  - \mbox{\normalsize $\sum$}_{\mu,\nu,\alpha,\dot{\beta}}\,
					   \theta^\alpha \sigma^\mu_{1\dot{\beta}}
					                                 \sigma^\nu_{\alpha\dot{2}}\bar{\theta}^{\dot{\beta}}
							\mbox{\large $\frac{\partial^4}{\partial x^\mu\partial x^\nu
							                                            \partial\bar{\theta}^{\dot{1}} \partial\theta^2}$}\\
       && \hspace{2em}																	
	               + \mbox{\normalsize $\sum$}_{\mu,\nu,\alpha,\dot{\beta}}\,
					   \theta^\alpha \sigma^\mu_{1\dot{\beta}}
					                                 \sigma^\nu_{\alpha\dot{1}}\bar{\theta}^{\dot{\beta}}
							\mbox{\large $\frac{\partial^4}{\partial x^\mu\partial x^\nu
							                                            \partial\bar{\theta}^{\dot{2}}\partial\theta^2}$}
			      + \mbox{\normalsize $\sum$}_{\mu,\nu}\,
				            \theta^1\theta^2
						    (\sigma^\mu_{1\dot{1}}\sigma^\nu_{2\dot{2}}
							    - \sigma^\mu_{2\dot{1}}\sigma^\nu_{1\dot{2}})
							\mbox{\large $\frac{\partial^4}
							                                       {\partial x^\mu \partial x^\nu
																       \partial\theta^2 \partial\theta^1}$}	
			           \mbox{\Large $)$}\,f\,.  	
  \end{eqnarray*}} 

\bigskip

\begin{lemma} {\bf [$g_{\mbox{\tiny $\bullet$}}$ as representation on super-$C^{\infty}$-ring]}\;
 For $(\xi,\bar{\eta})$ a constant sections of $S^{\prime}\oplus S^{\prime\prime}$,
  $g_{(\xi,\eta)}$ is an automorphism of $C^{\infty}(\widehat{X})$ as a super $C^{\infty}$-ring.
\end{lemma}
 
\begin{proof}
 Since $\xi^{\alpha}$, $\alpha=1,2$,
    (resp.\ $\bar{\eta}^{\dot{\beta}}$, $\dot{\beta}=\dot{1},\dot{2}$)
 	is a ${\Bbb C}$-linear combination of $\theta^1$ and $\theta^2$
   (resp.\ $\bar{\theta}^{\dot{1}}$ and $\bar{\theta}^{\dot{2}}$),
   $g_{(\xi,\bar{\eta})}(x^{\mu})$
     lies in $C^{\infty}$-{\it hull}\,$(C^{\infty}(\widehat{X}))$, for $\mu=0,1,2,3$.
 This implies that $g_{(\xi,\bar{\eta})}$ leaves
   $C^{\infty}$-{\it hull}\,$(C^{\infty}(\widehat{X}))$ invariant.
 It remains to show that $g_{(\xi,\bar{\eta})}$ is compatible with the operations
  associated to $h\in \cup_{l=1}^{\infty} C^{\infty}({\Bbb R}^l)$
  on the $C^{\infty}$-hull of $C^{\infty}(\widehat{X})$.
  
 Let
   $h\in C^{\infty}({\Bbb R}^l) $ and
   $f_1=f_1(x,\theta,\bar{\theta}),\,\cdots\,,\, f_l=f_l(x,\theta,\bar{\theta})
      \in C^{\infty}$-{\it hull}\,$(C^{\infty}(\widehat{X}))$.
 Then,
   \begin{eqnarray*}
    \lefteqn{
	  \mbox{\Large $($}
	    g_{(\xi,\bar{\eta})}
		 \mbox{\large $($}  h(f_1,\,\cdots\,,\, f_l)\mbox{\large $)$}
	  \mbox{\Large $)$}(x,\theta,\bar{\theta})   }\\[.8ex]
	 && =\;
	    h \mbox{\Large $($}
		    f_1 \mbox{\large $($}
			             g_{(\xi,\bar{\eta})}(x),
		                 g_{(\xi,\bar{\eta})}(\theta), g_{(\xi,\bar{\eta})}(\bar{\theta})
				   \mbox{\large $)$}  \,,\,
               \cdots\,, \,
			  f_l \mbox{\large $($}
			            g_{(\xi,\bar{\eta})}(x),
		                g_{(\xi,\bar{\eta})}(\theta), g_{(\xi,\bar{\eta})}(\bar{\theta})
					\mbox{\large $)$}
		   \mbox{\Large $)$}  \\[.8ex]	
     && = h \mbox{\Large $($}
            	  \mbox{\large $($}g_{(\xi,\bar{\eta})}(f_1)\mbox{\large $)$}
				            (x,\theta,\bar{\theta})\,,\,
   	                    \cdots\,,\,
				  \mbox{\large $($}g_{(\xi,\bar{\eta})}(f_l)\mbox{\large $)$}
				            (x,\theta,\bar{\theta})
			    \mbox{\Large $)$}\,.
   \end{eqnarray*}
  Which says
   $$
    g_{(\xi,\bar{\eta})}(h(f_1,\,\cdots\,,\, f_l))\;
	 =\;   h(  g_{(\xi,\bar{\eta})}(f_1)\,,\,\cdots\,,\, g_{(\xi,\bar{\eta})}(f_l) )\,.
   $$
   
  This proves the lemma.
  
\end{proof}

\bigskip

\begin{flushleft}
{\bf Infinitesimal generators of the supersymmetry transformations on $\widehat{X}$}
\end{flushleft}
Observe that
 $g_{(\xi,\bar{\eta})}(f(x,\theta,\bar{\theta}))$
 has an expression as a polynomial in the anticommuting parameters
  $(\xi,\bar{\eta})=(\xi^1,\xi^2, \bar{\eta}^{\dot{1}},\bar{\eta}^{\dot{2}})$
   from the constant sections of $S^{\prime}\oplus S^{\prime\prime}$,
 with coefficients in $C^{\infty}(\widehat{X})$, of the form
 \begin{eqnarray*}
  \lefteqn{
    g_{(\xi,\bar{\eta})}(f(x,\theta,\bar{\theta}))\;
       =\;  f(x,\theta,\bar{\theta})\,
	     +\, \sum_{\alpha=1}^2\xi^{\alpha}\cdot [f]_{(\alpha)}(x,\theta,\bar{\theta})\,
		 +\, \sum_{\dot{\beta}=\dot{1}}^{\dot{2}}
		          \bar{\eta}^{\dot{\beta}}\cdot [f]_{(\dot{\beta})}(x,\theta,\bar{\theta})   } \\[-.8ex]
    && \hspace{7em}				
		 +\, \mbox{(terms of total-$(\xi,\bar{\eta})$-degree $\ge 2$)}\,. \hspace{9em}
 \end{eqnarray*}
It follows from the identity
  $g_{(\xi,\bar{\eta})}(f_1(x,\theta,\bar{\theta})f_2(x,\theta,\bar{\theta}))
     = g_{(\xi,\bar{\eta})}(f_1(x,\theta,\bar{\theta}))
	     \cdot g_{(\xi,\bar{\eta})}(f_1(x,\theta, \bar{\theta}))$
that
 $$
  \begin{array}{ll}
   [f_1f_2]_{(\alpha)}\;
     =\; [f_1]_{(\alpha)}\cdot f_2 \, +\, ^{\varsigma}\!(f_1)\cdot [f_2]_{(\alpha)}\,,
	  & \hspace{2em}\alpha=1,2\,; \\[1.2ex]
  [f_1f_2]_{(\dot{\beta})}\;
     =\; [f_1]_{(\dot{\beta})}\cdot f_2 \, +\, ^{\varsigma}\!(f_1)\cdot [f_2]_{(\dot{\beta})}\,,
	  & \hspace{2em}\dot{\beta}=\dot{1}, \dot{2}\,.
  \end{array}
 $$
Here, recall that $^{\varsigma}\!(f_1)= {f_1}_{(\even)}-{f_1}_{(\odd)}$
 is the parity-conjugation of $f_1$; cf. Definition~1.3.1.
This implies that
 all the correspondences
 $$
  f\; \longmapsto\; [f]_{(\alpha)}\,,\hspace{2em} f\; \longmapsto\; [f]_{\dot{\beta}}\,,
 $$
 $\alpha=1,2$, $\dot{\beta}=\dot{1}, \dot{2}$, are odd derivations on $C^{\infty}(\widehat{X})$
 since $g_{(\xi,\bar{\eta})}$ is ${\Bbb C}$-linear
    and hence these correspondences are ${\Bbb C}$-linear as well.

The explicit expansion of
  $g_{(\xi,\bar{\eta})}(f)\;
	 =\;  f(x + \sqrt{-1}(\theta\sigma \bar{\eta}^t - \xi\sigma \bar{\theta}^t),
	           \theta+\xi , \bar{\theta}+\bar{\eta})$
   from the $C^{\infty}$-hull structure of $C^{\infty}(\widehat{X})$	
  gives an explicit expression for these odd derivations:\,\footnote{The
                                                                      `$-$'-sign before the summand $\sum_{\dot{\beta}}$
	       															  is chosen so that the resulting expressions for $Q_{\alpha}$
																	  and $\bar{Q}_{\dot{\beta}}$ match exactly with
																	  [We-B: Eq.\ (4.4)] of Wess \& Bagger except the the labelling indices.
								                                      } 
 \begin{eqnarray*}
  \lefteqn{
    g_{(\xi,\bar{\eta})}(f(x,\theta,\bar{\theta}))\;
       =\;  f(x,\theta,\bar{\theta})\,
	     +\, \sum_{\alpha=1}^2\xi^{\alpha}\, Q_{\alpha} f(x,\theta,\bar{\theta})\,
		 -\, \sum_{\dot{\beta}=\dot{1}}^{\dot{2}}
		          \bar{\eta}^{\dot{\beta}}
				   \cdot \bar{Q}_{\dot{\beta}} f(x,\theta,\bar{\theta})   } \\[-.8ex]
    && \hspace{7em}				
		 +\, \mbox{(terms of total-$(\xi,\bar{\eta})$-degree $\ge 2$)}\,. \hspace{9em}
 \end{eqnarray*}
 where
 $$
   Q_{\alpha}\;
    =\;  \frac{\partial}{\partial \theta^{\alpha}}
            -\, \sqrt{-1}\,\sum_{\mu=0}^3 \sum_{\dot{\beta}=\dot{1}}^{\dot{2}}
			          \sigma^{\mu}_{\alpha\dot{\beta}}\bar{\theta}^{\dot{\beta}}
					   \frac{\partial}{\partial x^{\mu}}
    \hspace{2em}\mbox{and}\hspace{2em}
  \bar{Q}_{\dot{\beta}}	\;
    =\; -\, \frac{\partial}{\rule{0ex}{.8em}\partial \bar{\theta}^{\dot{\beta}}}\,
	      +\, \sqrt{-1}\sum_{\mu=0}^3 \sum_{\alpha=1}^2
		              \theta^{\alpha} \sigma^\mu_{\alpha\dot{\beta}}\frac{\partial}{\partial x^{\mu}}\,.
 $$
			
\bigskip

\begin{definition} {\bf [standard (infinitesimal) supersymmetry generator]}\; {\rm
 The odd derivations $Q_{\alpha}$, $\bar{Q}_{\dot{\beta}}$,
   $\alpha=1,2$, $\dot{\beta}=\dot{1},\dot{2}$, of $C^{\infty}(\widehat{X})$
   are called the {\it standard (infinitesimal) supersymmetry generators}
   for the supersymmetry transformations on the superspace $\widehat{X}$,
  with $Q_{\alpha}$ (resp.\ $-\,\bar{Q}_{\dot{\beta}}$)
     the infinitesimal generator of the action of the $1$-parameter subgroup
	 (of the Abelian group of constant sections of $S^{\prime}\oplus S^{\prime\prime}$)
	 parameterized by $\xi^{\alpha}$
	 (resp.\ $\bar{\eta}^{\dot{\beta}}$).
}\end{definition}

\bigskip

By construction, $Q_{\alpha}$'s and $\bar{Q}_{\dot{\beta}}$'s satisfy the following
 super Lie-bracket relations in $\Der_{\Bbb C}(\widehat{X})$:
$$
  \{Q_{\alpha}\,,\, \bar{Q}_{\dot{\beta}}\}\;
   =\; 2\sqrt{-1}\, \sum_{\mu=0}^3
                 \sigma^{\mu}_{\alpha\dot{\beta}}\, \frac{\partial}{\partial x^{\mu}}
   \hspace{1em}\mbox{and}\hspace{1em}				
   \{Q_{\alpha}\,,\, Q_{\beta}\}\;
   =\; \{\bar{Q}_{\dot{\alpha}}\,,\, \bar{Q}_{\dot{\beta}}\}\;=\; 0\,.				
$$
(Cf.\ E.g.\ [We-B: Eq.\ (4.5)].)

\bigskip

\begin{remark} $[$super-Poincar\'{e}-group action on $\widehat{X}\,]$\; {\rm
 In addition to the supersymmetry transformations on $\widehat{X}$,
  the built-in action of the Poincar\'{e} group on $X$ extends also to an action on $\widehat{X}$.
 First, it acts on the standard coordinate-functions $(x,\theta,\bar{\theta})$
  ${\Bbb R}$-affine-linearly on $x$ and ${\Bbb C}$-linearly on $(\theta,\bar{\theta})$.
 This then defines an action of the Poincar\'{e} group on $C^{\infty}(\widehat{X})$
  from the $C^{\infty}$-hull structure of $C^{\infty}(\widehat{X})$.
 The corresponding infinitesimal generators of the action, as derivations of $C^{\infty}(\widehat{X})$,
  can be worked out similarly to the case of supersymmetric transformations.
 All together, they form a super Lie algebra with the super Lie bracket from $\Der_{\Bbb C}(\widehat{X})$.
 %
 %
 %
 %
 This reproduces the standard results of symmetries of the $d=4$, $N=1$ superspace in, for example,
     [We-B: Eq.~(1.26)], [G-G-R-S: Sec.\ 3.2.c], [West1: Sec.\ 14.2]
  in the context of Super $C^{\infty}$-Algebraic Geometry.  	
 Details are left to readers as an exercise.
}\end{remark}

\bigskip

\begin{flushleft}
{\bf Supersymmetrically invariant flow and supersymmetrically invariant frame on $\widehat{X}$}
\end{flushleft}
Consider another action on $\widehat{X}$
  of the Abelian group of constant sections
  $(\xi,\bar{\eta})=(\xi^1,\xi^2,\bar{\eta}^{\dot{1}},\bar{\eta}^{\dot{2}})$
  of $S^{\prime}\oplus S^{\prime\prime}$,
 given by the following transformations of the basic coordinate-functions $(x,\theta,\bar{\theta})$
 $$
   g^{\prime}_{(\xi,\bar{\eta})}\;:\;
   \left\{
    \begin{array}{lll}
      x^{\mu}   & \longmapsto
	   & x^{\mu}
	        + \sqrt{-1}
		      (-\theta \sigma^{\mu}\bar{\eta}^t  +  \xi \sigma^{\mu}\bar{\theta}^t)\\
	  \theta^{\alpha}    & \longmapsto   &  \theta^{\alpha} + \xi^{\alpha} \\
	  \bar{\theta}^{\dot{\beta}}
	  & \longmapsto    & \bar{\theta}^{\dot{\beta}}+ \bar{\eta}^{\dot{\beta}}
	     \hspace{10em},
    \end{array}\right.
   $$
   $\mu=0,1,2,3, \alpha=1,2, \dot{\beta}=\dot{1}, \dot{2}$\,.
Since $(\xi,\eta)$ is odd, this is again a ${\Bbb Z}/2$-grading-preserving  transformation and
 $$
   g^{\prime}_{(\xi,\bar{\eta})}(f(x,\theta,\bar{\theta}))\;
    :=\;  f(x+\sqrt{-1}
	                 (- \theta \mbox{\boldmath $\sigma$} \bar{\eta}^t
					     + \xi \mbox{\boldmath $\sigma$} \bar{\theta}^t ),
                \theta + \xi, \bar{\theta}+ \bar{\eta}						        )
 $$
 is defined by the $C^{\infty}$-structure of $C^{\infty}(\widehat{X})$,
 as in the case of $g_{(\xi,\eta)}(f(x,\theta,\bar{\theta}))$.
Similar to the representation $g_{\mbox{\tiny $\bullet$}}$,
 one has
 
\bigskip

\begin{lemma}
{\bf [$g^{\prime}_{\mbox{\tiny $\bullet$}}$ as representation on super-$C^{\infty}$-ring]}\;
 For $(\xi,\bar{\eta})$ a constant sections of $S^{\prime}\oplus S^{\prime\prime}$,
  $g^{\prime}_{(\xi,\eta)}$
  is an automorphism of $C^{\infty}(\widehat{X})$ as a super $C^{\infty}$-ring.
\end{lemma}

\bigskip

Furthermore, one has the following commutativity property:

\bigskip

\begin{lemma} {\bf [$g^{\prime}$ and $g$ commute]}\;
 As super-$C^{\infty}$-ring-automorphisms of $C^{\infty}(\widehat{X})$,
 $$
    g^{\prime}_{(\xi^{\prime},\bar{\eta}^{\prime})} \circ g_{(\xi,\bar{\eta})}\;
	=\; g_{(\xi,\bar{\eta})}\circ g^{\prime}_{(\xi^{\prime}, \bar{\eta}^{\prime})}
 $$
 for all constant sections $(\xi^{\prime},\bar{\eta}^{\prime}),\, (\xi,\bar{\eta})$
      of $S^{\prime}\oplus S^{\prime\prime}$.
\end{lemma}

\medskip

\begin{proof}
 This follows from the direct computation that
   $g^{\prime}_{(\xi^{\prime}, \bar{\eta}^{\prime})}\circ g_{(\xi,\bar{\eta})}$  and
   $g_{(\xi,\bar{\eta})}\circ g^{\prime}_{(\xi^{\prime},\bar{\eta}^{\prime})}$
   give the same transformation of the basic coordinate-functions $(x,\theta,\bar{\theta})$
   in $C^{\infty}(\widehat{X})$:
 $$
   \begin{array}{lll}
    x  & \longmapsto
	    &  x + \sqrt{-1}\,
		          ( \theta \mbox{\boldmath $\sigma$} (\bar{\eta}-\bar{\eta}^{\prime})^t
				      + (-\xi+\xi^{\prime})\mbox{\boldmath $\sigma$}\bar{\theta}^t
					  + \xi^{\prime} \mbox{\boldmath $\sigma$}\bar{\eta}^t
					  - \xi \mbox{\boldmath $\sigma$}{\bar{\eta}^{\prime}}\,\!^t )    \\[.8ex]
     \theta   & \longmapsto   & \theta + \xi + \xi^{\prime}\\[.8ex]
	 \bar{\theta} & \longmapsto
	    & \bar{\theta} +  \bar{\eta} + \bar{\eta}^{\prime}  \hspace{20em} \,.
   \end{array}
 $$
\end{proof}

\bigskip

Motivated by the notion of invariant flows on a manifold under a group action, one has

\bigskip

\begin{definition} {\bf [symmetrically invariant (even) flow (with odd parameters)]}\; {\rm
 We say that $g^{\prime}$
  defines a (${\Bbb Z}/2$-grading preserving)
  {\it supersymmetrically invariant flow} on the superspace $\widehat{X}$,
  parameterized by the odd parameters $(\xi,\bar{\eta})$
  in the Abelian group of constant sections of $S^{\prime}\oplus S^{\prime\prime}$.
}\end{definition}
 
\bigskip

As in the case of  $g_{(\xi,\bar{\eta})}(f(x,\theta,\bar{\theta}))$,
one can also express  $g^{\prime}_{(\xi,\bar{\eta})}(f(x,\theta,\bar{\theta}))$
 as a polynomial in $(\xi,\eta)$ to give the linearization of the action $g^{\prime}$ on $\widehat{X}$:
 \begin{eqnarray*}
  \lefteqn{
    g^{\prime}_{(\xi,\bar{\eta})}(f(x,\theta,\bar{\theta}))\;
       =\;  f(x,\theta,\bar{\theta})\,
	     +\, \sum_{\alpha=1}^2
		          \xi^{\alpha}\, e_{\alpha^{\prime}} f(x,\theta,\bar{\theta})\,
		 -\, \sum_{\dot{\beta}=\dot{1}}^{\dot{2}}
		          \bar{\eta}^{\dot{\beta}}
				   \cdot e_{\beta^{\prime\prime}} f(x,\theta,\bar{\theta})   } \\[-.8ex]
    && \hspace{7em}				
		 +\, \mbox{(terms of total-$(\xi,\bar{\eta})$-degree $\ge 2$)}\,. \hspace{9em}
 \end{eqnarray*}
 where\footnote{In
                                terms of [We-B] of Wess \& Bagger,
							   $e_{\alpha^{\prime}}$ and $e_{\beta^{\prime\prime}}$ here
							     are denoted respectively by $D_{\alpha}$ and $\bar{D}_{\dot{\beta}}$ ibidem.
                               While the latter notations were essentially already carved in stone in the supersymmetry literature,
							    in view that we need to reserve the notation $D$ for a connection,
								to avoid confusion from same letter $D$ representing various different objects, 								
								 we have no choice but to change the classical notation in the supersymmetry literature.
								Here, {\it the index $\alpha^{\prime}$ is to match the spinor index $\alpha$
								  while the index $\beta^{\prime\prime}$ is to match the spinor index $\dot{\beta}$}.
							   } 
 $$
   e_{\alpha^{\prime}}\;
    =\;  \frac{\partial}{\partial \theta^{\alpha}}
            +\, \sqrt{-1}\,\sum_{\mu=0}^3 \sum_{\dot{\beta}=\dot{1}}^{\dot{2}}
			          \sigma^{\mu}_{\alpha\dot{\beta}}\bar{\theta}^{\dot{\beta}}
					   \frac{\partial}{\partial x^{\mu}}
    \hspace{2em}\mbox{and}\hspace{2em}
  e_{\beta^{\prime\prime}}	\;
    =\; -\, \frac{\partial}{\rule{0ex}{.8em}\partial \bar{\theta}^{\dot{\beta}}}\,
	      -\, \sqrt{-1}\sum_{\mu=0}^3 \sum_{\alpha=1}^2
		              \theta^{\alpha} \sigma^\mu_{\alpha\dot{\beta}}\frac{\partial}{\partial x^{\mu}}\,.
 $$
They satisfy the following super Lie bracket relations
 $$
    \{e_{\alpha^{\prime}}\,,\, e_{\beta^{\prime\prime}} \}\;
	 =\; -2\sqrt{-1}\, \sum_{\mu=0}^3 \sigma^{\mu}_{\alpha\dot{\beta}}\,
	         \frac{\partial}{\partial x^{\mu}}\hspace{1em}\mbox{and}\hspace{1em}
	\{e_{\alpha^{\prime}}\,,\, e_{\beta^{\prime}}\}\;=\;
   	\{e_{\alpha^{\prime\prime}}\,,\, e_{\beta^{\prime\prime}}\}\;=\; 0
 $$
 while anti-commuting with $Q_{\alpha}$'s and $\bar{Q}_{\dot{\beta}}$'s:
 $$
   \{e_{\alpha^{\prime}}\,,\, Q_{\beta}\}\;
   =\;   \{e_{\alpha^{\prime}}\,,\, \bar{Q}_{\dot{\beta}}\}\;
   =\;   \{e_{\alpha^{\prime\prime}}\,,\, Q_{\beta}\}\;
   =\;   \{e_{\alpha^{\prime\prime}}\,,\, \bar{Q}_{\dot{\beta}}\}\; =0 \,,
 $$
 for $\alpha^\prime = 1^\prime, 2^\prime$,\,
      $\alpha^{\prime\prime}=1^{\prime\prime}, 2^{\prime\prime}$,\,
      $\beta=1,2$,\,
	  $\dot{\beta}=\dot{1}, \dot{2}$.
(Cf.\ [We-B: Eq.\ (4.6) and  Eq.'s\ (4.7) \& (4.8)].)
 
\bigskip

\begin{definition} {\bf [standard supersymmetrically invariant frame on $\widehat{X}$]}\; {\rm
  Let  $e_{\mu}:=\partial/\partial x^{\mu}$, $\mu=0,1,2,3$.
  Recall that they commute with $Q_{\alpha}$'s and $\bar{Q}_{\dot{\beta}}$'s.
  Then the $8$-tuple of derivations
    $$	
	 (e_I)_I\;
	   :=\; (e_{\mu}, e_{\alpha^{\prime}}, e_{\beta^{\prime\prime}})	
	             _{\mu, \alpha^{\prime}, \beta^{\prime\prime}}\;
	   :=\; (e_0,e_1,e_2,e_3\,;\, e_{1^{\prime}}, e_{2^{\prime}}\,;\,
	              e_{1^{\prime\prime}}, e_{2^{\prime\prime}})
    $$
   is called the {\it standard supersymmetrically invariant frame} on the superspace $\widehat{X}$,
  with $e_{\alpha^{\prime}}$ (resp.\ $-\,e_{\beta^{\prime\prime}}$)
     the infinitesimal generator of the supersymmetrically invariant flow
	 parameterized by $\xi^{\alpha}$
	 (resp.\ $\bar{\eta}^{\dot{\beta}}$).
}\end{definition}
   
\medskip

\begin{definition} {\bf [standard supersymmetrically invariant coframe on $\widehat{X}$]}\; {\rm
 The dual frame
   $$
    (e^I)_I\;
       :=\; (e^{\,\mu}, e^{\alpha^{\prime}}, e^{\beta^{\prime\prime}})	
	             _{\mu, \alpha^{\prime}, \beta^{\prime\prime}}\;
	   :=\; (e^0, e^1, e^2, e^3\,;\, e^{1^{\prime}}, e^{2^{\prime}}\,;\,
	              e^{1^{\prime\prime}}, e^{2^{\prime\prime}})
   $$
     to $  (e_{\mu}, e_{\alpha^{\prime}}, e_{\beta^{\prime\prime}})
                                                                                      _{\mu, \alpha^{\prime},\beta^{\prime\prime}}$
   is called the {\it standard supersymmetrically invariant coframe} on the superspace $\widehat{X}$.
 It is a collection of $1$-forms $e^I$ on $\widehat{X}$ characterized by
  $e^I(e_J):=  (e_J)\,\!^{\leftarrow}\!\!\!e^I=\delta_{IJ}$.
 In terms of $(x,\theta,\bar{\theta})$, they are given by
  $$
   \begin{array}{l}
   e^{\mu}\;
     =\; dx^{\mu}\,
            +\, \sqrt{-1}\, \sum_{\alpha,\dot{\beta}}
			       \sigma^\mu_{\alpha\dot{\beta}}\,\bar{\theta}^{\dot{\beta}}
 				   \cdot d\theta^{\alpha}\,
            +\, \sqrt{-1}\, \sum_{\alpha,\dot{\beta}}
			       \theta^{\alpha}\,\sigma^\mu_{\alpha\dot{\beta}}
				   \cdot d\bar{\theta}^{\dot{\beta}}\,,\\[2ex]
	e^{\alpha^{\prime}}\; =\; d\theta^{\alpha}\,, \hspace{4em}
    e^{\beta^{\prime\prime}}\; =\;  -\, d\bar{\theta}^{\dot{\beta}}\,.
  \end{array}
  $$
 One has
  $$
    de^{\mu}\;
	  =\;   2\sqrt{-1}\,\sum_{\alpha,\dot{\beta}} \sigma^{\mu}_{\alpha\dot{\beta}}\,
	            d\theta^{\alpha}\wedge d\bar{\theta}^{\dot{\beta}}\,, \hspace{2em}
	de^{\alpha^{\prime}}\;=\; 0\,, \hspace{2em}
	de^{\beta^{\prime\prime}}\;=\; 0\,.                                        	
  $$
 If writing $[e_I, e_J\}=\sum_K c_{IJ}^Ke_K$, then one has\,
  $de^K= \frac{1}{2} \sum_{I,J}c_{IJ}^K e^I\wedge e^J$.
}\end{definition}
      
\bigskip

\begin{flushleft}
{\bf The supersymmetrically invariant flat geometry with torsion on $\widehat{X}$}
\end{flushleft}
One may use the standard supersymmetrically invariant frame\,
   $(e_{\mu}, e_{\alpha^{\prime}}, e_{\beta^{\prime\prime}})	
	             _{\mu, \alpha^{\prime}, \beta^{\prime\prime}}$\, on $\widehat{X}$
 to define a left connection on the tangent sheaf ${\cal T}_{\widehat{X}}$ of $\widehat{X}$
 as follows.
(See Sec.\ 2.1 for a general study of left connections on a left $\widehat{\cal O}_X$-module
    and more explanations of the curvature tensor of a connection on a $\widehat{\cal O}_X$-module.)
  
\bigskip

\begin{definition} {\bf [canonical connection on ${\cal T}_{\widehat{X}}$]}\; {\rm
 The tangent sheaf ${\cal T}_{\widehat{X}}$ of the superspace $\widehat{X}$
  is a left free $\widehat{\cal O}_X$-module with basis
  $(e_I)_I:=(e_{\mu}, e_{\alpha^{\prime}}, e_{\beta^{\prime\prime}})	
	             _{\mu, \alpha^{\prime}, \beta^{\prime\prime}}$.				
 Thus, one can define the {\it canonical connection} $\nabla^{\scriptsizecan}$
    on ${\cal T}_{\widehat{X}}$ by setting
	 $$
	   \nabla^{\scriptsizecan}_{e_I} e_J\; =\;   [e_I, e_J\}\,,
	 $$
	 for $I, J=0,1,2,3, 1^{\prime}, 2^{\prime}, 1^{\prime\prime}, 2^{\prime\prime}$.
  Explicitly,
    $$
     \nabla^{\scriptsizecan}_{e_{\alpha^\prime}} e_{\beta^{\prime\prime}}\;
	 =\; \nabla^{\scriptsizecan}_{e_{\beta^{\prime\prime}}} e_{\alpha^{\prime}}\;
	 =\; -2\sqrt{-1}\,\sigma^{\nu}_{\alpha\dot{\beta}} e_{\nu}\;\;
	 \mbox{and all other $\nabla^{\scriptsizecan}_{e_I}e_J$'s are $0$.}	
    $$   	
  This then determines
     $\nabla^{\scriptsizecan}_u v$, for general $u, v \in {\cal T}_{\widehat{X}}$,
	 by the $\widehat{\cal O}_X$-linearity in the $u$-argument    and
	      the ${\Bbb C}$-linearity and the  ${\Bbb Z}/2$-graded Leibniz rule in the $v$-argument:
  (See Definition~2.1.2.)
	  %
	{\small
	  \begin{eqnarray*}
	    \nabla^{\tinycan}_u v\;
		 & :=\: &  \sum_{I,J} u^I \mbox{\Large $($}
		             (e_Iv^J)\cdot e_J\,
					     +\,\,\!^{\varsigma_{e_{\!I}}}\!v^J
                                  \cdot \nabla^{\tinycan}_{e_I}e_J						
		                                          \mbox{\Large $)$}\; \\
         & =  & \sum_\mu \mbox{\Large $($}
		                  uv^\mu\,
						  -\,2\sqrt{-1}\,\mbox{\large $\sum$}_{\alpha^{\prime},\beta^{\prime\prime}}
						        \,\!^\varsigma\!v^{\alpha^\prime}\sigma^\mu_{\alpha\dot{\beta}}\,
                          -\,2\sqrt{-1}\,\mbox{\large $\sum$}_{\alpha^{\prime},\beta^{\prime\prime}}
						        \,\!^\varsigma\!v^{\beta^{\prime\prime}}\sigma^\mu_{\alpha\dot{\beta}}\,						
                             	           \mbox{\Large $)$}\cdot e_\mu\,    \\[-.6ex]
        && +\, \sum_{\alpha^\prime}(uv^{\alpha^\prime})\cdot e_{\alpha^\prime}\,
			   +\, \sum_{\beta^{\prime\prime}}(uv^{\beta^{\prime\prime}})
					         \cdot e_{\beta^{\prime\prime}},,
	  \end{eqnarray*}}for 
	 $u= \sum_I u^Ie_I$ and $v= \sum_J v^J e_J$.
}\end{definition}

\bigskip

From the explicit expression of $\nabla^{\scriptsizecan}_{e_I}e_J$,
 one concludes that
   $\nabla^{\scriptsizecan}$ is even and of cohomological degree $1$, and is
  invariant under the (even) transformations on ${\cal T}_{\widehat{X}}$
   induced by the supersymmetrically invariant flow-with-odd-parameters on $\widehat{X}$.
%
%
This re-creates the {\it flat geometry with torsion} on $\widehat{X}$ described in, e.g.,
   [We-B: Chap.XIV], [G-G-R-S: Sec.\ 3.4.c], and [West1: Sec.\ 14.2].
See also
   [K-N: Sec.\ II.11 \& Sec.\ X.2] and [Gi: Sec.\ I.9]
  for related studies.

\bigskip
 
\begin{flushleft}
{\bf The chiral sector and the antichiral sector of $\widehat{X}$\\
         {\rm (Cf.\ [We-B: Chap.\ V] of Wess \& Bagger)}}
\end{flushleft}
(1) {\it The chiral sector of $\widehat{X}$}

\medskip

\noindent
Let
 $$
   x^\prime\,\!^\mu\; :=\;    x^\mu+ \sqrt{-1}\,\theta\sigma^\mu\bar{\theta}^t\,,\,  \hspace{1em}
   \theta^{\prime}\,\!^\alpha\;  :=\; \theta^\alpha\,, \hspace{1em}
   \bar{\theta}^\prime\,\!^{\dot{\beta}}   \;
         :=\; \bar{\theta}^{\dot{\beta}}\,, \hspace{2em}
   \mu=0,1,2,3\,,\;\; \alpha=1,2\,,\;\; \ \dot{\beta}=\dot{1}, \dot{2},
 $$
 denoted collectively as $(x^{\prime}, \theta^{\prime}, \ \bar{\theta}^{\prime})$,
 be a new  set of coordinate-functions in $C^{\infty}(\widehat{X})$.
They satisfy
 $$
   \begin{array}{lclcl}
    e_{\mu}x^{\prime}\,\!^{\nu}\;=\; \delta_{\mu\nu}\,,
	  && e_{\mu} \theta^{\prime}\,\!^{\alpha}\; =\; 0\,,
	  && e_{\mu}\bar{\theta}^{\prime}\,\!^{\dot{\beta}}\; =\; 0\,, \\[1.2ex]
	e_{\alpha^{\prime}}x^{\prime}\,\!^{\mu}\;
	      =\; 2\sqrt{-1}\,\sum_{\dot{\beta}}
		                              \sigma^{\nu}_{\alpha\dot{\beta}}\,\bar{\theta}^{\dot{\beta}}\,,
      && e_{\alpha^{\prime}}\theta^{\prime}\,\!^{\beta}\;=\; \delta_{\alpha\beta}\,,
      && e_{\alpha^{\prime}}\bar{\theta}^{\prime}\,\!^{\dot{\beta}}	  \;=\; 0\,, \\[1.2ex]
	e_{\beta^{\prime\prime}} x^{\prime}\,\!^{\mu}\;=\; 0\,,
	  && e_{\beta^{\prime\prime}}\theta^{\prime}\,\!^{\alpha}\;=\; 0\,,
	  && e_{\beta^{\prime\prime}}\bar{\theta}^{\prime}\,\!^{\dot{\alpha}}\;
	         =\; -\,\delta_{\dot{\beta}\dot{\alpha}}
   \end{array}
 $$
and, hence, in terms of the new coordinate functions
 $(x^{\prime}, \theta^{\prime}, \bar{\theta}^{\prime})$, one has
 $$
   e_{\mu}\;=\; \frac{\partial}{\partial x^{\prime}\,\!^{\mu}}\,,\hspace{2em}
   e_{\alpha^{\prime}}\;
      =\;  \frac{\partial}{\partial \theta^{\prime}\,\!^{\alpha}}\,
	           +\, 2\sqrt{-1}\, \sum_{\nu,\dot{\beta}}
			          \sigma^{\nu}_{\alpha\dot{\beta}}\bar{\theta}^{\prime}\,\!^{\dot{\beta}}
					  \frac{\partial}{\partial x^{\prime}\,\!^{\nu}}\,,\hspace{2em}
   e_{\beta^{\prime\prime}}\;
      =\; -\,\frac{\partial}{\rule{0ex}{.8em}\partial \bar{\theta}^{\prime}\,\!^{\dot{\beta}}}\,.
 $$
Here,
 $\partial/\partial x^{\prime}\,\!^{\mu} $, $\partial/\partial\theta^{\prime}\,\!^{\alpha}$,
   $\partial/\partial \bar{\theta}^{\prime}\,\!^{\dot{\theta}}$,
     $\mu=0,1,2,3$, $\alpha=1,2$, $\dot{\beta}=\dot{1}, \dot{2}$,
   are by definition the derivations of $C^{\infty}(\widehat{X})$
   associated to $(x^{\prime},\theta^{\prime},\bar{\theta}^{\prime})$,
 which are characterized by
  $\frac{\partial}{\partial x^{\prime}\,\!^{\mu}} x^{\prime}\,\!^{\nu}=\delta_{\mu\nu}$,\,
  $\frac{\partial}{\partial x^{\prime}\,\!^{\mu}} \theta^{\prime}\,\!^{\alpha}
    = \frac{\partial}{\partial x^{\prime}\,\!^{\mu}} \bar{\theta}^{\prime}\,\!^{\dot{\beta}} = 0$,\;
  $\frac{\partial}{\partial \theta^{\prime}\,\!^{\alpha}} \theta^{\prime}\,\!^{\beta}
      =\delta_{\alpha\beta}$,\,
  $\frac{\partial}{\partial \theta^{\prime}\,\!^{\alpha}} x^{\prime}\,\!^\mu
    = \frac{\partial}{\partial \theta^{\prime}\,\!^{\alpha}} \bar{\theta}^{\prime}\,\!^{\dot{\beta}}
	= 0$,\;    and\;
  $\frac{\partial}{\rule{0ex}{.8em}\partial \bar{\theta}^{\prime}\,\!^{\dot{\beta}}}
       \bar{\theta}^{\prime}\,\!^{\dot{\alpha}}  = \delta_{\dot{\beta}\dot{\alpha}}$,\,
  $\frac{\partial}{\rule{0ex}{.8em}\partial \bar{\theta}^{\prime}\,\!^{\dot{\beta}}}
       x^{\prime}\,\!^\mu
    = \frac{\partial}{\rule{0ex}{.8em}\partial \bar{\theta}^{\prime}\,\!^{\dot{\beta}}}
	      \theta^{\prime}\,\!^\alpha
	= 0$.
   
\bigskip

\begin{example}
{\bf [$e_{\alpha^\prime}$ in new coordinate-functions
             $(x^{\prime},\theta, \bar{\theta}^{\prime})$]}\footnote{As
			                                      illustrated by this example,
												   the process goes the same as in the case of representing a vector field on a smooth manifold
												    in different coordinate charts
												   because
												    we take {\it Der}$_{\Bbb C}(\widehat{X})$
													   as a left $C^{\infty}(\widehat{X})$-module    and
													all the derivations here are left  derivations.	
												  Cf.\ Footnote~7.
			                                              }\;
{\rm
 Let\\
   $e_{\alpha^{\prime}}
     = \sum_{\nu} c_{\alpha^{\prime}}^{\,\nu}\frac{\partial}{\partial x^{\prime}\,\!^{\nu}}
	    + \sum_{\beta} c_{\alpha^{\prime}}^{\,\beta^{\prime}}
		                                  \frac{\partial}{\partial \theta^{\prime}\,\!^{\beta}}
	    + \sum_{\dot{\beta}} c_{\alpha^{\prime}}^{\,\beta^{\prime\prime}}
		                                                 \frac{\partial}
														   {\rule{0ex}{.8em}
														       \partial \bar{\theta}^{\prime}\,\!^{\dot{\beta}} }$,
 where $c_{\alpha^{\prime}}^{\,\nu}$, $c_{\alpha^{\prime}}^{\,\beta^{\prime}}$,
            $c_{\alpha^{\prime}}^{\,\beta^{\prime\prime}}\in C^{\infty}(\widehat{X})$.
 Then,
   $$
    \begin{array}{c}
	  c_{\alpha^{\prime}}^{\,\nu}\;
	    =\; e_{\alpha^{\prime}}x^{\prime}\,\!^{\nu}\;
	    =\;  2\sqrt{-1}\,\sum_{\dot{\beta}}
			 \sigma^{\nu}_{\alpha\dot{\beta}}\,\bar{\theta}^{\dot{\beta}}\;
        =\;  2\sqrt{-1}\,\sum_{\dot{\beta}}
			 \sigma^{\nu}_{\alpha\dot{\beta}}\,\bar{\theta}^{\prime}\,\!^{\dot{\beta}}\,,\\[1.2ex]	
      c_{\alpha^{\prime}}^{\,\beta^{\prime}}\;
	    =\; e_{\alpha^{\prime}} \theta^{\prime}\,\!^{\beta}\; = \delta_{\alpha\beta}\,, \hspace{2em}
      c_{\alpha^{\prime}}^{\,\beta^{\prime\prime}}\;
	    =\; e_{\alpha^{\prime}} \bar{\theta}^{\prime}\,\!^{\dot{\beta}}\; =\; 0\,.
	\end{array}
   $$
 Thus,
   $e_{\alpha^{\prime}}
     = \frac{\partial}{\partial\theta^{\prime \alpha}}
	     + 2\sqrt{-1}\,
		     \sum_{\nu, \dot{\beta}}\sigma^{\nu}_{\alpha\dot{\beta}}
			   \bar{\theta}^{\dot{\beta}}\,\frac{\partial}{\partial x^{\prime \nu}}$,
   as is given above.			
}\end{example}

\medskip

\begin{definition} {\bf [standard chiral coordinate-functions on $\widehat{X}$]}\;\\ {\rm
 $\,(x^{\prime},\theta^{\prime}, \bar{\theta}^{\prime})$
              are called the {\it standard chiral coordinate-functions} on $\widehat{X}$.
}\end{definition}

\medskip

\begin{definition} {\bf [chiral function-ring \& chiral structure sheaf of $\widehat{X}$]}\; {\rm  
 (1)
 $\; f\in C^{\infty}(\widehat{X})$ is called {\it chiral}\,
       if $\;e_{1^{\prime\prime}}f= e_{2^{\prime\prime}}f =0\,$.
 The set of chiral functions on $\widehat{X}$ is a ${\Bbb C}$-subalgebra of
   $C^{\infty}(\widehat{X})$,
 called the {\it chiral function-ring } of $\widehat{X}$,
  denoted by $C^{\infty}(\widehat{X})^{\scriptsizech}$.
  
 (2)
 Replacing
   $\widehat{X}$ by $\widehat{U}$ for $U\subset X$ open   and
   $e_{1^{\prime\prime}}$, $e_{2^{\prime\prime}}$
     by   $e_{1^{\prime\prime}}|_{\widehat{U}}$, $e_{2^{\prime\prime}}|_{\widehat{U}}$
     in the above setup,
 one obtains the {\it sheaf  of chiral functions}
  $\widehat{\cal O}_X^{\scriptsizech}\subset \widehat{\cal O}_X$,
  also called the {\it chiral structure sheaf} of $\widehat{X}$.
}\end{definition}

\bigskip

Since $e_{\beta^{\prime\prime}}= \partial /\partial \bar{\theta}^{\prime}\,\!^{\dot{\beta}}$
  in terms of the coordinate-functions $(x^{\prime},\theta^{\prime},\bar{\theta}^{\prime})$,
$f$ is chiral if and only if
 $$
    f\; =\;  f^{\prime}_{(0)}(x^{\prime})\,
	           +\, \sum_{\alpha}f^{\prime}_{(\alpha)}(x^{\prime})
			                                                                             \theta^{\prime}\,\!^{\alpha}\,
			   +\, f^{\prime}_{(12)}(x^{\prime})\theta^{\prime}\,\!^1 \theta^{\prime}\,\!^2
 $$
 in chiral coordinate-functions $(x^{\prime},\theta^{\prime},\bar{\theta}^{\prime})$.
 
\bigskip

\begin{lemma}
{\bf [induced $C^{\infty}$-hull structure on $C^{\infty}(\widehat{X})^{\scriptsizech}$]}\;
 $C^{\infty}(\widehat{X})^{\scriptsizech}$ is closed under the $C^{\infty}$-hull structure
 of $C^{\infty}(\widehat{X})$   and,
 hence, has an induced $C^{\infty}$-hull structure from that of $C^{\infty}(\widehat{X})$.
\end{lemma}

\medskip

\begin{proof}
 This is a consequence of the fact that
  in the standard chiral coordinate-functions $(x^{\prime}, \theta^{\prime}, \bar{\theta}^{\prime})$
  for $\widehat{X}$,
 the commuting coordinate-functions
   $x^{\prime}\,\!^0, x^{\prime}\,\!^1, x^{\prime}\,\!^2, x^{\prime}\,\!^3$
   lie in the $C^{\infty}$-hull of $C^{\infty}(\widehat{X})$
 while the nilpotent anticommuting coordinate functions
   $\theta^{\prime}\,\!^1,  \theta^{\prime}\,\!^2,
      \bar{\theta}^{\prime}\,\!^{\dot{1}},  \bar{\theta}^{\prime}\,\!^{\dot{2}},   $
   are identical to $\theta^1,  \theta^2,  \bar{\theta}^{\dot{1}},  \bar{\theta}^{\dot{2}}$.

 Explicitly,  let $h\in C^{\infty}({\Bbb R}^l)$  and
   $f_1,\,\cdots\,, f_l
      \in C^{\infty}(\widehat{X})^{\scriptsizech}\cap C^{\infty}$-{\it hull}\,$(\widehat{X})$.
 Then\\
   $f_i=f^{\prime}_{i,(0)}(x^{\prime})
             + f^{\prime}_{i,(12)}(x^{\prime})\,\theta^{\prime}\,\!^1\theta^{\prime}\,\!^2$,
   for $i=1,\,\ldots\,,\, l$, and hence
   $h(f_1,\,\cdots\,,\, f_l)\in C^{\infty}$-{\it hull}\,$(C^{\infty}(\widehat{X}))$
   can be expressed as
   $\sum_{k=1}^l
      (\partial_kh) (f^{\prime}_{1,(0)}(x^{\prime}),\,\cdots\,, f^{\prime}_{i,(0)}   )
           f^{\prime}_{k,(12)}\theta^{\prime}\,\!^1 \theta^{\prime}\,\!^2$,
   which is chiral.		
   
\end{proof}

\medskip

\begin{lemma}
{\bf [chiral function on $\widehat{X}$ in terms of $(x, \theta, \bar{\theta})$]}\;
 In terms of the standard coordinate functions $(x,\theta, \theta^{\prime})$,
  a chiral function $f$ on $\widehat{X }$ is determined by the components
   $f_{(0)}$, $f_{(\alpha)}$, and $f_{(12)}$ of $f$
   via the following formula
  {\small
  \begin{eqnarray*}
   f & = &
     f_{(0)}(x)\,
      +\, \sqrt{-1}\sum_{\mu}(\partial_\mu f_{(0)})(x)
	                 \cdot \theta\sigma^\mu \bar{\theta}^t\,
      -\frac{1}{2}\sum_{\mu,\nu} (\partial_\mu\partial_\nu f_{(0)})(x)
	                 \cdot \theta\sigma^\mu\bar{\theta}^t \theta\sigma^\nu\bar{\theta}^t\\
   && \hspace{2em}
      +\, \sum_{\alpha}f_{(\alpha)}(x)\cdot \theta^\alpha \,
      +\, \sqrt{-1}\, \sum_{\mu, \alpha} (\partial_\mu f_{(\alpha)})(x)
                          \cdot \theta\sigma^\mu\bar{\theta}^t\theta^\alpha\,	
      +\, f_{(12)}(x)\cdot \theta^1\theta^2             \\ 					
  & = & f_{(0)}(x)\,
      +\, \sum_{\alpha}f_{(\alpha)}(x)\cdot \theta^\alpha\,
	  +\, f_{(12)}(x)\cdot \theta^1\theta^2\,
      +\, \sum_{\alpha, \dot{\beta}}	
	         \sqrt{-1}\,\mbox{\normalsize $\sum$}_\mu
			   \sigma^\mu_{\alpha\dot{\beta}}\cdot (\partial_\mu f_{(0)})(x)
			\cdot\theta^\alpha\bar{\theta}^{\dot{\beta}}\, \\
   && \hspace{2em}
       +\, \sum_{\dot{\beta}}\sqrt{-1}\,\mbox{\normalsize $\sum$}_\mu
	           \mbox{\large $($}
			     \sigma^\mu_{2\dot{\beta}}\cdot(\partial_\mu f_{(1)})(x)
				 - \sigma^\mu_{1\dot{\beta}}\cdot(\partial_\mu f_{(2)})(x)
			   \mbox{\large $)$}
	              \cdot \theta^1\theta^2\bar{\theta}^{\dot{\beta}}\,
       +\, (\square f_{(0)})(x)
             \cdot \theta^1\theta^2\bar{\theta}^{\dot{1}}\bar{\theta}^{\dot{2}}\,	   \\
  & = & f_{(0)}(x) \,
      +\, f_{(1)}(x)\cdot\theta^1\, + f_{(2)}(x)\cdot\theta^2\,
	  +\, f_{(12)}(x)\cdot\theta^1\theta^2\, \\
   && \hspace{2em}
     +\, \sqrt{-1}\, (-\,\partial_0 f_{(0)}+\partial_3f_{(0)})(x)
                      \cdot \theta^1\bar{\theta}^{\dot{1}}\,
     +\, \sqrt{-1}\, (\partial_1 f_{(0)}-\sqrt{-1}\partial_2 f_{(0)})(x)
                      \cdot \theta^1\bar{\theta}^{\dot{2}}   \\
   && \hspace{2em}					
     +\, \sqrt{-1}\, (\partial_1 f_{(0)}+\sqrt{-1}\partial_2f_{(0)})(x)
                      \cdot \theta^2\bar{\theta}^{\dot{1}} \,
     -\, \sqrt{-1}\, (\partial_0 f_{(0)}+\partial_3f_{(0)})(x)
                      \cdot \theta^2\bar{\theta}^{\dot{2}} 					     \\
   && \hspace{2em}
   +\, \sqrt{-1}\, ( \partial_1 f_{(1)}+\sqrt{-1}\partial_2f_{(1)}\,
                               + \partial_0 f_{(2)}- \partial_3 f_{(2)})(x)                             
                             \cdot \theta^1\theta^2\bar{\theta}^{\dot{1}}   \\
   && \hspace{2em}					
   +\, \sqrt{-1}\,(-\partial_0 f_{(1)}- \partial_3f_{(1)}
                               -\partial_1 f_{(2)}+\sqrt{-1}\partial_2f_{(2)})(x)
                             \cdot \theta^1\theta^2\bar{\theta}^{\dot{2}}  \\
   && \hspace{2em}							
   +\, ((\partial_0^2-\partial_1^2-\partial_2^2-\partial_3^2)f_{(0)})(x)
               \cdot\theta^1\theta^2\bar{\theta}^{\dot{1}}\bar{\theta}^{\dot{2}}\,.
  \end{eqnarray*}} 
\end{lemma}

\medskip

\begin{proof}
 Though elementary, we give two methods to prove the statement.
 The first is slick but
    is unclear how it can be generalized to the situation when $f$ takes values on a bundle with a connection.
 The second looks unnecessarily tedious for the current situation but has a direct generalization to the case
   where $f$ takes values on a bundle with a connection on $\widehat{X}$.

 \medskip
 
 \noindent
 $(a)$ {\it First proof}
 
 \smallskip
 
 \noindent
 In terms of the chiral coordinate functions
   $(x^\prime, \theta^\prime, \bar{\theta}^\prime)$,
 $$
   f\;
   =\;  f^\prime(x^\prime, \theta^\prime, \bar{\theta}^\prime)\;
   =\;  f^\prime_{(0)}(x^\prime)\,
	           +\, \sum_\alpha f^\prime_{(\alpha)}(x^\prime)
			                                                                             \theta^\prime\,\!^\alpha\,
			   +\, f^\prime_{(12)}(x^\prime)\theta^\prime\,\!^1 \theta^\prime\,\!^2
 $$
 where
   $f^\prime_{(0)},\,  f^\prime_{(\alpha)},\,  f^\prime_{(12)}
       \in C^\infty({\Bbb R}^4)^{\Bbb C}$.
 By Lemma~1.1.3,  applied to the real and the imaginary component of
        $f^\prime_{(0)},\,  f^\prime_{(\alpha)},\,  f^\prime_{(12)}$,
  \begin{eqnarray*}
   f^\prime_{(\tinybullet)}(x^{\prime})
   & = &
    f^{\prime}_{(\tinybullet)}
     (x+\sqrt{-1}\theta\boldsigma\bar{\theta}^t)   \\
   & = & f^{\prime}_{(\tinybullet)}(x)\,
        +\,\sqrt{-1}\,  \sum_\mu (\partial_{\mu}f^{\prime}_{(\tinybullet)})(x)\,
		         \theta\sigma^\mu \bar{\theta}^t\,
		-\,\frac{1}{2}\,\sum_{\mu,\nu}
		     (\partial_\mu\partial_\nu f^{\prime}_{(\tinybullet)})(x)\,
			     (\theta\sigma^\mu\bar{\theta}^t)(\theta\sigma^\nu\bar{\theta}^t)\,.
  \end{eqnarray*}
 The claim follows from
  applying the expansion to
     $f^\prime_{(0)}(x^\prime)$, $f^\prime_{(\alpha)}(x^\prime)$,  and
     $f^\prime_{(12)}(x^\prime)$,
  collecting like terms in $(\theta,\bar{\theta})$,  and
  re-denoting $f^{\prime}_{(\tinybullet)}(x)$ by $f_{(\tinybullet)}(x)$.
  
 \bigskip

 \noindent
 $(b)$  {\it Second proof}
 
 \medskip
 
 \noindent
 To better organize the terms of polynomial in $(\theta,\bar{\theta})$
   and since only summations of terms are involved,
  we will write a general $f\in C^{\infty}(\widehat{X})$ in a bookkeeping block-matrix\footnote{More
                                                       precisely, one may write $f$ as
													   $(1,  \theta^1, \theta^2, \theta^1\theta^2 )
														     M (1, \bar{\theta}^{\dot{1}},
															         \bar{\theta}^{\dot{2}},
														 \bar{\theta}^{\dot{1}}\bar{\theta}^{\dot{2}})^t$,
                                                         where $M$ is a $4\times 4$ matrix with entries
														 in $C^{\infty}(X)^{\Bbb C}$, and then use $M$ to represent $f$.
																	 } 
  \begin{eqnarray*}
   f & = &  f_{(0)}
	  + \sum_\alpha f_{(\alpha)}\theta^\alpha
	  + \sum_{\dot{\beta}} f_{(\dot{\beta})} \bar{\theta}^{\dot{\beta}}
	  + f_{(12)}\theta^1\theta^2
	  + \sum_{\alpha, \dot{\beta}}f_{(\alpha\dot{\beta})}
	          \theta^\alpha\bar{\theta}^{\dot{\beta}}
	  + f_{(\dot{1}\dot{2})}\bar{\theta}^{\dot{1}}\bar{\theta}^{\dot{2}}\\
    && \hspace{2em} 	  	
	  + \sum_\alpha f_{(\alpha\dot{1}\dot{2})}
	                  \theta^\alpha\bar{\theta}^{\dot{1}}\bar{\theta}^{\dot{2}}
      + \sum_{\dot{\beta}} f_{(12\dot{\beta})}	
                      \theta^1\theta^2 \bar{\theta}^{\dot{\beta}}	
      + f_{(12\dot{1}\dot{2})}					
	                 \theta^1\theta^2\bar{\theta}^{\dot{1}}\bar{\theta}^{\dot{2}}\\
   &=&
     \left[
	  \begin{array} {c|cc|c}
	   f_{(0)}  &  f_{(\dot{1})} & f_{(\dot{2})}  & f_{(\dot{1}\dot{2})}
	                                                                 \\[.6ex] \hline  \rule{0ex}{1em}
	   f_{(1)}  &  f_{(1\dot{1})}& f_{(1\dot{2})}& f_{(1\dot{1}\dot{2})}
	                                                                 \\[.6ex]
	   f_{(2)}  & f_{(2\dot{1})} & f_{(2\dot{2})}& f_{(2\dot{1}\dot{2})}
	                                                                 \\[.6ex] \hline  \rule{0ex}{1em}
	   f_{(12)}& f_{(12\dot{1})} & f_{(12\dot{2})}	& f_{(12\dot{1}\dot{2})}
	  \end{array}
     \right]\,.	
  \end{eqnarray*}
  Here,
    all the coefficients/entries $f_{(\tinybullet)}$ are in $C^{\infty}(X)^{\Bbb C}$   and
	are functions of $x=(x^0, x^1,x^2,x^3)$.
  Recall that
    $e_{\beta^{\prime\prime}}
	  :=  -\frac{\partial}{\rule{0ex}{.8em}\partial \bar{\theta}^{\dot{\beta}}}
	       -\sqrt{-1}\sum_{\mu,\alpha}
		       \theta^\alpha\sigma^\mu_{\alpha\dot{\beta}}\partial_\mu$,
	 with $\partial_\mu := \frac{\partial}{\partial x^\mu}$.
  Thus,
   {\footnotesize
   \begin{eqnarray*}
    e_{1^{\prime\prime}}f     & =
	 &   -\frac{\partial}{\rule{0ex}{.8em}\partial \bar{\theta}^{\dot{1}}}f\,
	       -\, \sqrt{-1}\sum_{\mu,\alpha}
		        \theta^\alpha\sigma^\mu_{\alpha\dot{1}}\partial_\mu f   \\
   &=& -\,\left[
	       \begin{array}{c|cc|c}
		     f_{(\dot{1})}    & 0   &  f_{(\dot{1}\dot{2})} &  0 \\[.6ex]
			                                                             \hline \rule{0ex}{1em}
		    -f_{(1\dot{1})}    & 0   &  -f_{(1\dot{1}\dot{2})}&   0\\[.6ex]
            -f_{(2\dot{1})}	   & 0   &  -f_{(2\dot{1}\dot{2})}&  0\\[.6ex]	
			                                                             \hline	\rule{0ex}{1em}
            f_{(12\dot{1})} & 0   &  f_{(12\dot{1}\dot{2})}& 0
	       \end{array}
	      \right]\,      \\
   &&\hspace{-10em}
      -\,\sqrt{-1}\,
	    \left[
		 \begin{array}{c|cc|c}
		  0 & 0 & 0 & 0  \\ \hline \rule{0ex}{1.2em}
		  \sum_\mu \sigma^\mu_{1\dot{1}}\partial_\mu f_{(0)}
		     & \sum_\mu  \sigma^\mu_{1\dot{1}}\partial_\mu f_{(\dot{1})}
		     & \sum_\mu \sigma^\mu_{1\dot{1}}\partial_\mu f_{(\dot{2})}
			 & \sum_\mu \sigma^\mu_{1\dot{1}}\partial_\mu f_{(\dot{1}\dot{2})} \\[1ex]
		  \sum_\mu \sigma^\mu_{2\dot{1}}\partial_\mu f_{(0)}
		     & \sum_\mu \sigma^\mu _{2\dot{1}}\partial_\mu f_{(\dot{1})}
			 & \sum_\mu \sigma^\mu_{2\dot{1}}\partial_\mu f_{(\dot{2})}
             & \sum_\mu \sigma^\mu_{2\dot{1}}\partial_\mu f_{(\dot{1}\dot{2})}	
			                                                     \\[1ex] \hline \rule{0ex}{1.2em}
		  \sum_\mu(\sigma^\mu_{1\dot{1}}\partial_\mu f_{(2)}
                                - \sigma^\mu_{2\dot{1}}\partial_\mu f_{(1)}		  )
			 & \sum_\mu (\sigma^\mu_{1\dot{1}}\partial_\mu f_{(2\dot{1})}
			                          -\sigma^\mu_{2\dot{1}}\partial_\mu f_{(1\dot{1})})
			 & \sum_\mu (\sigma^\mu_{1\dot{1}}\partial_\mu f_{(2\dot{2})}
                                      - \sigma^\mu_{2\dot{1}}\partial_\mu f_{(1\dot{2})})
			 & \sum_\mu (\sigma^\mu_{1\dot{1}}\partial_\mu f_{(2\dot{1}\dot{2})}
			                         -\sigma^\mu_{2\dot{1}}\partial_\mu f_{(1\dot{1}\dot{2})})
		 \end{array}
		\right]   \\[1.2ex]
   \end{eqnarray*}}
   
 \noindent
 and
   {\footnotesize
   \begin{eqnarray*}
    e_{2^{\prime\prime}}f     & =
	 &   -\frac{\partial}{\rule{0ex}{.8em}\partial \bar{\theta}^{\dot{2}}}f\,
	       -\, \sqrt{-1}\sum_{\mu,\alpha}
		        \theta^\alpha\sigma^\mu_{\alpha\dot{2}}\partial_\mu f   \\
   &=& -\,\left[
	       \begin{array}{c|cc|c}
		     f_{(\dot{2})}    &   -f_{(\dot{1}\dot{2})} & 0 &  0 \\[.6ex]
			                                                             \hline \rule{0ex}{1em}
		    -f_{(1\dot{2})}    &   f_{(1\dot{1}\dot{2})}& 0 &  0\\[.6ex]
            -f_{(2\dot{2})}	   &   f_{(2\dot{1}\dot{2})}& 0 & 0\\[.6ex]	
			                                                             \hline	\rule{0ex}{1em}
            f_{(12\dot{2})} &   -f_{(12\dot{1}\dot{2})}& 0 & 0
	       \end{array}
	      \right]\,      \\
   &&\hspace{-10em}
      -\,\sqrt{-1}\,
	    \left[
		 \begin{array}{c|cc|c}
		  0 & 0 & 0 & 0  \\ \hline \rule{0ex}{1.2em}
		  \sum_\mu \sigma^\mu_{1\dot{2}}\partial_\mu f_{(0)}
		     & \sum_\mu \sigma^\mu_{1\dot{2}}\partial_\mu f_{(\dot{1})}
		     & \sum_\mu \sigma^\mu_{1\dot{2}}\partial_\mu f_{(\dot{2})}
			 & \sum_\mu \sigma^\mu_{1\dot{2}}\partial_\mu f_{(\dot{1}\dot{2})} \\[1ex]
		  \sum_\mu \sigma^\mu_{2\dot{2}}\partial_\mu f_{(0)}
		     & \sum_\mu \sigma^\mu _{2\dot{2}}\partial_\mu f_{(\dot{1})}
			 & \sum_\mu \sigma^\mu_{2\dot{2}}\partial_\mu f_{(\dot{2})}
             & \sum_\mu \sigma^\mu_{2\dot{2}}\partial_\mu f_{(\dot{1}\dot{2})}	
			                                                     \\[1ex] \hline \rule{0ex}{1.2em}
		  \sum_\mu(\sigma^\mu_{1\dot{2}}\partial_\mu f_{(2)}
                                - \sigma^\mu_{2\dot{2}}\partial_\mu f_{(1)}		  )
			 & \sum_\mu (\sigma^\mu_{1\dot{2}}\partial_\mu f_{(2\dot{1})}
			                          -\sigma^\mu_{2\dot{2}}\partial_\mu f_{(1\dot{1})})
			 & \sum_\mu (\sigma^\mu_{1\dot{2}}\partial_\mu f_{(2\dot{2})}
                                      - \sigma^\mu_{2\dot{2}}\partial_\mu f_{(1\dot{2})})
			 & \sum_\mu (\sigma^\mu_{1\dot{2}}\partial_\mu f_{(2\dot{1}\dot{2})}
			                         -\sigma^\mu_{2\dot{2}}\partial_\mu f_{(1\dot{1}\dot{2})})
		 \end{array}
		\right]\,.
   \end{eqnarray*}}
  
 \bigskip
 
 \noindent
 Setting $e_{1^{\prime\prime}}f=e_{2^{\prime\prime}}f=0$, one obtains
 $$
  \begin{array} {c}
   \begin{array}{ll}
    f_{(\dot{1})}\;=\;   f_{(\dot{2})}\;   =\; f_{(\dot{1}\dot{2})}\;=\; 0\,,
      &  f_{(\alpha\dot{\beta})}\;
           =\; \sqrt{-1}\sum_\mu\sigma^\mu_{\alpha\dot{\beta}}\partial_\mu f_{(0)}\,,\\[2ex]
    f_{(12\dot{1})}\;
       =\; 	 \sqrt{-1} \sum_\mu
	             (\sigma^\mu_{2\dot{1}}\partial_\mu f_{(1)}
                      - \sigma^\mu_{1\dot{1}}\partial_\mu f_{(2)}	)\,,
      & f_{(12\dot{2})}\;
         =\; 	 \sqrt{-1} \sum_\mu
	             (\sigma^\mu_{2\dot{2}}\partial_\mu f_{(1)}
                      - \sigma^\mu_{1\dot{2}}\partial_\mu f_{(2)}	)\,, \\[2ex]
    f_{(1\dot{1}\dot{2})}\;
      =\; \sqrt{-1}\sum_\mu\sigma^\mu_{1\dot{1}}\partial_\mu f_{(\dot{2})} \;
	  =\; 0\,,
       & f_{(2\dot{1}\dot{2})}\;
          =\; \sqrt{-1}\sum_\mu\sigma^\mu_{2\dot{1}}\partial_\mu f_{(\dot{1})} \;
	      =\; 0\,,
   \end{array}
        \\[7ex]
   \begin{array}{ccl}
      f_{(12\dot{1}\dot{2})}
       & = & 	 \sqrt{-1}\sum_\mu
	              (\sigma^\mu_{2\dot{1}}\partial_\mu f_{(1\dot{2})}
	                   -\sigma^\mu_{1\dot{1}}\partial_\mu f_{(2\dot{2})}) \\[1.2ex]
       & = &  	 \sqrt{-1}\sum_\mu
	              (\sigma^\mu_{2\dot{2}}\partial_\mu f_{(1\dot{1})}
	                   -\sigma^\mu_{1\dot{2}}\partial_\mu f_{(2\dot{1})})\;\;
  	           =\;\; (\partial_0^2-\partial_1^2-\partial_2^2-\partial_3^2)f_{(0)}
   \end{array}
  \end{array}
 $$
 and a redundant collection of equations that are automatically satisfied under the above system of equations:
 $$
  \begin{array}{ll}
   -\sqrt{-1}\sum_\mu\sigma^\mu_{1\dot{1}}\partial_\mu f_{(\dot{1})}\;=\; 0\,,
      & -\sqrt{-1}\sum_\mu\sigma^\mu_{1\dot{2}}\partial_\mu f_{(\dot{2})}\;=\; 0\,,  \\ [2ex]
  -\sqrt{-1}\sum_\mu\sigma^\mu_{2\dot{1}}\partial_\mu f_{(\dot{1})}\;=\; 0\,,	
      & -\sqrt{-1}\sum_\mu\sigma^\mu_{2\dot{2}}\partial_\mu f_{(\dot{2})}\;=\; 0\,,   \\[2ex]
	\sqrt{-1}\sum_\mu (\sigma^\mu_{2\dot{1}}\partial_\mu f_{(1\dot{1})}
	               - \sigma^\mu_{1\dot{1}}\partial_\mu f_{(2\dot{1})}) \;=\; 0\,,
	  & \sqrt{-1}\sum_\mu (\sigma^\mu_{2\dot{2}}\partial_\mu f_{(1\dot{2})}
	               - \sigma^\mu_{1\dot{2}}\partial_\mu f_{(2\dot{2})}) \;=\; 0\,,             \\[2ex]
   -\sqrt{-1}\sum_\mu\sigma^\mu_{1\dot{1}}\partial_\mu f_{(\dot{1}\dot{2})}\;=\; 0\,,
     & -\sqrt{-1}\sum_\mu\sigma^\mu_{1\dot{2}}\partial_\mu f_{(\dot{1}\dot{2})}\;=\; 0\,,
	                                                       \\[2ex]
   -\sqrt{-1}\sum_\mu\sigma^\mu_{2\dot{1}}\partial_\mu f_{(\dot{1}\dot{2})}\;=\; 0\,,
     & -\sqrt{-1}\sum_\mu\sigma^\mu_{2\dot{2}}\partial_\mu f_{(\dot{1}\dot{2})}\;=\; 0\,,
	                                                       \\[2ex]														   
    \sqrt{-1}\sum_\mu(\sigma^\mu_{2\dot{1}}\partial_\mu f_{(1\dot{1}\dot{2})}
                         -\sigma^\mu_{1\dot{1}}\partial_\mu f_{(2\dot{1}\dot{2})}	)\;=\; 0\,,
	& \sqrt{-1}\sum_\mu(\sigma^\mu_{2\dot{2}}\partial_\mu f_{(1\dot{1}\dot{2})}
                         -\sigma^\mu_{1\dot{2}}\partial_\mu f_{(2\dot{1}\dot{2})}	)\;=\; 0\,.
  \end{array}
 $$
 This proves the lemma.
 
\end{proof}

\bigskip

\noindent
(2) {\it The antichiral sector of $\widehat{X}$}

\medskip

\noindent
The above chiral sector of $\widehat{X}$  has an antichiral partner,
which follows a very similar setup and is summarized below.

Let
 $$
   x^{\prime\prime}\,\!^{\mu}\; :=\;    x- \sqrt{-1}\,\theta\sigma^{\mu}\bar{\theta}^t\,,\,  \hspace{1em}
   \theta^{\prime\prime}\,\!^{\alpha}\;  :=\; \theta\,, \hspace{1em}
   \bar{\theta}^{\prime\prime}\,\!^{\dot{\beta}}   \; :=\; \bar{\theta}^{\dot{\beta}}\,, \hspace{2em}
   \mu=0,1,2,3\,,\;\; \alpha=1,2\,,\;\; \ \dot{\beta}=\dot{1}, \dot{2},
 $$
 denoted collectively as $(x^{\prime\prime}, \theta^{\prime\prime}, \ \bar{\theta}^{\prime\prime})$,
 be another  set of coordinate-functions in $C^{\infty}(\widehat{X})$.
They satisfy
 $$
   \begin{array}{lclcl}
    e_{\mu}x^{\prime\prime}\,\!^{\nu}\;=\; \delta_{\mu\nu}\,,
	  && e_{\mu} \theta^{\prime\prime}\,\!^{\alpha}\; =\; 0\,,
	  && e_{\mu}\bar{\theta}^{\prime\prime}\,\!^{\dot{\beta}}\; =\; 0\,, \\[1.2ex]
	e_{\alpha^{\prime}}x^{\prime\prime}\,\!^{\mu}\;
	      =\; 0\,,
      && e_{\alpha^{\prime}}\theta^{\prime\prime}\,\!^{\beta}\;=\; \delta_{\alpha\beta}\,,
      && e_{\alpha^{\prime}}\bar{\theta}^{\prime\prime}\,\!^{\dot{\beta}}	  \;=\; 0\,, \\[1.2ex]
	e_{\beta^{\prime\prime}} x^{\prime\prime}\,\!^{\mu}\;
	     =\;  -\, 2\sqrt{-1}\,  \sum_{\alpha}\theta^{\alpha}\sigma^{\mu}_{\alpha\dot{\beta}}\,,
	  && e_{\beta^{\prime\prime}}\theta^{\prime\prime}\,\!^{\alpha}\;=\; 0\,,
	  && e_{\beta^{\prime\prime}}\bar{\theta}^{\prime\prime}\,\!^{\dot{\alpha}}\;
	         =\; -\,\delta_{\dot{\beta}\dot{\alpha}}
   \end{array}
 $$
and, hence, in terms of the new coordinate functions
 $(x^{\prime\prime}, \theta^{\prime\prime}, \bar{\theta}^{\prime\prime})$, one has
 $$
   e_{\mu}\;=\; \frac{\partial}{\partial x^{\prime\prime}\,\!^{\mu}}\,,\hspace{2em}
   e_{\alpha^{\prime}}\;
      =\;  \frac{\partial}{\partial \theta^{\prime\prime}\,\!^{\alpha}}\,,\hspace{2em}
   e_{\beta^{\prime\prime}}\;
      =\; -\,\frac{\partial}{\rule{0ex}{.8em}\partial \bar{\theta}^{\prime\prime}\,\!^{\dot{\beta}}}\,
	             -\,2\sqrt{-1}\,\sum_{\nu, \alpha} \theta^{\prime\prime}\,\!^{\alpha}
				                      \sigma^{\nu}_{\alpha\dot{\beta}}\,
									  \frac{\partial}{\partial x^{\prime\prime}\,\!^{\nu}}\,.
 $$
Here,
 $\partial/\partial x^{\prime\prime}\,\!^{\mu} $, $\partial/\partial\theta^{\prime\prime}\,\!^{\alpha}$,
   $\partial/\partial \bar{\theta}^{\prime\prime}\,\!^{\dot{\theta}}$,
     $\mu=0,1,2,3$, $\alpha=1,2$, $\dot{\beta}=\dot{1}, \dot{2}$,
   are the derivations of $C^{\infty}(\widehat{X})$
   associated to $(x^{\prime\prime},\theta^{\prime\prime},\bar{\theta}^{\prime\prime})$,
  characterized by
  $\frac{\partial}{\partial x^{\prime\prime}\,\!^{\mu}}
      x^{\prime\prime}\,\!^{\nu}=\delta_{\mu\nu}$,\,
  $\frac{\partial}{\partial x^{\prime\prime}\,\!^{\mu}} \theta^{\prime\prime}\,\!^{\alpha}
    = \frac{\partial}{\partial x^{\prime\prime}\,\!^{\mu}}
	              \bar{\theta}^{\prime\prime}\,\!^{\dot{\beta}} = 0$,\; \\
  $\frac{\partial}{\partial \theta^{\prime\prime}\,\!^{\alpha}} \theta^{\prime\prime}\,\!^{\beta}
      =\delta_{\alpha\beta}$,\,
  $\frac{\partial}{\partial \theta^{\prime\prime}\,\!^{\alpha}} x^{\prime\prime}\,\!^\mu
    = \frac{\partial}{\partial \theta^{\prime\prime}\,\!^{\alpha}}
	              \bar{\theta}^{\prime\prime}\,\!^{\dot{\beta}}
	= 0$,\;    and\;
  $\frac{\partial}{\rule{0ex}{.8em}\partial \bar{\theta}^{\prime\prime}\,\!^{\dot{\beta}}}
       \bar{\theta}^{\prime\prime}\,\!^{\dot{\alpha}}  = \delta_{\dot{\beta}\dot{\alpha}}$,\,
  $\frac{\partial}{\rule{0ex}{.8em}\partial \bar{\theta}^{\prime\prime}\,\!^{\dot{\beta}}}
       x^{\prime\prime}\,\!^\mu
    = \frac{\partial}{\rule{0ex}{.8em}\partial \bar{\theta}^{\prime\prime}\,\!^{\dot{\beta}}}
	      \theta^{\prime\prime}\,\!^\alpha
	= 0$.
    
\bigskip

\begin{definition} {\bf [standard antichiral coordinate-functions on $\widehat{X}$]}\\ {\rm
 $(x^{\prime\prime},\theta^{\prime\prime}, \bar{\theta}^{\prime\prime})$
              are called the {\it standard antichiral coordinate-functions} on $\widehat{X}$.
}\end{definition}

\medskip

\begin{definition} {\bf [antichiral function-ring \& antichiral structure sheaf of $\widehat{X}$]}\\
{\rm
 (1)
 $\; f\in C^{\infty}(\widehat{X})$ is called {\it antichiral}\,
       if $\;e_{1^\prime}f= e_{2^\prime}f =0\,$.
 The set of antichiral functions on $\widehat{X}$ is a ${\Bbb C}$-subalgebra of
   $C^{\infty}(\widehat{X})$,
 called the {\it antichiral function-ring } of $\widehat{X}$,
  denoted by $C^{\infty}(\widehat{X})^{\scriptsizeach}$.\\
 (2)
 Replacing
   $\widehat{X}$ by $\widehat{U}$ for $U\subset X$ open   and
   $e_{1^\prime}$, $e_{2^\prime}$
     by   $e_{1^\prime}|_{\widehat{U}}$, $e_{2^\prime}|_{\widehat{U}}$
     in the above setup,
 one obtains the {\it sheaf  of antichiral functions}
  $\widehat{\cal O}_X^{\scriptsizeach}\subset \widehat{\cal O}_X$,
  also called the {\it antichiral structure sheaf} of $\widehat{X}$.
}\end{definition}

\bigskip

Since $e_{\alpha^\prime}= \partial /\partial \theta^{\prime\prime}\,\!^{\alpha}$
  in terms of the coordinate-functions
  $(x^{\prime\prime},\theta^{\prime\prime},\bar{\theta}^{\prime\prime})$,
$f$ is antichiral if and only if
 $$
    f\; =\;  f^{\prime\prime}_{(0)}(x^{\prime\prime})\,
	           +\, \sum_{\dot{\beta}}f^{\prime\prime}_{(\dot{\beta})}(x^{\prime\prime})\,
         			   \bar{\theta}^{\prime\prime}\,\!^{\dot{\beta}}\,
			   +\, f^{\prime\prime}_{(12)}(x^{\prime\prime})\,
			           \bar{\theta}^{\prime\prime}\,\!^{\dot{1}}
					   \bar{\theta}^{\prime\prime}\,\!^{\dot{2}}
 $$
 in antichiral coordinate-functions
 $(x^{\prime\prime},\theta^{\prime\prime},\bar{\theta}^{\prime\prime})$.
 
\bigskip

\begin{lemma}
{\bf [induced $C^{\infty}$-hull structure on $C^{\infty}(\widehat{X})^{\scriptsizeach}$]}\;
 $C^{\infty}(\widehat{X})^{\scriptsizeach}$ is closed under the $C^{\infty}$-hull structure
 of $C^{\infty}(\widehat{X})$   and,
 hence, has an induced $C^{\infty}$-hull structure from that of $C^{\infty}(\widehat{X})$.
\end{lemma}

\medskip

\begin{lemma}
{\bf [antichiral function on $\widehat{X}$ in terms of $(x, \theta, \bar{\theta})$]}\;
 In terms of the standard coordinate functions $(x,\theta, \theta^{\prime})$,
  an antichiral function $f$ on $\widehat{X }$ is determined by the components
   $f_{(0)}$, $f_{(\dot{\beta})}$, and $f_{(\dot{1}\dot{2})}$ of $f$
  via the following formula
  {\small
  \begin{eqnarray*}
   f & = &
     f_{(0)}(x)\,
      -\, \sqrt{-1}\sum_{\mu}(\partial_\mu f_{(0)})(x)
	                 \cdot \theta\sigma^\mu \bar{\theta}^t\,
      -\frac{1}{2}\sum_{\mu,\nu} (\partial_\mu\partial_\nu f_{(0)})(x)
	                 \cdot \theta\sigma^\mu\bar{\theta}^t \theta\sigma^\nu\bar{\theta}^t\\
   && \hspace{2em}
      +\, \sum_{\dot{\beta}}f_{(\dot{\beta})}(x)\cdot \bar{\theta}^{\dot{\beta}} \,
      -\, \sqrt{-1}\, \sum_{\mu, \dot{\beta}} (\partial_\mu f_{(\dot{\beta})})(x)
                          \cdot \theta\sigma^\mu\bar{\theta}^t \bar{\theta}^{\dot{\beta}}\,	
      +\, f_{(\dot{1}\dot{2})}(x)\cdot \bar{\theta}^{\dot{1}}\bar{\theta}^{\dot{2}}\\
  & = & f_{(0)}(x)\,
      +\, \sum_{\dot{\beta}}f_{(\dot{\beta})}(x)\cdot \bar{\theta}^{\dot{\beta}}\,
	  +\, f_{(\dot{1}\dot{2})}(x)\cdot \bar{\theta}^{\dot{1}}\bar{\theta}^{\dot{2}}\,
	  -\, \sum_{\alpha,\dot{\beta}}
	        \sqrt{-1}\,\mbox{\normalsize $\sum$}_\mu
			    \sigma^\mu_{\alpha\dot{\beta}}\cdot(\partial_\mu f_{(0)})(x)
             \cdot \theta^\alpha\bar{\theta}^{\dot{\beta}}\,			\\
   && \hspace{2em}			
      +\, \sum_{\alpha}
	          \sqrt{-1}\,\mbox{\normalsize $\sum$}_\mu	
                \mbox{\large $($}
				 \sigma^\mu_{\alpha\dot{2}}\cdot(\partial_\mu f_{(\dot{1})})(x)
                   - \sigma^\mu_{\alpha\dot{1}}\cdot(\partial_\mu f_{(\dot{2})})(x)
				 \mbox{\large $)$}
              \cdot \theta^\alpha\bar{\theta}^{\dot{1}}\bar{\theta}^{\dot{2}}
	 +\, (\square f_{(0)})(x)
			 \cdot \theta^1\theta^2\bar{\theta}^{\dot{1}}\bar{\theta}^{\dot{2}} \\           		
  & = & f_{(0)}(x) \,
      +\, f_{(\dot{1})}(x)\cdot\bar{\theta}^{\dot{1}}\,
	  +\, f_{(\dot{2})}(x)\cdot\bar{\theta}^{\dot{2}}\,
	  +\, f_{(\dot{1}\dot{2})}(x)\cdot\bar{\theta}^{\dot{1}}\bar{\theta}^{\dot{2}}\, \\
   && \hspace{2em}
     -\, \sqrt{-1}\, (-\,\partial_0 f_{(0)}+\partial_3f_{(0)})(x)
                      \cdot \theta^1\bar{\theta}^{\dot{1}}\,
     -\, \sqrt{-1}\, (\partial_1 f_{(0)}-\sqrt{-1}\partial_2 f_{(0)})(x)
                      \cdot \theta^1\bar{\theta}^{\dot{2}}   \\
   && \hspace{2em}					
     -\, \sqrt{-1}\, (\partial_1 f_{(0)}+\sqrt{-1}\partial_2f_{(0)})(x)
                      \cdot \theta^2\bar{\theta}^{\dot{1}} \,
     +\, \sqrt{-1}\, (\partial_0 f_{(0)}+\partial_3f_{(0)})(x)
                      \cdot \theta^2\bar{\theta}^{\dot{2}} 					     \\
   && \hspace{2em}
   +\, \sqrt{-1}\, ( \partial_1 f_{(\dot{1})}-\sqrt{-1}\partial_2f_{(\dot{1})}\,
                               + \partial_0 f_{(\dot{2})}- \partial_3 f_{(\dot{2})})(x)                             
                             \cdot \theta^1\bar{\theta}^{\dot{1}}\bar{\theta}^{\dot{2}}   \\
   && \hspace{2em}					
   -\, \sqrt{-1}\,(\partial_0 f_{(\dot{1})}+ \partial_3f_{(\dot{1})}
                               +\partial_1 f_{(\dot{2})}+\sqrt{-1}\partial_2f_{(\dot{2})})(x)
                             \cdot \theta^2\bar{\theta}^{\dot{1}}\bar{\theta}^{\dot{2}}  \\
   && \hspace{2em}							
   +\, ((\partial_0^2-\partial_1^2-\partial_2^2-\partial_3^2)f_{(0)})(x)
               \cdot\theta^1\theta^2\bar{\theta}^{\dot{1}}\bar{\theta}^{\dot{2}}\,.
  \end{eqnarray*}} 
\end{lemma}

\bigskip

\section{$d=4$, $N=1$ Azumaya/matrix superspaces $\widehat{X}^{\!A\!z}$ with a fundamental module
                    with a connection}

In Sec.\ 1.3, we set the convention that a derivation $\xi\in \Der_{\Bbb C}(\widehat{X}) $
 apply to $f\in C^{\infty}(\widehat{X})$ from the left (of $f$).
For this reason, it is natural to take as convention that the covariant derivation
 $\widehat{\nabla}_{\xi}$ defined by a connection $\widehat{\nabla}$
 on a vector bundle $\widehat{E}$ on $\widehat{X}$
  apply to a section $\widehat{s}$ of $\widehat{E}$ also from the left.
This turns out fine if one only consider connections that are purely even.
While mathematically a theory of purely even, left connections gives a sound theory,
 the physicist's construction of a connection from a vector multiplet gives a connection that include also an odd part.
This suggests that one needs to consider a general left connections,
 which include not only the even part, such as $d$, but also the odd part.
Once making such a generalization, one finds new subtleties (Sec.\ 2.1).
The resolution of these subtleties leads us to the notion of {\it hybrid connections} (Sec.\ 2.2),
 which best fit physicists' construction of a connection from a vector superfield (Sec.\ 2.3).
Using this,
 together with the notion of {\it Azumaya/matrix supermanfolds with a fundamental module}
 developed in [L-Y4] (D(11.2)),
 one obtains the $d=4$, $N=1$ {\it Azumaya/matrix superspace with a fundamental module with a connection}
 (Sec.\ 2.4).
This describes the world-volume of fermionic D3-branes with $N=1$ world-volume supersymmetry.

\bigskip
   
\subsection{Lessons
   from left connections on the Chan-Paton bundle $\widehat{E}$ over $\widehat{X}$}
The notion of (left) derivations on $\widehat{X}$ can be generalized naturally to
  the notion of `{\it left connections}' on a bundle $\widehat{E}$ over $\widehat{X}$.
This is a good mathematical theory by itself and they behave well under parity-preserving gauge transformations.
But as we will learn the notion does not fit well under general gauge transformations that mix
 the even component of $\widehat{E}$ and the odd component of $\widehat{E}$.
The lesson will guide us how to ``correct"  this in Sec.~2.2.

\bigskip

\begin{flushleft}
{\bf The old setup}
\end{flushleft}
Recall first the basic setup, notations, and a fact: ([L-Y4] (D(11.2)) and [L-Y9] (D(11.4.1)))
 \begin{itemize}
  \item[\LARGE $\cdot$]
   Denote by $\widehat{\Bbb C}$
     the algebra ${\Bbb C}[\theta^1,\theta^2,\theta^{\dot{1}},\theta^{\dot{2}}]^{\anticommuting}$
     of complex Grassmann numbers.
  
  \item[\LARGE $\cdot$]
   Let $\widehat{\frak m}
             :=(\theta^1, \theta^2, \bar{\theta}^{\dot{1}}, \bar{\theta}^{\dot{2}})$
    be  the ideal sheaf of the $4$-dimensional Minkowski space-time $X$
    in the $d=4$, $N=1$ superspace $\widehat{X}$ as a super $C^{\infty}$-subscheme.
   
  \item[\LARGE $\cdot$]
   Let $E$ be a complex vector bundle of rank $r$ on $X$.
   The corresponding sheaf of smooth sections is denoted by ${\cal E}$.
   Denote by $\End_{\Bbb C}(E)$ (resp.\ $\Aut_{\Bbb C}(E)$)
     the bundle of endomorphisms (resp.\ the bundle of automorphisms) of $E$.
   The corresponding sheaves are denoted by
     $\Endsheaf_{{\cal O}_X^{\,\Bbb C}}({\cal E})$   and
	 $\Autsheaf_{{\cal O}_X^{\,\Bbb C}}({\cal E})$
	 respectively.
  In this old setup, by default,
     $\End_{\Bbb C}(E)$ and $\Endsheaf_{{\cal O}_X^{\,\Bbb C}}({\cal E})$
     act respectively on $E$ and ${\cal E}$ from the left.
          
  \item[\LARGE $\cdot$]	
   Let $\widehat{E}$ be the complex vector bundle of rank $r$ on $\widehat{X}$ that extends $E$;
     each fiber of $\widehat{E}$ over $X$ is a free bi-$\widehat{\Bbb C}$-module of rank $r$.
   The corresponding sheaf of smooth sections is a bi-$\widehat{\cal O}_X$-modules,
     denoted by $\widehat{\cal E}$.
   $\widehat{\cal E}= {\cal E}\otimes_{{\cal O}_X^{\,\Bbb C}}\widehat{\cal O}_X$
      is ${\Bbb Z}/2$-graded;
     and the left and the right locally free $\widehat{\cal O}_X$-module structure of $\widehat{\cal E}$
	 are related by $sa = (-1)^{p(s)p(a)}as$
	 for $a\in \widehat{\cal O}_X,\, s\in \widehat{\cal E} $ parity homogeneous.
  
  \item[\LARGE $\cdot$]
   Let $\End_{\Right\widehat{\Bbb C}}(\widehat{E})$
     be the bundle of endomorphisms of $\widehat{E}$ as a right $\widehat{\Bbb C}$-module over $X$.
   The corresponding sheaf of endomorphisms of $\widehat{\cal E}$ is denoted by
     $\Endsheaf_{\Right\widehat{\cal O}_X}(\widehat{\cal E})$.   		
   By default,
     $\End_{\Right\widehat{\Bbb C}}(\widehat{E})$ and
	    $\Endsheaf_{\widehat{\cal O}_X}(\widehat{\cal E})$
     act respectively on $\widehat{E}$ and $\widehat{\cal E}$ from the left.
 
  \item[\LARGE $\cdot$]
   Let $\Aut_{\Right\widehat{\Bbb C}}(\widehat{E})
             \subset \End_{\Right\widehat{\Bbb C}}(\widehat{E})$
     be the bundle of automorphisms of $\widehat{E}$ as a right $\widehat{\Bbb C}$-module over $X$.
   The corresponding sheaf of automorphisms of $\widehat{\cal E}$ is denoted by
     $\Autsheaf_{\Right\widehat{\cal O}_X}(\widehat{\cal E})$.   		
   Recall [L-Y9: Lemma 2.2.1.1 \& Corollary 2.2.1.2] (D(11.4.1)) that
      $$
         \Autsheaf_{\Right\widehat{\cal O}_X}(\widehat{\cal E})\;
          \simeq\; \Autsheaf_{{\cal O}_X^{\,\Bbb C}}({\cal E})\,
		                    \oplus\,  \Endsheaf_{{\cal O}_X^{\,\Bbb C}}({\cal E})		
							                                 \otimes_{{\cal O}_X^{\,\Bbb C}} \widehat{\frak m}\,.
      $$
  The set $C^{\infty}(\Aut_{\Right\widehat{\Bbb C}}(\widehat{E}))$	
     of smooth sections of $\Aut_{\Right\widehat{\Bbb C}}(\widehat{E})$
    forms the group of {\it gauge transformations} of $\widehat{E}$.
 \end{itemize}

\bigskip

\begin{flushleft}
{\bf General aspects of left connections on $\widehat{\cal E}$ and their curvature tensor}
\end{flushleft}
\begin{definition} {\bf [even part vs.\ odd part of ${\Bbb C}$-bilinear pairing]}\; {\rm
 Let $V_1$, $V_2$, and $W$ be ${\Bbb Z}/2$-graded ${\Bbb C}$-vector spaces and
   $$
    \begin{array}{ccccc}
	 A & : & V_1 \times V_2
	     & \longrightarrow   & W  \\[1.2ex]
    && (v_1, v_2)              &  \longmapsto    &   v_1Av_2		 	
	\end{array}
  $$
  be a ${\Bbb C}$-bilinear pairing that is applied to $V_1$ from the right
   and to $V_2$ from the left.
 The {\it even part} $A^{(\even)}$ of  $A$
    is defined to be the ${\Bbb C}$-bilinear pairing $V_1\times V_2\rightarrow W$ with
   $$
     v_1 A^{(\even)}v_2\; :=\;
      \left\{
	    \begin{array}{ll}
		(v_1 A v_2)^{(\even)}
		      & \hspace{2em}\mbox{if $v_1$ and $v_2$ are both even or both odd}\,,\\[1.2ex]
		(v_1 A v_2)^{(\odd)}
		      & \hspace{2em}\mbox{if either ($v_1$ even, $v_2$ odd) or  ($v_1$ odd, $v_2$ even)}
        \end{array}		
	  \right.	
   $$
  while the {\it odd part} $A^{(\odd)}$ of $A$
    is defined to be the ${\Bbb C}$-bilinear pairing $V_1\times V_2\rightarrow W$ with
   $$
     v_1 A^{(\odd)}v_2\; :=\;
      \left\{
	    \begin{array}{ll}
		(v_1A v_2)^{(\odd)}
		      & \hspace{2em}\mbox{if $v_1$ and $v_2$ are both even or both odd}\,,\\[1.2ex]
		(v_1 A v_2)^{(\even)}
		      & \hspace{2em}\mbox{if either ($v_1$ even, $v_2$ odd) or  ($v_1$ odd, $v_2$ even)}
        \end{array}		
	  \right.	
   $$
 Here,
    $(v_1Av_2)^{(\even)}$\, (resp.\ $(v_1Av_2)^{(\odd)}$) is the even component
  (resp.\ odd component) of $v_1Av_2 \in W$.
 By construction, $A=A^{(\even)}+ A^{(\odd)}$.	
}\end{definition}

\medskip

\begin{definition} {\bf [left connection on $\widehat{\cal E}$]}\; {\rm
 A {\it left connection} $\widehat{\nabla}$ on $\widehat{\cal E}$
   is a ${\Bbb C}$-bilinear pairing\footnote{The
                                                                       notion of {\it right connection} can be defined similarly
										                               with the right ${\Bbb Z}/2$-graded Leibniz rule.
										                              This is a generalization of the notion of right derivation.
                                                                                        }   
  $$
    \begin{array}{ccccc}
	 \widehat{\nabla} & : & {\cal T}_{\widehat{X}} \times \widehat{\cal E}
	     & \longrightarrow   & \widehat{\cal E}  \\[1.2ex]
    && (\xi, s)              &  \longmapsto    &  \widehat{\nabla}_{\xi}s		 	
	\end{array}
  $$
  such that
	\begin{itemize}
	 \item[(1)]  [{\it $\widehat{\cal O}_X$-linearity
	                                   in the ${\cal T}_{\widehat{X}}$-argument}]\\[.6ex]	
	  $\mbox{\hspace{1em}}$
	  $\widehat{\nabla}_{f_1\xi_1 + f_2\xi_2}s\;
	     =\; f_1 \widehat{\nabla}_{\xi_1}s + f_2 \widehat{\nabla}_{\xi_2}s$, \hspace{1em}
      for $f_1, f_2 \in \widehat{\cal O}_X$, $\xi_1, \xi_2 \in {\cal T}_{\widehat{X}}$,  and
	       $s\in \widehat{\cal E}$;

     \item[(2)]  [{\it ${\Bbb C}$-linearity in the $\widehat{\cal E}$-argument}]\\[.6ex]
	 $\mbox{\hspace{1em}}$
	 $\widehat{\nabla}_\xi(c_1s_1+c_2s_2)\;
	     =\;  c_1 \widehat{\nabla}_\xi s_1 + c_2 \widehat{\nabla}_\xi s_2$, \hspace{1em}
	  for $c_1, c_2\in {\Bbb C}$, $\xi\in {\cal T}_{\widehat{X}}$, and
	       $s_1, s_2\in \widehat{\cal E}$;
	
	 \item[(3)] [{\it generalized ${\Bbb Z}/2$-graded Leibniz rule in the $\widehat{\cal E}$-argument}]\\[.6ex]
	 $\mbox{\hspace{1em}}$
	 $\widehat{\nabla}_\xi(fs)\;
	   =\; (\xi f)s
	           + (-1)^{p(f)p(\xi)}\,f\cdot\,\!^{\varsigma_{\!f}}\!(\widehat{\nabla})_\xi s$,\\[.6ex]
      for $f\in \widehat{\cal O}_X$, $\xi\in {\cal T}_{\widehat{X}}$ parity homogeneous
	       and $s\in\widehat{\cal E}$.
     Here,
	   $^{\varsigma_{\!f}}\!(\widehat{\nabla})$
       is the parity-conjugation of $\widehat{\nabla}$ induced by $f$;
       i.e., 	
	   $^{\varsigma_{\!f}}\!(\widehat{\nabla})
	     = \widehat{\nabla}$,  if $f$ is even, or
		    $\,\!^{\varsigma}\widehat{\nabla}
			        := \mbox{(even part of $\widehat{\nabla}$)}\,
					      -\, \mbox{(odd part of $\widehat{\nabla}$)}$  if $f$ is odd;
     (cf.\ Definition~1.3.1).
	\end{itemize}
  As an operation on the pairs $(\xi, s)$,
   a connection $\nabla$ on $\widehat{\cal E}$ is applied to $\xi$ from the right
   while applied to $s$ from the left;\footnote{In the ${\Bbb Z}/2$-graded world,
                                                                     it is instructive to denote $\widehat{\nabla}_{\!\xi}s$ as
																	  $\xi \widehat{\nabla} s$ or $_{\xi}\!\widehat{\nabla} s$
																	 (though we do not adopt it as a regularly used notation in this work).
																	In particular,
  																	  from $_{f\xi}\!\widehat{\nabla} s$
																	  to $f (\,\!_{\xi}\!\widehat{\nabla} s)$,
                                                                     $f$ and $\widehat{\nabla}$ do {\it not} pass each other.
                                                                      }  
  cf.\ Lemma~1.3.7 and Remark~1.3.8.
}\end{definition}

\medskip

\begin{explanation} {\bf [generalized ${\Bbb Z}/2$-graded Leibniz rule in Definition~2.1.2]}\; {\rm
 For better illumination, we'll denote $\widehat{\nabla}_{\!\xi}s$ in this Explanation as
   $\,_{\xi}\!\widehat{\nabla}s$ (cf.\ Footnote~22).
 Notice that, unlike the exterior differential operator $d$,
  we do not assume that a connection $\widehat{\nabla}$,
   as a ${\Bbb C}$-bilinear pairing of the two sheaves of
     ${\Bbb C}$-vector spaces ${\cal T}_{\widehat{X}}$ and
     $\widehat{\cal E}$, is even.
 Thus, when imposing a ${\Bbb Z}/2$-graded Leibniz rule for $\,_{\xi}\!\widehat{\nabla}(fs)$,
   for $\xi\in {\cal T}_{\widehat{X}}$, $f\in \widehat{\cal O}_X$, and $s\in \widehat{\cal E}$,    
   one has to take into account not only how $f$ passes $\xi$ but also how $f$ passes $\widehat{\nabla}$
   so that one can reach a term of the form $f\cdot (\,\!_{\xi}(\cdots))s$.
 Applying the general rule for passing ${\Bbb Z}/2$-graded objects not necessarily of the same kind,
   one then has
   $$
      _{\xi}\!\widehat{\nabla}(fs)\;
      \rightsquigarrow\;    \,\!_{\xi} (f\cdot \,\!^{\varsigma_{\!f}}\!(\widehat{\nabla})s)\;
	  \rightsquigarrow\;   (-1)^{p(f)p(\xi)}\,f
	                                        \cdot\,\!_{\xi}\,\!^{\varsigma_{\!f}}\!(\widehat{\nabla})s\;
	   =:\; (-1)^{p(f)p(\xi)}\,f\cdot\,\!^{\varsigma_{\!f}}\!(\widehat{\nabla})_\xi s\,.
   $$
 This explains the generalized ${\Bbb Z}/2$-graded Leibniz rule
   for a connection $\widehat{\nabla}$ on $\widehat{\cal E}$.

 When $\widehat{\nabla}$ is even,
   $\,\!^{\varsigma_{\!f}}(\widehat{\nabla})=\widehat{\nabla}$
   for all $f\in \widehat{\cal O}_X$.
 The generalized ${\Bbb Z}/2$-graded Leibniz rule then resumes to the ordinary
  ${\Bbb Z}/2$-graded Leibniz rule:
  $\;\widehat{\nabla}_\xi(fs)\;
	   =\; (\xi f)s
	           + (-1)^{p(f)p(\xi)}\,f\cdot \widehat{\nabla}_{\!\xi} s$\,,
   for $f\in \widehat{\cal O}_X$, $\xi\in {\cal T}_{\widehat{X}}$ parity homogeneous
	       and $s\in\widehat{\cal E}$.
}\end{explanation}

\medskip

\begin{lemma} {\bf [even vs.\ odd part of connection $\widehat{\nabla}$]}\;
  Given a connection $\widehat{\nabla}$ on $\widehat{\cal E}$,
   the even part $\widehat{\nabla}^{(\even)}$  of $\widehat{\nabla}$
    is another connection on $\widehat{\cal E}$
  while the odd part of $\,\widehat{\nabla}$ is an odd
    $\,\Endsheaf_{\widehat{\cal O}_X^{\,\Bbb C}}(\widehat{\cal E})$-valued
  $1$-form $A^{(\odd)}$ on $\widehat{X}$.
  In notation, $\,\widehat{\nabla}= \widehat{\nabla}^{(\even)}+ A^{(\odd)}$.
\end{lemma}

\medskip

\begin{proof}
 That $\widehat{\nabla}^{(\even)}$
   is $\widehat{\cal O}_X$-linear in the ${\cal T}_{\widehat{X}}$-argument
   and ${\Bbb C}$-linear in the $\widehat{\cal E}$ is immediate.
 For example, let $f\in \widehat{\cal O}_X$ odd and $p(f\xi)=p(s)$; then
  $f\xi$ and $s$ have the opposite parities and, hence,
  $$
    \widehat{\nabla}_{\!f\xi}^{(\even)}s\;
	 :=\;  (\widehat{\nabla}_{\!f\xi}s)^{(\odd)}\;
	  =\;  (f\cdot \widehat{\nabla}_{\,\xi}s)^{(\odd)}\;
	  =\; f\cdot (\widehat{\nabla}_{\!\xi}s)^{(\even)}\;
	  =:\; f\cdot \widehat{\nabla}^{(\even)}_{\xi}s\,.	
  $$
 That $\widehat{\nabla}^{(\even)}$ satisfies the ordinary
  ${\Bbb Z}/2$-graded Leibniz rule (cf.\ Explanation~2.1.3)
  follows from:													
  \begin{eqnarray*}
   \lefteqn{
     \widehat{\nabla}^{(\even)}_\xi (fs)
	   =\; (\widehat{\nabla}_{\!\xi}(fs))^{(p(\xi)+ p(fs))}   }\\
    && =\; \left(
	          (\xi f)s
	           + (-1)^{p(f)p(\xi)}\,f\cdot\,\!^{\varsigma_{\!f}}\!(\widehat{\nabla})_\xi s
			    \right)^{(p(\xi)+ p(f) + p(s))} \\[.8ex]
    && =\; (\xi f)s\, +\, (-1)^{p(f)p(\xi)}\,	
	                \left( f\cdot\,\!^{\varsigma_{\!f}}\!(\widehat{\nabla})_\xi s
			         \right)^{(p(\xi)+ p(f) + p(s))}                                      \\[.8ex]
	&& =\; (\xi f)s\,+\, (-1)^{p(f)p(\xi)}\, f\cdot \widehat{\nabla}^{(\even)}_{\xi}s\,.
  \end{eqnarray*}
 Here,
   $(\,\cdots\,)^{(k)}:= (\,\cdots\,)^{(\even)}$, if $k$ is even,
   or $(\,\cdots\,)^{(\odd)}$ if $k$ is odd.

 Finally,
  \begin{eqnarray*}
   \lefteqn{
    (A^{(\odd)}(\xi))(fs)\;
	  =\; \widehat{\nabla}_{\!\xi}(fs)- \widehat{\nabla}_{\xi}^{(\even)}(fs)  }\\[.6ex]
	 && =\;\; (-1)^{p(f)p(\xi)}\,f
	                \cdot \left(\,\!^{\varsigma_{\!f}}\!(\widehat{\nabla})_{\xi}s
                             - \widehat{\nabla}^{(\even)}_\xi s \right)\;\;
	  =\;\;   (-1)^{p(f)(p(\xi)+1) }\, f\cdot (A^{(\odd)}(\xi))(s)  \,,
  \end{eqnarray*}
 for $\xi\in{\cal T}_{\widehat{X}}$ and $f\in \widehat{\cal O}_X$ parity homogeneous
      and  $s\in\widehat{\cal E}$.
 Note that since $A^{(\odd)}$ is odd, $p(\xi)+1$ is nothing but $p(A^{(\odd)}(\xi))$.
 This says that
   $A^{(\odd)}(\xi)\in \Endsheaf_{\widehat{\cal O}_X}(\widehat{\cal E}) $,
   acting on $\widehat{\cal E}$ from the left.
  
 This completes the proof.
 
\end{proof}

\medskip

\begin{convention}
$[ \Endsheaf_{\widehat{\cal O}_X}\!(\widehat{\cal E})$-valued
     $1$-form$]$\; {\rm
 Normally we take an $\Endsheaf_{\widehat{\cal O}_X}\!(\widehat{\cal E})$-valued
   $1$-form $A$ as a section in
   $\Endsheaf_{\widehat{\cal O}_X}\!(\widehat{\cal E})
       \otimes_{\widehat{\cal O}_X}\!
	   {\cal T}^{\ast}_{\widehat{X}}$.\footnote{This is based on the convention that we write
	                                                                   a differential form on $\widehat{X}$ as a combination of\\
																	    $f_0\, df_1\wedge \,\cdots\,\wedge df_l$,
																		  where $f_0, f_1\,, \cdots\,, f_l \in C^{\infty}(\widehat{X})$.
	                                                                   One has the other choice:
																       ${\cal T}^{\ast}_{\widehat{X}}
																           \otimes_{\widehat{\cal O}_X}
																	       \Endsheaf_{\widehat{\cal O}_X}
																	        (\widehat{\cal E})$.																			
																	  Different choices of conventions may influence
																	     the $(-1)^{\mbox{\tiny $\bullet$}}$-factor
      																     in an explicit computation.
	                                                                   }  
 For example, in terms of the supersymmetrically invariant coframe
  $(e^{\mu}, e^{\alpha^{\prime}}, e^{\beta^{\prime\prime}})
      _{\mu, \alpha^{\prime},\beta^{\prime\prime}}$,
  $A= \sum_{\mu=0}^3 A_\mu e^\mu
          + \sum_{\alpha=1^{\prime}, 2^{\prime}} A_{\alpha^{\prime}} e^{\alpha^{\prime}}
		  + \sum_{\beta^{\prime\prime}=1^{\prime\prime}, 2^{\prime\prime}}
		            A_{\beta^{\prime\prime}} e^{\beta^{\prime\prime}}$,
 with $A_{\mu},  A_{\alpha^{\prime}}, A_{\beta^{\prime\prime}}
           \in \Endsheaf_{\widehat{\cal O}_X}\!(\widehat{\cal E}) $.
 However, one may also take an $\Endsheaf_{\widehat{\cal O}_X}\!(\widehat{\cal E})$-valued
   $1$-form $A$ as a section in
   ${\cal T}^{\ast}_{\widehat{X}}\otimes_{\widehat{\cal O}_X}\!
      \Endsheaf_{\widehat{\cal O}_X}\!(\widehat{\cal E})$
	 and write $A$ as $\sum_I e^I A^\prime_I$ with matching ${\Bbb Z}/2$-passing sign rule.
 Cf.~Footnote~23.	
}\end{convention}

\medskip

\begin{lemma-definition}
{\bf [even left connection associated to even trivialization of $\widehat{\cal E}$]}\;
 A trivialization of $\widehat{\cal E}$ on $\widehat{X}$ by a basis of even global sections
  $(s_1,\,\cdots\,, s_r)$
  defines an even left connection on $\widehat{\cal E}$
  by the assignment $(\xi, s)\mapsto  \sum_{i=1}^r(\xi f^i)s_i$,
  for $\xi\in {\cal T}_{\widehat{X}}$ and $s=\sum_{i=1}^r f^is_i\in \widehat{\cal E}$.

  {\rm We will call such a trivialization an {\it even trivialization} of $\widehat{\cal E}$,
      denote such a left connection by $d$,  and
      call it the {\it trivial left connection associated to an even trivialization of $\widehat{\cal E}$}.}
\end{lemma-definition}

\medskip

\begin{proof}
 Denote the ${\Bbb C}$-bilinear pairing $(\xi, s)\mapsto  \sum_{i=1}^r(\xi f^i)s_i$ in the Statement
  by $P$.
 Then since $s_i$'s are all even,
   $P$ sends $(\xi:\mbox{even},  s:\mbox{even})$ and $(\xi:\mbox{odd}, s:\mbox{odd})$
      (resp.\ $(\xi:\mbox{even}, s:\mbox{odd})$ and $(\xi:\mbox{odd}, s:\mbox{even})$)
   to even (resp.\ odd) sections of $\widehat{\cal E}$.
 This implies that $P$ is an even pairing,
  cf.~Definition~2.1.1.
 It is straightforward now to check that it satisfies all the defining properties of a left connection.
 
\end{proof}
 
\medskip

\begin{explanation} {\bf [even trivialization condition in Lemma/Definition~2.1.6]}\; {\rm
{\it If $s_i$'s in the basis $(s_1,\,\cdots\,, s_r)$ of $\widehat{\cal E}$ are not all even,
  then the pairing
     $P:(\xi, s)\mapsto  \sum_{i=1}^r(\xi f^i)s_i$,
      for $\xi\in {\cal T}_{\widehat{X}}$ and $s=\sum_{i=1}^r f^is_i\in \widehat{\cal E}$
	 in general does not satisfy the generalized ${\Bbb Z}/2$-graded Leibniz rule   and,
	hence, does not define a left connection on $\widehat{\cal E}$.}
 The following simplified counterexample serves to illustrate this.

 Let $Z={\Bbb R}^1$ with coordinate-function $x$ and
   $\widehat{Z}$ be the super line with
     $C^{\infty}(\widehat{Z})=C^{\infty}({\Bbb R})[\theta]$, $\theta^2=0$.
 Take $\widehat{\cal E}=\widehat{\cal O}_Z$ and $s_1= 1+\theta$ as a basis.
 A section $s=a(x)+b(x)\theta \in \widehat{\cal E}$ can be expressed as
   $(a(x)+(b(x)-a(x))\theta)s_1=: f^1s_1$.
 Thus $P(\xi, s)= (\xi f^1) s_1$.
 In particular,
   $P(\partial_{\theta}, b(x)\theta)
    = P(\partial_{\theta}, (b(x)\theta) s_1)
	=  b(x)\cdot(1+\theta)$.
 Should the generalized ${\Bbb Z}/2$-graded Leibniz rule for $P$ hold,
  this would equal
  $\partial_{\theta}(b(x)\theta)s_1
      + (-1)^{p(\partial_\theta)p(b(x)\theta) }  (b(x)\theta)
	       \cdot(P^{(\even)}(\partial_{\theta},  s_1)
		                 - P^{(\odd)}(\partial_{\theta}, s_1) )$,
 Which would imply that
  $I:= -\,(b(x)\theta)
	          \cdot(P^{(\even)}(\partial_{\theta},  s_1)
		                 - P^{(\odd)}(\partial_{\theta}, s_1) )$ vanishes.

 Now,
  $\;P^{(\even)}(\partial_{\theta}, s_1)
     = P^{(\even)}(\partial_{\theta}, 1+\theta)
	 = (P(\partial_{\theta}, 1))^{(\odd)}+ (P(\partial_{\theta}, \theta))^{(\even)}
     = 1-\theta\;$
 while\\
  $\;P^{(\odd)}(\partial_{\theta}, s_1)
     = P^{(\odd)}(\partial_{\theta}, 1+\theta)
	 = (P(\partial_{\theta}, 1))^{(\even)}+ (P(\partial_{\theta}, \theta))^{(\odd)}
     = -1+\theta\,$.
 It follows that\\
  $I= -\,(b(x)\theta)\cdot (2-2\theta)= -\,2b(x)\theta \ne 0$.
 This shows that $P$  is not a left connection on $\widehat{\cal E}$.\footnote{{\it Connection
                                                                          under gauge transformations}\hspace{1em}
                                                                        Behind
                                                                         this simple example is the more general statement/observation that
																	     {\it conjugation by a gauge symmetry $g$ of a vector bundle
																	        $\widehat{E}$ over the super space $\widehat{X}$
																			does not take a covariant derivative to a covariant derivative
																		    unless $g$ is even (i.e.\ ${\Bbb Z}/2$-grading preserving)}.
																		  On the other hand, in physics literature gauge transformations of
																		    the form {\it exp}\,$(\Phi)$,
																			where $\Phi$ is a superfield on $\widehat{X}$ frequently appear.
																		  Such transformations are in general not even.
																		  How to deal with such transformations coherently with connections
																		    for the case we need for D-branes
																		    motivates the notion of left-right hybrid connection in~Sec.~2.2.																		   
																			}
}\end{explanation}
 
\bigskip

The same argument as in the proof of  Lemma~2.1.4 says that:
 
\bigskip

\begin{lemma-definition} {\bf [connection $1$-form]}\;  {\rm
 {\it Let
   $d$ be a trivial left connection on $\widehat{\cal E}$
      (associated to the trivialization of $\widehat{\cal E}$ by some even basis)   and
   $\widehat{\nabla} = \widehat{\nabla}^{(\even)}+ A^{(odd)}$
     be a connection on $\widehat{\cal E}$.
  Then
    $$
      \widehat{\nabla}^{(\even)}\;  =\;  d\,+\, A^{(\even)}
    $$
	for an even
	  $\Endsheaf_{\widehat{\cal O}_X^{\,\Bbb C}}(\widehat{\cal E})$-valued
	  $1$-form.}
 Together,
  $\widehat{\nabla}= d+ A^{(\even)}+A^{(\odd)}=: d+ A$.
 $A$ is called the {\it connection $1$-form}   associated to $\widehat{\nabla}$
  (with respect to the underlying (even) trivialization of $\widehat{\cal E}$).
 By construction, $A^{(\even)}$ is the even part of $A$, which depends on the trivialization,
   while $A^{(\odd)}$ is the odd part of $A$, which is independent of the trivialization.
}\end{lemma-definition}

\bigskip
Having had the notion of a left connection $\widehat{\nabla}$, 
 one naturally wants to define the curvature tensor $F^{\widehat{\nabla}}$ of $\widehat{\nabla}$.
Unfortunately 
 the standard formula for $F^{\widehat{\nabla}}$ when $\widehat{\nabla}$ is purely even 
 does not work due to the odd part $A^{(\odd)}$ in $\widehat{\nabla}$.
Some modification to the standard formula is required.
See Lemma/Definition 2.1.9 in the next theme.

\bigskip

\begin{flushleft}
{\bf Subtleties behind the notion of `left connection'}
\end{flushleft}
$(a)$\;\;{\it Left connections under general gauge transformations}

\medskip

\noindent
Explanation~2.1.7 gives an example that
  %
  \begin{itemize}
   \item[\LARGE $\cdot$]
    {\it In general, a non-even gauge transformation $g$ on $\widehat{E}$
	        does not take a (left) covariant derivative $\widehat{\nabla}_{\!\xi}$ on $\widehat{E}$
			to another (left) covariant derivative on $\widehat{E}$ by conjugation.}
  \end{itemize}
%
%
Indeed, in general the operation
 $$
   (\xi,\widehat{s})\;\longmapsto\;
      g  \mbox{\large $($}\widehat{\nabla}_{\!\xi} (g^{-1}\widehat{s})\mbox{\large $)$}
 $$
 does not define a connection on $\widehat{E}$.
The correct rule is given by
  $$
     (\xi,\widehat{s})\;\longmapsto\;
   (\xi)\,\!^\leftarrow\!\!\!(g\circ\nabla\circ g^{-1})(\widehat{s})\;
       :=\; \,\!^{\varsigma_\xi}\!g\mbox{\large $($}
	                                \widehat{\nabla}_{\!\xi} (g^{-1}\widehat{s})
	                                                                \mbox{\large $)$}\,,
  $$
  for $\xi$ parity homogeneous.\footnote{Caution
                                                         that, by the convention of current notes,
													     $\widehat{\nabla}_{\!\xi}
											                := (\xi)\!\,\!^{\raisebox{.3ex}{$\leftarrow$}}\!\!\!\widehat{\nabla}$.
														 Thus, there is {\it no} passing between $\widehat{\nabla}$ and $\xi$.
														
														 Similarly, $A^{\tinyeven}(\xi):= (\xi)\!^{\leftarrow}\!\!\!A^{\tinyeven}$
														   and $A^{\tinyodd}(\xi):= (\xi)\!^{\leftarrow}\!\!\!A^{\tinyodd}$.
													    } 

\vspace{12em}														

\noindent
$(b)$\;\;
{\it The curvature tensor $F^{\widehat{\nabla}}$ associated to a connection $\widehat{\nabla}$}

\medskip

\noindent
Naively, one would define the
 {\it $\End_{\widehat{\Bbb C}}(\widehat{E})$-valued curvature $2$-tensor} on $\widehat{X}$
 associated to a connection $\widehat{\nabla}$ on $\widehat{E}$ by the assignment
 $$
   (\xi_1,\xi_2; \widehat{s})\;\longmapsto\;
     [\widehat{\nabla}_{\!\xi_1}, \widehat{\nabla}_{\!\xi_2}\}s\,
	 -\, \widehat{\nabla}_{[\xi_1,\xi_2\}}s
 $$
 for $\xi_1,\xi_2\in \Der_{\Bbb C}(\widehat{X})$ and $\widehat{s}\in C^{\infty}(\widehat{E})$.
Yet, as given above, this is not a $2$-tensor on $\widehat{X}$.

The correct definition is given by:

\begin{lemma-definition} {\bf [curvature $2$-tensor associated to left connection]}\;
 Let $\widehat{\nabla}$ be a (general) left connection on $\widehat{E}$.
 Then the correspondence
  $$
   F^{\widehat{\nabla}}\;:\;
   (\xi_1,\xi_2; \widehat{s})\;\longmapsto\;
        \mbox{\Large $($}
		  [\widehat{\nabla}_{\!\xi_1}, \widehat{\nabla}_{\!\xi_2}\}\,
	         -\,\widehat{\nabla}_{[\xi_1,\xi_2\}}\,
             -\,\mbox{\large $($}1-(-1)^{p(\xi_2)}\mbox{\large $)$}
			     \cdot [A^{(\odd)}(\xi_1), \xi_2\}
		\mbox{\Large $)$}\, s\,,
  $$
  for $\xi_1,\xi_2\in \Der_{\Bbb C}(\widehat{X})$  and
       $\widehat{s}\in C^{\infty}(\widehat{E})$,
  defines an $\End_{\widehat{\Bbb C}}(\widehat{E})$-valued $2$-tensor on $\widehat{X}$.
 {\rm We shall call $F^{\widehat{\nabla}}$ thus defined
     the {\it curvature tensor} on $\widehat{X}$
	   associated to the left connection $\widehat{\nabla}$ on $\widehat{E}$.}
\end{lemma-definition}

\bigskip

\noindent
Since we mean to use this only as a contrast to and motivation for the new setting in Sec.\ 2.2,
 we leave the proof to interested readers as an exercise.
 
%
%
%
%
%
%
%

\bigskip

\begin{flushleft}
{\bf Left connections on $\widehat{\cal E}$ that are adapted to vector multiplets of supersymmetry representations}
\end{flushleft}
(See Sec.\ 2.2 for more thorough explanations in the case of simple hybrid connections.
    The current theme is for comparison with Sec.\ 2.2 only.)
One can press on, following, e.g., [G-G-R-S: Sec.\ 4.2],
 to define the notion of
 {\it SUSY-rep compatible left connection associated to a vector superfield $V$}
 by first fixing an even trivialization of of $\widehat{E}$ and then setting\footnote{Caution that
                                                                    in general $V$ is not purely even and, hence,
										                            $e^V$ is not an even gauge transformation of $\widehat{E}$.
                                                                   Because of this,
																	 $\widehat{\nabla}_{e_{\alpha^\prime}}$ is defined as
		                                                             $e^{-V}\circ e_{\alpha^\prime}\circ\,\!^\varsigma\!(e^V)$,
		                                                              rather than $e^{-V}\circ e_{\alpha^\prime}\circ\,e^V$,
		                                                             so that it is a left covariant derivation on $\widehat{E}$.
																   One may define $\widehat{\nabla}_{e_{\alpha^\prime}}$
																	  instead as
																	    $\,\!^\varsigma\!(e^{-V})\circ e_{\alpha^\prime}\circ e^V$.
                                                                    Cf.\  Item (a) of the previous theme.																	 
		                                                             }  
 $$
   \widehat{\nabla}_{e_{\beta^{\prime\prime}}}\;=\; e_{\beta^{\prime\prime}}\,,\;\;
   \widehat{\nabla}_{e_{\alpha^\prime}}\;
     =\; e^{-V}\circ e_{\alpha^\prime}\circ\,\!^\varsigma\!(e^V)\,,\;\;
   \widehat{\nabla}_{e_\mu}\;
	    =\; 	\mbox{\small $\frac{\sqrt{-1}}{2}$}\,
		         \sum_{\alpha, \dot{\beta}} \breve{\sigma}_{\mu}^{\alpha\dot{\beta}}
				  \cdot \{\widehat{\nabla}_{e_{\alpha^{\prime}}},
				                       \widehat{\nabla}_{e_{\beta^{\prime\prime}}}\}\,.	
 $$
 Here,
	  $\breve{\sigma}_{\mu}
	     =(\breve{\sigma}_{\mu}^{\alpha\dot{\beta}})_{\alpha\dot{\beta}}$
	  with
	 {\footnotesize
     $$
      \breve{\sigma}_0\;:=\;
       \frac{1}{2}\left[\!\begin{array}{rr} -1 & 0 \\ 0 & -1\end{array}\!\right]\!,\;\;
	  \breve{\sigma}_1\;:=\;
       \frac{1}{2}\left[\!\begin{array}{rr} 0 & 1 \\ 1 & 0\end{array}\!\right]\!,\;\;
      \breve{\sigma}_2\;:=\;
       \frac{1}{2}
	    \left[\!\begin{array}{rr} 0 & \sqrt{-1} \\ -\sqrt{-1} & 0\end{array}\!\right]\!,\;\;	   
      \breve{\sigma}_3\;:=\;
       \frac{1}{2}\left[\!\begin{array}{rr} 1 & 0 \\ 0 & -1\end{array}\!\right]\,.
     $$} 
One can check that $\widehat{\nabla}$	satisfies some curvature vanishing properties.

The complexity of expressions propagated from the seemingly harmless subtleties
   of not-purely-even left connections mentioned in the previous theme, when one presses on further,
 suggests that considering only purely left connections may not give the best context to
  study connections on bundles on a superspace.
This leads us to a new notion: {\it left-right hybrid connections}, which we now turn to.
   	
%
%
%
%
%
%
%

\bigskip
   
\subsection{Hybrid connections on the Chan-Paton bundle $\widehat{E}$ over $\widehat{X}$}
Guided by the lesson learned from Sec.~2.1,
we introduce and study in this subsection the notion of  `{\it hybrid connections}':
Ordinary derivations still act from the left,
 but endomorphisms from the evaluation of endomorphisms-valued $1$-forms on derivations
 --- which need to be added to make a good notion of covariant derivations ---
 act from the right.
Our basic setup needs to be adjusted accordingly, which we take as our starting point.

\bigskip

\begin{flushleft}
{\bf The new setup}
\end{flushleft}
Up to now, we have taken the endomorphism sheaf
   $\Endsheaf_{{\cal O}_X^{\,\Bbb C}}({\cal E})$
 (resp.\ $\Endsheaf_{\widehat{\cal O}_X} (\widehat{\cal E})$) of ${\cal E}$
 (resp.\ $\widehat{\cal E}$) as acting on ${\cal E}$ (resp.\ $\widehat{\cal E}$) from the left.
In the commutative world, this is the most natural convention.
In the ${\Bbb Z}/2$-graded  world,
  as we have chosen the convention that derivations act on their subjects from the left,
 to vary the notion of connections from Sec.~2.1
we are forced to reset the convention to that
  the endomorphism sheaf $\Endsheaf_{{\cal O}_X^{\,\Bbb C}}({\cal E})$
 (resp.\ $\Endsheaf_{\widehat{\cal O}_X}(\widehat{\cal E})$) of ${\cal E}$
 (resp.\ $\widehat{\cal E}$) as acting on ${\cal E}$ (resp.\ $\widehat{\cal E}$) {\it from the right}
 so that one can encompass more situations in physics literature and obtain cleaner formula in various situations.
Similarly, for endomorphism bundles $\End_{\Bbb C}(E)$  and
 $\End_{\widehat{\Bbb C}}(\widehat{E})$.
 \begin{itemize}
  \item[\LARGE $\cdot$]
   Denote by $\widehat{\Bbb C}$
     the ${\Bbb C}$-algebra
	 ${\Bbb C}[\theta^1,\theta^2,\theta^{\dot{1}},\theta^{\dot{2}}]^{\anticommuting}$
     of complex Grassmann numbers with the standard ${\Bbb Z}/2$-grading.
   	
  \item[\LARGE $\cdot$]
   Let $\widehat{\frak m}
             :=(\theta^1, \theta^2, \bar{\theta}^{\dot{1}}, \bar{\theta}^{\dot{2}})$
    be  the ideal sheaf of the $4$-dimensional Minkowski space-time $X$
    in the $d=4$, $N=1$ superspace $\widehat{X}$ as a super $C^{\infty}$-subscheme.
   
  \item[\LARGE $\cdot$]
   Let $E$ be a complex vector bundle of rank $r$ on $X$.
   The corresponding sheaf of smooth sections is denoted by ${\cal E}$.
   Denote by $\End_{\Bbb C}(E)$ (resp.\ $\Aut_{\Bbb C}(E)$)
     the bundle of endomorphisms (resp.\ the bundle of automorphisms) of $E$.
   The corresponding sheaves are denoted by
     $\Endsheaf_{{\cal O}_X^{\,\Bbb C}}({\cal E})$   and
	 $\Autsheaf_{{\cal O}_X^{\,\Bbb C}}({\cal E})$
	 respectively.
   Here, we reset the convention to that
	\begin{itemize}
	 \item[\LARGE $\cdot$]
     {\it $\End_{\Bbb C}(E)$ and $\Aut_{\Bbb C}(E)$  act on $E$ from the right}  	
	\end{itemize}
	and, similarly,
	\begin{itemize}
	 \item[\LARGE $\cdot$]
	 {\it   $\Endsheaf_{{\cal O}_X^{\,\Bbb C}}({\cal E})$ and
	           $\Autsheaf_{{\cal O}_X^{\,\Bbb C}}({\cal E})$ act on ${\cal E}$ from the right}.
	\end{itemize}
                       
  \item[\LARGE $\cdot$]	
   Let $\widehat{E}$ be the ${\Bbb Z}/2$-graded complex vector bundle of rank $r$ on $\widehat{X}$
     that extends $E$;
   each fiber of $\widehat{E}$ over $X$ is a free bi-$\widehat{\Bbb C}$-module of rank $r$.
   By construction 	
   the corresponding sheaf of smooth sections is a bi-$\widehat{\cal O}_X$-modules,
     denoted by $\widehat{\cal E}$.
   $\widehat{\cal E}= {\cal E}\otimes_{{\cal O}_X^{\,\Bbb C}}\widehat{\cal O}_X$
      is naturally ${\Bbb Z}/2$-graded;
     and the left and the right locally free $\widehat{\cal O}_X$-module structure of $\widehat{\cal E}$
	 are related by $sa = (-1)^{p(s)p(a)}as$
	 for $a\in \widehat{\cal O}_X,\, s\in \widehat{\cal E} $ parity homogeneous.
  
  \item[\LARGE $\cdot$]
   Let $\End_{\widehat{\Bbb C}}(\widehat{E})$ be the bundle of endomorphisms
     of $\widehat{E}$ as a {\it left} $\widehat{\Bbb C}$-module over $X$.
   The corresponding sheaf of endomorphisms of $\widehat{\cal E}$ is denoted by
     $\Endsheaf_{\widehat{\cal O}_X}(\widehat{\cal E})$.
   Under the new setup
	\begin{itemize}
	 \item[\LARGE $\cdot$]
	  {\it  $\End_{\widehat{\Bbb C}}(\widehat{E})$\,
	            (resp.\ $\Endsheaf_{\widehat{\cal O}_X}(\widehat{\cal E})$)\,
			   acts on $\widehat{E}$ (resp.\ $\widehat{\cal E}$) from the right.}	
	\end{itemize}
   $\End_{\widehat{\Bbb C}}(\widehat{E})$ is naturally ${\Bbb Z}/2$-graded:
    an {\it even endomorphism} sends
	  even elements to even elements and  odd elements to odd elements in $\widehat{E}$
    while an {\it odd endomorphism} sends
 	  even elements to odd elements and  odd elements to even elements in $\widehat{E}$.
   Similarly, for $\Endsheaf_{\widehat{\cal O}_X}(\widehat{\cal E})$.
   
  \item[\LARGE $\cdot$]
   Let
      $\Aut_{\widehat{\Bbb C}}(\widehat{E})
	                                        \subset \End_{\widehat{\Bbb C}}(\widehat{E})$
     be the bundle of automorphisms of $\widehat{E}$ as a {\it left} $\widehat{\Bbb C}$-module over $X$.
   The corresponding sheaf of automorphisms of $\widehat{\cal E}$ is denoted by
     $\Autsheaf_{\widehat{\cal O}_X}(\widehat{\cal E})$.   		
   By convention,
	\begin{itemize}
	 \item[\LARGE $\cdot$]
	  {\it  $\Aut_{\widehat{\Bbb C}}(\widehat{E})$\,
	            (resp.\ $\Autsheaf_{\widehat{\cal O}_X}(\widehat{\cal E})$)\,
			   acts on $\widehat{E}$ (resp.\ $\widehat{\cal E}$) from the right.}	
	\end{itemize}
   Recall [L-Y9: Lemma 2.2.1.1 \& Corollary 2.2.1.2] (D(11.4.1)) that
      $$
         \Autsheaf_{\widehat{\cal O}_X}(\widehat{\cal E})\;
          \simeq\; \Autsheaf_{{\cal O}_X^{\,\Bbb C}}({\cal E})\,
		                    \oplus\,  \Endsheaf_{{\cal O}_X^{\,\Bbb C}}({\cal E})		
							                                 \otimes_{{\cal O}_X^{\,\Bbb C}} \widehat{\frak m}\,.
      $$
  The set $C^{\infty}(\Aut_{\widehat{\Bbb C}}(\widehat{E}))$	
     of smooth sections of $\Aut_{\widehat{\Bbb C}}(\widehat{E})$
    forms the group of {\it gauge transformations} of $\widehat{E}$.
 \end{itemize}

\medskip

\begin{convention} $[$left operators, right operators, and their compositions$]$\; {\rm
  Given sets $A$ and $C$ of operators (e.g.\ derivations, endomorphisms, ...)
   that act on a set-with-structure $B$ (e.g.\ rings, modules, ...)
   with $A$ acting from the left (of $B$) and $C$ acting from the right (of $B$),\footnote{Caution
                                                                        that a left operation and a right operation may not commute.
																		}      
  we set the following notational conventions:
  \begin{itemize}
   \item[\LARGE $\cdot$]
    For $a\in A$, $c\in C$, and $z\in B$,
	 we write
	  $az$ also as $a(z)$,  and
	  $zc$ also as
       $(z)\,\!^{\leftarrow}\!\!\!c$	 or $c^{\circ}(z)$ whichever is more convient.
    	
   \item[\LARGE $\cdot$]
    Let $z\in B$ and, for example,  $a_1, a_2, a_3\in A$ and $c_1, c_2\in C$.
	Then, denote the composition
      $$
	     a_3 \mbox{\Large $($}
	                \mbox{\large $($}
				      a_2 ( \mbox{\small $($}a_1z \mbox{\small $)$}c_1)
				    \mbox{\large $)$}  c_2
		          \mbox{\Large $)$}\;
	     =:\; (a_3 \circ  c_2^{\circ}  \circ  a_2 \circ c_1^{\circ} \circ a_1 )(z)\,.
	  $$
	
	\item[\Large $\cdot$]
     When an operation $P$ on $B$
	      is a sum of an operator-from-left $a\in A$ and an operator-from-right $c\in C$,
       we write $P$ as $a+c^{\circ}$.
	 By definition, $P(z)= (a+c^{\circ})(z)= az + zc$ for $z\in B$.
	 For convenience we say that $P$ is applied to $B$ {\it formally from the left}.
	
	\item[\Large $\cdot$]
	 Similar notations and conventions apply to operators on $B$ with values in another set-with-structure.
  \end{itemize}
}\end{convention}

\bigskip

\begin{flushleft}
{\bf  The notion of simple hybrid connections on $\widehat{\cal E}$}
\end{flushleft}
\begin{definition} {\bf [pre-connection on $\widehat{\cal E}$]}\; {\rm
 A {\it pre-connection} $\widehat{\nabla}$ on $\widehat{\cal E}$
   is a ${\Bbb C}$-bilinear pairing
  $$
    \begin{array}{ccccc}
	 \widehat{\nabla} & : & {\cal T}_{\widehat{X}} \times \widehat{\cal E}
	     & \longrightarrow   & \widehat{\cal E}  \\[1.2ex]
    && (\xi, s)              &  \longmapsto    &  \widehat{\nabla}_{\!\xi}s		 	
	\end{array}
  $$
  such that
	\begin{itemize}
	 \item[(1)]  [{\it $\widehat{\cal O}_X$-linearity
	                                   in the ${\cal T}_{\widehat{X}}$-argument}]\\[.6ex]	
	  $\mbox{\hspace{1em}}$
	  $\widehat{\nabla}_{\!f_1\xi_1 + f_2\xi_2}s\;
	     =\; f_1 \widehat{\nabla}_{\!\xi_1}s + f_2 \widehat{\nabla}_{\!\xi_2}s$, \hspace{1em}
      for $f_1, f_2 \in \widehat{\cal O}_X$, $\xi_1, \xi_2 \in {\cal T}_{\widehat{X}}$,  and
	       $s\in \widehat{\cal E}$;

     \item[(2)]  [{\it ${\Bbb C}$-linearity in the $\widehat{\cal E}$-argument}]\\[.6ex]
	 $\mbox{\hspace{1em}}$
	 $\widehat{\nabla}_{\!\xi}(c_1s_1+c_2s_2)\;
	     =\;  c_1 \widehat{\nabla}_{\!\xi} s_1 + c_2 \widehat{\nabla}_{\!\xi} s_2$, \hspace{1em}
	  for $c_1, c_2\in {\Bbb C}$, $\xi\in {\cal T}_{\widehat{X}}$, and
	       $s_1, s_2\in \widehat{\cal E}$;
	
	 \item[(3)] [{\it ${\Bbb Z}/2$-graded Leibniz rule in the $s$-argument}]\\[.6ex]
	 $\mbox{\hspace{1em}}$
	 $\widehat{\nabla}_{\!\xi}(fs)\;
	   =\; (\xi f)s
	           + (-1)^{p(f)p(\xi)}\,f\cdot \widehat{\nabla}_{\!\xi} s$,\\[.6ex]
      for $f\in \widehat{\cal O}_X$, $\xi\in {\cal T}_{\widehat{X}}$ parity homogeneous
	       and $s\in\widehat{\cal E}$.
	\end{itemize}
  As an operation on the pairs $(\xi, s)$,
   a pre-connection $\widehat{\nabla}$ on $\widehat{\cal E}$ is applied to $\xi$ from the right
   while applied to $s$ possibly only formally from the left;\footnote{Similarly
                                                                     to the situation for left connections,
                                                                     it is instructive to denote $\widehat{\nabla}_{\!\xi}s$ as
																	  $\xi \widehat{\nabla} s$ or $_{\xi}\!\widehat{\nabla} s$
																	 (though we do not adopt it here).
																	In particular,
  																	  from $_{f\xi}\!\widehat{\nabla} s$
																	  to $f (\,\!_{\xi}\!\widehat{\nabla} s)$,
                                                                     $f$ and $\widehat{\nabla}$ do {\it not} pass each other.
                                                                      }  
  cf.\ Lemma~1.3.7 and Remark~1.3.8.
 Note that, similar to $d$ on $\widehat{\cal O}_X$,
   the ${\Bbb Z}/2$-graded Leibniz rule for $\widehat{\nabla}$ on $\widehat{\cal E}$
   can written equivalently as
   $$
    \mbox{(3$^{\prime}$)}\hspace{13.2em}
    \widehat{\nabla}(fs)\;=\; (df)\cdot s\,+\, f\cdot  \widehat{\nabla}s    \hspace{13.6em}
   $$
   since $\widehat{\nabla}_{\!\xi}:=  (\xi)\,\!^\leftarrow\!\!\!\widehat{\nabla}$
     for $\xi\in {\cal T}_{\widehat{X}}$.
}\end{definition}

\medskip

\begin{explanation} {\bf [${\Bbb Z}/2$-graded Leibniz rule in Definition~2.2.2]}\; {\rm
 The $\,\!^{\varsigma_{\!f}}\!(\widehat{\nabla})$
    in the generalized ${\Bbb Z}/2$-graded Leibniz rule
    in Definition 2.1.2 of a left connection $\widehat{\nabla}$ on $\widehat{\cal E}$
	indicates that all ingredients of $\widehat{\nabla}$ apply to $s$ from the left.
 Here for a pre-connection $\widehat{\nabla}$, only the usual $\widehat{\nabla}$ appears	
    in the ${\Bbb Z}/2$-graded Leibniz rule.
 This is an indication that ingredients of $\widehat{\nabla}$
   are either even and applied to $s$ from the left,
   or even-odd-mixed and applied to $s$ from the right.
}\end{explanation}
																													 
\medskip

\begin{lemma-definition} {\bf [pre-connection associated to trivialization of $\widehat{\cal E}$]}\;
 A trivialization of $\widehat{\cal E}$ on $\widehat{X}$ by a basis of global sections
  $(s^1,\,\cdots\,, s^r)$ defines a pre-connection on $\widehat{\cal E}$
  by the assignment $(\xi, s)\mapsto  \sum_{i=1}^r(\xi f_i)s^i$,
  for $\xi\in {\cal T}_{\widehat{X}}$ and $s=\sum_{i=1}^r f_is^i\in \widehat{\cal E}$.
 {\rm
  We will call it the {\it trivial pre-connection} associated to the given trivialization of $\widehat{\cal E}$.}
\end{lemma-definition}

\bigskip

Note that the even component
  $(s^1_{(\even)},\,\cdots\,, s^r_{(\even)}   )$ of a basis
  $(s^1,\,\cdots\,, s^r)$ of $\widehat{\cal E}$ (as a left $\widehat{\cal O}_X$-module)
  gives another basis of $\widehat{\cal O}_X$, thus none of $s^i_{(\even)}$ can be zero.
Compared with Lemma/Definition~2.1.6,
 here, however,  $s^i$'s can have non-zero odd component $s^i_{(\odd)}$.

\bigskip

\noindent
{\it Proof of Lemma/Definition~2.2.4.}\;
 Denote the ${\Bbb C}$-linear pairing $(\xi,s)\mapsto \sum_{i=1}^r (\xi f_i)s^i$
   in the Statement by $P$.
 Then, its clear that $P$ satisfies the
   $\widehat{\cal O}_X$-Linearity-in-the ${\cal T}_{\widehat{X}}$-Component Condition
   and ${\Bbb C}$-Linearity-in-$\widehat{\cal E}$-Component  Condition.
 For the ${\Bbb Z}/2$-graded Leibniz Rule Condition,
  $P(\xi,  f_0s)= \sum_{i=1}^r\xi(f_0f_i)s^i
      =   (\xi_if_0)s + (-1)^{p(\xi)p(f_0)}f_0\cdot P(\xi, s)$
   from the ${\Bbb Z}/2$-graded Leibniz rule for $\xi$.
 This completes the proof.

\noindent\hspace{40.7em}$\square$

\bigskip

\begin{remark} $[$relook at {\bf Explanation~2.1.7}$]$\; {\rm
 (Continuing Explanation~2.1.7.)
 Let $s^1=1+\theta$ and repeat the computation in Explanation~2.1.7.
 This time $P(\partial_\theta, s_1)$ on one hand equals zero by definition,
  and one the other hand equals
   $P(\partial_\theta, 1)+ P(\partial_\theta, \theta)
      = -s^1 + s^1=0$ since $1=(1-\theta)s^1$ and $\theta=\theta s^1$.
 Thus, there is no contradiction now.	
}\end{remark}

\medskip

\begin{lemma-definition} {\bf [simple hybrid connection on $\widehat{\cal E}$]}\;
 Let
  $(s^1,\,\cdots\,, s^r)$ be a basis of $\widehat{\cal E}$ (as an $\widehat{\cal O}_X$-module) and
  $A\in  \Endsheaf_{\widehat{\cal O}_X}(\widehat{\cal E})
              \otimes_{\widehat{\cal O}_X} {\cal T}_{\widehat{X}}^{\ast}$
     be an $ \Endsheaf_{\widehat{\cal O}_X}(\widehat{\cal E})$-valued $1$-from
	 on $\widehat{X}$.
 Recall that
   $ \Endsheaf_{\widehat{\cal O}_X}(\widehat{\cal E})$
   now acts on $\widehat{\cal E}$ from the right.	
 Then the ${\Bbb C}$-bilinear pairing\footnote{There
                                                                       is a subtle point here that does not occur in the case of left connections:
		                       The ${\cal E}$\!{\it nd}$_{\widehat{\cal O}_X}(\widehat{\cal E})$-valued
							                                           $1$-form $A$ now applies to $s\in \widehat{\cal E}$  and
																	   evaluates on $\xi\in {\cal T}_{\widehat{X}}$ both from their right.
																	   {\it Which of the two performs first?}
																	   Here, we set the conventions
																	     that the pair is $(\xi,s)$, not $(s,\xi)$, and
																	     $(\xi, s)\,\!^{\leftarrow}\!\!\!A
																		      := (\xi)\,\!^{\leftarrow}\!\!\!(sA)$
                                                                         so that one can express
																		 $\widehat{\nabla}_{\!\xi}s$ neatly as
																		 $(\xi)\,\!^{\leftarrow}\!\!\!(ds + sA)$.
																	Details of the expression will be different
																	   if one chooses the other convention $(s,\xi)$,
																	   due to that different passings of ${\Bbb Z}/2$-graded objects
																	   are involved.
                                                                       }   
   $$
    \begin{array}{ccccl}
	 \widehat{\nabla} & : & {\cal T}_{\widehat{X}} \times \widehat{\cal E}
	     & \longrightarrow   & \hspace{2.4em}\widehat{\cal E}  \\[1.2ex]
    && (\xi, s)              &  \longmapsto
	   &  \widehat{\nabla}_{\!\xi}s	\;
	         :=\; \sum_{i=1}^r
			          \left(
					     (\xi f_i) s^i\,
						    +\, \,\!^{\varsigma_\xi}\!f_i\cdot(\,\!^{\varsigma_{\xi}}\!s^i) A(\xi)
					  \right)  \\[1.2ex]
    &&&& \hspace{2.4em}
	    =:\; \sum_{i=1}^r
	               \left( df_i(\xi)\cdot s^i
                               +\, \,\!^{\varsigma_\xi}\!f_i
					  		          \cdot  A(\xi)^{\circ}(\,\!^{\varsigma_{\xi}}\!s^i)
                   \right)\\[1.2ex]
    &&&& \hspace{2.4em}
        =:\; 	   (\xi)\,\!^{\leftarrow}\!\!\!(ds +   s A)\\[1.2ex]
	&&&& \hspace{2.4em}
	    =:\;   (ds)(\xi)\,+\,     (A(\xi)^\circ)(\,\!^{\varsigma_\xi}s)\,,
	\end{array}
   $$
  for $\xi\in{\cal T}_{\widehat{X}} $ parity homogeneous and $s=\sum_{i=1}^r f_is^i$,
   is a pre-connection on $\widehat{\cal E}$.
 Here, recall (Definition~1.3.1) that
   $\,\!^{\varsigma_\xi}\!s^i :=s^i$ for $\xi$ even, or $s^i_{(\even)}-s^i_{(\odd)}$ for  $\xi$ odd;
   and similarly for~$\,\!^{\varsigma_\xi}\!f_i$. \hspace{2em}
 In particular, for $(s^1,\,\cdots\,, s^r)$ even,
     $\widehat{\nabla}_{\!\xi}s
	    = \sum_{i=1}^r
	               \left( df_i(\xi)\cdot s^i
                               +\, (-1)^{p(\xi)p(f_i)}\,f_i \cdot  A(\xi)^{\circ}(s^i)
                   \right)$
	for \hspace{2em} $f_i\in\widehat{\cal O}_X$, $\xi\in{\cal T}_{\widehat{X}}$ parity homogeneous.
     
  {\rm
  A pre-connection on $\widehat{\cal E}$ of such particular type is called
             a {\it simple hybrid connection} on $\widehat{\cal E}$.
  The $ \Endsheaf_{\widehat{\cal O}_X}\!(\cal E) $-valued $1$-from $A$ in the setting
     is called the {\it connection $1$-form} of $\widehat{\nabla}$ with respect to the trivialization
       of $\widehat{\cal E}$.
   From the expression
    $\widehat{\nabla}_{\!\xi}s=(\xi)\,\!^{\leftarrow}\!\!\!(ds+ sA)$,
    we will write
	  $$
	     \widehat{\nabla}s \; =\;  ds +sA   \hspace{2em}\mbox{or}\hspace{2em}
		 \widehat{\nabla}\;=\; d+A^{\circ}
	  $$	
	in short hand.	
  In particular, when $A=0$, $\widehat{\nabla}= d$ resumes to
	 the trivial pre-connection associated to the given trivialization of $\widehat{\cal E}$.
  We will call such trivial pre-connection simply a {\it trivial connection}. 	
	  }
\end{lemma-definition}

\medskip

\begin{proof}
 (1) {\it $\widehat{\cal E}$-linearity in the ${\cal T}_{\widehat{X}}$-argument}
  $$
  \begin{array}{rcl}
   \widehat{\nabla}_{f_0\xi}\,s       & =
     & \sum_{i=1}^r
	      \left(
		    (f_0\xi)f_i\cdot s^i\,+\,
			  \,\!^{\varepsilon_{(f_0\xi)}}f_i \cdot\,\!^{\varsigma_{(f_0\xi)}}s^i A(f_0\xi)
          \right)  \\[1.2ex]
     & = &
        \sum_{i=1}^r	
         \left(
		  f_0\cdot \xi f_i \cdot s^i \,
		  +\, f_0\cdot\,\!^{\varsigma_\xi}f_i\cdot\,\!^{\varsigma_\xi}s^i A(\xi)
         \right) 		 \\[1.2ex]
	 & = &
	   f_0\,\widehat{\nabla}_{\!\xi}s\,.
  \end{array}
  $$
  Here, we've used
   the identity that $A(f_0\xi)= (f_0\xi)\,\!^{\leftarrow}\!\!\!A = f_0\,A(\xi)$    and
   the observations that
   \begin{itemize}
    \item[\LARGE $\cdot$] {\it For $f_0$ even}\,:\hspace{1em}
	  $p(f_0\xi)=p(\xi)$ and hence
	  $\,\!^{\varsigma_{(f_0\xi)}}(\,\mbox{\tiny $\bullet$}\,)
	     =\,\!^{\varsigma_\xi}(\,\mbox{\tiny $\bullet$}\,)$;\;
	  $\,\!^{\varsigma_\xi}f_i\cdot \,\!^{\varsigma_\xi}s^i \cdot f_0
	     = f_0\cdot \,\!^{\varsigma_\xi}f_i\cdot \,\!^{\varsigma_\xi}s^i $.

    \item[\LARGE $\cdot$] {\it For $f_0$ odd}\,:\hspace{1em}
	   \begin{itemize}
	    \item[\LARGE $\cdot$] {\it For $\xi$ even}\,:\hspace{1em}
		 $f_0\xi$ is odd;\;
		 $\,\!^{\varsigma_{(f_0\xi)}}(\,\mbox{\tiny $\bullet$}\,)
		    = \,\!^\varsigma  (\,\mbox{\tiny $\bullet$}\,)$;\;
         $\,\!^{\varsigma}f_i\cdot \,\!^{\varsigma}s^i\cdot f_0
            = f_0\cdot f_i\cdot s^i  = f_0 \cdot \,\!^{\varsigma_\xi}f_i \cdot \,\!^{\varsigma_\xi}s^i$		
 			
		\item[\LARGE $\cdot$] {\it For $\xi$ odd}\,:\hspace{1em}
         $f_0\xi$ is even;\;
	     $\,\!^{\varsigma_{(f_0\xi)}}(\,\mbox{\tiny $\bullet$}\,)
		    = (\,\mbox{\tiny $\bullet$}\,)$;\;
         $f_i\cdot s^i\cdot f_0
		    = f_0 \cdot \,\!^{\varsigma}\!f_i\cdot \,\!^{\varsigma}s^i
   		    = f_0 \cdot \,\!^{\varsigma_\xi}f_i \cdot \,\!^{\varsigma_\xi}s^i$.	
	   \end{itemize}
   \end{itemize}
  
 \medskip
 
 \noindent
 (2) {\it ${\Bbb Z}/2$-graded Leibniz rule in the $\widehat{\cal E}$-argument}
  $$
   \begin{array}{rcl}
     \widehat{\nabla}_{\,\xi}(f_0s)      & =
	   & \sum_{i=1}^r
	         \left (
		       \xi(f_0f_i)\cdot s^i  \,
			    +\,  \,\!^{\varsigma_\xi}\!(f_0f_i)\cdot \,\!^{\varsigma_\xi}s^i A(\xi)
		     \right) \\[1.2ex]
	 & = &
	    \sum_{i=1}^r
		 \left(
		    (\xi f_0)f_is^i\, +\, (-1)^{p(\xi)p(f_0)}f_0(\xi f_i)\cdot s^i\,
			  +\, (-1)^{p(\xi) p(f_0)}\,
			          f_0\cdot \,\!^{\varsigma_\xi}\!f_i
					  \cdot\,\!^{\varsigma_\xi}s^i A(\xi)
		 \right) \\[1.2ex]
     & = &
        (\xi f_0)s\,+\, (-1)^{p(\xi)p(f_0)}\, f_0\,\widehat{\nabla}_{\!\xi}s\,.	
   \end{array}
 $$
 Here, we've used the identities that, for $\xi$ parity homogeneous,
   $\,\!^{\varsigma_\xi}(f_0f_1)= \,\!^{\varsigma_\xi}\!f_0\cdot \,\!^{\varsigma_\xi}f_1$
      for general $f_0, f_1$,  and
   $\,\!^{\varsigma_\xi}f_0=(-1)^{p(\xi)p(f_0)}\,f_0$ for $f_0$ parity homogeneous.
	
 This completes the proof.

\end{proof}

\bigskip

\begin{lemma} {\bf [simple hybrid connection under gauge transformation]}\;
 Let $\widehat{\nabla}$ be a simple hybrid connection on $\widehat{\cal E}$ and
  $g\in \Autsheaf_{\widehat{\cal O}_X}(\widehat{\cal E})$
  be a gauge transformation, acting on $\widehat{\cal E}$ from the right.
 Then the conjugated ${\Bbb C}$-bilinear pairing
  $$
    \begin{array}{ccccc}
	 \widehat{\nabla}^g & : & {\cal T}_{\widehat{X}} \times \widehat{\cal E}
	     & \longrightarrow   & \widehat{\cal E}  \\[1.2ex]
    && (\xi, s)              &  \longmapsto
	   &  \widehat{\nabla}^g_\xi s\;
	         :=\; ( g^{\circ}  \circ  \widehat{\nabla}_{\xi} \circ  g^{-1\,\circ}) (s)\;
             :=\;      \left(     \widehat{\nabla}_{\!\xi}(sg^{-1})	   \right)g
	\end{array}
   $$
   defines a simple hybrid connection on $\widehat{\cal E}$.
 Suppose that $\widehat{\nabla}s=  ds +sA$ (with respect to a fixed trivialization of $\widehat{\cal E}$).
 Then
  $$
    \widehat{\nabla}^g s\;=\;   ds \,+\,  s( dg^{-1} g +  g^{-1} A g)\;
	            =:\; ds + s A^g
  $$
 (with respect to the fixed trivialization of $\widehat{\cal E}$).
 Here,
  $d$ in $dg^{-1}$ is the induced trivial connection on
    $\Autsheaf_{\widehat{\cal O}_X}(\widehat{\cal E})$
    from the fixed trivial connection $d$ on $\widehat{\cal E}$  and
  $dg^{-1}, dg^{-1}g
       \in  \Endsheaf_{\widehat{\cal O}_X}(\widehat{\cal E})
              \otimes_{\widehat{\cal O}_X} {\cal T}_{\widehat{X}}^{\ast}$.
 Note that $\, dg^{-1}g = - g^{-1}dg$.
\end{lemma}

\medskip

\begin{proof}
 That $\widehat{\nabla}^g_{f\xi}s= f\widehat{\nabla}^g_\xi s$ is clear.
 The ${\Bbb Z}/2$-Graded Leibniz Rule Condition on $\widehat{\nabla}^g$ follows from
  $$
   \begin{array}{rcl}
    \widehat{\nabla}^g_\xi  s   & =    &
	   \left(   \widehat{\nabla}_{\!\xi} ( f\cdot sg^{-1}) \right)g \;\;
	      =\;\;      \left(    (\xi f)\cdot sg^{-1}\,
		                                +\, (-1)^{p(\xi)p(f)}\, f\cdot \widehat{\nabla}_{\!\xi}(sg^{-1})
				   	   \right) g    \\[1.2ex]
    & = & 					
      (\xi f)s\,+\, (-1)^{p(\xi)p(f)}\, f\cdot
                         \left(   \widehat{\nabla}_{\!\xi}(sg^{-1})   \right) g\;\;
       =\;\; (\xi f)s\,+\, (-1)^{p(\xi)p(f)}\, f
	                                       \cdot \widehat{\nabla}^g_{\xi}s\,.						 
   \end{array}
  $$
 This shows that $\widehat{\nabla}^g$ is a pre-connection.
 
 Now fix a trivialization of $\widehat{\cal E}$.
 This induces a trivialization of $\Autsheaf_{\widehat{\cal O}_X}(\widehat{\cal E})$.
 With respect to these trivializations,
  $$
   \begin{array}{rcl}
    \widehat{\nabla}^g_\xi s     & = &
	   \left(\widehat{\nabla}_{\!\xi}(sg^{-1})    \right)g\;\;
	   =\;\; \left(   \xi(sg^{-1})\,+\, \,\!^{\varsigma_\xi}\!(sg^{-1}\,A(\xi))
	           \right)g   \\[1.2ex]
    & = &
	 \left( (\xi s)g^{-1}\,+\,  \,\!^{\varsigma_\xi}s(\xi g^{-1}) \,
	                  + \, \,\!^{\varsigma_\xi}s\cdot\,\!^{\varsigma_\xi}g^{-1}\cdot A(\xi)
	 \right) g    \\[1.2ex]
	& = &   \xi s \,
	               +\, \,\!^{\varsigma_\xi}s(\xi g^{-1}) g\,
	               +\, \,\!^{\varsigma_\xi}s\cdot\,\!^{\varsigma_\xi}g^{-1}\cdot A(\xi)  g \\[1.2ex]
	& = &
	  \xi s \,+\, (\xi)\,\!^{\leftarrow}\!\!\!\left(
	                              s(dg^{-1}g)\,+\, s(g^{-1}A g)
								                                            \right)          \\[1.2ex]
    & = & 			(\xi)\,\!^{\leftarrow}\!\!\!\left(
                                  ds + s(dg^{-1}g+ g^{-1}A g)
	                                                                        \right)     \;\;
		=:\;\; (\xi)\,\!^{\leftarrow}\!\!\!\left( ds + s A^g \right)\,.
  \end{array}
  $$
 This shows that $\widehat{\nabla}^g$ is indeed a simple hybrid connection.
 
\end{proof}

\bigskip

Note that, as illustrated in Explanation~2.1.7,
similar statement holds for a left connection only when the gauge transformation is parity-preserving.

\vspace{8em}

\begin{flushleft}
{\bf  The curvature tensor of a simple hybrid connection on $\widehat{\cal E}$ and its components}
\end{flushleft}
%
%
%
\begin{lemma-definition} {\bf [curvature tensor of simple hybrid connection $\widehat{\nabla}$]}\;
 Let $\widehat{\nabla}$ be a simple hybrid connection on $\widehat{\cal E}$
   with $\widehat{\nabla}s = ds + sA$ with respect to a fixed trivialization of $\widehat{\cal E}$.
 Let
   $$
      F^{\widehat{\nabla}}\; :=\;  dA\, -\, A\wedge A\,,
   $$
  where
    \begin{itemize}
	 \item[\LARGE $\cdot$]
      the $d$ on
       $\Endsheaf_{\widehat{\cal O}_X}(\widehat{\cal E})
           \otimes_{\widehat{\cal O}_X}\!{\cal T}^{\ast}_{\widehat{X}}$
       comes from the usual exterior differential $d$ on ${\cal T}^{\ast}_{\widehat{X}}$
	   and the trivial left connection $d$ associated to the induced trivialization
	    of  $\Endsheaf_{\widehat{\cal O}_X}(\widehat{\cal E})$ from that of $\widehat{\cal E}$,  and

     \item[\LARGE $\cdot$]
      the wedge-product $A\wedge A$ comes
	    from the wedge-product of $1$-forms on $\widehat{X}$ and
		the ring-multiplication in $\Endsheaf_{\widehat{\cal O}_X}(\widehat{\cal E})$
		subject to the ${\Bbb Z}\times {\Bbb Z}/2$-bi-grading passing rule Convention~1.3.5.
   \end{itemize}
  Then,  under  a gauge transformation
    $g\in  \Autsheaf_{\widehat{\cal O}_X}(\widehat{\cal E})
	   \subset    \Endsheaf_{\widehat{\cal O}_X}(\widehat{\cal E})$,
	$$
	   F^{\widehat{\nabla}^g}\;=\; g^{-1}\,F^{\widehat{\nabla}}\, g\,.
	$$
  In particular,
    $F^{\widehat{\nabla}}
	  \in \Endsheaf_{\widehat{\cal O}_X}(\widehat{\cal E})
               \otimes_{\widehat{\cal O}_X}\! \bigwedge^2{\cal T}^{\ast}_{\widehat{X}}$
	 is independent of the trivialization of $\widehat{\cal E}$.
 {\rm  We shall call $F^{\widehat{\nabla}}$ the {\it curvature tensor}
   of the simple hybrid connection $\widehat{\nabla}$.}	
\end{lemma-definition}

\medskip

\begin{proof}
 Recall Lemma~2.2.7 that
  $\widehat{\nabla}s = ds + s A^g$, with $A^g=  dg^{-1}g + g^{-1}A g$
  with respect to the same fixed trivialization of $\widehat{\cal E}$.
 Thus,
  $$
	 dA^g \;
	   =\;  -\, dg^{-1}\wedge dg\,
	          + \, dg^{-1}\wedge Ag \,+\, g^{-1}dA g\,  -\, g^{-1}A \wedge dg
 $$
 while
 \begin{eqnarray*}
  \lefteqn{
     A^g\wedge A^g\;
	   =\; ( dg^{-1}g + g^{-1}A g) \wedge ( dg^{-1}g + g^{-1}A g) }\\[-.2ex]
  && =\; (dg^{-1}g)\wedge (dg^{-1}g)\, +\, (dg^{-1}g)\wedge (g^{-1}Ag)\,
                +\, (g^{-1}Ag)\wedge (dg^{-1}g)\,+\, (g^{-1}Ag)\wedge (g^{-1}Ag)\\[-.2ex]
  && =\; -\, dg^{-1}\wedge dg\, +\, dg^{-1}\wedge Ag \,
              -\, g^{-1}A\wedge dg\, +\, g^{-1}A\wedge A g\,.
 \end{eqnarray*}
 Here,  the identity $dg^{-1}g= - g^{-1}dg$ is used.
 It follows that
  $$
   dA^g-A^g\wedge A^g\;
     =\;  g^{-1}(dA - A\wedge A)g\,.
  $$
 This proves the lemma.
    
\end{proof}

\medskip

\begin{lemma} {\bf [components of curvature tensor]}\;
 In terms of the supersymmetrically invariant frame $\{e_I\}_I$
  and coframe $\{e^J\}_J$ on $\widehat{X}$
    (Definition~1.4.7 and Definition~1.4.8),
 let\\ $F^{\widehat{\nabla}}=\sum_{I,J}e^I\wedge e^J F_{IJ}$.
 Then,
    \begin{eqnarray*}
    (e_I, e_J, s)\,\!^{\leftarrow}\!\!\!F^{\widehat{\nabla}}
	 & = & (-1)^{p(s)(p(e_I)+p(e_j))}\,s\,
	        \mbox{\Large $($}(-1)^{p(e_I)p(e_J)}F_{IJ}\,-\, F_{JI}
			\mbox{\Large $)$}  \\
	& = &  \left(
	          \,\!^L\![\, \widehat{\nabla}_{e_I}\,,\, \widehat{\nabla}_{e_J}\}\,
			       -\, \widehat{\nabla}_{[e_I, e_J\}}
			\!\right) s\,.	
   \end{eqnarray*}
  Here
   for two left-right-mixed operators $O_1=a_1+ c_1^\circ $ and $O_2=a_2+c_2^{\circ}$
     with $a_1$ and $a_2$ parity homogeneous and $c_1$ and $c_2$ in arbitrary parity situation,
   $$
      \,\!^L\![O_1, O_2   \}\;  :=\;  O_1O_2- (-1)^{p(a_1)p(a_2)}\,O_2O_1\,.
   $$
\end{lemma}

\medskip

\begin{proof}
 The first equality follows from the identity
    \begin{eqnarray*}
	 \lefteqn{
    (e_I, e_J)\,\!^\leftarrow\!\!\!(e^K\wedge e^L)\;
	    =\;     (e_I, e_J) \,\!^{\leftarrow}\!\!\!
                   \mbox{\Large $($}
				     e^K\otimes e^L - (-1)^{p(e_K)p(e_L)}\, e^L\otimes e^K
					\mbox{\Large $)$}		}\\[-.2ex]
    && =\; (-1)^{p(e_I)p(e_J)}\,\delta_{IK}\delta_{JL}\,-\, \delta_{IL}\delta_{JK}	
	               \hspace{12em}
   \end{eqnarray*}
   and the sign-rule for passings (Convention~1.3.5).
 We now turn to the second equality.

 To reduce some load of notations, write
    $A = \sum_I e^I A_I$   and
	$[e_I, e_K\}= \sum_K c_{IJ}^K e_K$.
  Recall that $c_{IJ}^K\in {\Bbb C}$ (and hence are even)  and
     that $de^K=\frac{1}{2}\sum_{I,J}c^K_{IJ}e^I\wedge e^J$; Definition~1.4.8.
 Then,
   \begin{eqnarray*}
    \lefteqn{
	 F^{\widehat{\nabla}}\;
	      =\; dA-A\wedge A\;
	      =\; \sum_I de^I A_I\,-\, \sum_I e^I\wedge dA_I\, -\, \sum_{I,J}(e^IA_I)\wedge (e^J A_J)
	                }\\
	 && =\; \sum_{I,J}e^I\wedge e^J
                 	 \left(\rule{0ex}{1em}\right.
	                    \frac{1}{2}\,\sum_K c^K_{IJ}A_K\,
						 -\, e_JA_I\, -\, \,\!^{\varsigma_{e_{\!J}}}\!A_I A_J
	                 \left.\rule{0ex}{1em}\right)\;
            =:\; \sum_{I,J}e^I\wedge e^J\,F_{IJ}\,.					
   \end{eqnarray*}
 Thus, for $s\in \widehat{\cal E}$ parity homogeneous,
  \begin{eqnarray*}
    \lefteqn{
	  (e_I, e_J, s)\,\!^\leftarrow \!\!\! F^{\widehat{\nabla}}\;
	     =\;  \sum_{K,L}
		            (-1)^{p(s)(p(e^K)+p(e_L))}
		           (e_I,e_L)\,\!^\leftarrow\!\!\!(e^K\wedge e^L) \cdot  s F_{KL}
	           }\\
    &&
	 =\; (-1)^{p(s)(p(e_I)+p(e_j))}\,s\,
	        \mbox{\Large $($}(-1)^{p(e_I)p(e_J)}F_{IJ}\,-\, F_{JI}
			\mbox{\Large $)$}  \\[.6ex]
	&&
     =\; 	(-1)^{p(s)(p(e_I)+p(e_J))}\, s\,
	          \left(\rule{0ex}{1em}\right.
			   -\,(-1)^{p(e_I)p(e_J)}\,e_JA_I + e_IA_J  \\
    &&  \hspace{12em}			
		       -\, (-1)^{p(e_I)p(e_J)}\cdot\,\!^{\varsigma_{e_{\!J}}}\!A_I A_J\,
               +\, \,\!^{\varsigma_{e_{\!I}}}\!A_J A_I\,			
	           -\, \sum_K c_{IJ}^K A_K
	          \left.\rule{0ex}{1em}\right)\,.
  \end{eqnarray*}
  Here, note that $p(e_I)=p(e^I)$ for all $I$  and
  we've used the observation that
   $$
     (-1)^{p(s)(p(e_I)+p(e_J))+ p(e_I)p(e_J)}s\,\sum_k c_{IJ}^K A_K\;
	   =\; -\, (-1)^{p(s)(p(e_I)+p(e_J))}\, s\,\sum_K c_{IJ}^K A_K
   $$
   since $c_{IJ}^K\ne 0$
      only when $(I,J)=(\alpha^{\prime}, \beta^{\prime\prime})$
	            or $(\beta^{\prime\prime},\alpha^{\prime})$,
		in which case $p(e_I)p(e_J)=1$.
 
 On the other hand, since the applying-from-the-left part of $\widehat{\nabla}$ is $d$, which is even,
  $$
   \begin{array}{rll}
      \mbox{\Large $($}
	          \,\!^L\![\, \widehat{\nabla}_{\!e_I}\,,\, \widehat{\nabla}_{\!e_J}\}\,
			       -\, \widehat{\nabla}_{[e_I, e_J\}}
				\mbox{\Large $)$}s
    & =  & \widehat{\nabla}_{\!e_I} \widehat{\nabla}_{\!e_J}s\,
                -\, (-1)^{p(e_I)p(e_J)}\,
 		               \widehat{\nabla}_{\!e_J} \widehat{\nabla}_{\!e_I} s\,
		        -\, \widehat{\nabla}_{[e_I, e_J\}}s  \\[1.2ex]
	& =: &  \mbox{(I)}\,+\, \mbox{(II)}\, +\, \mbox{(III)}\,.
   \end{array}	
  $$
  We now proceed to work out the three summands (I), (II), and (III).
  
  $$
  \begin{array}{rrl}
	  \mbox{(I)} &  :=
        &	  \widehat{\nabla}_{\!e_I} \widehat{\nabla}_{\!e_J}s\;\;
	         =\;\;  \widehat{\nabla}_{\!e_I}
		             \mbox{\Large $($}e_Js +(-1)^{p(s)p(e_J)}sA_J \mbox{\Large $)$}  \\[1.2ex]
     &  =
	   &   e_I  \mbox{\Large $($}e_Js +(-1)^{p(s)p(e_J)}sA_J \mbox{\Large $)$}\,
	          +\, \,\!^{\varsigma_{e_{\!I}}}\!
			     \mbox{\Large $($}e_Js +(-1)^{p(s)p(e_J)}sA_J \mbox{\Large $)$}A_I  \\[1.2ex]
	 & =
       &	e_Ie_J s\,
	        +\,  (-1)^{p(s)p(e_J)}\,(e_Is)A_J\,
			+\, (-1)^{p(s)(p(e_J)+p(e_I))}\,s\cdot e_IA   \\[.8ex]
	 && \hspace{4em}
			+\, (-1)^{p(e_I)p(e_J)+ p(e_I)p(s)}\, (e_Js)A_I\,
			+\, (-1)^{p(s)(p(e_J)+p(e_I))}\,
			      s\cdot \,\!^{\varsigma_{e_{\!I}}}\!A_J A_I\,.
  \end{array}
  $$
  With $I\leftrightarrow J$ and the sign-factor added,
  $$
  \begin{array}{rrl}
	  \mbox{(II)}    &  :=
         & -\, (-1)^{p(e_I)p(e_J)}\,
			   \widehat{\nabla}_{\!e_J} \widehat{\nabla}_{\!e_I} s      \\[1.2ex]
      &  =
         &	  -\,(-1)^{p(e_I)p(e_J)}\, e_Je_I s  \\[.8ex]
	  && \hspace{2em}
	        -\,  (-1)^{p(e_I)p(e_J)+p(s)p(e_I)}\,(e_Js)A_I\,
			-\, (-1)^{p(e_I)p(e_J)+p(s)(p(e_I)+p(e_J))}\,s\cdot e_JA   \\[.8ex]
	 && \hspace{4em}
			-\, (-1)^{p(e_J)p(s)}\, (e_Is)A_J\,
			-\, (-1)^{p(e_I)p(e_J)+p(s)(p(e_I)+p(e_J))}\,
			      s\cdot \,\!^{\varsigma_{e_{\!J}}}\!A_I A_J\,.
  \end{array}
  $$
  $$
   \begin{array}{rrl}
	  \mbox{(III)}  &  :=
        &    -\, \widehat{\nabla}_{[e_I, e_J\}}s\;\;
           =\;\; -\,[e_I, e_J \}s   \,
	           -\, (-1)^{p(s)(p(e_I)+p(e_J))}\, s \cdot ([e_I, e_J\})\,\!^\leftarrow\!\!\!A  \\[1.2ex]
	 & =
       & -\,[e_I, e_J \}s   \,
	              -\, (-1)^{p(s)(p(e_I)+p(e_J))}\, s\, \sum_K  c_{IJ}^K A_K\, .
   \end{array}
  $$		
 After summing (I), (II), and (III),
   one find that all terms that involve derivations on $s$ cancel:
  \begin{eqnarray*}
   \lefteqn{
     \mbox{\Large $($}
	          \,\!^L\![\, \widehat{\nabla}_{\!e_I}\,,\, \widehat{\nabla}_{\!e_J}\}\,
			       -\, \widehat{\nabla}_{[e_I, e_J\}}
				\mbox{\Large $)$}s\;\;
	      =\;\;   \mbox{(I)}\,+\, \mbox{(II)}\, +\, \mbox{(III)}     }\\[.6ex]
	 &&
	  =\;   (-1)^{p(s)(p(e_J)+p(e_I))}\,s\cdot e_IA\,
			    +\, (-1)^{p(s)(p(e_J)+p(e_I))}\,
			             s\cdot \,\!^{\varsigma_{e_{\!I}}}\!A_J A_I    \\
	  && \hspace{2em}
			-\, (-1)^{p(e_I)p(e_J)+p(s)(p(e_I)+p(e_J))}\,s\cdot e_JA
			-\, (-1)^{p(e_I)p(e_J)+p(s)(p(e_I)+p(e_J))}\,
			      s\cdot \,\!^{\varsigma_{e_{\!J}}}\!A_I A_J   \\
      && \hspace{4em}
	  	  -\, (-1)^{p(s)(p(e_I)+p(e_J))}\, s\, \mbox{$\sum$}_K  c_{IJ}^K A_K \\[.6ex]
	 && =\;  	(e_I, e_J, s)\,\!^\leftarrow\!\!\! F^{\widehat{\nabla}}\,.		
  \end{eqnarray*}
  
  This completes the proof.
		
\end{proof}

\bigskip

\begin{flushleft}
{\bf The induced connection on $\widehat{\cal E}^{\vee}$ and
         $\Endsheaf_{\widehat{\cal O}_X}\!(\widehat{\cal E})$}
\end{flushleft}
Let
 $\widehat{\cal E}^{\vee}
   := \Homsheaf_{\Left\widehat{\cal O}_X}(\widehat{\cal E}, \widehat{\cal O}_X)$
   be the right dual bi-$\widehat{\cal O}_X$-module of $\widehat{\cal E}$.
It applies to $\widehat{\cal E}$ from the right (of $\widehat{\cal E}$)  and
there is a built-in evaluation map
 $\widehat{\cal E}\otimes_{\widehat{\cal O}_X} \widehat{\cal E}^{\vee}
    \rightarrow \widehat{\cal O}_X$,
 with $s\otimes \tilde{t} \mapsto (s)\,\!^\leftarrow\!\!\! \tilde{t}$.
 %
For a right endomorphism
 $\widehat{m}
    \in\Endsheaf_{\widehat{\cal O}_X}\!(\widehat{\cal E}) $ of $\widehat{\cal E}$,
 denote by
  $\widehat{m}^{\!\vee}\in
    \Endsheaf_{\Right\widehat{\cal O}_X}\!(\widehat{\cal E}^{\vee})$
 its dual left endomorphism of $\widehat{\cal E}^{\vee}$,
 defined by
 $(s\widehat{m})\,\!^\leftarrow\!\!\! \tilde{t}
    = (s)\,\!^\leftarrow\!\!\! (\widehat{m}^{\!\vee} \tilde{t})$
 for all $s\in \widehat{\cal E}$ and $\tilde{t}\in \widehat{\cal E}^{\vee}$.
In terms of these notations, one has the following two Lemma/Definitions.

\bigskip

\begin{lemma-definition} {\bf [induced left connection on right dual $\widehat{\cal E}^{\vee}$]}\;
 Let $\widehat{\nabla}$ be a simple hybrid connection on $\widehat{\cal E}$
   with $\widehat{\nabla}s= ds+sA$ with respect to a fixed trivialization of $\widehat{\cal E}$.
 Then
  $$
    \widehat{\nabla}^{\vee}t\;   :=\;  dt\, -\, A^{\!\vee}t\,,
  $$
  with respect to the dual trivialization of $\widehat{\cal E}^{\vee}$,
  defines a left connection on $\widehat{\cal E}^{\vee}$.
 It satisfies the identity
  $$
    d((s)\,\!^\leftarrow\!\!\! t)\;
	 =\;  (\widehat{\nabla}s)\,\!^\leftarrow\!\!\! \tilde{t}\,
	        +\, (s)\,\!^\leftarrow\!\!\! (\widehat{\nabla}^{\vee}\tilde{t})\,.
  $$
  
 {\rm $\widehat{\nabla}^{\vee}$ is called the
     {\it induced (left) connection} on $\widehat{\cal E}^{\vee}$
    from $\widehat{\nabla}$ on $\widehat{\cal E}$.}
\end{lemma-definition}

\medskip

\begin{proof}
 That $\widehat{\nabla}^\vee$ is a left connection on $\widehat{\cal E}^{\vee}$
  follows from Sec.\ 2.1.
 The identity follows from a direct computation:
 \begin{eqnarray*}
  (\widehat{\nabla}s)\,\!^\leftarrow\!\!\! \tilde{t}\,
	        +\, (s)\,\!^\leftarrow\!\!\! (\widehat{\nabla}^{\vee}\tilde{t})
     & = &   (ds+sA)\,\!^\leftarrow\!\!\! \tilde{t}\,
	        +\, (s)\,\!^\leftarrow\!\!\! (d\tilde{t}-A^{\!\vee}\tilde{t})\\
    & = &   \mbox{\Large $($}
	             (ds)\,\!^\leftarrow\!\!\! \tilde{t}\,
	        +\, (s)\,\!^\leftarrow\!\!\! (d\tilde{t})
			    \mbox{\Large $)$}\,
            +\, \mbox{\Large $($}
			     (sA)\,\!^\leftarrow\!\!\! \tilde{t}\,
	                -\, (s)\,\!^\leftarrow\!\!\! (A^{\!\vee}\tilde{t})
			     \mbox{\Large $)$} \;\; =\;\;    d((s)\,\!^\leftarrow\!\!\! t)\,.
 \end{eqnarray*}
\end{proof}

\medskip

\begin{remark} $[\,A^{\vee}$ in form of matrix\,$]$\; {\rm
 If $s\in \widehat{\cal E}$ are taken as row vectors under the trivialization and
    $\tilde{t}\in \widehat{\cal E}^{\vee}$ are taken as column vectors under the dual trivialization,
 let $A$ be in the form of a matrix that applies to $\widehat{\cal E}$ from the right:
  $s\mapsto sA$ as matrix multiplications.
 Then in the form of a matrix that applies to $\widehat{\cal E}$ from the left,
   $A^{\vee}= A$. I.e.\ $A^{\vee}: \tilde{t}\mapsto A\tilde{t}$.
}\end{remark}

\medskip

\begin{lemma-definition} {\bf [induced hybrid connection $\widehat{D}$ on
            $\Endsheaf_{\widehat{\cal O}_X}\!(\widehat{\cal E}) $]}\;
 Let $\widehat{\nabla}$ be a simple hybrid connection on $\widehat{\cal E}$
   with $\widehat{\nabla}s= ds+sA$ with respect to a fixed even trivialization of $\widehat{\cal E}$.
 Then, (assuming $\widehat{m}$ parity homogeneous
              when using $\,\!^{\varsigma_{\widehat{m}}}\!(\mbox{\tiny $\bullet$})$)
  $$
    \widehat{D}\widehat{m}\;   :=\;  d\widehat{m}\, -\, [A, \widehat{m}\}\;
	 =\; (d\widehat{m} - A\widehat{m} )
	       + \widehat{m}\, \,\!^{\varsigma_{\widehat{m}}}\!\!A\;\;
	 =:\; \widehat{D}^{(\leftscriptsize)} \widehat{m}\,
	       +\widehat{m}\, \,\!^{\varsigma_{\widehat{m}}}\!\!A
  $$
  with respect to the induced trivialization of
   $\Endsheaf_{\widehat{\cal O}_X}\!(\widehat{\cal E}) $,
  defines a hybrid connection on
  $\Endsheaf_{\widehat{\cal O}_X}\!(\widehat{\cal E}) $.
 As a ${\Bbb C}$-bilinear pairing
  $\widehat{D} : {\cal T}_{\widehat{X}} \times
     \Endsheaf_{\widehat{\cal O}_X}\!(\widehat{\cal E})
	  \rightarrow   \Endsheaf_{\widehat{\cal O}_X}\!(\widehat{\cal E})$
	with $(\xi, \widehat{m}) \mapsto \widehat{D}_\xi \widehat{m}$,
  it satisfies the following three properties:\footnote{In
                                                    this work,
                                                    we've defined the notion of
													  {\sl left connections} and {\sl simple hybrid connectons}.
	                                               The induced connection $\widehat{D}$ on
												      $\Endsheaffootnotesize_{\widehat{\cal O}_X}\!
													   (\widehat{\cal E}) $
													 from the simple hybrid connection
													   $\widehat{\nabla}$ on $\widehat{\cal E}$
												     is of neither kind.
												   Since this is the only connection we will use
												     that has not been covered yet in this work,
													we list its specific characterizing properties here   and
                                                    leave the development of the notion of general {\it hybrid connections}
													 on an $\widehat{\cal O}_X$-module to the future.
   													}  
   \begin{itemize}
    \item[$(1)$]
     $[${\sl $\widehat{\cal O}_X$-linearity
	                                   in the ${\cal T}_{\widehat{X}}$-argument}$]$\\[.6ex]	
	  $\mbox{\hspace{.2em}}$
	  $\widehat{D}_{f_1\xi_1 + f_2\xi_2}\widehat{m}\;
	     =\; f_1 \widehat{D}_{\xi_1}\widehat{m}
		      + f_2 \widehat{D}_{\xi_2}\widehat{m}$, \hspace{.6em}
      for $f_1, f_2 \in \widehat{\cal O}_X$, $\xi_1, \xi_2 \in {\cal T}_{\widehat{X}}$,  and
	       $\widehat{m}\in \Endsheaf_{\widehat{\cal O}_X}\!(\widehat{\cal E})$;

    \item[$(2)$]
	 $[${\sl ${\Bbb C}$-linearity in the
	        $\Endsheaf_{\widehat{\cal O}_X}\!(\widehat{\cal E})$-argument}$]$\\[.6ex]
	 $\mbox{\hspace{.2em}}$
	 $\widehat{D}_\xi(c_1\widehat{m}_1+c_2 \widehat{m}_2)\;
	     =\;  c_1 \widehat{D}_\xi \widehat{m}_1
		    + c_2 \widehat{D}_\xi \widehat{m}_2$,\;\;
	  for $c_1, c_2\in {\Bbb C}$, $\xi\in {\cal T}_{\widehat{X}}$, and
	       $\widehat{m}_1, \widehat{m}_2
		     \in \Endsheaf_{\widehat{\cal O}_X}(\widehat{\cal E})$;
	
    \item[$(3)$]
	 $[${\sl general ${\Bbb Z}/2$-graded Leibniz rule
	    in the $\Endsheaf_{\widehat{\cal O}_X}\!(\widehat{\cal E})$-argument}$]$\\[.6ex]
	 $\mbox{\hspace{.2em}}$
	 $\widehat{D}_\xi(f\widehat{m})\;
	   =\; (\xi f)\widehat{m}
	           + (-1)^{p(f)p(\xi)}\,f
			       \cdot\,\!^{\varsigma_{\!f}}\!(\widehat{D})_\xi
				                  \widehat{m}$,\\[.6ex]
      for $f\in \widehat{\cal O}_X$, $\xi\in {\cal T}_{\widehat{X}}$ parity homogeneous
	       and $\widehat{m}\in\Endsheaf_{\widehat{\cal O}_X}\!(\widehat{\cal E})$.
   \end{itemize}
 Which justifies $\widehat{D}$ to be taken as a connection.
 Note that Property (3) can be written as
   $$
     \widehat{D}(f\widehat{m})\;
	 =\; df\cdot \widehat{m}\,
	       +\, f\cdot \,\!^{\varsigma_{\!f}}\!\!\widehat{D}\,\widehat{m}
   $$
  as well.
 Furthermore,
  \begin{itemize}
   \item[$(4)$]
   $[${\sl ${\Bbb Z}/2$-graded Leibniz rule with respect to product
	             in $\Endsheaf_{\widehat{\cal O}_X}\!(\widehat{\cal E}) $}$]$\\[.6ex]
    $\mbox{\hspace{.2em}}$
	$\widehat{D}(\widehat{m}_1\widehat{m}_2)\;
	   =\; (\widehat{D}\widehat{m}_1)\widehat{m}_2\,
	           +\, \widehat{m}_1
			         (\,\!^{\varsigma_{\widehat{m}_{\!1}}}\!\!\widehat{D}\,
					                                                                     \widehat{m}_2)$,\\[.6ex]
     for $\widehat{m}_1$ (parity homogeneous),
	      $\widehat{m}_2\in\Endsheaf_{\widehat{\cal O}_X}\!(\widehat{\cal E}) $.
  
   \item[$(5)$]
   $[${\sl relation with $\widehat{\nabla}^{\vee}$ and $\widehat{\nabla}$ under
                canonical isomorphism
		        $\Endsheaf_{\widehat{\cal O}_X}\!(\widehat{\cal E})
				   \simeq \widehat{\cal E}^{\vee}
				                \otimes_{\widehat{\cal O}_X}\!\widehat{\cal E}$}$]$\\[.6ex]
   $\mbox{\hspace{.2em}}$								
    for $\widehat{m}=\tilde{t}\otimes s$ parity homogeneous,
	 $\;\widehat{D}\widehat{m}
	      = (\widehat{\nabla}^{\vee}\tilde{t})\otimes s
		       + \tilde{t}\otimes\,\!^{\varsigma_{\widehat{m}}}\!\widehat{\nabla}\,s$\,;
   
   \item[$(6)$]
   $[${\sl relation with evaluation}$]$\\[.6ex]
   $\mbox{\hspace{.2em}}$
     $\widehat{\nabla}(s\widehat{m})\;
	    =\;   (\widehat{\nabla}s)\widehat{m}\,+\, s(\widehat{D}\widehat{m})\,
		         +\, s\widehat{m}(A-\,\!^{\varsigma_{\widehat{m}}}\!\!A)$,\;\;
	 for $s\in \widehat{\cal E}$ and
	      $\widehat{m}\in \Endsheaf_{\widehat{\cal O}_X}(\widehat{\cal E})$;\\[.6ex]	
   note that
     $A-\,\!^{\varsigma_{\widehat{m}}}\!\!A
	    =(1-(-1)^{p(\widehat{m})})\, A^{(\odd)}$
     vanishes if $\widehat{m}$ or $A$ is even;
	
   \item[$(7)$]
   $[${\sl $\widehat{D}$ as extension of $d$}$]$\\[.6ex]
   $\mbox{\hspace{.2em}}$
    $\widehat{D}$ restricts to $d$
      under the built-in ${\Bbb Z}/2$-graded-ring-inclusion
        $\widehat{\cal O}_X
	      \hookrightarrow  \Endsheaf_{\widehat{\cal O}_X}\!(\widehat{\cal E})$.
  \end{itemize}		
    
 {\rm $\widehat{D}$ is called the
     {\it induced hybrid connection} on
	   $\Endsheaf_{\widehat{\cal O}_X}\!(\widehat{\cal E})$ 	
    from the simple hybrid connection $\widehat{\nabla}$ on $\widehat{\cal E}$.}
\end{lemma-definition}

\medskip

\begin{proof}
 Properties (1), (2), and (7) are immediate.
 
 For Property (3), note that
   $[A, f\widehat{m}\}
     = [A,f\}\widehat{m}+ f[\,\!^{\varsigma_{\!f}}\!\!A, \widehat{m}\}
	 = f[\,\!^{\varsigma_{\!f}}\!\!A, \widehat{m}\}$,
	for $\widehat{m}$ parity homogeneous, since $[A, f\}=0$ for $f\in \widehat{\cal O}_X$.
 Thus,
  \begin{eqnarray*}
    \widehat{D}_\xi (f\widehat{m})
	& =  &  \xi(f\widehat{m})\, - \,(\xi)\,\!^\leftarrow\!\!\![A, f\widehat{m}\}    \\[.2ex]
	& =  &  (\xi f)\widehat{m}\,
	        +\, (-1)^{p(\xi)p(f)}\, f\cdot
			   (\xi\widehat{m}\,
			       -\, (\xi)\,\!^\leftarrow\!\!\! [\,\!^{\varsigma_{\!f}}\!\!A, \widehat{m} \})\\[.2ex]
	& = & (\xi f)\widehat{m}\,
	          +\, (-1)^{p(\xi)p(f)}\,f
			         \cdot  \,\!^{\varsigma_{\!f}}\!(\widehat{D})_{\xi} \widehat{m}\,.				   			
  \end{eqnarray*}

 Property (4) follows from the identity
  $[A, \widehat{m}_1\widehat{m}_2\}
     = [A, \widehat{m}_1\}\,\widehat{m}_2
	     + \widehat{m}_1\,[ \,\!^{\varsigma_{\widehat{m}_{\!1}}}\!\!A ,
		                                     \widehat{m}_2\}$.
    
 For Property (5), recall Lemma/Definition~2.2.10 and Remark~2.2.11.
 Then,						
  \begin{eqnarray*}
   \lefteqn{
   (\widehat{\nabla}^{\vee}\tilde{t})\otimes s
		       + \tilde{t}\otimes\,\!^{\varsigma_{\widehat{m}}}\!\widehat{\nabla}\,s\;
    =\; (d\tilde{t}-A\tilde{t})\otimes s\,
	       +\, \tilde{t}\otimes (ds+ s\, \,\!^{\varsigma_{\widehat{m}}}\!\!A )  }\\
    && =\;    (d\tilde{t})\otimes s\, +\, \tilde{t}\otimes ds\,
	                 -\,  A\tilde{t}\otimes s\,
					 +\, \tilde{t}\otimes s\, \,\!^{\varsigma_{\widehat{m}}}\!\!A\;
     =\; d\widehat{m}- [A, \widehat{m}\}\; =\; \widehat{D}\widehat{m}\,.   		
  \end{eqnarray*}
  
 For Property (6), recall also Lemma/Definition~2.2.10 and Remark~2.2.11
   and let $\widehat{m}=\tilde{t}\otimes s$ under the canonical isomorphism
   $\Endsheaf_{\widehat{\cal O}_X}\!(\widehat{\cal E})
      \simeq \widehat{\cal E}^{\vee}\otimes_{\widehat{\cal O}_X}\!\widehat{\cal E}$.
 Then,
  \begin{eqnarray*}
   \lefteqn{
  (\widehat{\nabla}s^{\prime})\widehat{m}\,
      +\, s^{\prime}(\widehat{D}\widehat{m}) \;
    =\; (ds^{\prime}+s^{\prime}A)\,\!^\leftarrow\!\!\!\tilde{t}\cdot s\,
	            +\, s^{\prime}\,\!^\leftarrow\!\!\!(d\widehat{m})\,
				-\,(s^{\prime}A\tilde{t})\cdot s\,
			   +\, (s^{\prime}\tilde{t})\otimes s\,
			              \,\!^{\varsigma_{\widehat{m}}}\!\!A) }\\
   && =\; \widehat{\nabla}(s^{\prime}\widehat{m})\,
                  +\, (s^{\prime}\tilde{t})\otimes s\,
			              \,\!^{\varsigma_{\widehat{m}}}\!\!A)\,
				  -\, (s^{\prime}\widehat{m})A\,. \hspace{16em}
  \end{eqnarray*}
    
 This completes the proof.	

\end{proof}

\bigskip

More explicitly, writing $A=\sum_I e^I A_I$, then
    $$
	 \begin{array}{c}
	  \,\!^{\varsigma_{\widehat{m}}}\!\!A \;
	  =\;  \sum_\mu e^\mu \cdot \,\!^{\varsigma_{\widehat{m}}}\!\!A_\mu
		 + (-1)^{p(\widehat{m})}\sum_{\alpha^\prime} e^{\alpha^\prime}
		                \cdot  \,\!^{\varsigma_{\widehat{m}}}\!\!A_{\alpha^\prime}
		 + (-1)^{p(\widehat{m})}\sum_{\beta^{\prime\prime}}
		              e^{\beta^{\prime\prime}}
		                \cdot  \,\!^{\varsigma_{\widehat{m}}}\!\!A_{\beta^{\prime\prime}}
    \end{array}						
   $$
  and
   $(e_I)
       \,\!^\leftarrow\!\!\!(\widehat{m}\,\,\!^{\varsigma_{\widehat{m}}}\!\!A)
	  = (-1)^{p(e_I)p(\widehat{m})}\widehat{m}\cdot
           (e_I)\,\!^\leftarrow\!\!\!\,\!^{\varsigma_{\widehat{m}}}\!\!A
	  = (-1)^{p(e_I)p(\widehat{m})}\, \,\!^{\varsigma_{\widehat{m}}}\!\!A_I$.
And one has
 $$
   \widehat{D}_{e_I}\widehat{m}\;
	 =\; e_I\widehat{m}\,-\, (e_I)\,\!^\leftarrow\!\!\![A, \widehat{m}\} \;
	 =\;  e_I\widehat{m}\,-\, A_I\widehat{m}\,
		          +\,\widehat{m} \,\,\!^{\varsigma_{\widehat{m}}}\!\!A_I\;
     =\;  e_I \widehat{m}\,-\, [A_I, \widehat{m}\}\,,
 $$
 for $\widehat{m}\in \Endsheaf_{\widehat{\cal O}_X}\!(\widehat{\cal E})$
      parity homogeneous.

\bigskip

\begin{remark} $[$\,$[\mbox{\LARGE $\cdot$} , \mbox{\LARGE $\cdot$}]$
   vs.\ $[\mbox{\LARGE $\cdot$} , \mbox{\LARGE $\cdot$}\}$ in Lemma/Definition~2.2.12\,$]$\;
{\rm							
  If one defines instead
  $$
    \widehat{D}^{\prime}\widehat{m}\;   :=\;  d\widehat{m}\, -\, [A, \widehat{m}]\;
	 =\; (d\widehat{m} - A\widehat{m} ) + \widehat{m}A\;\;
	 =:\; \widehat{D}^{(\leftscriptsize)} \widehat{m}\,+\widehat{m}A\,,
  $$
  then
  $\widehat{D}^{\prime}$ satisfies Properties (1), (2) and
   \begin{itemize}  	
    \item[$(3^\prime)$]
	 $[${\sl mixed general ${\Bbb Z}/2$-graded Leibniz rule
	    in the $\Endsheaf_{\widehat{\cal O}_X}\!(\widehat{\cal E})$-argument}$]$\\[.6ex]
	 $\mbox{\hspace{.2em}}$
	 $\widehat{D}^\prime_\xi(f\widehat{m})\;
	   =\; (\xi f)\widehat{m}
	           + (-1)^{p(f)p(\xi)}\,f
			       \cdot\,\!^{\varsigma_{\!f}}\!(\widehat{D}^{(\leftscriptsize)})_\xi
				                  \widehat{m}\,
			  + (-1)^{p(f)p(\xi)}\, f\cdot (\xi)\,\!^\leftarrow\!\!\!(\widehat{m}A)  $,\\[.6ex]
      for $f\in \widehat{\cal O}_X$, $\xi\in {\cal T}_{\widehat{X}}$ parity homogeneous
	       and $\widehat{m}\in\Endsheaf_{\widehat{\cal O}_X}\!(\widehat{\cal E})$.

   \item[$(4^\prime)$]
   $[${\sl Leibniz rule with respect to product
	             in $\Endsheaf_{\widehat{\cal O}_X}\!(\widehat{\cal E}) $}$]$\\[.6ex]
    $\mbox{\hspace{.2em}}$
	$\widehat{D}^\prime(\widehat{m}_1\widehat{m}_2)\;
	   =\; (\widehat{D}^\prime\widehat{m}_1)\widehat{m}_2\,
	           +\, \widehat{m}_1 \widehat{D}^\prime \widehat{m}_2$,
      for $\widehat{m}_1, \widehat{m}_2\in\Endsheaf_{\widehat{\cal O}_X}\!(\widehat{\cal E}) $.
		
   \item[$(5^\prime)$]		
   $[${\sl relation with $\widehat{\nabla}^{\vee}$ and $\widehat{\nabla}$ under
                canonical isomorphism
		        $\Endsheaf_{\widehat{\cal O}_X}\!(\widehat{\cal E})
				   \simeq \widehat{\cal E}^{\vee}
				                \otimes_{\widehat{\cal O}_X}\!\widehat{\cal E}$}$]$\\[.6ex]
   $\mbox{\hspace{.2em}}$								
    for $\widehat{m}=\tilde{t}\otimes s$ parity homogeneous,
	 $\;\widehat{D}^{\prime}\widehat{m}
	      = (\widehat{\nabla}^{\vee}\tilde{t})\otimes s
		       + \tilde{t}\otimes \widehat{\nabla}\,s$\,;	
		
   \item[$(6^\prime)$]
   $[${\sl compatibility with evaluation}$]$\\[.6ex]
   $\mbox{\hspace{.2em}}$
     $\widehat{\nabla}(s\widehat{m})\;
	    =\;   (\widehat{\nabla}s)\widehat{m}\,+\, s(\widehat{D}^{\prime}m)$,\;\;
	 for $s\in \widehat{\cal E}$ and
	      $\widehat{m}\in \Endsheaf_{\widehat{\cal O}_X}(\widehat{\cal E})$;	
	
   \item[$(7^\prime)$]
   $[${\sl $\widehat{D}^{\prime}$ restricted to $\widehat{\cal O}_X$}$]$\\[.6ex]
   $\mbox{\hspace{.2em}}$
    $\widehat{D}^{\prime}(f\cdot \Id)\;
	   =\; df\cdot \Id + f\cdot (A-\,\!^{\varsigma_{\!f}}\!\!A)\;
	   (=\; df\cdot \Id + (1-(-1)^{p(f)}) )\, fA^{(\odd)}$, \\[.6ex]
	for $f\in \widehat{\cal O}_X$ parity homogeneous.
  \end{itemize}
 Comparisons of $(3)$ vs.\ $(3^{\prime})$  and $(7)$ vs.\ $(7^{\prime})$
  indicate that $\widehat{D}$ is the more natural one to take for our problem.\footnote{{\it Induced
                                                connection and evaluation$\,:\,
												(5)$ vs.\ $(5^{\prime})$ and $(6)$ vs.\ $(6^{\prime})$}
												\hspace{1em}
											   The equality in Property $(5^{\prime})$	
											    is usually the identity used to define the induced connection
												on a tensor product.
											   Unfortunately, it gives $\widehat{D}^{\prime}$, not $\widehat{D}$. 	
											  Also, it is a little puzzling that
											   $\widehat{D}^{\prime}$ is compatible with the evaluation completely
											    while $\widehat{D}$ isn't
											    except for the case when $\widehat{D}$ is purely even.
											  These two observations seem to say that $\widehat{D}$,
											    despite satisfying Properties (1), (2), (3), (4), (7),
												may still not be the most functorially-natural induced connection on
												$\Endsheaffootnotesize_{\widehat{\cal O}_X}\!
												 (\widehat{\cal E})$
												from the connection $\widehat{\nabla}$ on $\widehat{\cal E}$.
											} 
}\end{remark}

\medskip

\begin{remark} $[$left-part vs.\ right-part in $\widehat{D}$$]$\, {\rm
 From Property (3) that $\widehat{D}$ satisfies,
   $\widehat{D}$ behaves like a left connection
   as long as the related ${\Bbb Z}/2$-graded Leibniz rule is concerned.
 On the other hand,
   the term $\widehat{m}\,\,\!^{\varsigma_{\widehat{m}}}\!\!A$
   in the expanded expression
   $\widehat{D}\widehat{m}
     = d\widehat{m}-A\widehat{m}
	    + \widehat{m}\,\,\!^{\varsigma_{\widehat{m}}}\!\!A$
  for $\widehat{D}= d+[A, \mbox{\LARGE $\cdot$}\}$		
  suggests that $\widehat{D}$ has non-vanishing right-part.
 Indeed, the identity $[A, \widehat{m}f\}= [A, \widehat{m}\}f$
  suggests that $\widehat{D}$ behaves like a simple hybrid connection as well.
}\end{remark}

\bigskip
   
\subsection{Simple hybrid connections on $\widehat{\cal E}$
    that ring with the vector representation of the $d=4$, $N=1$ supersymmetry}
	
Continuing the new setup at the beginning of Sec.\ 2.2.
The simple hybrid connection introduced in Sec.\ 2.2
  has too many components
  when compared to a ($\Endsheaf_{\widehat{\cal O}_X}(\widehat{\cal E})$-valued)
  vector multiplet of the $d=4$, $N=1$ supersymmetry algebra  and, hence,
 needs to be constrained appropriately.
In this subsection, we adopt the lesson learned from physicists
 to directly construct a simple hybrid connection on a bundle over the superspace from a vector multiplet.
This guarantees that it is SUSY-rep compatible.

\bigskip

\begin{flushleft}
{\bf SUSY-rep compatible simple hybrid connection on $\widehat{\cal E}$}
\end{flushleft}
%
%
%
%
%
%
The study in the textbook
  [G-G-R-S: Sec.4.2.a.3] of
  S.\ James Gates, Jr., Marcus Grisaru, Martin Ro\u{c}ek, and Warren Siegel
 motivate the following definitions.

\bigskip

\begin{definition} {\bf [endomorphism in Wess-Zumino gauge]}\; {\rm
 An $h\in\Endsheaf_{\widehat{\cal O}_X}\!(\widehat{\cal E})$
  is said to be {\it in Wess-Zumino gauge}
  if $h$ is of the following form in the coordinate-functions $(x,\theta,\bar{\theta})$ on $\widehat{X}$:
  {\small
  $$
    h(x,\theta,\bar{\theta})\;
	   =\; \sum_{\alpha,\dot{\beta}}
	            h_{(\alpha\dot{\beta})}(x)\theta^{\alpha}\bar{\theta}^{\dot{\beta}}\,
			 +\, \sum_{\alpha}
			        h_{(\alpha 12)}(x)
					  \theta^{\alpha}\bar{\theta}^{\dot{1}}\bar{\theta}^{\dot{2}}\,
	         +\, \sum_{\dot{\beta}}
			        h_{(12\dot{\beta})}(x)\theta^1\theta^2\bar{\theta}^{\dot{\beta}}\,
	         +\, \sum_{\alpha}
			        h_{(12\dot{1}\dot{2})}(x)
					  \theta^1\theta^2\bar{\theta}^{\dot{1}}\bar{\theta}^{\dot{2}}\,.
  $$}
}\end{definition}
 
\medskip

\begin{definition} {\bf [vector superfield]}\; {\rm
 Given $\widehat{\cal E}$,
 a {\it vector superfield} $V$ on $\widehat{X}$ is an element of
  $\Endsheaf_{\widehat{\cal O}_X}\!(\widehat{\cal E})$
  such that there exist
   a gauge transformation $g\in \Autsheaf_{\widehat{\cal O}_X}\!(\widehat{\cal E})$  and
   a $V^{\prime}\in \Endsheaf_{\widehat{\cal O}_X}\!(\widehat{\cal E})$ 		
     in Wess-Zumino gauge
   such that $g^{\circ} e^{V\circ } g^{-1\circ }= e^{V^{\prime}\circ}$
   (equivalently $g^{-1} e^V g= e^{V^{\prime}}$, cf.\ Convention~2.2.1).
 The sheaf of vector superfields on $\widehat{X}$ associated to $\widehat{\cal E}$	
   is denoted by ${\cal V}_{\widehat{X}\!,\widehat{\cal E}}$.
  (See footnote~32 for words on the {\it exponential} $e^V$ of $V$.)
}\end{definition}

\bigskip

Note that since
 $g_{(0)}^{-1}e^{V_{(0)}}g_{(0)}= e^{V^{\prime}_{(0)}}$
   if $g^{-1} e^V g= e^{V^{\prime}}$,
   a necessary condition for $h\in\Endsheaf_{\widehat{\cal O}_X}\!(\widehat{\cal E})$
   to be a vector superfield is that $h_{0}=0$.
Also note that $g^{-1}e^Vg= e^{\,g^{-1}V g}$; thus
 the condition $g^{-1} e^V g= e^{V^{\prime}}$ in Definition~2.3.1 above is equivalent to
 the condition  $g^{-1} V g= V^{\prime}$.
Finally, note that $(V^{\prime})^3=0$ for $V^{\prime}$ in Wess-Zumino gauge.
Thus for $V$ a vector superfield in the sense of Definition~2.3.2, $V^3= g (V^{\prime})^3g^{-1}=0$
 and, hence, $e^V= \Id+V+\frac{1}{2}V^2$.\footnote{{\it Exponential $e^V$}\hspace{1em}
                                                            Since
                                                             we do not introduce the notion of
															  {\it $C^{\infty}$-group-scheme} and
															  {\it super $C^{\infty}$-group-scheme} in this work,
															 rigorously speaking,
															 we haven't defined what the exponential $e^V$ of $V$ means
                                                             in the context of super $C^{\infty}$-Algebraic Geometry.
															 However, the only situation of $e^V$ that is relevant to the current work
  															  is for $V$ a vector super field in the sense of Definition~2.3.2.
                                                             Thus, as a hindsight, one may take
															  {\it Id}$\,+V+\frac{1}{2}V^2$ as the definition of
															  $e^V$ for $V$ a vector superfield.
															 From this, one also see that  the parity-conjugate
															  $\,\!^\varsigma\!e^V:= \,\!^\varsigma\!(e^V)= e\,\!^{^\varsigma\!V}$.
															 } 
 
\bigskip

\begin{definition} {\bf [SUSY-rep compatible simple hybrid connection on $\widehat{\cal E}$]}\;
{\rm
 Fix a trivialization of $\widehat{\cal E}$;
  the corresponding left connection on $\widehat{\cal E}$ is denoted by $d$.
 Let $\widehat{\nabla}$ be a simple hybrid connection on $\widehat{\cal E}$.
 $\widehat{\nabla}$ is said to be {\it SUSY-rep compatible}
  if $\widehat{\nabla}$ is of the following form up to a gauge transformation:
   \begin{itemize}
    \item[\LARGE $\cdot$]
	 There is a vector superfield
	 $V\in {\cal V}_{\widehat{X}\!,\widehat{\cal E}}
	     \subset \Endsheaf_{\widehat{\cal O}_X}\!(\widehat{\cal E})$
	 such that
     $$
      \widehat{\nabla}_{e_{\beta^{\prime\prime}}}\;
	    =\;  e_{\beta^{\prime\prime}}\,,\hspace{2em}	
      \widehat{\nabla}_{e_{\alpha^{\prime}}}\;
	    =\; 	 e^{V\circ}	\circ e_{\alpha^{\prime}} \circ e^{-V\circ}\,,\hspace{2em}
	  \widehat{\nabla}_{e_\mu}\;
	    =\; 	\mbox{\small $\frac{\sqrt{-1}}{2}$}\,
		         \sum_{\alpha, \dot{\beta}} \breve{\sigma}_{\mu}^{\alpha\dot{\beta}}
				  \cdot \,\!^L\![\widehat{\nabla}_{e_{\alpha^{\prime}}},
				                       \widehat{\nabla}_{e_{\beta^{\prime\prime}}}\}\,,
     $$
    for $\beta^{\prime\prime}=1^{\prime\prime}, 2^{\prime\prime}$,
         $\alpha^{\prime}=1^{\prime}, 2^{\prime}$, and $\mu=0,1,2,3$.
    Here,
	  $\breve{\sigma}_{\mu}
	     =(\breve{\sigma}_{\mu}^{\alpha\dot{\beta}})_{\alpha\dot{\beta}}$
	  with
	 {\footnotesize
     $$
      \breve{\sigma}_0\;:=\;
       \frac{1}{2}\left[\!\begin{array}{rr} -1 & 0 \\ 0 & -1\end{array}\!\right]\!,\;\;
	  \breve{\sigma}_1\;:=\;
       \frac{1}{2}\left[\!\begin{array}{rr} 0 & 1 \\ 1 & 0\end{array}\!\right]\!,\;\;
      \breve{\sigma}_2\;:=\;
       \frac{1}{2}
	    \left[\!\begin{array}{rr} 0 & \sqrt{-1} \\ -\sqrt{-1} & 0\end{array}\!\right]\!,\;\;	   
      \breve{\sigma}_3\;:=\;
       \frac{1}{2}\left[\!\begin{array}{rr} 1 & 0 \\ 0 & -1\end{array}\!\right]
     $$}and\;    
	 $\,\!^L\![\widehat{\nabla}_{e_{\alpha^{\prime}}},
				                       \widehat{\nabla}_{e_{\beta^{\prime\prime}}}\}
       = \{ \widehat{\nabla}_{e_{\alpha^{\prime}}},
				                       \widehat{\nabla}_{e_{\beta^{\prime\prime}}}\}$\;
    (cf.\ Lemma~2.2.9).
   \end{itemize}
 In particular, a SUSY-rep compatible simple hybrid connection on $\widehat{\cal E}$
   is determined by a vector superfield $V\in  {\cal V}_{\widehat{X}\!, \widehat{\cal E}}$,
   up to a gauge transformation.
}\end{definition}

\bigskip

More explicitly,
 let $\widehat{\nabla}s= ds+ sA$ with $A= \sum_Ie^IA_I$
   in terms of the supersymmetrically invariant coframe $(e^I)_I$ on $\widehat{X}$.
Then, straightforward computations give
  $$
  \begin{array}{lllll}
   \widehat{\nabla}_{e_{\beta^{\prime\prime}}}s
    & = &  e_{\beta^{\prime\prime}}s\,,                                        \\[1.2ex]
   \widehat{\nabla}_{e_{\alpha^{\prime}}}s
    & = &        \mbox{\Large $($} e_{\alpha^{\prime}}(se^{-V}) \mbox{\Large $)$}
                      e^V	
	& = &
  	   e_{\alpha^{\prime}}s\,
	         +\, (-1)^{p(s)}\, s\, (e_{\alpha^{\prime}}e^{-V})e^V\,,     \\[1.2ex]
   \widehat{\nabla}_{e_{\mu}}s
    & = & \frac{\sqrt{-1}}{2}\, \sum_{\alpha, \dot{\beta}}
               \breve{\sigma}_{\mu}^{\alpha\dot{\beta}}
			    \cdot \,\!^L\![\widehat{\nabla}_{e_{\alpha^{\prime}}},
				                      \widehat{\nabla}_{e_{\beta^{\prime\prime}}}  \}\,s	
	& = & e_{\mu}s\,
	          +\, s\cdot \mbox{\Large $($}
			     \frac{\sqrt{-1}}{2}\, \sum_{\alpha, \dot{\beta}}
                  \breve{\sigma}_{\mu}^{\alpha\dot{\beta}}
				  \Theta_{\alpha\dot{\beta}}       \mbox{\Large $)$}  \,,
  \end{array}
  $$
 where
  $$
   \Theta_{\alpha\dot{\beta}}\;
    =\; (e_{\beta^{\prime\prime}} e_{\alpha^{\prime}} e^{-V}  )\, e^V\,
	         -\,  (e_{\alpha^{\prime}} \,\!^\varsigma\!e^{-V})
			        \cdot (e_{\beta^{\prime\prime}} e^V)\;
	=\;  e_{\beta^{\prime\prime}}
	           \mbox{\Large $($}(e_{\alpha^\prime}e^{-V})\, e^V \mbox{\Large $)$}\,.
  $$
This justifies that $\widehat{\nabla}$ is indeed a simple hybrid connection on $\widehat{\cal E}$
 with the the connection $1$-form $A$, with respect to the underlying trivialization of $\widehat{\cal E}$,
 determined by $V$:
 $$
  \begin{array}{c}
   A\; =\; \sum_I e^I\cdot A_I\;
         =\; \sum_{\mu=0}^3
              e^{\mu}\cdot
			   \mbox{\Large $($}
			     \frac{\sqrt{-1}}{2}\, \sum_{\alpha, \dot{\beta}}
                  \breve{\sigma}_{\mu}^{\alpha\dot{\beta}}
				  \Theta_{\alpha\dot{\beta}}       \mbox{\Large $)$}\,
             +\, \sum_{\alpha^{\prime}=1^{\prime}}^{2^{\prime}}
			        e^{\alpha^{\prime}}
					 \cdot    (e_{\alpha^{\prime}}e^{-V})e^V\,.				
  \end{array}
 $$
For
 {\small
 $$
    V\ :=\;  \sum_{\gamma, \dot{\delta}}
	                V_{(\gamma\dot{\delta})}\theta^\gamma\bar{\theta}^{\dot{\delta}}\,
       	     +\, \sum_{\dot{\delta}}V_{(12\dot{\delta})}		
	                \theta^1\theta^2\bar{\theta}^{\dot{\gamma}}\,
			 +\, \sum_{\gamma} V_{(\gamma\dot{1}\dot{2})}
			        \theta^\gamma\bar{\theta}^{\dot{1}}\bar{\theta}^{\dot{2}}\,
			 +\, V_{(12\dot{1}\dot{2})}
			        \theta^1\theta^2\bar{\theta}^{\dot{1}}\bar{\theta}^{\dot{2}}\;\;
	 \in\; \widehat{\cal O}_X^{A\!z}\,,
 $$}this
 is given in terms of components of $V$ by
{\small
 \begin{eqnarray*}
  \Theta_{1\dot{1}}
    & = & V_{(1\dot{1})}\,
	          -\,  V_{(12\dot{1})} \theta^2\,
			  +\, V_{(1\dot{1}\dot{2})} \bar{\theta}^{\dot{2}}\,
			  +\, \sqrt{-1}\,\mbox{\normalsize $\sum$}_\mu \sigma^\mu_{1\dot{1}}
			         \partial_\mu V_{(1\dot{1})}\theta^1\bar{\theta}^{\dot{1}}\, \\
	&& +\, \mbox{\Large $($}
                   [V_{(1\dot{1})}, V_{(1\dot{2})}]
				-  \sqrt{-1}\, \mbox{\normalsize $\sum$}_\mu \sigma^\mu_{1\dot{2}}
				       \partial_\mu V_{(1\dot{1})}
				+2\,\sqrt{-1}\,\mbox{\normalsize $\sum$}_\mu \sigma^\mu_{1\dot{1}}
				       \partial_\mu V_{(1\dot{2})}
                \mbox{\Large $)$}\, \theta^1\bar{\theta}^{\dot{2}}   \\	
    && +\, \sqrt{-1}\,\mbox{\normalsize $\sum$}_\mu \sigma^\mu_{2\dot{1}}
	                   \partial_\mu V_{(1\dot{1})}  \theta^2\bar{\theta}^{\dot{1}}\, \\
    && +\, \mbox{\Large $($}
                  -  V_{(12\dot{1}\dot{2})}
				  + \frac{1}{2}\,[V_{(1\dot{1})}, V_{(2\dot{2})}]
				  + \frac{1}{2}\,[V_{(2\dot{1})}, V_{(1\dot{2})}]
				  + \sqrt{-1}\,\mbox{\normalsize $\sum$}_\mu \sigma^\mu_{2\dot{1}}
				        \partial_\mu V_{(1\dot{2})}    \\
       && \hspace{10em}						
                  -  \sqrt{-1}\,\mbox{\normalsize $\sum$}_\mu \sigma^\mu_{1\dot{2}}
				        \partial_\mu V_{(2\dot{1})}
                  + \sqrt{-1}\,\mbox{\normalsize $\sum$}_\mu \sigma^\mu_{1\dot{1}}
				        \partial_\mu V_{(2\dot{2})}						
                \mbox{\Large $)$}\,\theta^2\bar{\theta}^{\dot{2}}\,	\\
    && +\, \sqrt{-1}\,\mbox{\normalsize $\sum$}_\mu \sigma^\mu_{1\dot{1}}
	                  \partial_\mu V_{(12\dot{1})}\,\theta^1\theta^2\bar{\theta}^{\dot{1}}\,  \\
    && +\, \mbox{\Large $($}
                   [V_{(1\dot{1})}, V_{(12\dot{2})}]
				   -  [V_{(1\dot{2})}, V_{(12\dot{1})}]
				   - \sqrt{-1}\, \mbox{\normalsize $\sum$}_\mu \sigma^\mu_{1\dot{2}}
				              \partial_\mu V_{(12\dot{1})}
				   + 2\,\sqrt{-1}\,\mbox{\normalsize $\sum$}_\mu \sigma^\mu_{1\dot{1}}
				              \partial_\mu V_{(12\dot{2})}
                \mbox{\Large $)$}\,\theta^1\theta^2\bar{\theta}^{\dot{2}}	\,  \\
    && +\, \sqrt{-1}\,\mbox{\normalsize $\sum$}_\mu \sigma^\mu_{1\dot{1}}	
                              \partial_\mu V_{(1\dot{1}\dot{2})} 	
	               \theta^1\bar{\theta}^{\dot{1}}\bar{\theta}^{\dot{2}}\,
		   +\, \sqrt{-1}\, \mbox{\normalsize $\sum$}_\mu \sigma^\mu_{2\dot{1}}
		                      \partial_\mu V_{(1\dot{1}\dot{2})}
		            \theta^2\bar{\theta}^{\dot{1}}\bar{\theta}^{\dot{2}} \\
    && +\, \mbox{\Large $($}	
                  \sqrt{-1}\,\mbox{\normalsize $\sum$}_\mu \sigma^\mu_{1\dot{1}}
				        \partial_\mu V_{(12\dot{1}\dot{2})}
				 +\sqrt{-1}\,\mbox{\normalsize $\sum$}_\mu \sigma^\mu_{2\dot{1}}
				        \partial_\mu [V_{(1\dot{1})}, V_{(1\dot{2})}]
				 - \frac{1}{2}\,\sqrt{-1}\,\mbox{\normalsize $\sum$}_\mu \sigma^\mu_{1\dot{1}}
				        \partial_\mu [V_{(1\dot{1})}, V_{(2\dot{2})}]    \\
        && \hspace{2em}					
                 - \frac{1}{2}\,\sqrt{-1}\,\mbox{\normalsize $\sum$}_\mu \sigma^\mu_{1\dot{1}}
				        \partial_\mu [V_{(2\dot{1})}, V_{(1\dot{2})}]			
                 + \mbox{\normalsize $\sum$}_{\mu, \nu}
				      \sigma^\mu_{2\dot{1}}\sigma^\nu_{1\dot{2}}
					  \partial_\mu \partial_\nu V_{(1\dot{1})}
                 -  \mbox{\normalsize $\sum$}_{\mu, \nu}
				      \sigma^\mu_{2\dot{1}}\sigma^\nu_{1\dot{1}}
					  \partial_\mu \partial_\nu V_{(1\dot{2})}    \\
        && \hspace{10em}					
                 -  \mbox{\normalsize $\sum$}_{\mu, \nu}
				      \sigma^\mu_{1\dot{1}}\sigma^\nu_{1\dot{2}}
					  \partial_\mu \partial_\nu V_{(2\dot{1})}					
                 + \mbox{\normalsize $\sum$}_{\mu, \nu}
				      \sigma^\mu_{1\dot{1}}\sigma^\nu_{1\dot{1}}
					  \partial_\mu \partial_\nu V_{(2\dot{2})}
               	\mbox{\Large $)$}\,
				    \theta^1\theta^2\bar{\theta}^{\dot{1}}\bar{\theta}^{\dot{2}}\,, \\[2ex]
 \end{eqnarray*}
 \begin{eqnarray*}					
  \Theta_{1\dot{2}}
    & = & V_{(1\dot{2})}\,
	          -\,  V_{(12\dot{2})} \theta^2\,
			  -\,  V_{(1\dot{1}\dot{2})} \bar{\theta}^{\dot{1}}\,  \\
    && +\, \mbox{\Large $($}
                -  [V_{(1\dot{1})}, V_{(1\dot{2})}]
				+ 2\,\sqrt{-1}\, \mbox{\normalsize $\sum$}_\mu \sigma^\mu_{1\dot{2}}
				       \partial_\mu V_{(1\dot{1})}
			    -  \sqrt{-1}\,\mbox{\normalsize $\sum$}_\mu \sigma^\mu_{1\dot{1}}
				       \partial_\mu V_{(1\dot{2})}
                \mbox{\Large $)$}\, \theta^1\bar{\theta}^{\dot{1}}   \\
    && +\, \sqrt{-1}\,\mbox{\normalsize $\sum$}_\mu \sigma^\mu_{1\dot{2}}
	                   \partial_\mu V_{(1\dot{2})}  \theta^1\bar{\theta}^{\dot{2}}\, \\
    && +\, \mbox{\Large $($}
                     V_{(12\dot{1}\dot{2})}
				  - \frac{1}{2}\,[V_{(1\dot{1})}, V_{(2\dot{2})}]
				  - \frac{1}{2}\,[V_{(2\dot{1})}, V_{(1\dot{2})}]
				  + \sqrt{-1}\,\mbox{\normalsize $\sum$}_\mu \sigma^\mu_{2\dot{2}}
				        \partial_\mu V_{(1\dot{1})}    \\
       && \hspace{10em}						
                  +  \sqrt{-1}\,\mbox{\normalsize $\sum$}_\mu \sigma^\mu_{1\dot{2}}
				        \partial_\mu V_{(2\dot{1})}
                   -  \sqrt{-1}\,\mbox{\normalsize $\sum$}_\mu \sigma^\mu_{1\dot{1}}
				        \partial_\mu V_{(2\dot{2})}						
                \mbox{\Large $)$}\,\theta^2\bar{\theta}^{\dot{1}}\,	\\
    && +\, \sqrt{-1}\,\mbox{\normalsize $\sum$}_\mu \sigma^\mu_{2\dot{2}}
	                   \partial_\mu V_{(1\dot{2})}  \theta^2\bar{\theta}^{\dot{2}}\, \\
    && +\, \mbox{\Large $($}
                   -  [V_{(1\dot{1})}, V_{(12\dot{2})}]
				   + [V_{(1\dot{2})}, V_{(12\dot{1})}]
				   + 2\,\sqrt{-1}\, \mbox{\normalsize $\sum$}_\mu \sigma^\mu_{1\dot{2}}
				              \partial_\mu V_{(12\dot{1})}
				   -  \sqrt{-1}\,\mbox{\normalsize $\sum$}_\mu \sigma^\mu_{1\dot{1}}
				              \partial_\mu V_{(12\dot{2})}
                \mbox{\Large $)$}\,\theta^1\theta^2\bar{\theta}^{\dot{1}}	\,  \\
    && +\, \sqrt{-1}\,\mbox{\normalsize $\sum$}_\mu \sigma^\mu_{1\dot{2}}
	                  \partial_\mu V_{(12\dot{2})}\,\theta^1\theta^2\bar{\theta}^{\dot{2}}\,
           +\, \sqrt{-1}\,\mbox{\normalsize $\sum$}_\mu \sigma^\mu_{1\dot{2}}	
                              \partial_\mu V_{(1\dot{1}\dot{2})} 	
	               \theta^1\bar{\theta}^{\dot{1}}\bar{\theta}^{\dot{2}}\,
		   +\, \sqrt{-1}\, \mbox{\normalsize $\sum$}_\mu \sigma^\mu_{2\dot{2}}
		                      \partial_\mu V_{(1\dot{1}\dot{2})}
		            \theta^2\bar{\theta}^{\dot{1}}\bar{\theta}^{\dot{2}} \\ 		
    && +\, \mbox{\Large $($}	
                  \sqrt{-1}\,\mbox{\normalsize $\sum$}_\mu \sigma^\mu_{1\dot{2}}
				        \partial_\mu V_{(12\dot{1}\dot{2})}
				 +\sqrt{-1}\,\mbox{\normalsize $\sum$}_\mu \sigma^\mu_{2\dot{2}}
				        \partial_\mu [V_{(1\dot{1})}, V_{(1\dot{2})}]
				 - \frac{1}{2}\,\sqrt{-1}\,\mbox{\normalsize $\sum$}_\mu \sigma^\mu_{1\dot{2}}
				        \partial_\mu [V_{(1\dot{1})}, V_{(2\dot{2})}]    \\
        && \hspace{2em}					
                 - \frac{1}{2}\,\sqrt{-1}\,\mbox{\normalsize $\sum$}_\mu \sigma^\mu_{1\dot{2}}
				        \partial_\mu [V_{(2\dot{1})}, V_{(1\dot{2})}]			
                 + \mbox{\normalsize $\sum$}_{\mu, \nu}
				      \sigma^\mu_{2\dot{2}}\sigma^\nu_{1\dot{2}}
					  \partial_\mu \partial_\nu V_{(1\dot{1})}
                 -  \mbox{\normalsize $\sum$}_{\mu, \nu}
				      \sigma^\mu_{2\dot{2}}\sigma^\nu_{1\dot{1}}
					  \partial_\mu \partial_\nu V_{(1\dot{2})}    \\
        && \hspace{10em}					
                 -  \mbox{\normalsize $\sum$}_{\mu, \nu}
				      \sigma^\mu_{1\dot{2}}\sigma^\nu_{1\dot{2}}
					  \partial_\mu \partial_\nu V_{(2\dot{1})}					
                 + \mbox{\normalsize $\sum$}_{\mu, \nu}
				      \sigma^\mu_{1\dot{2}}\sigma^\nu_{1\dot{1}}
					  \partial_\mu \partial_\nu V_{(2\dot{2})}
               	\mbox{\Large $)$}\,
				    \theta^1\theta^2\bar{\theta}^{\dot{1}}\bar{\theta}^{\dot{2}}\,, \\[2ex]
 \end{eqnarray*}
 \begin{eqnarray*}
  \Theta_{2\dot{1}}
    & = & V_{(2\dot{1})}\,
	          +\,  V_{(12\dot{1})} \theta^1\,
			  +\, V_{(2\dot{1}\dot{2})} \bar{\theta}^{\dot{2}}\,
			  +\, \sqrt{-1}\,\mbox{\normalsize $\sum$}_\mu \sigma^\mu_{1\dot{1}}
			         \partial_\mu V_{(2\dot{1})}\theta^1\bar{\theta}^{\dot{1}}\, \\
    && +\, \mbox{\Large $($}
                   V_{(12\dot{1}\dot{2})}
				  + \frac{1}{2}\,[V_{(1\dot{1})}, V_{(2\dot{2})}]
				  + \frac{1}{2}\,[V_{(2\dot{1})}, V_{(1\dot{2})}]
				  -  \sqrt{-1}\,\mbox{\normalsize $\sum$}_\mu \sigma^\mu_{2\dot{2}}
				        \partial_\mu V_{(1\dot{1})}    \\
       && \hspace{10em}						
                  + \sqrt{-1}\,\mbox{\normalsize $\sum$}_\mu \sigma^\mu_{2\dot{1}}
				        \partial_\mu V_{(1\dot{2})}
                  + \sqrt{-1}\,\mbox{\normalsize $\sum$}_\mu \sigma^\mu_{1\dot{1}}
				        \partial_\mu V_{(2\dot{2})}						
                \mbox{\Large $)$}\,\theta^1\bar{\theta}^{\dot{2}}\,	\\
    && +\, \sqrt{-1}\,\mbox{\normalsize $\sum$}_\mu \sigma^\mu_{2\dot{1}}
			         \partial_\mu V_{(2\dot{1})}\theta^2\bar{\theta}^{\dot{1}}\, \\
	&& +\, \mbox{\Large $($}
                   [V_{(2\dot{1})}, V_{(2\dot{2})}]
				-  \sqrt{-1}\, \mbox{\normalsize $\sum$}_\mu \sigma^\mu_{2\dot{2}}
				       \partial_\mu V_{(2\dot{1})}
				+2\,\sqrt{-1}\,\mbox{\normalsize $\sum$}_\mu \sigma^\mu_{2\dot{1}}
				       \partial_\mu V_{(2\dot{2})}
                \mbox{\Large $)$}\, \theta^2\bar{\theta}^{\dot{2}}   \\
    && +\, \sqrt{-1}\,\mbox{\normalsize $\sum$}_\mu \sigma^\mu_{2\dot{1}}
	                  \partial_\mu V_{(12\dot{1})}\,\theta^1\theta^2\bar{\theta}^{\dot{1}}\,  \\
    && +\, \mbox{\Large $($}
                   [V_{(2\dot{1})}, V_{(12\dot{2})}]
				   -  [V_{(2\dot{2})}, V_{(12\dot{1})}]
				   - \sqrt{-1}\, \mbox{\normalsize $\sum$}_\mu \sigma^\mu_{2\dot{2}}
				              \partial_\mu V_{(12\dot{1})}
				   + 2\,\sqrt{-1}\,\mbox{\normalsize $\sum$}_\mu \sigma^\mu_{2\dot{1}}
				              \partial_\mu V_{(12\dot{2})}
                \mbox{\Large $)$}\,\theta^1\theta^2\bar{\theta}^{\dot{2}}	\,  \\
    && +\, \sqrt{-1}\,\mbox{\normalsize $\sum$}_\mu \sigma^\mu_{1\dot{1}}	
                              \partial_\mu V_{(2\dot{1}\dot{2})} 	
	               \theta^1\bar{\theta}^{\dot{1}}\bar{\theta}^{\dot{2}}\,
		   +\, \sqrt{-1}\, \mbox{\normalsize $\sum$}_\mu \sigma^\mu_{2\dot{1}}
		                      \partial_\mu V_{(2\dot{1}\dot{2})}
		            \theta^2\bar{\theta}^{\dot{1}}\bar{\theta}^{\dot{2}} \\
   && +\, \mbox{\Large $($}	
                  \sqrt{-1}\,\mbox{\normalsize $\sum$}_\mu \sigma^\mu_{2\dot{1}}
				        \partial_\mu V_{(12\dot{1}\dot{2})}
				 - \sqrt{-1}\,\mbox{\normalsize $\sum$}_\mu \sigma^\mu_{1\dot{1}}
				        \partial_\mu [V_{(2\dot{1})}, V_{(2\dot{2})}]
				 +\frac{1}{2}\,\sqrt{-1}\,\mbox{\normalsize $\sum$}_\mu \sigma^\mu_{2\dot{1}}
				        \partial_\mu [V_{(1\dot{1})}, V_{(2\dot{2})}]    \\
        && \hspace{2em}					
                 +\frac{1}{2}\,\sqrt{-1}\,\mbox{\normalsize $\sum$}_\mu \sigma^\mu_{2\dot{1}}
				        \partial_\mu [V_{(2\dot{1})}, V_{(1\dot{2})}]			
                 + \mbox{\normalsize $\sum$}_{\mu, \nu}
				      \sigma^\mu_{2\dot{1}}\sigma^\nu_{2\dot{2}}
					  \partial_\mu \partial_\nu V_{(1\dot{1})}
                 -  \mbox{\normalsize $\sum$}_{\mu, \nu}
				      \sigma^\mu_{2\dot{1}}\sigma^\nu_{2\dot{1}}
					  \partial_\mu \partial_\nu V_{(1\dot{2})}    \\
        && \hspace{10em}					
                 -  \mbox{\normalsize $\sum$}_{\mu, \nu}
				      \sigma^\mu_{1\dot{1}}\sigma^\nu_{2\dot{2}}
					  \partial_\mu \partial_\nu V_{(2\dot{1})}					
                 + \mbox{\normalsize $\sum$}_{\mu, \nu}
				      \sigma^\mu_{1\dot{1}}\sigma^\nu_{2\dot{1}}
					  \partial_\mu \partial_\nu V_{(2\dot{2})}
               	\mbox{\Large $)$}\,
				    \theta^1\theta^2\bar{\theta}^{\dot{1}}\bar{\theta}^{\dot{2}}\,,\\[2ex]
 \end{eqnarray*}
 \begin{eqnarray*}					
  \Theta_{2\dot{2}}
    & = & V_{(2\dot{2})}\,
	          +\,  V_{(12\dot{2})} \theta^1\,
			  -\, V_{(2\dot{1}\dot{2})} \bar{\theta}^{\dot{1}}\,\\
    && +\, \mbox{\Large $($}
                  -  V_{(12\dot{1}\dot{2})}
				  -  \frac{1}{2}\,[V_{(1\dot{1})}, V_{(2\dot{2})}]
				  -  \frac{1}{2}\,[V_{(2\dot{1})}, V_{(1\dot{2})}]
				  + \sqrt{-1}\,\mbox{\normalsize $\sum$}_\mu \sigma^\mu_{2\dot{2}}
				        \partial_\mu V_{(1\dot{1})}    \\
       && \hspace{10em}						
                  -  \sqrt{-1}\,\mbox{\normalsize $\sum$}_\mu \sigma^\mu_{2\dot{1}}
				        \partial_\mu V_{(1\dot{2})}
                  + \sqrt{-1}\,\mbox{\normalsize $\sum$}_\mu \sigma^\mu_{1\dot{2}}
				        \partial_\mu V_{(2\dot{1})}						
                \mbox{\Large $)$}\,\theta^1\bar{\theta}^{\dot{1}}\,	\\			
    && +\, \sqrt{-1}\,\mbox{\normalsize $\sum$}_\mu \sigma^\mu_{1\dot{2}}
	                   \partial_\mu V_{(2\dot{2})}  \theta^1\bar{\theta}^{\dot{2}}\, \\
	&& +\, \mbox{\Large $($}
                -  [V_{(2\dot{1})}, V_{(2\dot{2})}]
				+2\,\sqrt{-1}\, \mbox{\normalsize $\sum$}_\mu \sigma^\mu_{2\dot{2}}
				       \partial_\mu V_{(2\dot{1})}
				- \sqrt{-1}\,\mbox{\normalsize $\sum$}_\mu \sigma^\mu_{2\dot{1}}
				       \partial_\mu V_{(2\dot{2})}
                \mbox{\Large $)$}\, \theta^2\bar{\theta}^{\dot{1}}   \\
    && +\, \sqrt{-1}\,\mbox{\normalsize $\sum$}_\mu \sigma^\mu_{2\dot{2}}
	                   \partial_\mu V_{(2\dot{2})}  \theta^2\bar{\theta}^{\dot{2}}\, \\
    && +\, \mbox{\Large $($}
                   -  [V_{(2\dot{1})}, V_{(12\dot{2})}]
				   + [V_{(2\dot{2})}, V_{(12\dot{1})}]
				   + 2\,\sqrt{-1}\, \mbox{\normalsize $\sum$}_\mu \sigma^\mu_{2\dot{2}}
				              \partial_\mu V_{(12\dot{1})}
				   -  \sqrt{-1}\,\mbox{\normalsize $\sum$}_\mu \sigma^\mu_{2\dot{1}}
				              \partial_\mu V_{(12\dot{2})}
                \mbox{\Large $)$}\,\theta^1\theta^2\bar{\theta}^{\dot{1}}	\,  \\
    && +\, \sqrt{-1}\,\mbox{\normalsize $\sum$}_\mu \sigma^\mu_{2\dot{2}}
	                  \partial_\mu V_{(12\dot{2})}\,\theta^1\theta^2\bar{\theta}^{\dot{2}}\,
           +\, \sqrt{-1}\,\mbox{\normalsize $\sum$}_\mu \sigma^\mu_{1\dot{2}}	
                              \partial_\mu V_{(2\dot{1}\dot{2})} 	
	               \theta^1\bar{\theta}^{\dot{1}}\bar{\theta}^{\dot{2}}\,
		   +\, \sqrt{-1}\, \mbox{\normalsize $\sum$}_\mu \sigma^\mu_{2\dot{2}}
		                      \partial_\mu V_{(2\dot{1}\dot{2})}
		            \theta^2\bar{\theta}^{\dot{1}}\bar{\theta}^{\dot{2}} \\
    && +\, \mbox{\Large $($}	
                  \sqrt{-1}\,\mbox{\normalsize $\sum$}_\mu \sigma^\mu_{2\dot{2}}
				        \partial_\mu V_{(12\dot{1}\dot{2})}
				 - \sqrt{-1}\,\mbox{\normalsize $\sum$}_\mu \sigma^\mu_{1\dot{2}}
				        \partial_\mu [V_{(2\dot{1})}, V_{(2\dot{2})}]
				 +\frac{1}{2}\,\sqrt{-1}\,\mbox{\normalsize $\sum$}_\mu \sigma^\mu_{2\dot{2}}
				        \partial_\mu [V_{(1\dot{1})}, V_{(2\dot{2})}]    \\
        && \hspace{2em}					
                 +\frac{1}{2}\,\sqrt{-1}\,\mbox{\normalsize $\sum$}_\mu \sigma^\mu_{2\dot{2}}
				        \partial_\mu [V_{(2\dot{1})}, V_{(1\dot{2})}]			
                 + \mbox{\normalsize $\sum$}_{\mu, \nu}
				      \sigma^\mu_{2\dot{2}}\sigma^\nu_{2\dot{2}}
					  \partial_\mu \partial_\nu V_{(1\dot{1})}
                 -  \mbox{\normalsize $\sum$}_{\mu, \nu}
				      \sigma^\mu_{2\dot{2}}\sigma^\nu_{2\dot{1}}
					  \partial_\mu \partial_\nu V_{(1\dot{2})}    \\
        && \hspace{10em}					
                 -  \mbox{\normalsize $\sum$}_{\mu, \nu}
				      \sigma^\mu_{1\dot{2}}\sigma^\nu_{2\dot{2}}
					  \partial_\mu \partial_\nu V_{(2\dot{1})}					
                 + \mbox{\normalsize $\sum$}_{\mu, \nu}
				      \sigma^\mu_{1\dot{2}}\sigma^\nu_{2\dot{1}}
					  \partial_\mu \partial_\nu V_{(2\dot{2})}
               	\mbox{\Large $)$}\,
				    \theta^1\theta^2\bar{\theta}^{\dot{1}}\bar{\theta}^{\dot{2}}\,;  \\[2ex]	
 \end{eqnarray*}
 \begin{eqnarray*}					
  A_0 & = & -\,\sqrt{-1}\, (\Theta_{1\dot{1}}+\Theta_{2\dot{2}})/4\,,\\[1ex]
  A_1 & = &  \sqrt{-1}\,(\Theta_{1\dot{2}}+\Theta_{2\dot{1}})/4\,, \\[1ex]
  A_2 & = & -\,(\Theta_{1\dot{2}}-\Theta_{2\dot{1}})/4\,, \\[1ex]
  A_3 & = &  \sqrt{-1}\,(\Theta_{1\dot{1}}-\Theta_{2\dot{2}})/4\,;  \\[2ex]
 \end{eqnarray*}
 \begin{eqnarray*}
  A_{1^\prime}	
    & = & -\, V_{(1\dot{1})}\bar{\theta}^{\dot{1}}\,
               -\, V_{(1\dot{2})} \bar{\theta}^{\dot{2}}\,
               -\, V_{(12\dot{1})} \theta^2\bar{\theta}^{\dot{1}}\,
               -\, V_{(12\dot{2})} \theta^2\bar{\theta}^{\dot{2}}\,
               -\, V_{(1\dot{1}\dot{2})}\bar{\theta}^{\dot{1}} \bar{\theta}^{\dot{2}}\, \\
	&&   +\, \mbox{\Large $($}
	                [V_{(1\dot{1})}, V_{(1\dot{2})}]
					+\sqrt{-1}\,\mbox{\normalsize $\sum$}_\mu \sigma^\mu_{1\dot{1}}
					      \partial_\mu V_{(1\dot{2})}
					-\sqrt{-1}\,\mbox{\normalsize $\sum$}_\mu \sigma^\mu_{1\dot{2}}
					      \partial_\mu V_{(1\dot{1})}					
	              \mbox{\Large $)$}\, \theta^1\bar{\theta}^{\dot{1}}\bar{\theta}^{\dot{2}}\, \\
    && +\,   \mbox{\Large $($}
	               - V_{(12\dot{1}\dot{2})}
                    + \frac{1}{2}\,[V_{(1\dot{1})}, V_{(2\dot{2})}]
				    + \frac{1}{2}\,[V_{(2\dot{1})}, V_{(1\dot{2})}]	   \\
	  && \hspace{10em}
					+ \sqrt{-1}\,\mbox{\normalsize $\sum$}_\mu \sigma^\mu_{1\dot{1}}
					       \partial_\mu V_{(2\dot{2})}
                    -  \sqrt{-1}\,\mbox{\normalsize $\sum$}_\mu \sigma^\mu_{1\dot{2}}
					       \partial_\mu V_{(2\dot{1})}										
                  \mbox{\Large $)$}\, \theta^2\bar{\theta}^{\dot{1}}\bar{\theta}^{\dot{2}}\,	\\
    && +\, \mbox{\Large $($}
                  -  [V_{(1\dot{1})}, V_{(12\dot{2})}]
				  + [V_{(1\dot{2})}, V_{(12\dot{1})}] \\
	  && \hspace{10em}
				  + \sqrt{-1}\,\mbox{\normalsize $\sum$}_\mu \sigma^\mu_{1\dot{2}}
				       \partial_\mu V_{(12\dot{1})}
				  -  \sqrt{-1}\, \mbox{\normalsize $\sum$}_\mu \sigma^\mu_{1\dot{1}}
				       \partial_\mu V_{(12\dot{2})}
                \mbox{\Large $)$}\,
				   \theta^1\theta^2\bar{\theta}^{\dot{1}}\bar{\theta}^{\dot{2}}\,,   \\[2ex]	
  \end{eqnarray*}
  \begin{eqnarray*}
  A_{2^\prime}
   & = & -\, V_{(2\dot{1})}\bar{\theta}^{\dot{1}}\,
               -\, V_{(2\dot{2})} \bar{\theta}^{\dot{2}}\,
               +\, V_{(12\dot{1})} \theta^1\bar{\theta}^{\dot{1}}\,
               +\, V_{(12\dot{2})} \theta^1\bar{\theta}^{\dot{2}}\,
               -\, V_{(2\dot{1}\dot{2})}\bar{\theta}^{\dot{1}} \bar{\theta}^{\dot{2}}\, \\
   && +\,   \mbox{\Large $($}
	                V_{(12\dot{1}\dot{2})}
                    + \frac{1}{2}\,[V_{(1\dot{1})}, V_{(2\dot{2})}]
				    + \frac{1}{2}\,[V_{(2\dot{1})}, V_{(1\dot{2})}]	   \\
	  && \hspace{10em}
					+ \sqrt{-1}\,\mbox{\normalsize $\sum$}_\mu \sigma^\mu_{2\dot{1}}
					       \partial_\mu V_{(1\dot{2})}
                    -  \sqrt{-1}\,\mbox{\normalsize $\sum$}_\mu \sigma^\mu_{2\dot{2}}
					       \partial_\mu V_{(1\dot{1})}										
                  \mbox{\Large $)$}\, \theta^1\bar{\theta}^{\dot{1}}\bar{\theta}^{\dot{2}}\,	\\
   &&   +\, \mbox{\Large $($}
	                [V_{(2\dot{1})}, V_{(2\dot{2})}]
					+\sqrt{-1}\,\mbox{\normalsize $\sum$}_\mu \sigma^\mu_{2\dot{1}}
					      \partial_\mu V_{(2\dot{2})}
					-\sqrt{-1}\,\mbox{\normalsize $\sum$}_\mu \sigma^\mu_{2\dot{2}}
					      \partial_\mu V_{(2\dot{1})}					
	              \mbox{\Large $)$}\, \theta^2\bar{\theta}^{\dot{1}}\bar{\theta}^{\dot{2}}\, \\	
   && +\, \mbox{\Large $($}
                  -  [V_{(2\dot{1})}, V_{(12\dot{2})}]
				  + [V_{(2\dot{2})}, V_{(12\dot{1})}] \\
	  && \hspace{10em}
				  + \sqrt{-1}\,\mbox{\normalsize $\sum$}_\mu \sigma^\mu_{2\dot{2}}
				       \partial_\mu V_{(12\dot{1})}
				  -  \sqrt{-1}\, \mbox{\normalsize $\sum$}_\mu \sigma^\mu_{2\dot{1}}
				       \partial_\mu V_{(12\dot{2})}
                \mbox{\Large $)$}\,
				   \theta^1\theta^2\bar{\theta}^{\dot{1}}\bar{\theta}^{\dot{2}}\,,   \\[1ex]
 \end{eqnarray*}
 \begin{eqnarray*}				
  A_{1^{\prime\prime}}   & = & A_{2^{\prime\prime}}\;\; =\;\; 0\,.	
 \end{eqnarray*}}In  
terms of the coframe
  $(dx^\mu, d\theta^\alpha, d\bar{\theta}^{\dot{\beta}})_{\mu, \alpha, \dot{\beta}}$
  on $\widehat{X}$,
let
 $$
  A\;=\; \mbox{$\sum$}_{\mu=0}^3 dx^\mu \cdot a_\mu\,
             +\, \mbox{$\sum$}_{\alpha=1}^2 d\theta^\alpha \cdot b_\alpha\,
			 +\, \mbox{$\sum$}_{\dot{\beta}=\dot{1}}^{\dot{2}}
			        d\bar{\theta}^{\dot{\beta}} \cdot b_{\dot{\beta}}\,,
 $$
then, by Definition~1.4.8,
  %
{\small
 \begin{eqnarray*}
  a_{\mu} & = & A_\mu\,, \\[1ex]
  b_{\alpha}	
    & = & A_{\alpha^\prime} \,
	           -\,\sqrt{-1}\,\mbox{\normalsize $\sum$}_{\mu,\dot{\beta}}\,
                             \sigma^\mu_{\alpha\dot{\beta}}\bar{\theta}^{\dot{\beta}}A_\mu\,, \\[1ex]
  b_{\dot{\beta}}	
     & = & -\,\sqrt{-1}\,\mbox{\normalsize $\sum$}_{\mu,\alpha}
                             \sigma^\mu_{\alpha\dot{\beta}}\theta^\alpha A_\mu\,. \\
 \end{eqnarray*}}
 
Recall Lemma/Definition~2.2.8 and Lemma~2.2.9.
Then, the design of a SUSY-rep compatible simple hybrid connection on $\widehat{\cal E}$
 renders the following statement immediate:

\bigskip

\begin{corollary} {\bf [flat directions of SUSY-rep compatible simple hybrid connection]}\;
 Let $\widehat{\nabla}$ be a SUSY-rep compatible simple hybrid connection on $\widehat{\cal E}$.
 Then,
   with respect to the supersymmetrically invariant coframe $(e^I)_I$ on $\widehat{X}$,
  the components
    of the curvature tensor $F^{\widehat{\nabla}}$ of $\widehat{\nabla}$
  in purely fermionic directions all vanish:
 $$
   F_{\alpha^{\prime}\beta^{\prime}}\;
   =\; F_{\alpha^{\prime\prime}\beta^{\prime\prime}}\;
   =\; F_{\alpha^{\prime}\beta^{\prime\prime}}\;=\; 0\,.
 $$
\end{corollary}

\medskip

\begin{proof}
 $F_{\alpha^{\prime\prime}\beta^{\prime\prime}}
   = \{e_{\alpha^{\prime\prime}}, e_{\beta^{\prime\prime}}\}=0$.
 $F_{\alpha^\prime \beta^\prime}
   =    e^{-V \circ}\circ
           \{e_{\alpha^{\prime\prime}}, e_{\beta^{\prime\prime}}\}
		   \circ e^{V\circ }=0$.
 And\\
   $F_{\alpha^{\prime}\beta^{\prime\prime}}
      = \{\widehat{\nabla}_{\!e_{\alpha}^{\prime}},
		          \widehat{\nabla}_{\!e_{\beta^{\prime\prime}}}\}
	        - \widehat{\nabla}_{ \{e_{\alpha^{\prime}}, e_{\beta^{\prime\prime}}\}}
	  =0$
   by tautology
  since
   $\widehat{\nabla}_{ \{e_{\alpha^{\prime}}, e_{\beta^{\prime\prime}}\}}
      = -2\sqrt{-1}\,\sum_{\mu}\sigma^{\mu}_{\alpha\dot{\beta}}\,
	          \widehat{\nabla}_{\!e_{\mu}}$
    and the design of $\widehat{\nabla}_{\!e_{\mu}}$
	 as a ${\Bbb C}$-combination of
	 $\widehat{\nabla}_{ \{e_{\alpha^{\prime}}, e_{\beta^{\prime\prime}}\}}$'s	
	 means to make $F_{\alpha^{\prime}\beta^{\prime\prime}}$ vanish.
 
\end{proof}

\bigskip

\begin{remark} $[${\it parity-conjugate $\,\!^{\varsigma}\!\widehat{\nabla}$ of
     SUSY-rep compatible connection $\widehat{\nabla}$}$]$\; {\rm
 Recall that for a simple hybrid connection $\widehat{\nabla}$ on $\widehat{\cal E}$
  with $\widehat{\nabla}s= ds+ sA$,
 the {\it parity-conjugate} $\,\!^\varsigma\!\widehat{\nabla}$ of $\widehat{\nabla}$
  is given by $\,\!^\varsigma\!\widehat{\nabla}s= ds + \,\!^\varsigma\!\!A  s$.
 For a SUSY-rep compatible simple hybrid connection $\widehat{\nabla}$ on $\widehat{\cal E}$,
   say, associated to a vector superfield $V$,
  $\,\!^\varsigma\!\widehat{\nabla}$ is given by
  the SUSY-rep compatible simple hybrid connection $\widehat{\nabla}$ on $\widehat{\cal E}$
   associated to the parity conjugate $\,\!^\varsigma\!V$ of $V$.
 Let $A=\sum_Ie^IA_I$ be the connection $1$-form associated to $\widehat{\nabla}$.
 Then the connection $1$-from associated to $\,\!^\varsigma\!\widehat{\nabla}$ is given by
  $\,\!^\varsigma\!\!A = \sum_I(-1)^{p(e_I)}\,e^I \,\,\!^\varsigma\!(\!A_I\!)$.
}\end{remark}

\bigskip

\begin{flushleft}
{\bf Induced SUSY-rep compatible hybrid connection on
         $\Endsheaf_{\widehat{\cal O}_X}\!(\widehat{\cal E})$}
\end{flushleft}
Continuing the notations.
Let $\widehat{\nabla}$ be the SUSY-rep compatible simple hybrid connection on $\widehat{\cal E}$
  associated to a vector superfield $V$ on $\widehat{X}$,
 with $\widehat{\nabla}s=ds+sA$ for $s\in \widehat{\cal E}$.
Then, by Lemma/Definition~2.2.12,
 $\widehat{D}\widehat{m}\; :=\; d\widehat{m}- [A, \widehat{m}\}$,
  for $\widehat{m}\in\Endsheaf_{\widehat{\cal O}_X}\!(\widehat{\cal E})$,
 defines a hybrid connection on $\Endsheaf_{\widehat{\cal O}_X}\!(\widehat{\cal E})$
 that restricts to $d$ on $\widehat{\cal O}_X$.
  
\bigskip

\begin{lemma}
{\bf [$\widehat{D}_{e_{\beta^{\prime\prime}}}$ and
          $\widehat{D}_{e_{\alpha^\prime}}$]}\;
 $\widehat{D}_{e_{\beta^{\prime\prime}}}$ and
  $\widehat{D}_{e_{\alpha^\prime}}$
  have the following alternative expressions
 $$
   \widehat{D}_{e_{\beta^{\prime\prime}}}\widehat{m}\;
       =\;  e_{\beta^{\prime\prime}}\widehat{m}
     \hspace{2em}\mbox{and}\hspace{2em}
   \widehat{D}_{e_{\alpha^{\prime}}} \widehat{m}\;
       =\;   \,\!^{\varsigma}\!e^{-V} \mbox{\Large $($}
	      e_{\alpha^{\prime}}
		               (e^V \widehat{m}\, \,\!^{\varsigma_{\widehat{m}}}\!e^{-V})
	                                                    \mbox{\Large $)$}\,
							\,\!^{\varsigma_{\widehat{m}}}\!e^V\,,
 $$
 for $\beta^{\prime\prime}=1^{\prime\prime}, 2^{\prime\prime}$,
	  $\alpha^{\prime}= 1^{\prime}, 2^{\prime}$  and
	  $\widehat{m}$ parity homogeneous in the second equality.
 Here,
   $\,\!^{\varsigma}\!e^{-V}$
     (resp.\  $\,\!^{\varsigma_{\widehat{m}}}\!e^{-V}$,
	               $\,\!^{\varsigma_{\widehat{m}}}\!e^V$)
    is the shorthand for
	$\,\!^{\varsigma}\!(e^{-V})$
     (resp.\  $\,\!^{\varsigma_{\widehat{m}}}\!(e^{-V})$,
	               $\,\!^{\varsigma_{\widehat{m}}}\!(e^V)$).
\end{lemma}

\medskip

\begin{proof}
 Recall that $\widehat{D}_{e_I}\widehat{m}= e_I\widehat{m}-[A_I, \widehat{m}\}$.
 Thus,
 $\widehat{D}_{e_{\beta^{\prime\prime}}}\widehat{m}
    = e_{\beta^{\prime\prime}}\widehat{m}
	    - [A_{\beta^{\prime\prime}}, \widehat{m}\}
    =  e_{\beta^{\prime\prime}}\widehat{m}$ since $A_{\beta^{\prime\prime}}=0$,
 for $\beta^{\prime\prime}=1^{\prime\prime}, 2^{\prime\prime}$   and,
 for $\widehat{m}$ parity homogeneous,
 \begin{eqnarray*}
  \widehat{D}_{e_{\alpha^\prime}}\widehat{m}
    & = &  e_{\alpha^{\prime}} \widehat{m}\,-\, [A_{\alpha^\prime}, \widehat{m}\}\;\;	
	   =\;\; e_I\widehat{m} \,
	           -\, A_{\alpha^\prime}\widehat{m}\,
	           +\, \widehat{m}\,\,\!^{\varsigma_{\widehat{m}}}\!\!A_{\alpha^\prime}\\
    & = & e_{\alpha^\prime}\widehat{m}\,
	          -\, (e_{\alpha^\prime}e^{-V})e^V  \cdot\widehat{m}\,
              +\,  \widehat{m}\cdot
			         \,\!^{\varsigma_{\widehat{m}}}\!((e_{\alpha^\prime}e^{-V}) e^V)\\
    & = &  \,\!^{\varsigma}\!e^{-V} \mbox{\Large $($}
	               e_{\alpha^{\prime}}
		               (e^V \widehat{m}\, \,\!^{\varsigma_{\widehat{m}}}\!e^{-V})
	                                                    \mbox{\Large $)$}\,
							\,\!^{\varsigma_{\widehat{m}}}\!e^V\,,
 \end{eqnarray*}
 for $\alpha^{\prime}= 1^{\prime}, 2^{\prime}$.
 Here, the following identities are used:
  $$
   \begin{array}{c}
    (e_{\alpha^\prime}e^{-V}) e^V\;
     =\;  -\, \,\!^\varsigma\!e^{-V}\,(e_{\alpha^\prime}e^V)\,,\hspace{2em}
     \,\!^{\varsigma_{\widehat{m}}}\!((e_{\alpha^\prime}e^{-V}) e^V)\;
      =\;   \,\!^{\varsigma_{\widehat{m}}}\!(e_{\alpha^\prime}e^{-V})\,
	          \,\!^{\varsigma_{\widehat{m}}}\!e^V\,,  \\[1.2ex]
   \,\!^{\varsigma_{\widehat{m}}}\!(e_{\alpha^\prime}e^{-V})\;
      =\;   (-1)^{p(\widehat{m})}\,  \,\!^{\varsigma_{\widehat{m}}}\!e^{-V}\,,
	                                                                                  \\[1.2ex]
	e_{\alpha^{\prime}}
		               (e^V \widehat{m}\, \,\!^{\varsigma_{\widehat{m}}}\!e^{-V})\;
	=\;  (e_{\alpha^{\prime}}e^V)
	       \cdot \widehat{m}\cdot \,\!^{\varsigma_{\widehat{m}}}\!e^{-V}\,
		 +\,  \,\!^\varsigma\!e^V
		            \cdot  (e_{\alpha^\prime}\widehat{m})
			      \cdot \,\!^{\varsigma_{\widehat{m}}}\!e^{-V}\,
	     +\, \,\!^\varsigma\! e^V
		          \cdot (-1)^{p(\widehat{m})}\,\widehat{m}
                  \cdot	(e_{\alpha^\prime}\,\!^{\varsigma_{\widehat{m}}}\!e^{-V})\,.
   \end{array}
   $$
  
\end{proof}
  
\bigskip

A similar statement for $\,\!^\varsigma\!\widehat{D}$ also holds:

\bigskip

\begin{lemma}
{\bf [$\,\!^\varsigma\!\widehat{D}_{e_{\beta^{\prime\prime}}}$ and
          $\,\!^\varsigma\!\widehat{D}_{e_{\alpha^\prime}}$]}\;
 $\,\!^\varsigma\!\widehat{D}_{e_{\beta^{\prime\prime}}}$ and
  $\,\!^\varsigma\!\widehat{D}_{e_{\alpha^\prime}}$
  have the following alternative expressions
 $$
   \,\!^\varsigma\!\widehat{D}_{e_{\beta^{\prime\prime}}}\widehat{m}\;
       =\;  e_{\beta^{\prime\prime}}\widehat{m}
     \hspace{2em}\mbox{and}\hspace{2em}
   \,\!^\varsigma\!\widehat{D}_{e_{\alpha^{\prime}}} \widehat{m}\;
       =\;   e^{-V} \mbox{\Large $($}
	          e_{\alpha^{\prime}}
		               (\,\!^\varsigma\!e^V \widehat{m}\,
					      \,\!^{\varsigma_{\widehat{m}}}\!\,\!^\varsigma\!e^{-V})
	                                                    \mbox{\Large $)$}\,
							\,\!^{\varsigma_{\widehat{m}}}\!\,\!^\varsigma\!e^V\,,
 $$
 for $\beta^{\prime\prime}=1^{\prime\prime}, 2^{\prime\prime}$,
	  $\alpha^{\prime}= 1^{\prime}, 2^{\prime}$  and
	  $\widehat{m}$ parity homogeneous in the second equality.
 Here,\\
   $\;\,\!^\varsigma\!\widehat{D}_{e_{\beta^{\prime\prime}}}
     :=  (\,\!^\varsigma\!\widehat{D})_{e_{\beta^{\prime\prime}}}\,$,
   $\;\,\!^\varsigma\!\widehat{D}_{e_{\alpha^\prime}}
     := (\,\!^\varsigma\!\widehat{D})_{e_{\alpha^\prime}}\,$,   and
   $\;\,\!^{\varsigma}\!e^V\;$
     (resp.\  $\;\,\!^{\varsigma_{\widehat{m}}}\!\,\!^\varsigma\!e^{-V}\,$,
	               $\;\,\!^{\varsigma_{\widehat{m}}}\!\,\!^\varsigma\!e^V$)
    is the shorthand for
	$\;\,\!^{\varsigma}\!(e^V)\,$
     (resp.\  $\;\,\!^{\varsigma_{\widehat{m}}}\!(\,\!^\varsigma\!(e^{-V}))\,$,
	               $\;\,\!^{\varsigma_{\widehat{m}}}\!(\,\!^\varsigma\!(e^V))$).
\end{lemma}

%
%

\bigskip

\subsection{$d=4$, $N=1$ Azumaya/matrix superspaces $\widehat{X}^{\!A\!z}$}

Let
 \begin{itemize}
  \item[\LARGE $\cdot$]
   $X$ be the $4$-dimensional Minkowski space-time with structure sheaf ${\cal O}_X$
     as a $C^{\infty}$-scheme   and
   $E$ be a complex vector bundle (say, of rank $r$) on $X$ with the associated sheaf
	   of sections denoted by ${\cal E}$ as an ${\cal O}_X^{\,\Bbb C}$-module;
	
  \item[\LARGE $\cdot$]
     $\widehat{X}$ be the $4$-dimensional, $N=1$ superspace with the structure sheaf
       $\widehat{\cal O}_X$ as a super $C^{\infty}$-scheme   and
     $\widehat{\cal E}
	   := {\cal E}\otimes_{{\cal O}_X^{\,\Bbb C}}\!\widehat{\cal O}_X$
     be the extension of ${\cal E}$ to $\widehat{X}$;  and
	
  \item[\LARGE $\cdot$]
	$\Endsheaf_{\widehat{\cal O}_X}\!(\widehat{\cal E})$
      be the sheaf of ${\Bbb Z}/2$-graded rings of right endomorphisms of $\widehat{\cal E}$
	  (with $\widehat{\cal E}$ as a left $\widehat{\cal O}_X$-module);
    $\Endsheaf_{\widehat{\cal O}_X}\!(\widehat{\cal E})$
	   acts on $\widehat{\cal E}$ from the right by default  and
	   hence is canonically isomorphic to
       $\widehat{\cal E}^{\vee}\otimes_{\widehat{\cal O}_X}\! \widehat{\cal E}$.
 \end{itemize}		
In this subsection we
 introduce the $4$-dimensional $N=1$ Azumaya superspace
   as a super  Azumaya $C^{\infty}$-scheme with underlying topology $X$  and
 study its chiral function ring and antichiral function ring
   associated to a simple hybrid connection on its fundamental module.

\bigskip
  
\begin{flushleft}
{\bf The partial $C^{\infty}$-ring structure on
         $\Endsheaf_{\widehat{\cal O}_X}\!(\widehat{\cal E})$}
\end{flushleft}		
The following theorem is a substatement of [L-Y9: Theorem 2.1.8] (D(11.4.1)),
 rephrased in the form that fits the current situation directly:
 
\bigskip

\begin{theorem} {\bf [$C^{\infty}$ partial operations on
     $\Endsheaf_{\widehat{\cal O}_X}\!(\widehat{\cal E})$]}\;
 Let
  $h\in C^{\infty}({\Bbb R}^l)$ and
  $\widehat{m}_i = m_{i,(0)}+ \widehat{m}_{i, (\ge 1)}
    \in  \Endsheaf_{\widehat{\cal O}_X}\!(\widehat{\cal E})$,
	$i=1,\,\cdots\,, l$
  with
    $m_{i,(0)}\in  \Endsheaf_{{\cal O}_X^{\,\Bbb C}}\!({\cal E})$   and
	$m_{{i, (\ge 1)}}$ contains all the terms in $\widehat{m}_i$
	  that involve the fermionic generators in $\widehat{\cal O}_X$.
 Suppose that $\widehat{m}_1,\,\cdots\,,\, \widehat{m}_l$ satisfy the following two properties
  \begin{itemize}
   \item[$(1)$] $[${\sl commutativity}$]$\hspace{1em}
    $\widehat{m}_i\widehat{m}_j=\widehat{m}_j \widehat{m}_i $\,,\;\;
	for $i,j=1,\,\cdots\,,\, l$.
  
   \item[$(2)$] $[${\sl realness}$]$\hspace{1em}
     \parbox[t]{32.8em}{
    For every $p\in X$,
	 the eigenvalues of the restriction $m_{i, (0)}(p)$
	  of $m_{i, (0)}$ to the fiber
	  $\End_{\Bbb C}(E|_p)\simeq M_{r\times r}({\Bbb C})$
	  of $\Endsheaf_{{\cal O}_X^{\,\Bbb C}}\!({\cal E})$ over $p$	
	are all real.}
  \end{itemize}
  Then,
   $h(\widehat{m}_1,\,\cdots\,,\, \widehat{m}_l)$ is uniquely well-defined.
\end{theorem}

\medskip

\begin{definition} {\bf [partial $C^{\infty}$-ring structure on \& weak $C^{\infty}$-hull of
     $\Endsheaf_{\widehat{\cal O}_X}\!(\widehat{\cal E})$]}\; {\rm
 In view of Theorem~2.4.1,
  we say that $\Endsheaf_{\widehat{\cal O}_X}\!(\widehat{\cal E})$
    has a {\it partial $C^{\infty}$-ring structure}.
 For convenience, though with a slight abuse of terminology,
 a finite set of elements in  $\Endsheaf_{\widehat{\cal O}_X}\!(\widehat{\cal E})$
  that satisfy the two conditions ibidem are said to {\it lie in the weak $C^{\infty}$-hull}
  of $\Endsheaf_{\widehat{\cal O}_X}\!(\widehat{\cal E})$.
 In notation,
  $$
   \mbox{$\{\widehat{m}_1,\,\cdots\,,\, \widehat{m}_l\}
     \in$ {\it weak}-$C^{\infty}$-{\it hull}\,($\Endsheaf_{\widehat{\cal O}_X}
	                                                                         \!(\widehat{\cal E})$)}\,.
  $$
}\end{definition}
		
\medskip

\begin{notation}
{\bf [{\it weak-$C^{\infty}$-hull\,$(\Endsheaf_{\widehat{\cal O}_X}
                                             \!(\widehat{\cal E}) )\cap{\widehat{\cal F}} $}]}\;
{\rm
 For
  $\{\widehat{m}_1,\,\cdots\,,\, \widehat{m}_l\}\in$
  {\it weak}-$C^{\infty}$-{\it hull}\,$(\Endsheaf_{\widehat{\cal O}_X}
	                                                                         \!(\widehat{\cal E}))$																			
   with $\widehat{m}_1,\,\cdots\,,\, \widehat{m}_l$
   all in a subsheaf
    $\widehat{\cal F}$ of $\Endsheaf_{\widehat{\cal O}_X}\!(\widehat{\cal E})$,
 we will denote
  $$
	 \{\widehat{m}_1,\,\cdots\,,\, \widehat{m}_l\}\;\in\;
     \mbox{\it weak-$C^{\infty}$-hull\,$(\Endsheaf_{\widehat{\cal O}_X}\!
	                                                               (\widehat{\cal E}) )
		     \cap{\widehat{\cal F}}$}\,.
  $$										
}\end{notation}

\bigskip

\begin{flushleft}
{\bf $d=4$, $N=1$ Azumaya/matrix superspace as super D3-brane world-volume}
\end{flushleft}
We are finally ready to give the mathematical structure on the world-volume
 of a super (stacked) D3-brane along the line of
 [L-Y1] (D(1)), [L-Y3] (D(11.1)), and [L-Y4] (D(11.2)).		
 
\bigskip

\begin{definition} {\bf [Azumaya/matrix superspace as super D3-brane world-volume]}\; {\rm
 The ${\Bbb Z}/2$-graded-locally-ringed space
  $$
     \widehat{X}^{\!A\!z}\; :=\;
	 (X,  \widehat{\cal O}_X^{A\!z}
		       :=  \Endsheaf_{\widehat{\cal O}_X}\!(\widehat{\cal E}))\,,
  $$
 with $   \widehat{\cal O}_X^{A\!z}$ endowed with the partial $C^{\infty}$-ring structure
 from Theorem~2.4.1,
 is called a {\it $4$-dimensional $N=1$ Azumaya superspace} or interchangeably
  {\it $4$-dimensional $N=1$ matrix superspace}, or in short
   {\it $d=4$, $N=1$ Azumaya/matrix superspace}.
 The pair
  $(\widehat{X}^{\!A\!z}, \widehat{\cal E})
     = (X, \widehat{\cal O}_X^{A\!z}, \widehat{\cal E})$
  is called a {\it $d=4$, $N=1$ Azumaya/matrix superspace with a fundamental module}. 	

 String-theoretically,
  the Azumaya/matrix superspace $\widehat{X}^{\!A\!z}$
    is the {\it world-volume of a super D3-brane} and
 $\widehat{\cal E}$ is the {\it Chan-Paton sheaf} on the super D-brane world-volume.
 SUSY-rep compatible simple hybrid connections $\widehat{\nabla}$ on $\widehat{\cal E}$
   encode the {\it gauge fields}, their super partners {\it gauginos}, and some {\it auxiliary fields}
   on the super D-brane world-volume;
 they are the fields that correspond to the vector multiplet in the representations
  of the $d=4$, $N=1$ supersymmetry algebra.
}\end{definition}

\medskip

\begin{definition}
{\bf [standard/reference structures on $(\widehat{X}^{\!A\!z}, \widehat{\cal E})$]}\;
{\rm
 We collect from previous (sub)sections a few standard/reference structures related to
   $(\widehat{X}^{\!A\!z}, \widehat{\cal E})$ here:
  \begin{itemize}
   \item[(1)]
   [{\it special coordinate functions on $\widehat{X}$}]\hspace{1em}
    The {\it standard coordinate-functions} $(x,\theta, \bar{\theta})$,
	    cf. Notation~1.2.6;
	the {\it standard chiral coordinate functions}
	    $(x^\prime, \theta^\prime, \bar{\theta}^\prime)$,
		cf.\ Definition~1.4.11;
    the {\it standard antichiral coordinate functions}
	    $(x^{\prime\prime}, \theta^{\prime\prime}, \bar{\theta}^{\prime\prime})$,
		cf.\ Definition~1.4.15.

   \item[\LARGE $\cdot$]
    [{\it special derivations and $1$-forms on $\widehat{X}$}]\hspace{1em}
    The {\it standard supersymmetry generators} $Q_\alpha, \bar{Q}_{\dot{\beta}}$;
	the {\it standard supersymmetrically invariant derivations}
	   $e_{\alpha^{\prime}}, e_{\beta^{\prime\prime}}$;
	the {\it standard supersymmetrically invariant coframe} $(e_I)_I$; all on $\widehat{X}$.
  
   \item[(2)]
    [{\it reference trivialization of sheaves of modules}]\hspace{1em}
	Fix a reference trivialization of ${\cal E}$ on $X$.
	This induces then reference trivializations for ${\cal E}^{\vee}$ and
	   $\Endsheaf_{{\cal O}_X^{\,\Bbb C}}({\cal E})$ on $X$.
    These trivializations extend to reference trivializations for $\widehat{\cal E}$,
      $\widehat{\cal E}^{\vee}$, and
	  $\Endsheaf_{\widehat{\cal O}_X}\!(\widehat{\cal E})$ on $\widehat{X}$
	  via the extension of the structure sheaf
	    $(\,\mbox{\tiny $\bullet$}\,)
		       \otimes_{{\cal O}_X^{\,\Bbb C}}\!\widehat{\cal O}_X $.
			
  \item[\LARGE $\cdot$]	
   [{\it reference connection}]\hspace{1em}
   By construction, each such trivialization corresponds to a particular choice of an even basis.
	The connection associated to the reference trivialization is denoted by $d$.

   \item[(3)]
    [{\it $\widehat{\cal E}$, $\widehat{\cal E}^{\vee}$, and
      $\Endsheaf_{\widehat{\cal O}_X}\!(\widehat{\cal E})$
	  as equivariant sheaves under supersymetry}]\hspace{1em}	
    The reference trivialization of
	  $\widehat{\cal E}$, $\widehat{\cal E}^{\vee}$, and
      $\Endsheaf_{\widehat{\cal O}_X}\!(\widehat{\cal E})$ in Item (2)
   	  induces a lifting of the flows on $\widehat{X}$ to flows on
	  $\widehat{\cal E}$, $\widehat{\cal E}^{\vee}$, and
      $\Endsheaf_{\widehat{\cal O}_X}\!(\widehat{\cal E})$
      rendering them equivariant sheaves under supersymmetry.	
	  %
	
   \item[\LARGE $\cdot$]	
    [{\it reference lifting of derivations on $\widehat{X}$}]\hspace{1em}
	Via the reference connection $d$ on
    	$\widehat{\cal E}$, $\widehat{\cal E}^{\vee}$, and
        $\Endsheaf_{\widehat{\cal O}_X}\!(\widehat{\cal E})$,
	 a derivation $\xi$ on $\widehat{X}$ can now lift to apply on
	  $\widehat{\cal E}$, $\widehat{\cal E}^{\vee}$, and
      $\Endsheaf_{\widehat{\cal O}_X}\!(\widehat{\cal E})$ as well.
	In particular,
	  $Q_\alpha, \bar{Q}_{\dot{\beta}},
	   e_{\alpha^{\prime}}, e_{\beta^{\prime\prime}}$
	 all apply on $\widehat{\cal E}$, $\widehat{\cal E}^{\vee}$, and
       $\Endsheaf_{\widehat{\cal O}_X}\!(\widehat{\cal E})$
	 via this {\it reference lifting}.
  %
  \end{itemize}
}\end{definition}

\bigskip

\section{The $\widehat{D}$-chiral and the $\widehat{D}$-antichiral structure sheaf
     of $\widehat{X}^{\!A\!z}$}

In considering the notion of a `{\it chiral stricture sheaf}'  or `{\it antichiral structure sheaf}'
 on the $d=4$, $N=1$ Azumaya/matrix superspace $\widehat{X}^{\!A\!z}$,
we are guided by three wished-for properties:
   %
 %
 \begin{itemize}
  \item[(1)]
   [{\it SUSY-rep compatible}]\hspace{1em}
  It has to reflect the multiplets associated to
	the $\widehat{\cal O}_X^{A\!z}$-valued scalar representation of $d=4$, $N=1$ supersymmetry.
	
  \item[(2)]	
   [{\it ${\Bbb Z}/2$-graded with partial $C^{\infty}$-ring structure}]\hspace{1em}
   Being in the ${\Bbb Z}/2$-world,  it preferably should be a sheaf of ${\Bbb Z}/2$-graded rings.
   For $C^{\infty}$-Algebraic Geometry to apply,
    it better has a reasonably natural partial $C^{\infty}$-ring structure.
   
  \item[(3)]
   [{\it Useful}]\hspace{1em}
   It should lead us to
     a good notion of `{\it chiral maps}' or `{\it antichiral maps}' from $\widehat{X}^{\!A\!z}$
	 and be useful in the construction of a supersymmetric action functional for them
	so that one can use them to study the dynamics of fermionic D3-branes in a target space(-time).
 \end{itemize}
In this subsection, we introduce the notion of
 the $\widehat{D}$-chiral structure sheaf and the $\widehat{D}$-antichiral structure sheaf
     of $\widehat{X}^{\!A\!z}$ that meet the above requirements
 and give the normal form of their sections.

\bigskip

\subsection{The $\widehat{D}$-chiral structure sheaf and the $\widehat{D}$-antichiral structure sheaf
     of $\widehat{X}^{\!A\!z}$}
 
\bigskip

\begin{flushleft}
{\bf The reference chiral/antichiral structure sheaf on $\widehat{X}^{\!A\!z}$ associated to $d$}
\end{flushleft}
The notion of chiral functions and antichiral functions on $\widehat{X}$ extends to that
 on $\widehat{X}^{\!A\!z}$ very naively
 via the reference trivial connection $d$ on $\widehat{\cal O}_X^{A\!z}$.
The SUSY-Rep Compatibility and ${\Bbb Z}/2$-Graded-Ring Property become automatic.

\bigskip
		
\begin{definition-lemma}
{\bf [reference chiral/antichiral structure sheaf of $\widehat{X}^{\!A\!z}$]}\; {\rm
 Recall the fixed reference connection $d$ on $\widehat{\cal O}_X^{A\!z}$.
 An $\widehat{m}\in \widehat{\cal O}_X^{A\!z}$ is called {\it $d$-chiral}
  (resp.\ {\it $d$-antichiral})
  if $e_{1^{\prime\prime}}\widehat{m}=e_{2^{\prime\prime}}\widehat{m}=0$
  (resp.\ $e_{1^\prime}\widehat{m}= e_{2^\prime}\widehat{m}=0$).
 
 {\it The subsheaf of d-chiral sections in $\widehat{\cal O}_X^{A\!z}$
   is a sheaf of ${\Bbb Z}/2$-graded subrings of $\widehat{\cal O}_X^{A\!z}$},
  called the {\it reference chiral structure sheaf} or {\it $d$-chiral structure sheaf}
   of $\widehat{X}^{\!A\!z}$, denoted by $\widehat{\cal O}_X^{A\!z, \scriptsizedch}$.
 Similarly, {\it the subsheaf of $d$-antichiral sections in $\widehat{\cal O}_X^{A\!z}$
   is a sheaf of ${\Bbb Z}/2$-graded subrings of $\widehat{\cal O}_X^{A\!z}$},
  called the {\it reference antichiral structure sheaf} or {\it $d$-antichiral structure sheaf}
   of $\widehat{X}^{\!A\!z}$, denoted by $\widehat{\cal O}_X^{A\!z, \scriptsizedach}$.
}\end{definition-lemma}

\medskip

\begin{lemma} {\bf [partial $C^{\infty}$-ring structure on
    $\widehat{\cal O}_X^{A\!z, \scriptsizedch}$ and
    $\widehat{\cal O}_X^{A\!z, \scriptsizedach}$]}\;
 The partial $C^{\infty}$-ring structure on
  $\widehat{\cal O}_X^{A\!z}$ restricts to a partial $C^{\infty}$-ring structure on
   $\widehat{\cal O}_X^{A\!z, \scriptsizedch}$ and
   $\widehat{\cal O}_X^{A\!z, \scriptsizedach}$.
 I.e.\
 let
  $\{\widehat{m}_1, \,\cdots\,, \widehat{m}_l\}\in $
    weak-$C^{\infty}$-hull\,$(\widehat{\cal O}_X^{A\!z})
  \cap \widehat{\cal O}_X^{A\!z, \scriptsizedch}$
 (resp.\
    weak-$C^{\infty}$-hull\,$(\widehat{\cal O}_X^{A\!z})
    \cap \widehat{\cal O}_X^{A\!z, \scriptsizedach}$)    and
  $h\in C^{\infty}({\Bbb R}^l)$.
 Then
  $h(\widehat{m}_1, \,\cdots\,, \widehat{m}_l)
     \in  \widehat{\cal O}_X^{A\!z, \scriptsizedch}$
 (resp.\ $\widehat{\cal O}_X^{A\!z, \scriptsizedach}$).
\end{lemma}

\bigskip

\noindent
These two lemmas are special cases of Definition/Lemma~3.1.3 and Lemma~3.1.4 in the next theme.

The reference trivialization in Definition~2.4.5
 specifies an isomorphism
 $\widehat{\cal O}_X^{A\!z}
   \simeq  \widehat{O}_X\otimes_{\Bbb C}M_{r\times r}({\Bbb C}) $.
Under this isomorphism,
a section of $\widehat{\cal O}_X^{A\!z}$ can be written as
 $\widehat{m}=(f_{ij})_{ij}$, $f_{ij}\in \widehat{\cal O}_X$, and
 $\xi \widehat{m}= (\xi f_{ij})_{ij}$ for a derivation $\xi$ on $\widehat{X}$.
It follows that under this isomorphism,
 $$
  \widehat{\cal O}_X^{A\!z, \scriptsizedch}\;
     \simeq\; \widehat{\cal O}_X^{\scriptsizech}
	                  \otimes_{\Bbb C}M_{r\times r}({\Bbb C})
    \hspace{2em}\mbox{and}\hspace{2em}
  \widehat{\cal O}_X^{A\!z, \scriptsizedach}\;
    \simeq\; \widehat{\cal O}_X^{\scriptsizeach}
	                  \otimes_{\Bbb C}M_{r\times r}({\Bbb C})\,.
 $$

\bigskip

\begin{flushleft}
{\bf The $\widehat{D}$-chiral and the $\widehat{D}$-antichiral structure sheaf
          on $\widehat{X}^{\!A\!z}$}
\end{flushleft}
Let
 $\widehat{\nabla}$ be the SUSY-rep compatible simple hybrid connection on $\widehat{\cal E}$
   associated to a vector superfield $V$ on $\widehat{X}^{\!A\!z}$    and
 $\widehat{D}$ be the induced hybrid connection on
   $\widehat{\cal O}_X^{A\!z}
     :=\Endsheaf_{\widehat{\cal O}_X}\!({\widehat{\cal E}})$.

\bigskip

\begin{definition-lemma}
{\bf [$\widehat{D}$-chiral/antichiral structure sheaf of $\widehat{X}^{\!A\!z}$]}\; {\rm
 An $\widehat{m}\in \widehat{\cal O}_X^{A\!z}$ is called {\it $\widehat{D}$-chiral}
  (resp.\ {\it $\widehat{D}$-antichiral}) if\,\footnote{For
                                                                               a connection like $\widehat{D}$ that is not purely even,
																			   the {\it naive $\widehat{D}$-chiral condition}
																			     $\widehat{D}_{e_{1^{\prime\prime}}}\widehat{m}
																					=\widehat{D}_{e_{2^{\prime\prime}}}\widehat{m}
																					=0$ alone
																				or the {\it naive $\widehat{D}$-antichiral condition}
																				    $\widehat{D}_{e_{1^\prime}}\widehat{m}
                                                                                      = \widehat{D}_{e_{2^\prime}}\widehat{m}																									 
																					  =0$ alone
																			     does not seem to give a good chiral or antichiral theory.
																				 Cf.\ Remark~3.1.5.
																				 }  
   $$
      \widehat{D}_{e_{\beta^{\prime\prime}}}\widehat{m}\;
	  =\; \,\!^{\varsigma}\!\widehat{D}_{e_{\beta^{\prime\prime}}}\widehat{m}\;
	  =\; 0\,,
   $$
   for $\beta^{\prime\prime}=1^{\prime\prime}, 2^{\prime\prime}$
  (resp.\
   $$
      \widehat{D}_{e_{\alpha^\prime}}\widehat{m}\;
	  =\; \,\!^{\varsigma}\!\widehat{D}_{e_{\alpha^\prime}}\widehat{m}\;
	  =\; 0\,, 	
   $$
   for $\alpha^\prime= 1^\prime, 2^\prime$).
 Note that the $\widehat{D}$-chiral conditions
  are equivalent to the conditions
   $$
     (\widehat{D}^{(\even)})_{e_{\beta^{\prime\prime}}}\widehat{m}\;
	  =\;  (\widehat{D}^{(\odd)})_{e_{\beta^{\prime\prime}}}\widehat{m}\;
	  =\; 0\,;
   $$
 and also to the conditions
  $$
   \widehat{D}_{e_{\beta^{\prime\prime}}}\widehat{m}_{(\even)}\;
   =\; \widehat{D}_{e_{\beta^{\prime\prime}}}\widehat{m}_{(\odd)}\;
   =\; 0\,,
  $$
  and
  to the conditions
  $$
   \,\!^\varsigma\!\widehat{D}_{e_{\beta^{\prime\prime}}}\widehat{m}_{(\even)}\;
   =\; \,\!^\varsigma\!\widehat{D}_{e_{\beta^{\prime\prime}}}\widehat{m}_{(\odd)}\;
   =\; 0
  $$
 while the $\widehat{D}$-antichiral conditions
  are equivalent to the conditions
   $$
     (\widehat{D}^{(\even)})_{e_{\alpha^\prime}}\widehat{m}\;
	  =\;  (\widehat{D}^{(\odd)})_{e_{\alpha^\prime}}\widehat{m}\;
	  =\; 0\,;
   $$
 and also to the conditions
  $$
   \widehat{D}_{e_{\alpha^\prime}}\widehat{m}_{(\even)}\;
   =\; \widehat{D}_{e_{\alpha^\prime}}\widehat{m}_{(\odd)}\;
   =\; 0\,,
  $$
  and
  to the conditions
  $$
   \,\!^\varsigma\!\widehat{D}_{e_{\alpha^\prime}}\widehat{m}_{(\even)}\;
   =\; \,\!^\varsigma\!\widehat{D}_{e_{\alpha^\prime}}\widehat{m}_{(\odd)}\;
   =\; 0\,.
  $$
 Here,
   $\widehat{D}=\widehat{D}^{(\even)}+\widehat{D}^{(\odd)}$
     is the decomposition of $\widehat{D}$ into the even part and the odd part, and
   $\widehat{m}=\widehat{m}_{(\even)}+\widehat{m}_{(\odd)}$
     is the decomposition of $\widehat{m}$ into the even part and the odd part.
    
 {\it The subsheaf of $\widehat{D}$-chiral sections in $\widehat{\cal O}_X^{A\!z}$
   is a sheaf of ${\Bbb Z}/2$-graded subrings of $\widehat{\cal O}_X^{A\!z}$},
  called the {\it $\widehat{D}$-chiral structure sheaf} of $\widehat{X}^{\!A\!z}$  and
  denoted by $\widehat{\cal O}_X^{A\!z, \scriptsizewidehatDch}$.
 Similarly, {\it the subsheaf of $\widehat{D}$-antichiral sections in $\widehat{\cal O}_X^{A\!z}$
   is a sheaf of ${\Bbb Z}/2$-graded subrings of $\widehat{\cal O}_X^{A\!z}$},
  called the {\it $\widehat{D}$-antichiral structure sheaf} of $\widehat{X}^{\!A\!z}$  and
  denoted by $\widehat{\cal O}_X^{A\!z, \scriptsizewidehatDach}$.
}\end{definition-lemma}

\medskip

\begin{proof}
 For $\widehat{D}$-chiral sections of $\widehat{\cal O}_X^{A\!z}$,
  since
     $\,\!^{\varsigma}\!\widehat{D}_{e_{\beta^{\prime\prime}}}
	 = \widehat{D}_{e_{\beta^{\prime\prime}}}= e_{\beta^{\prime\prime}}$,
  $$
    \widehat{\cal O}_X^{A\!z, \scriptsizewidehatDch}\;
	 =\;   \widehat{\cal O}_X^{A\!z, \scriptsizedch}\;
	 \simeq\;   \widehat{\cal O}_X^{\scriptsizech}
	                                                   \otimes_{\Bbb C}\!M_{r\times r}({\Bbb C})
  $$
  and the claim follows.
 
 For $\widehat{D}$-antichiral sections of $\widehat{\cal O}_X^{A\!z}$,
  let $\widehat{D}=\widehat{D}^{(\even)}\,+\, \widehat{D}^{(\odd)}$
  be the decomposition of $\widehat{D}$ into the even part and the odd part.
 Then, $\,\!^{\varsigma}\!\widehat{D}
                  = \widehat{D}^{(\even)}- \widehat{D}^{(\odd)}$.
 Note also that
  $\,\!^\varsigma\!(\,\!^\varsigma\!\widehat{D}_{e_{\beta^{\prime\prime}}}
       \widehat{m})
	  = - \widehat{D}_{e_{\beta^{\prime\prime}}}\,\,\!^\varsigma\widehat{m}$.
 Thus, the $\widehat{D}$-antichiral conditions
    $  \widehat{D}_{e_{\alpha^\prime}}\widehat{m}
	  = \,\!^{\varsigma}\!\widehat{D}_{e_{\alpha^\prime}}\widehat{m} = 0$
  are equivalent to the conditions
   $$
     (\widehat{D}^{(\even)})_{e_{\alpha^\prime}}\widehat{m}\;
	  =\;  (\widehat{D}^{(\odd)})_{e_{\alpha^\prime}}\widehat{m}\;
	  =\; 0\,;
   $$
 and also to the conditions
  $$
   \widehat{D}_{e_{\alpha^\prime}}\widehat{m}_{(\even)}\;
   =\; \widehat{D}_{e_{\alpha^\prime}}\widehat{m}_{(\odd)}\;
   =\; 0\,,
  $$
  and
  to the conditions
  $$
   \,\!^\varsigma\!\widehat{D}_{e_{\alpha^\prime}}\widehat{m}_{(\even)}\;
   =\; \,\!^\varsigma\!\widehat{D}_{e_{\alpha^\prime}}\widehat{m}_{(\odd)}\;
   =\; 0\,,
  $$
 where
   $\widehat{m}=\widehat{m}_{(\even)}+\widehat{m}_{(\odd)}$
   is the decomposition of $\widehat{m}$ into the even part and the odd part.
 Thus, if $\widehat{m}$ is $\widehat{D}$-antichiral,
    then both $\widehat{m}_{(\even)}$ and $\widehat{m}_{(\odd)}$
	are $\widehat{D}$-antichiral.
 Furthermore,
  \begin{eqnarray*}
    \widehat{D}_{e_{\alpha^\prime}}(\widehat{m}_1\widehat{m}_2)
	 & = &  (\widehat{D}_{e_{\alpha^\prime}} \widehat{m}_1)\, \widehat{m}_2\,
	         +\,  \widehat{m}_{1, (\even)}\,
			         \widehat{D}_{e_{\alpha^\prime}} \widehat{m}_2  \,
			 -\,  \widehat{m}_{1, (\odd)}\,
			       \,\!^\varsigma\!\widehat{D}_{e_{\alpha^\prime}}\widehat{m}_2  \\[-.2ex]
     & = & (\widehat{D}_{e_{\alpha^\prime}} \widehat{m}_1)\, \widehat{m}_2\,
	         +\, \,\!^\varsigma\!\widehat{m}_1\,
                     \widehat{D}^{(\even)}_{e_{\alpha^\prime}}\widehat{m}_2\,
			 +\, \widehat{m}_1\,
			       \widehat{D}^{(\odd)}_{e_{\alpha^\prime}}\widehat{m}_2\; ; \\[.6ex]
   \,\!^\varsigma\!\widehat{D}_{e_{\alpha^\prime}}(\widehat{m}_1\widehat{m}_2)
	 & = &  (\,\!^\varsigma\!\widehat{D}_{e_{\alpha^\prime}} \widehat{m}_1)\,
	                  \widehat{m}_2\,
	         +\,  \widehat{m}_{1, (\even)}\,
			         \,\!^\varsigma\!\widehat{D}_{e_{\alpha^\prime}} \widehat{m}_2  \,
			 -\,  \widehat{m}_{1, (\odd)}\,
			       \,\widehat{D}_{e_{\alpha^\prime}}\widehat{m_2}  \\[-.2ex]
     & = & (\,\!^\varsigma\!\widehat{D}_{e_{\alpha^\prime}} \widehat{m}_1)\,
	              \widehat{m}_2\,
	         +\, \,\!^\varsigma\!\widehat{m}_1\,
                     \widehat{D}^{(\even)}_{e_{\alpha^\prime}}\widehat{m}_2\,
	         -\,  \widehat{m}_1\,
			       \,\widehat{D}^{(\odd)}_{e_{\alpha^\prime}}\widehat{m}_2  \,.
  \end{eqnarray*}
 Thus,
   if $\widehat{m}_1$ and $\widehat{m}_2$ are $\widehat{D}$-antichiral, then
    so is $\widehat{m}_1\widehat{m}_2$.
 This proves that the sheaf of $\widehat{D}$-antichiral sections of $\widehat{\cal O}_X^{A\!z}$	
   is a sheaf of ${\Bbb Z}/2$-graded subrings of $\widehat{\cal O}_X^{A\!z}$.
   
 This completes the proof.
  
\end{proof}

\medskip

\begin{lemma} {\bf [partial $C^{\infty}$-ring structure on
    $\widehat{\cal O}_X^{A\!z, \scriptsizewidehatDch}$ and
    $\widehat{\cal O}_X^{A\!z, \scriptsizewidehatDach}$]}\;
 The partial $C^{\infty}$-ring structure on
  $\widehat{\cal O}_X^{A\!z}$ restricts to a partial $C^{\infty}$-ring structure on
   $\widehat{\cal O}_X^{A\!z, \scriptsizewidehatDch}$ and
   $\widehat{\cal O}_X^{A\!z, \scriptsizewidehatDach}$.
 I.e.\
 let
  $\{\widehat{m}_1, \,\cdots\,, \widehat{m}_l\}\in $
    weak-$C^{\infty}$-hull\,$(\widehat{\cal O}_X^{A\!z})
  \cap \widehat{\cal O}_X^{A\!z, \scriptsizewidehatDch}$
 (resp.\
    weak-$C^{\infty}$-hull\,$(\widehat{\cal O}_X^{A\!z})
    \cap \widehat{\cal O}_X^{A\!z, \scriptsizewidehatDach}$)    and
  $h\in C^{\infty}({\Bbb R}^l)$.
 Then
  $h(\widehat{m}_1, \,\cdots\,, \widehat{m}_l)
     \in  \widehat{\cal O}_X^{A\!z, \scriptsizewidehatDch}$
 (resp.\ $\widehat{\cal O}_X^{A\!z, \scriptsizewidehatDach}$).
\end{lemma}

\medskip

\begin{proof}
 The proof is an application of the Malgrange Division Theorem ([Mal]; [Br]) and
 follows the same argument as in the proof for
    [L-Y5: Theorem 3.1.1] (D(11.3.1)),
	[L-Y6: Sec.\ 4] (D(13.1)),
    [L-Y8: Lemma~2.1.6] (D(13.3)), and
    [L-Y9: Theorem 2.1.5] (D(11.4.1)).
 We will only give a sketch here and refer readers to ibidem for missing details.

 Let ${\Bbb R}^l$ be equipped with the standard coordinate functions
  $\boldt:= (t^1,\,\cdots\,, t^l)$.
 Then, the commutativity of $\widehat{m}_1,\,\cdots\,,\, \widehat{m}_l$
  and realness of their eigenvalues imply that
  their characteristic polynomials are in
   $C^{\infty}(X)[t^1,\,\cdots\,t^l]\subset C^{\infty}(\widehat{X}\times {\Bbb R}^l)$
   and the ideal generalized by these characteristic polynomials describes
   a nonempty $C^{\infty}$-subscheme $\varSigma_{\widehat{\scriptsizeboldm}}$
   of $\widehat{X}\times {\Bbb R}^l$, finite over $\widehat{X}$.
 Here, $\widehat{\boldm}:= (\widehat{m}_1,\,\cdots\,,\, \widehat{m}_l )$.
 Via the pullback of the projection maps
    $\widehat{X}\times {\Bbb R}^l
	   \longrightaarrow X\times {\Bbb R}^l\longrightaarrow {\Bbb R}^l$,
  one can regard $h\in C^{\infty}({\Bbb R}^l)$ as in
   $C^{\infty}(X\times {\Bbb R}^l)
     \subset C^{\infty}(\widehat{X}\times {\Bbb R}^l)$.
  For $p\in X\subset \widehat{X}$, 	
   applying the Malgrange Division Theorem to the germ of $h$ over $p$
   with respect to the square of the characteristic polynomials of
   $\widehat{m}_1,\,\cdots\,,\, \widehat{m}_l$
   repeatedly, one concludes that:
   \begin{itemize}
    \item[\LARGE $\cdot$]
	 There exist open sets
	  $\widehat{U}\subset \widehat{X}$ that contains $p$ and
	  $V\subset {\Bbb R}^l$ that contains the projection of
	    $\varGamma_{\widehat{\scriptsizeboldm}}\cap \widehat{U}\times {\Bbb R}$
		in ${\Bbb R}^l$
	 such that
	 $$
	    h|_{\widehat{U}\times V}\; =\;  \widehat{f}_0[1]\,+\, \widehat{f}_1[1]\,,		
	 $$
	 with the following properties:
	 \begin{itemize}
	  \item[(1)]
	   $\widehat{f}_0[1]\;
	      =\; \sum_{\scriptsizeboldd}  \,
	                    a[1]_{\scriptsizeboldd}\,\boldt^{\scriptsizeboldd}\,
	       \in C^{\infty}(U)[t^1|_V,\,\cdots\,, t^l|_V]|$;
	
	  \item[(2)]
	   $\widehat{f}_1[1]\in  \widehat{I}_{\widehat{\scriptsizeboldm}}^{\,2}
	   \subset C^{\infty}(\widehat{U}\times V)$ satisfies
	   $$
	     \widehat{f}_1[1]|_{\scriptsizeboldt\rightsquigarrow \widehat{\scriptsizeboldm}}\;
		 =\;  (\widehat{\xi}\widehat{f}_1[1])|
		                 _{\scriptsizeboldt\rightsquigarrow \widehat{\scriptsizeboldm}}\;
         =\; 0\; \in C^{\infty}(\End_{\widehat{\Bbb C}}
		                (\widehat{E}|_{\widehat{U}}))
	   $$
	   for all $\widehat{\xi}\in \Der_{\Bbb C}(C^{\infty}(\widehat{U}\times V))$\,.
	 \end{itemize}
    Here,
	 \begin{itemize}
	  \item[\LARGE $\cdot$]
	   $\boldd :=(d_1,\,\cdots\,\, d_l)$\,,\;
       $a[1]_{\scriptsizeboldd}\in C^{\infty}(U)$\,,\; and 	
	   $\boldt^{\scriptsizeboldd}:= (t^1)^{d_1}\,\cdots\, (t^l)^{d_l}$\,;
	
	  \item[\LARGE $\cdot$]
	   $\widehat{I}_{\widehat{\scriptsizeboldm}}
	      \subset C^{\infty}(\widehat{U}\times V)$
         is the ideal associated to
		  $\varSigma_{\widehat{\scriptsizeboldm}}\cap (\widehat{U}\times V)
		    \subset \widehat{U}\times V$     and
	   $(\mbox{\tiny $\bullet$})|
	          _{\scriptsizeboldt\rightsquigarrow \widehat{\scriptsizeboldm}}$
        is the evaluation of $(\mbox{\tiny $\bullet$})$
		at $(t^1,\,\cdots\,,\,  t^l)=(\widehat{m}_1,\,\cdots\,,\, \widehat{m}_l)$.
   \end{itemize}
 \end{itemize}
 Property (1) and the first equality
	$\widehat{f}_1[1]|_{\scriptsizeboldt\rightsquigarrow \widehat{\scriptsizeboldm}}= 0$
    of Property (2)
 imply that
 $$
   h(\widehat{m}_1,\,\cdots\,,\, \widehat{m}_l)|_{\widehat{U}}\;
    =\;  \widehat{f}_0[1]|_{\scriptsizeboldt\rightsquigarrow \widehat{\scriptsizeboldm}}\;	
	=\;  \sum_{\scriptsizeboldd}  \,
	                    a[1]_{\scriptsizeboldd}\cdot \widehat{\boldm}^{\scriptsizeboldd}\,
	\in C^{\infty}(\End_{\widehat{\Bbb C}}(\widehat{E}|_{\widehat{U}}))\,,
 $$
 where
   $\widehat{\boldm}^{\scriptsizeboldd}
      := \widehat{m}_1^{d_1}\,\cdots\,\widehat{m}_l^{d_l}$.
 Since $\xi h=0$ for $\xi\in \Der_{\Bbb C}(C^{\infty}(\widehat{U}))$,
 Property (1) and the second equality
    $(\widehat{\xi}\widehat{f}_1[1])|
		                 _{\scriptsizeboldt\rightsquigarrow \widehat{\scriptsizeboldm}}= 0$
	for all $\widehat{\xi}\in \Der_{\Bbb C}(C^{\infty}(\widehat{U}\times V))$
  imply that 	
   $$
      \xi\widehat{f}_0[1]\;
      =\;  \sum_{\scriptsizeboldd}
	       (\xi a[1]_{\scriptsizeboldd})\cdot \boldt^{\scriptsizeboldd}\;
	  =\; 0
   $$
   for all $\xi\in \Der_{\Bbb C}(C^{\infty}(\widehat{U}))$.
   
 Now assume that $\widehat{m}_1,\,\cdots\,,\, \widehat{m}_l$ are $\widehat{D}$-antichiral.
 Then Definition/Lemma~3.1.3 implies that
   $\widehat{\boldm}^{\scriptsizeboldd}$ is also $\widehat{D}$-antichiral.
 Recall that $\widehat{D}$ restricts to $d$ on $\widehat{\cal O}_X$ (Lemma/Definition~2.2.12 (7)).
 It follows that, over $\widehat{U}$,
   $$
     \widehat{D}_{e_{\alpha^\prime}}
	   h(\widehat{m}_1, \,\cdots\,,\, \widehat{m}_l)\;
	 =\;   \sum_{\scriptsizeboldd}
	         (e_{\alpha^\prime} a[1]_{\scriptsizeboldd})\cdot \boldt^{\scriptsizeboldd}\,
			 +\,  \sum_{\scriptsizeboldd}
	                  a[1]_{\scriptsizeboldd}
					   \cdot  \widehat{D}_{e_{\alpha^\prime}}
					                      ( \widehat{\boldm}^{\scriptsizeboldd})\;
     =\; 0
   $$
   and
   $$
    \,\!^\varsigma\!\widehat{D}_{e_{\alpha^\prime}}
	   h(\widehat{m}_1, \,\cdots\,,\, \widehat{m}_l)\;
	 =\;   \sum_{\scriptsizeboldd}
	         (e_{\alpha^\prime} a[1]_{\scriptsizeboldd})\cdot \boldt^{\scriptsizeboldd}\,
			 +\,  \sum_{\scriptsizeboldd}
	                  a[1]_{\scriptsizeboldd}
					   \cdot  \,\!^\varsigma\!\widehat{D}_{e_{\alpha^\prime}}
					                      ( \widehat{\boldm}^{\scriptsizeboldd})\;
     =\; 0\,.
   $$
 Since $p\in X\subset \widehat{X}$ is arbitrary,
 this proves that $h(\widehat{m}_1, \,\cdots\,,\, \widehat{m}_l)$ is $\widehat{D}$-antichiral.

 This completes the proof.

\end{proof}

\medskip

\begin{remark} $[${weakly $\widehat{D}$-chiral/$\widehat{D}$-antichiral structure sheaf
       on $\widehat{X}^{\!A\!z}$}$]$\; {\rm
 There is a weaker notion of $\widehat{D}$-chirality/$\widehat{D}$-antichirality
   for sections of $\widehat{\cal O}_X^{A\!z}$:
  
  \begin{itemize}
   \item[\LARGE $\cdot$]
     [{\sl weakly $\widehat{D}$-chiral/$\widehat{D}$antichiral section
	            of $\widehat{\cal O}_X^{A\!z}$}]\;
    An $\widehat{m}\in \widehat{\cal O}_X^{A\!z}$ is called {\it weakly $\widehat{D}$-chiral}
     (resp.\ {\it weakly $\widehat{D}$-antichiral}) if
   $$
      \widehat{D}^{(\even)}\,\!_{e_{\beta^{\prime\prime}}}\widehat{m}\;
	  =\; 0\,,
   $$
   for $\beta^{\prime\prime}=1^{\prime\prime}, 2^{\prime\prime}$
  (resp.\
   $$
      \widehat{D}^{(\even)}\,\!_{e_{\alpha^\prime}}\widehat{m}\;
	  =\; 0\,, 	
   $$
   for $\alpha^\prime= 1^\prime, 2^\prime$).
  \end{itemize}
 
 \noindent
 Since
   $\widehat{D}^{(\even)}\,\!_{e_{\alpha^\prime}}$ is purely odd,
  if $\widehat{D}^{(\even)}\,\!_{e_{\alpha^\prime}} \widehat{m}=0$,
  then both $\widehat{D}^{(\even)}\,\!_{e_{\alpha^\prime}}\widehat{m}_{(\even)}$
      and $\widehat{D}^{(\even)}\,\!_{e_{\alpha^\prime}}\widehat{m}_{(\odd)}$
      must vanish.	
 Together with the fact
   $\widehat{D}^{(\even)}$ itself is a connection on $\widehat{\cal O}_X^{A\!z}$,
  the same argument as in Definition/Lemma~3.1.3
  proves the following statement:
  
  \begin{itemize}
   \item[\LARGE $\cdot$]
   {\it The subsheaf of weakly $\widehat{D}$-chiral sections in $\widehat{\cal O}_X^{A\!z}$
      is a sheaf of ${\Bbb Z}/2$-graded subrings of $\widehat{\cal O}_X^{A\!z}$;  and
	 so is the subsheaf of $\widehat{D}$-antichiral sections in $\widehat{\cal O}_X^{A\!z}$.}
   \end{itemize}
   
 \noindent
 Call the former the {\it weakly $\widehat{D}$-chiral structure sheaf} of $\widehat{X}^{\!A\!z}$
     and denote it by $\widehat{\cal O}_X^{A\!z, \scriptsizewidehatDch\,\!^w}$ and
 call the latter the {\it weakly $\widehat{D}$-antichiral structure sheaf} of $\widehat{X}^{\!A\!z}$
     and denote it by $\widehat{\cal O}_X^{A\!z, \scriptsizewidehatDach\,\!^w}$.
 Then similar arguments as in the proof of Lemma~3.1.4 give:
  									
   \begin{itemize}
    \item[\LARGE $\cdot$]
      [{\sl partial $C^{\infty}$-ring structure on
          $\widehat{\cal O}_X^{A\!z, \scriptsizewidehatDch\,\!^w}$ and
    $\widehat{\cal O}_X^{A\!z, \scriptsizewidehatDach\,\!^w}$}]\;
   {\it The partial $C^{\infty}$-ring structure on
      $\widehat{\cal O}_X^{A\!z}$ restricts to a partial $C^{\infty}$-ring structure on
      $\widehat{\cal O}_X^{A\!z, \scriptsizewidehatDch\,\!^w}$ and
      $\widehat{\cal O}_X^{A\!z, \scriptsizewidehatDach\,\!^w}$.
      I.e.\
      let
       $\{\widehat{m}_1, \,\cdots\,, \widehat{m}_l\}\in $
         weak-$C^{\infty}$-hull\,$(\widehat{\cal O}_X^{A\!z})
        \cap \widehat{\cal O}_X^{A\!z, \scriptsizewidehatDch\,\!^w}$
     (resp.\
          weak-$C^{\infty}$-hull\,$(\widehat{\cal O}_X^{A\!z})
          \cap \widehat{\cal O}_X^{A\!z, \scriptsizewidehatDach\,\!^w}$)    and
        $h\in C^{\infty}({\Bbb R}^l)$.
      Then
       $h(\widehat{m}_1, \,\cdots\,, \widehat{m}_l)
          \in  \widehat{\cal O}_X^{A\!z, \scriptsizewidehatDch\,\!^w}$
      (resp.\ $\widehat{\cal O}_X^{A\!z, \scriptsizewidehatDach\,\!^w}$).}
  \end{itemize}
}\end{remark}

\bigskip

\begin{flushleft}
{\bf An abstract characterization of
  $\widehat{\cal O}_X^{A\!z, \scriptsizewidehatDch}$ and
  $\widehat{\cal O}_X^{A\!z, \scriptsizewidehatDach}$}
\end{flushleft}
Recall from Lemma~2.3.6
 that
  $\widehat{D}_{e_{\alpha^{\prime}}} \widehat{m}
       =   \,\!^{\varsigma}\!e^{-V} \mbox{\Large $($}
	      e_{\alpha^{\prime}}
		               (e^V \widehat{m}\, \,\!^{\varsigma_{\widehat{m}}}\!e^{-V})
	                                                    \mbox{\Large $)$}\,
							\,\!^{\varsigma_{\widehat{m}}}\!e^V$
  and that
  $\widehat{D}_{e_{\beta^{\prime\prime}}}\widehat{m}
       = e_{\beta^{\prime\prime}}\widehat{m}$,
 for $\alpha^\prime=1^\prime, 2^\prime$	    and $\beta^{\prime\prime}=1^{\prime\prime}, 2^{\prime\prime}$.
In particular,
  $$
    \widehat{D}_{e_{\alpha^{\prime}}} \widehat{m}_{(\even)}
       =   \,\!^{\varsigma}\!e^{-V} \mbox{\Large $($}
	      e_{\alpha^{\prime}}
		               (e^V \widehat{m}_{(\even)}\,e^{-V})
	                                                    \mbox{\Large $)$}\, e^V
       \hspace{1em}\mbox{and}\hspace{1em} 							
    \widehat{D}_{e_{\alpha^{\prime}}} \widehat{m}_{(\odd)}
       =   \,\!^{\varsigma}\!e^{-V} \mbox{\Large $($}
	      e_{\alpha^{\prime}}
		               (e^V \widehat{m}_{(\odd)}\, \,\!^\varsigma\!e^{-V})
	                                                    \mbox{\Large $)$}\, \,\!^\varsigma\!e^V\,.
  $$
It follows that
 $\widehat{D}_{e_{1^\prime}} \widehat{m}_{(\even)}
  = \widehat{D}_{e_{2^\prime}} \widehat{m}_{(\even)}=0$
 if and only if
   $e^V\widehat{m}_{(\even)}\,e^{-V}\in \widehat{\cal O}_X^{A\!z,\scriptsizedach}$.
Similarly,
 $\widehat{D}_{e_{1^{\prime}}} \widehat{m}_{(\odd)}
   = \widehat{D}_{e_{2^{\prime}}} \widehat{m}_{(\odd)}=0$
 if and only if
   $e^V\widehat{m}_{(\odd)}\,\,\!^\varsigma\!e^{-V}
      \in \widehat{\cal O}_X^{A\!z,\scriptsizedach}$.
This proves the following lemma, which gives an abstract characterization of
 the $\widehat{D}$-chiral structure sheaf  $\widehat{\cal O}_X^{A\!z,\scriptsizewidehatDch}$ and
 the $\widehat{D}$-antichiral structure sheaf  $\widehat{\cal O}_X^{A\!z,\scriptsizewidehatDach}$
 of $\widehat{X}^{\!A\!z}$.
  
\bigskip

\begin{lemma} {\bf [characterization of $\widehat{\cal O}_X^{A\!z,\scriptsizewidehatDch}$ and
                                             $\widehat{\cal O}_X^{A\!z,\scriptsizewidehatDach}$]}\;
 As subsheaves of $\widehat{\cal O}_X^{A\!z}$,
  \begin{eqnarray*}
   \widehat{\cal O}_X^{A\!z,\scriptsizewidehatDch}
	  & = &   \widehat{\cal O}_X^{A\!z, \scriptsizedch}\,, \\	
   \widehat{\cal O}_X^{A\!z,\scriptsizewidehatDach}
	  & = &  \mbox{\Large $($}
	               (e^{-V}\widehat{\cal O}_X^{A\!z, \scriptsizedach}\, e^V)
                    \cap \widehat{\cal O}_X^{A\!z, (\even)}	
             \mbox{\Large $)$}\,
		    \oplus\,
		     \mbox{\Large $($}
	               (e^{-V}\widehat{\cal O}_X^{A\!z, \scriptsizedach}\, \,\!^\varsigma\!e^V)
                    \cap \widehat{\cal O}_X^{A\!z, (\odd)}	
             \mbox{\Large $)$}\,.
  \end{eqnarray*}
\end{lemma}

\bigskip

As a consequence of this characterization, one has the following proposition:

\bigskip

\begin{proposition} {\bf [invariance of $\widehat{\cal O}_X^{A\!z,\scriptsizewidehatDch}$ and
       $\widehat{\cal O}_X^{A\!z,\scriptsizewidehatDach}$ under supersymmetry]}\;
  Recall that the reference trivial connection $d$ on $\widehat{\cal O}_X^{A\!z}$ gives a lifting of
    the supersymmetry transformations on $\widehat{X}$ to
	transformations on  $\widehat{\cal O}_X^{A\!z}$.
  Then, with respect to this lifting of supersymmetry transformations,
    both $\widehat{\cal O}_X^{A\!z,\scriptsizewidehatDch}$ and
	$\widehat{\cal O}_X^{A\!z,\scriptsizewidehatDach}$
   are supersymmetrically invariant subsheaves of $\widehat{\cal O}_X^{A\!z}$.
\end{proposition}

\medskip

\begin{proof}
 The reference trivialization that gives the reference connection $d$
   on $\widehat{\cal O}_X^{A\!z}$ specifies the following isomorphisms
  $$
    \begin{array}{c}
     \widehat{\cal O}_X^{A\!z}
      = \widehat{\cal O}_X\otimes_{\Bbb C}M_{r\times r}({\Bbb C}), \;\;
     \widehat{\cal O}_X^{A\!z, (\even)}
      = \widehat{\cal O}_X^{(\even)}\otimes_{\Bbb C}M_{r\times r}({\Bbb C})\,, \;\;
     \widehat{\cal O}_X^{A\!z, (\odd)}
      = \widehat{\cal O}_X^{(\odd)}\otimes_{\Bbb C}M_{r\times r}({\Bbb C})\,,  \\[1.2ex]
     \widehat{\cal O}_X^{A\!z, \scriptsizedch}
      = \widehat{\cal O}_X^{\scriptsizech}\otimes_{\Bbb C}M_{r\times r}({\Bbb C})\,, \;\;
     \widehat{\cal O}_X^{A\!z, \scriptsizedach}
      = \widehat{\cal O}_X^{\scriptsizeach}\otimes_{\Bbb C}M_{r\times r}({\Bbb C})\,,
   \end{array}	
 $$
 From these isomorphisms, one concludes that all
    $\widehat{\cal O}_X^{A\!z, \scriptsizedch}$,
	$\widehat{\cal O}_X^{A\!z,\scriptsizedach }$,
	$\widehat{\cal O}_X^{A\!z, (\even)}$, and
	$\widehat{\cal O}_X^{A\!z, (\odd)}$
	are invariant subsheaves of $\widehat{\cal O}_X^{A\!z}$
	 under the lifted-via-$d$ supersymmetry transformations
  since
    $\widehat{\cal O}_X^{\scriptsizech}$,
	$\widehat{\cal O}_X^{\scriptsizeach }$,
	$\widehat{\cal O}_X^{(\even)}$, and
	$\widehat{\cal O}_X^{(\odd)}$
	are invariant subsheaves of $\widehat{\cal O}_X$ under supersymmetry transformations
	on $\widehat{X}$.
  The proposition now follows from Lemma~3.1.6.
 
\end{proof}

\bigskip

\subsection{Normal form of $\widehat{D}$-chiral sections and $\widehat{D}$-antichiral sections
    of $\widehat{\cal O}_X^{A\!z}$}

A normal form of $\widehat{D}$-chiral sections and $\widehat{D}$-antichiral sections of
 $\widehat{\cal O}_X^{A\!z}$ that generalizes Lemma~1.4.14 and Lemma~1.4.18 is given in this subsection.
We begin with the following four basic formulas from straightforward computations:

\bigskip
	
\begin{lemma} {\bf [basic formula]}\;
 For $s\in \widehat{\cal E}$, let
   $\widehat{\nabla}_{\frac{\partial}{\partial x^\mu}}s
       = \frac{\partial}{\partial x^\mu}s + s a_\mu$,
   $\widehat{\nabla}_{\frac{\partial}{\partial\theta^\alpha}}s
       = \frac{\partial}{\partial\theta^\alpha}s + \,\!^\varsigma\!s b_{\alpha}$, 	   	
   $\widehat{\nabla}_{\frac{\partial}
                                                  {\rule{0ex}{.7em}\partial\bar{\theta}^{\dot{\beta}}}}s 
       = \frac{\partial}{\rule{0ex}{.8em}\partial \bar{\theta}^{\dot{\beta}}}s
	        + \,\!^\varsigma\!s b_{\dot{\beta}}$,
   where $a_{\mu}, b_{\alpha}, b_{\dot{\beta}}\in \widehat{\cal O}_X^{A\!z}$.			
 Then, ($\partial_\mu:= \frac{\partial}{\partial x^\mu}$, $(-1)^{\dot{1}}:= -1, (-1)^{\dot{2}}:=1$)
  
  {\small
  \begin{eqnarray*}
   \lefteqn{ \widehat{D}_{e_{\alpha^\prime}} \widehat{m}_{(\tinyeven)}\;
     =\; \widehat{D}_{\frac{\partial}{\partial \theta^{\alpha}}}
	             \widehat{m}_{(\tinyeven)}\,
	            +\, \sqrt{-1}\,\sum_{\mu, \dot{\beta}}
				       \sigma^\mu_{\alpha\dot{\beta}}\bar{\theta}^{\dot{\beta}}
				       \widehat{D}_{\frac{\partial}{\partial x^\mu}} \widehat{m}_{(\tinyeven)}   }\\
   && =\;
      -\, [b_{\alpha(0)}, m_{(0)}]		
	  -\, \sum_{\gamma}
	        \mbox{\Large $($}
              (-1)^{\alpha}(1-\delta_{\alpha\gamma})m_{(12)}
			   + [b_{\alpha(\gamma)}, m_{(0)}  ]
            \mbox{\Large $)$}\,\theta^{\gamma}\,   	\\					
    && \hspace{1.3em}
	  -\, \sum_{\dot{\delta}}
				     \mbox{\Large $($}
				       m_{(\alpha\dot{\delta})} +  [b_{\alpha(\dot{\gamma})}, m_{(0)}]
                      + \sqrt{-1}\,\mbox{\normalsize $\sum$}_\mu
                            \sigma^\mu_{\alpha\dot{\delta}}
							 \mbox{\large $($}
							  \partial_\mu  m_{(0)} -[a_{\mu(0)}, m_{(0)}]
							 \mbox{\large $)$} 					
				     \mbox{\Large $)$}\, \bar{\theta}^{\dot{\delta}}\,  \\			
    && \hspace{1.3em}		
      -\, \mbox{\Large $($}
             [b_{\alpha(12)}, m_{(0)}] + [b_{\alpha(0)}, m_{(12)}]
           \mbox{\Large $)$}\, \theta^1\theta^2\, \\[.8ex]		
    && \hspace{1.3em}
      +\, \sum_{\gamma,\dot{\delta}}	
	        \mbox{\Large $($}
			 -  [b_{\alpha(\gamma\dot{\delta})}, m_{(0)}]
			 -  [b_{\alpha(0)}, m_{(\gamma\dot{\delta})}]\,
	         + \sqrt{-1}\,\mbox{\normalsize $\sum$}_\mu
				   \sigma^\mu_{\alpha\dot{\delta}}[a_{\mu(\gamma)}, m_{(0)}]
			\mbox{\Large $)$}\, \theta^\gamma \bar{\theta}^{\dot{\delta}} \,    \\			
    && \hspace{1.3em}
     +\, \mbox{\Large $($}
             -  [b_{\alpha(\dot{1}\dot{2})}, m_{(0)}]
			 -  [b_{\alpha(0)}, m_{(\dot{1}\dot{2})}]
		     + \sqrt{-1}\,\mbox{\normalsize $\sum$}_\mu
		        \sigma^\mu_{\alpha\dot{2}} [a_{\mu(\dot{1})}, m_{(0)}]	
            - \sqrt{-1}\,\mbox{\normalsize $\sum$}_\mu
			     \sigma^\mu_{\alpha\dot{1}} [a_{\mu(\dot{2})}, m_{(0)}]
           \mbox{\Large $)$}\, \bar{\theta}^{\dot{1}}\bar{\theta}^{\dot{2}}\,			   \\[.6ex]		
    && \hspace{1.3em}
      +\, \sum_{\dot{\delta}}
             \mbox{\Large $($}
			 -  [b_{\alpha(12\dot{\delta})}, m_{(0)}]
			 + [b_{\alpha(2)}, m_{(1\dot{\delta})}]
			 -  [b_{\alpha(1)}, m_{(2\dot{\delta})}]
			 -  [b_{\alpha(\dot{\delta})}, m_{(12)}]          \\[-1ex]
      && \hspace{5em}			
			 +\sqrt{-1}\,\mbox{\normalsize $\sum$}_\mu
			      \sigma^\mu_{\alpha\dot{\delta}}
				   \mbox{\large $($}
				     \partial_\mu m_{(12)} - [a_{\mu(0)}, m_{(12)}]					 					 
				   \mbox{\large $)$}
			  - \sqrt{-1}\,\mbox{\normalsize $\sum$}_\mu	
			      \sigma^\mu_{\alpha\dot{\delta}}[a_{\mu (12)}, m_{(0)}]
             \mbox{\Large $)$}\, \theta^1\theta^2\bar{\theta}^{\dot{\delta}}\, \\		
    && \hspace{1.3em}	
      +\, \sum_{\gamma}
             \mbox{\Large $($}
			-(-1)^\alpha\,(1-\delta_{\alpha\gamma})\,
			    m_{(12\dot{1}\dot{2})}\\[-1.2ex]
         && \hspace{5em} 			
		      - [b_{\alpha(\gamma\dot{1}\dot{2})}, m_{(0)}]			
		      - [b_{\alpha(\dot{2})}, m_{(\gamma\dot{1})}]
		     + [b_{\alpha(\dot{1})}, m_{(\gamma\dot{2})}]
              - [b_{\alpha(\gamma)}, m_{(\dot{1}\dot{2})}]			 \\[.6ex]
         && \hspace{5em}			
			 + \sqrt{-1}\,\mbox{\normalsize $\sum$}_\mu \sigma^\mu_{\alpha\dot{2}}
			     \mbox{\large $($}
				  \partial_\mu m_{(\gamma\dot{1})}-[a_{\mu(0)}, m_{(\gamma\dot{1})}]
				 \mbox{\large $)$}
			  - \sqrt{-1}\, \mbox{\normalsize $\sum$}_\mu \sigma^\mu_{\alpha\dot{1}}
			       \mbox{\large $($}
				    \partial_\mu m_{(\gamma\dot{2})}- [a_{\mu(0)}, m_{(\gamma\dot{2})}]
				   \mbox{\large $)$}   \\[.6ex]
         && \hspace{14em}				
		      - \sqrt{-1}\,\mbox{\normalsize $\sum$}_\mu
			         \sigma^\mu_{\alpha\dot{2}}[a_{\mu (\gamma\dot{1})}, m_{(0)}] 	
			 + \sqrt{-1}\,\mbox{\normalsize $\sum$}_\mu
				     \sigma^\mu_{\alpha\dot{1}}[a_{\mu (\gamma\dot{2})}, m_{(0)}]
			 \mbox{\Large $)$}\,\theta^\gamma\bar{\theta}^{\dot{1}}\bar{\theta}^{\dot{2}}\, \\[.6ex]	
    && \hspace{1.3em}
	    +\, \mbox{\Large $($}
		      - [b_{\alpha(12\dot{1}\dot{2})}, m_{(0)}]
			  - [b_{\alpha(\dot{1}\dot{2})} , m_{(12)}]
			 + [b_{\alpha(2\dot{2})} , m_{(1\dot{1})}]
			  - [b_{\alpha(2\dot{1})}, m_{(1\dot{2})}]   \\
       && \hspace{5em}			
			  - [b_{\alpha(1\dot{2})}, m_{(2\dot{1})}]
			 + [b_{\alpha(1\dot{1})}, m_{(2\dot{2})}]
			  - [b_{\alpha(12)} , m_{(\dot{1}\dot{2})}]
			  - [b_{\alpha(0)} , m_{(12\dot{1}\dot{2})}]  \\
       && \hspace{5em}
             + \sqrt{-1}\,\mbox{\normalsize $\sum$}_\mu
                 \sigma^\mu_{\alpha\dot{1}}
                   \mbox{\large $($}
                      - [a_{\mu (12\dot{2})}, m_{(0)}]
                      - [a_{\mu(\dot{2})}, m_{(12)}]					
					 + [a_{\mu (2)} , m_{(1\dot{2})}]
					  - [a_{\mu(1)}  , m_{(2\dot{2})}]					
                   \mbox{\large $)$} \\	
       && \hspace{5em}
	         + \sqrt{-1}\,\mbox{\normalsize $\sum$}_\mu
                 \sigma^\mu_{\alpha\dot{2}}
                   \mbox{\large $($}
                        [a_{\mu (12\dot{1})}  , m_{(0)}]
                     + [a_{\mu(\dot{1})},  m_{(12)}]
					  - [a_{\mu (2)} , m_{(1\dot{1})}]
					 + [a_{\mu(1)}  , m_{(2\dot{1})}]					
                   \mbox{\large $)$}
			  \mbox{\Large $)$}\, \theta^1\theta^2\bar{\theta}^{\dot{1}}\bar{\theta}^{\dot{2}}\,,
  \end{eqnarray*}} 

  {\small
  \begin{eqnarray*}
   \lefteqn{ \widehat{D}_{e_{\alpha^\prime}} \widehat{m}_{(\tinyodd)}\;
     =\; \widehat{D}_{\frac{\partial}{\partial \theta^{\alpha}}}
	             \widehat{m}_{(\tinyodd)}\,
	            +\, \sqrt{-1}\,\sum_{\mu, \dot{\beta}}
				       \sigma^\mu_{\alpha\dot{\beta}}\bar{\theta}^{\dot{\beta}}
				       \widehat{D}_{\frac{\partial}{\partial x^\mu}} \widehat{m}_{(\tinyodd)}   }\\
   &&  =\; m_{(\alpha)}\,
      -\, \sum_\gamma [b_{\alpha(0)}, m_{(\gamma)}]\, \theta^{\gamma}\,
	  -\, \sum_{\dot{\delta}}
		         [b_{\alpha(0)} , m_{(\dot{\delta})}]\,    \bar{\theta}^{\dot{\delta}}\\
   && \hspace{1.3em}
     +\, \mbox{\Large $($}
	           [b_{\alpha(2)}, m_{(1)}]
			 - [b_{\alpha(1)} , m_{(2)}]
           \mbox{\Large $)$}\, \theta^1\theta^2\,  	  \\[.6ex]				
   && \hspace{1.3em}
     +\, \sum_{\gamma, \dot{\delta}}
            \mbox{\Large $($}
		     -(-1)^{\alpha}(1-\delta_{\alpha\gamma})\,
			        m_{12\dot{\delta}}
			 + [b_{\alpha(\dot{\delta})} , m_{(\gamma)}]
			  - [b_{\alpha(\gamma)}  , m_{(\dot{\delta})}]        \\[-2ex]
       && \hspace{16em}
             - \sqrt{-1}\,\mbox{\normalsize $\sum$}_\mu
			      \sigma^\mu_{\alpha\dot{\delta}}
				    \mbox{\large $($}
                   	 \partial_\mu m_{(\gamma)} - [a_{\mu(0)}, m_{(\gamma)}]
					\mbox{\large $)$}
		    \mbox{\Large $)$}\, \theta^{\gamma}\bar{\theta}^{\dot{\delta}}\, \\ 	
   && \hspace{1.3em}		
	 +\, \mbox{\Large $($}
                 m_{(\alpha\dot{1}\dot{2})}
                 + [b_{\alpha(\dot{2})} , m_{(\dot{1})}]
                  - [b_{\alpha(\dot{1})} , m_{(\dot{2})}] 				   \\
       && \hspace{5em}				
				 - \sqrt{-1}\, \mbox{\normalsize $\sum$}_\mu
				      \sigma^\mu_{\alpha\dot{2}}
					   \mbox{\large $($}
					     \partial_\mu m_{(\dot{1})}- [a_{\mu(0)} , m_{(\dot{1})}]
					   \mbox{\large $)$}
                 + \sqrt{-1}\, \mbox{\normalsize $\sum$}_\mu
				      \sigma^\mu_{\alpha\dot{1}}
					   \mbox{\large $($}
					     \partial_\mu m_{(\dot{2})}- [a_{\mu(0)} , m_{(\dot{2})}]
					   \mbox{\large $)$}	
		        \mbox{\Large $)$}\,\bar{\theta}^{\dot{1}}\bar{\theta}^{\dot{2}}\, \\		
   && \hspace{1.3em}
     +\,\sum_{\dot{\delta}}
	       \mbox{\Large $($}
		     -  [b_{\alpha(2\dot{1})}, m_{(1)}]
			 + [b_{\alpha(\dot{1}\dot{1})}  , m_{(2)}]
			 -  [b_{\alpha(12)}  , m_{(\dot{\delta})}]
			 -  [b_{\alpha(0)}, m_{(12\dot{\delta})}]     \\[-2ex]
       && \hspace{16em}
          + \sqrt{-1}\,\mbox{\normalsize $\sum$}_\mu	
		       \sigma^\mu_{\alpha\dot{\delta}}
			    \mbox{\large $($}
				   [a_{\mu(2)}, m_{(1)}]
				 - [a_{\mu(1)}, m_{(2)}]
			    \mbox{\large $)$}
           \mbox{\Large $)$}\, \theta^1\theta^2\bar{\theta}^{\dot{\delta}}\\			
   && \hspace{1.3em}
     +\, \sum_{\gamma}
	       \mbox{\Large $($}
		      - [b_{\alpha(\dot{1}\dot{2})}, m_{(\gamma)}]
			 + [b_{\alpha(\gamma\dot{2})}, m_{(\dot{1})}]
			  - [b_{\alpha(\gamma\dot{1})} , m_{(\dot{2})}]
			  - [b_{\alpha(0)} , m_{(\gamma\dot{1}\dot{2})}]                 \\[-1.2ex]
       && \hspace{5em}			
	         +\sqrt{-1}\,\mbox{\normalsize $\sum$}_\mu
			     \sigma^\mu_{\alpha\dot{1}}
				  \mbox{\large $($}
				  - [a_{\mu(\dot{2})} , m_{(\gamma)}]
				  +[a_{\mu(\gamma)} , m_{(\dot{2})}]
				  \mbox{\large $)$}                                                  \\
       && \hspace{16em}				
             +\sqrt{-1}\,\mbox{\normalsize $\sum$}_\mu
			     \sigma^\mu_{\alpha\dot{2}}
				  \mbox{\large $($}
				     [a_{\mu(\dot{1})} , m_{(\gamma)}]
				   - [a_{\mu(\gamma)} , m_{(\dot{1})}]
				  \mbox{\large $)$}				
           \mbox{\Large $)$}\, \theta^\gamma\bar{\theta}^{\dot{1}}\bar{\theta}^{\dot{2}}\, \\
   && \hspace{1.3em}
     +\, \mbox{\Large $($}
	            [b_{\alpha(2\dot{1}\dot{2})}, m_{(1)}]
			 -  [b_{\alpha(1\dot{1}\dot{2})}, m_{(2)}]
			 + [b_{\alpha(12\dot{2})}, m_{(\dot{1})}]
			 -  [b_{\alpha(12\dot{1})}, m_{(\dot{2})}]    \\
       && \hspace{3em}
             + [b_{\alpha(\dot{2})}, m_{(12\dot{1})}]	
			 -  [b_{\alpha(\dot{1})}, m_{(12\dot{2})}]
			 + [b_{\alpha(2)}, m_{(1\dot{1}\dot{2})}]
			 -  [b_{\alpha(1)}, m_{(2\dot{1}\dot{2})}]    \\
       && \hspace{3em}
             -  \sqrt{-1}\, \mbox{\normalsize $\sum$}_\mu
			      \sigma^\mu_{\alpha\dot{2}}
				   \mbox{\large $($}
				    \partial_\mu m_{(12\dot{1})}
					   - [a_{\mu(0)} , m_{(12\dot{1})}]
				   \mbox{\large $)$}
             +  \sqrt{-1}\, \mbox{\normalsize $\sum$}_\mu
			      \sigma^\mu_{\alpha\dot{1}}
				   \mbox{\large $($}
				    \partial_\mu m_{(12\dot{2})}
					   - [a_{\mu(0)} , m_{(12\dot{2})}]
				   \mbox{\large $)$}	    \\
        && \hspace{3em}
          + \sqrt{-1}\, \mbox{\normalsize $\sum$}_\mu		
		      \sigma^\mu_{\alpha\dot{2}}
			   \mbox{\large $($}
			       [a_{\mu(2\dot{1})}, m_{(1)}]
				 - [a_{\mu(1\dot{1})}, m_{(\dot{2})}]
				 +[a_{\mu(12)}, m_{(\dot{\dot{1}})}]
			   \mbox{\large $)$}     \\
       && \hspace{10em}
          + \sqrt{-1}\, \mbox{\normalsize $\sum$}_\mu		
		      \sigma^\mu_{\alpha\dot{1}}
			   \mbox{\large $($}
			      - [a_{\mu(2\dot{2})}, m_{(1)}]
				  +[a_{\mu(1\dot{2})}, m_{(2)}]
				  - [a_{\mu(12)}, m_{(\dot{2})}]
			   \mbox{\large $)$}
	       \mbox{\Large $)$}\, \theta^1\theta^2\bar{\theta}^{\dot{1}}\bar{\theta}^{\dot{2}}\,,
  \end{eqnarray*}} 
 
  {\small
  \begin{eqnarray*}
   \lefteqn{ -\,\widehat{D}_{e_{\beta^{\prime\prime}}} \widehat{m}_{(\tinyeven)}\;
     =\; \widehat{D}_{\frac{\partial}{\rule{0ex}{.7em}\partial\bar{\theta}^{\dot{\beta}}}}
	             \widehat{m}_{(\tinyeven)}\,
	            +\, \sqrt{-1}\,\sum_{\mu, \alpha}\theta^\alpha  \sigma^\mu_{\alpha\dot{\beta}}
				       \widehat{D}_{\frac{\partial}{\partial x^\mu}} \widehat{m}_{(\tinyeven)}   }\\
    && =\;
      -\, [b_{\dot{\beta}(0)}, m_{(0)}]		
                - 	\sum_\gamma
				     \mbox{\Large $($}
				       m_{(\gamma\dot{\beta})} +  [b_{\dot{\beta}(\gamma)}, m_{(0)}]
                      + \sqrt{-1}\,\mbox{\normalsize $\sum$}_\mu
                            \sigma^\mu_{\gamma\dot{\beta}}
							 \mbox{\large $($}
							  \partial_\mu  m_{(0)} -[a_{\mu(0)}, m_{(0)}]
							 \mbox{\large $)$} 					
				     \mbox{\Large $)$}\, \theta^\gamma\,  \\
    && \hspace{1.3em}
      -\, \sum_{\dot{\delta}}
	        \mbox{\Large $($}
              (-1)^{\dot{\beta}}(1-\delta_{\dot{\beta}\dot{\delta}})m_{(\dot{1}\dot{2})}
			   + [b_{\dot{\beta}(\dot{\delta})}, m_{(0)}  ]
            \mbox{\Large $)$}\,\bar{\theta}^{\dot{\delta}}\,   	\\
    && \hspace{1.3em}			
      +\, \mbox{\Large $($}
            - [b_{\dot{\beta}(12)}, m_{(0)}] - [b_{\dot{\beta}(0)}, m_{(12)}]
		   + \sqrt{-1}\,\mbox{\normalsize $\sum$}_\mu
		        \sigma^\mu_{2\dot{\beta}} [a_{\mu(1)}, m_{(0)}]	
            - \sqrt{-1}\,\mbox{\normalsize $\sum$}_\mu
			     \sigma^\mu_{1\dot{\beta}} [a_{\mu(2)}, m_{(0)}]
           \mbox{\Large $)$}\, \theta^1\theta^2\,			   \\[.6ex]
    && \hspace{1.3em}
      -\, \sum_{\gamma,\dot{\delta}}	
	        \mbox{\Large $($}
			 [b_{\dot{\beta}(\gamma\dot{\delta})}, m_{(0)}]
			 + [b_{\dot{\beta}(0)}, m_{(\gamma\dot{\delta})}]\,
	         + \sqrt{-1}\,\mbox{\normalsize $\sum$}_\mu
				   \sigma^\mu_{\gamma\dot{\beta}}[a_{\mu(\dot{\delta})}, m_{(0)}]
			\mbox{\Large $)$}\, \theta^\gamma \bar{\theta}^{\dot{\delta}} \,    \\
    && \hspace{1.3em}
      -\, \mbox{\Large $($}
             [b_{\dot{\beta}(\dot{1}\dot{2})}, m_{(0)}]
			 + [b_{\dot{\beta}(0)}, m_{(\dot{1}\dot{2})}]
           \mbox{\Large $)$}\, \bar{\theta}^{\dot{1}}	  \bar{\theta}^{\dot{2}}\, \\[.8ex]
    && \hspace{1.3em}
      +\, \sum_{\dot{\delta}}
             \mbox{\Large $($}
			-(-1)^{\dot{\beta}}\,(1-\delta_{\dot{\beta}\dot{\delta}})\,
			    m_{(12\dot{1}\dot{2})}\\[-1.2ex]
         && \hspace{5em} 			
		    - [b_{\dot{\beta}(12\dot{\delta})}, m_{(0)}]
			- [b_{\dot{\beta}(\dot{\delta})}, m_{(12)}]
		   + [b_{\dot{\beta}(2)}, m_{(1\dot{\delta})}]
		    - [b_{\dot{\beta}(1)}, m_{(2\dot{\delta})}]      \\[.6ex]
         && \hspace{5em}			
			  - \sqrt{-1}\,\mbox{\normalsize $\sum$}_\mu \sigma^\mu_{2\dot{\beta}}
			     \mbox{\large $($}
				  \partial_\mu m_{(1\dot{\delta})}-[a_{\mu(0)}, m_{(1\dot{\delta})}]
				 \mbox{\large $)$}
			  + \sqrt{-1}\, \mbox{\normalsize $\sum$}_\mu \sigma^\mu_{1\dot{\beta}}
			       \mbox{\large $($}
				    \partial_\mu m_{(2\dot{\delta})}- [a_{\mu(0)}, m_{(2\dot{\delta})}]
				   \mbox{\large $)$}   \\[.6ex]
         && \hspace{14em}				
		       + \sqrt{-1}\,\mbox{\normalsize $\sum$}_\mu
			         \sigma^\mu_{2\dot{\beta}}[a_{\mu (1\dot{\delta})}, m_{(0)}] 			
			    - \sqrt{-1}\,\mbox{\normalsize $\sum$}_\mu
				     \sigma^\mu_{1\dot{\beta}}[a_{\mu (2\dot{\delta})}, m_{(0)}]
			 \mbox{\Large $)$}\,\theta^1\theta^2\bar{\theta}^{\dot{\delta}}\,          \\[.6ex]	
    && \hspace{1.3em}
      +\, \sum_{\gamma}
             \mbox{\Large $($}
			  - [b_{\dot{\beta}(\gamma\dot{1}\dot{2})}, m_{(0)}]
			  - [b_{\dot{\beta}(\dot{2})}, m_{(\gamma\dot{1})}]
			 + [b_{\dot{\beta}(\dot{1})}, m_{(\gamma\dot{2})}]
			  - [b_{\dot{\beta}(\gamma)}, m_{(\dot{1}\dot{2})}]          \\[-1ex]
      && \hspace{5em}			
			 +\sqrt{-1}\,\mbox{\normalsize $\sum$}_\mu
			      \sigma^\mu_{\gamma\dot{\beta}}
				   \mbox{\large $($}
				     \partial_\mu m_{(\dot{1}\dot{2})} - [a_{\mu(0)}, m_{(\dot{1}\dot{2})}]
				   \mbox{\large $)$}
			  - \sqrt{-1}\,\mbox{\normalsize $\sum$}_\mu	
			      \sigma^\mu_{\gamma\dot{\beta}}[a_{\mu (\dot{1}\dot{2})}, m_{(0)}]
             \mbox{\Large $)$}\, \theta^\gamma\bar{\theta}^{\dot{1}}\bar{\theta}^{\dot{2}}\, \\
    && \hspace{1.3em}
	    +\, \mbox{\Large $($}
		      - [b_{\dot{\beta}(12\dot{1}\dot{2})}, m_{(0)}]
			  - [b_{\dot{\beta}(\dot{1}\dot{2})} , m_{(12)}]
			 + [b_{\dot{\beta}(2\dot{2})} , m_{(1\dot{1})}]
			  - [b_{\dot{\beta}(2\dot{1})}, m_{(1\dot{2})}]   \\
       && \hspace{5em}			
			  - [b_{\dot{\beta}(1\dot{2})}, m_{(2\dot{1})}]
			 + [b_{\dot{\beta}(1\dot{1})}, m_{(2\dot{2})}]
			  - [b_{\dot{\beta}(12)} , m_{(\dot{1}\dot{2})}]
			  - [b_{\dot{\beta}(0)} , m_{(12\dot{1}\dot{2})}]  \\
       && \hspace{5em}
             + \sqrt{-1}\,\mbox{\normalsize $\sum$}_\mu
                 \sigma^\mu_{1\dot{\beta}}
                   \mbox{\large $($}
                     - [a_{\mu (2\dot{1}\dot{2})}  , m_{(0)}]
					 - [a_{\mu (\dot{2})} , m_{(2\dot{1})}]
					+ [a_{\mu(\dot{1})}  , m_{(2\dot{2})}]
					 - [a_{\mu(2)},  m_{(\dot{1}\dot{2})}]
                   \mbox{\large $)$} \\	
       && \hspace{5em}
	         + \sqrt{-1}\,\mbox{\normalsize $\sum$}_\mu
                 \sigma^\mu_{2\dot{\beta}}
                   \mbox{\large $($}
                       [a_{\mu (1\dot{1}\dot{2})}  , m_{(0)}]
					+ [a_{\mu (\dot{2})} , m_{(1\dot{1})}]
					 - [a_{\mu(\dot{1})}  , m_{(1\dot{2})}]
					 + [a_{\mu(1)},  m_{(\dot{1}\dot{2})}]
                   \mbox{\large $)$}
			  \mbox{\Large $)$}\, \theta^1\theta^2\bar{\theta}^{\dot{1}}\bar{\theta}^{\dot{2}}\,,
  \end{eqnarray*}} 
  
 {\small
  \begin{eqnarray*}
   \lefteqn{ -\,\widehat{D}_{e_{\beta^{\prime\prime}}} \widehat{m}_{(\tinyodd)}\;
     =\; \widehat{D}_{\frac{\partial}{\rule{0ex}{.7em}\partial\bar{\theta}^{\dot{\beta}}}}
	             \widehat{m}_{(\tinyodd)}\,
	            +\, \sqrt{-1}\,\sum_{\mu, \alpha}\theta^\alpha  \sigma^\mu_{\alpha\dot{\beta}}
				       \widehat{D}_{\frac{\partial}{\partial x^\mu}} \widehat{m}_{(\tinyodd)}   }\\
   &&  =\; m_{(\dot{\beta})}\,
      -\, \sum_\gamma [b_{\dot{\beta}(0)}, m_{(\gamma)}]\, \theta^{\gamma}\,
	  -\, \sum_{\dot{\delta}}
		         [b_{\dot{\beta}(0)} , m_{(\dot{\delta})}]\,    \bar{\theta}^{\dot{\delta}}\\
   && \hspace{1.3em}		
	 +\, \mbox{\Large $($}
                 m_{(12\dot{\beta})}
                 + [b_{\dot{\beta}(2)} , m_{(1)}]
                  - [b_{\dot{\beta}(1)} , m_{(2)}] 				   \\
       && \hspace{5em}				
				 - \sqrt{-1}\, \mbox{\normalsize $\sum$}_\mu
				      \sigma^\mu_{2\dot{\beta}}
					   \mbox{\large $($}
					     \partial_\mu m_{(1)}- [a_{\mu(0)} , m_{(1)}]
					   \mbox{\large $)$}
                 + \sqrt{-1}\, \mbox{\normalsize $\sum$}_\mu
				      \sigma^\mu_{1\dot{\beta}}
					   \mbox{\large $($}
					     \partial_\mu m_{(2)}- [a_{\mu(0)} , m_{(2)}]
					   \mbox{\large $)$}	
		        \mbox{\Large $)$}\,\theta^1\theta^2\,   \\
   && \hspace{1.3em}
     +\, \sum_{\gamma, \dot{\delta}}
            \mbox{\Large $($}
		     (-1)^{\dot{\beta}}(1-\delta_{\dot{\beta}\dot{\delta}})\,
			        m_{\gamma\dot{1}\dot{2}}
			 + [b_{\dot{\beta}(\dot{\delta})} , m_{(\gamma)}]
			  - [b_{\dot{\beta}(\gamma)}  , m_{(\dot{\delta})}]        \\[-2ex]
       && \hspace{16em}
             + \sqrt{-1}\,\mbox{\normalsize $\sum$}_\mu
			      \sigma^\mu_{\gamma\dot{\beta}}
				    \mbox{\large $($}
                   	 \partial_\mu m_{(\dot{\delta})} - [a_{\mu(0)}, m_{(\dot{\delta})}]
					\mbox{\large $)$}
		    \mbox{\Large $)$}\, \theta^{\gamma}\bar{\theta}^{\dot{\delta}}\, \\ 	
   && \hspace{1.3em}
     +\, \mbox{\Large $($}
	           [b_{\dot{\beta}(\dot{2})}, m_{(\dot{1})}]
			 - [b_{\dot{\beta}(\dot{1})} , m_{(\dot{2})}]
           \mbox{\Large $)$}\, \bar{\theta}^{\dot{1}}\bar{\theta}^{\dot{2}}\,  	  \\
   && \hspace{1.3em}
     +\, \sum_{\dot{\delta}}
	       \mbox{\Large $($}
		     -  [b_{\dot{\beta}(2\dot{\delta})} , m_{(1)}]
			 + [b_{\dot{\beta}(1\dot{\delta})}  , m_{(2)}]
			  - [b_{\dot{\beta}(12)} , m_{(\dot{\delta})}]
			  - [b_{\dot{\beta}(0)} , m_{(12\dot{\delta})}]                 \\
       && \hspace{5em}			
	         +\sqrt{-1}\,\mbox{\normalsize $\sum$}_\mu
			     \sigma^\mu_{1\dot{\beta}}
				  \mbox{\large $($}
				    [a_{\mu(\dot{\delta})} , m_{(2)}]
				  - [a_{\mu(2)} , m_{(\dot{\delta})}]
				  \mbox{\large $)$}                                                  \\
       && \hspace{16em}				
             +\sqrt{-1}\,\mbox{\normalsize $\sum$}_\mu
			     \sigma^\mu_{2\dot{\beta}}
				  \mbox{\large $($}
				  -  [a_{\mu(\dot{\delta})} , m_{(1)}]
				  + [a_{\mu(1)} , m_{(\dot{\delta})}]
				  \mbox{\large $)$}				
           \mbox{\Large $)$}\, \theta^1\theta^2\bar{\theta}^{\dot{\delta}}\,		        \\
   && \hspace{1.3em}
     +\,\sum_{\gamma}
	       \mbox{\Large $($}
		     -  [b_{\dot{\beta}(\dot{1}\dot{2})}, m_{(\gamma)}]
			 + [b_{\dot{\beta}(\gamma\dot{2})}  , m_{(\dot{1})}]
			 -  [b_{\dot{\beta}(\gamma\dot{1})}  , m_{(\dot{2})}]
			 -  [b_{\dot{\beta}(0)}, m_{(\gamma\dot{1}\dot{2})}]     \\[-2ex]
       && \hspace{16em}
          + \sqrt{-1}\,\mbox{\normalsize $\sum$}_\mu	
		       \sigma^\mu_{\gamma\dot{\beta}}
			    \mbox{\large $($}
				   [a_{\mu(\dot{2})}, m_{(\dot{1})}]
				 - [a_{\mu(\dot{1})}, m_{(\dot{2})}]
			    \mbox{\large $)$}
           \mbox{\Large $)$}\, \theta^\gamma\bar{\theta}^{\dot{1}}\bar{\theta}^{\dot{2}}\\		
   && \hspace{1.3em}
     +\, \mbox{\Large $($}
	            [b_{\dot{\beta}(2\dot{1}\dot{2})}, m_{(1)}]
			 -  [b_{\dot{\beta}(1\dot{1}\dot{2})}, m_{(2)}]
			 + [b_{\dot{\beta}(12\dot{2})}, m_{(\dot{1})}]
			 -  [b_{\dot{\beta}(12\dot{1})}, m_{(\dot{2})}]    \\
       && \hspace{3em}
             + [b_{\dot{\beta}(\dot{2})}, m_{(12\dot{1})}]	
			 -  [b_{\dot{\beta}(\dot{1})}, m_{(12\dot{2})}]
			 + [b_{\dot{\beta}(2)}, m_{(1\dot{1}\dot{2})}]
			 -  [b_{\dot{\beta}(1)}, m_{(2\dot{1}\dot{2})}]    \\
       && \hspace{3em}
             -  \sqrt{-1}\, \mbox{\normalsize $\sum$}_\mu
			      \sigma^\mu_{2\dot{\beta}}
				   \mbox{\large $($}
				    \partial_\mu m_{(1\dot{1}\dot{2})}
					   - [a_{\mu(0)} , m_{(1\dot{1}\dot{2})}]
				   \mbox{\large $)$}
             +  \sqrt{-1}\, \mbox{\normalsize $\sum$}_\mu
			      \sigma^\mu_{1\dot{\beta}}
				   \mbox{\large $($}
				    \partial_\mu m_{(2\dot{1}\dot{2})}
					   - [a_{\mu(0)} , m_{(2\dot{1}\dot{2})}]
				   \mbox{\large $)$}	    \\
        && \hspace{3em}
		 + \sqrt{-1}\, \mbox{\normalsize $\sum$}_\mu		
		      \sigma^\mu_{2\dot{\beta}}
			   \mbox{\large $($}
			        [a_{\mu(\dot{1}\dot{2})}, m_{(1)}]
				 -  [a_{\mu(1\dot{2})}, m_{(\dot{1})}]
				 + [a_{\mu(1\dot{1})}, m_{(\dot{2})}]
			   \mbox{\large $)$}         \\
       && \hspace{10em}	
           + \sqrt{-1}\, \mbox{\normalsize $\sum$}_\mu		
		      \sigma^\mu_{1\dot{\beta}}
			   \mbox{\large $($}
			     -  [a_{\mu(\dot{1}\dot{2})}, m_{(2)}]
				 + [a_{\mu(2\dot{2})}, m_{(\dot{1})}]
				 -  [a_{\mu(2\dot{1})}, m_{(\dot{2})}]
			   \mbox{\large $)$}      			   			
	       \mbox{\Large $)$}\, \theta^1\theta^2\bar{\theta}^{\dot{1}}\bar{\theta}^{\dot{2}}\,,					   
   \end{eqnarray*}} 
\end{lemma}

\noindent
for $\alpha^\prime=1^\prime, 2^\prime$,\, $\beta^{\prime\prime}=1^{\prime\prime}, 2^{\prime\prime}$.

\bigskip
					
The main goal of this subsection is to prove the following proposition
 that generalizes, e.g., [We-B: Chap.\ V, Eqs. (5.3) \& (5.5)]  of Wess and Bagger
 to the current situation:

\bigskip

\begin{proposition} {\bf [normal form of $\widehat{D}$-chiral/$\widehat{D}$-antichiral sections
                of $\widehat{\cal O}_X^{A\!z}$]}\;
 Continuing the notations from Lemma~3.2.1.
 Recall that
  $\widehat{D}$ is the induced connection on
    $\widehat{\cal O}_X^{A\!z}:= \Endsheaf_{\widehat{\cal O}_X}(\widehat{\cal E})$
	from a SUSY-rep compatible simple hybrid connection  $\widehat{\nabla}$ on $\widehat{\cal E}$
	that is associated to a vector superfield in Wess-Zumino gauge on $\widehat{X}$.
 Then,
  in terms of the standard coordinate functions $(x,\theta,\bar{\theta})$ on $\widehat{X}$,
  a $\widehat{D}$-chiral section $\widehat{m}$ of $\widehat{\cal O}_X^{A\!z}$
  is determined by four of its components, $m_{(0)}, m_{(1)}, m_{(2)}$, and $m_{(12)}$,
  in the following form
  {\small
   \begin{eqnarray*}
    \widehat{m} & =
	  & m_{(0)}\,+\, \sum_{\gamma}m_{(\gamma)}\theta^\gamma\,
	      +\, m_{(12)}\theta^1\theta^2 \\
     && +\, \sum_{\gamma, \dot{\delta}}
	            \mbox{\Large $($}
				 - [b_{\dot{\delta}(\gamma)}, m_{(0)}]
				 +\sqrt{-1}\,\mbox{\normalsize $\sum$}_\mu
				     \sigma^\mu_{\gamma\dot{\delta}} D_{\partial_\mu} m_{(0)}
				\mbox{\Large $)$}\, \theta^\gamma\bar{\theta}^{\dot{\delta}}\,
           +\, [b_{\dot{1}(\dot{2})}, m_{(0)}]\,
	                \bar{\theta}^{\dot{1}}\bar{\theta}^{\dot{2}}\, \\
     && +\, \sum_{\dot{\delta}}		
                \mbox{\Large $($}
                 -  [b_{\dot{\delta}(2)}, m_{(1)}]
				 + [b_{\dot{\delta}(1)}, m_{(2)}]
				 + \sqrt{-1}\,\mbox{\normalsize $\sum$}_\mu
				     \sigma^\mu_{2\dot{\delta}}D_{\partial_\mu}m_{(1)}
				 - \sqrt{-1}\,\mbox{\normalsize $\sum$}_\mu	
				     \sigma^\mu_{1\dot{\delta}}D_{\partial_\mu}m_{(2)}
                \mbox{\Large $)$}\, \theta^1\theta^2\bar{\theta}^{\dot{\delta}}\, \\				
     && +\, \sum_{\gamma} [b_{\dot{1}(\dot{2})}, m_{(\gamma)}]\,
	               \theta^\gamma\bar{\theta}^{\dot{1}}\bar{\theta}^{\dot{2}}\\	
     &&	+\, \mbox{\Large $($}
	 	      -  [b_{\dot{2}(\dot{1})}, m_{(12)}]
			  -  [b_{\dot{2}(12\dot{1})} , m_{(0)}]
			  -  [b_{\dot{2}(2)}, [b_{\dot{1}(1)}, m_{(0)}]]
			  + [b_{\dot{2}(1)}  , [b_{\dot{1}(2)}, m_{(0)}]]  \\[.6ex]
       && \hspace{1.7em}
	          + \sqrt{-1}\,\mbox{\normalsize $\sum$}_\mu			
	               \sigma^\mu_{2\dot{2}}[D_{\partial_\mu}b_{\dot{1}(1)}, m_{(0)}]
			  - \sqrt{-1}\,\mbox{\normalsize $\sum$}_\mu			
	               \sigma^\mu_{1\dot{2}}[D_{\partial_\mu}b_{\dot{1}(2)}, m_{(0)}]
			  +\sqrt{-1}\,\mbox{\normalsize $\sum$}_{\mu}
                   \sigma^\mu_{1\dot{1}}[b_{\dot{2}(2)} ,  D_{\partial_\mu}m_{(0)}]  \\[.6ex]
       && \hspace{1.7em}
              -  \sqrt{-1}\,\mbox{\normalsize $\sum$}_{\mu}
                   \sigma^\mu_{1\dot{2}}[b_{\dot{1}(2)} ,  D_{\partial_\mu}m_{(0)}]
              -  \sqrt{-1}\,\mbox{\normalsize $\sum$}_{\mu}
                   \sigma^\mu_{2\dot{1}}[b_{\dot{2}(1)} ,  D_{\partial_\mu}m_{(0)}]
             + \sqrt{-1}\,\mbox{\normalsize $\sum$}_{\mu}
                   \sigma^\mu_{2\dot{2}}[b_{\dot{1}(1)} ,  D_{\partial_\mu}m_{(0)}] 	\\[.6ex]				   
       && \hspace{1.7em}
	          + \sqrt{-1}\,\mbox{\normalsize $\sum$}_{\mu}
                   \sigma^\mu_{2\dot{2}}[a_{\mu(1\dot{1})} ,  m_{(0)}]
              -  \sqrt{-1}\,\mbox{\normalsize $\sum$}_{\mu}
                   \sigma^\mu_{1\dot{2}}[a_{\mu(2\dot{1})} ,  m_{(0)}] \\[.6ex] 				   
        &&	\hspace{1.7em}
       	      -  [F^{\nabla}_{03}+\sqrt{-1}F^{\nabla}_{12}, m_{(0)}]
			  + \square^D m_{(0)}
                  \mbox{\Large $)$}\,
				  \theta^1\theta^2\bar{\theta}^{\dot{1}}\bar{\theta}^{\dot{2}}\,.	
   \end{eqnarray*}}Here, 
   \begin{itemize}
    \item[\LARGE $\cdot$]  	
     $D$ is the restriction of the connection $\widehat{D}$ to ${\cal O}_X^{A\!z}$,
	  defined by $D_{\partial_\mu}m= \partial_\mu m - [a_{\mu(0)}, m]$\\
	  for $m= m_{(\cdot)},  b_{\cdot(\cdot)}  \in {\cal O}_X^{A\!z}$,
	
	\item[\LARGE $\cdot$]
	 $\square^D
	     := D_{\partial_0}D_{\partial_0}-D_{\partial_1}D_{\partial_1}
              - D_{\partial_2}D_{\partial_2}-D_{\partial_3}D_{\partial_3}$,
			
    \item[\LARGE $\cdot$]	
     $\nabla$ is the restriction of the connection $\widehat{\nabla}$  to  ${\cal E}$,
     defined by $\nabla_{\partial_\mu}s= \partial_\mu s + sa_{\mu(0)}$, 	
	 $F^\nabla$ is the curvature $2$-tensor of $\nabla$ and
	 $F^\nabla_{\mu\nu}
	   := F^\nabla(\partial_\mu, \partial_\nu)
	    = [\nabla_{\partial_\mu}, \nabla_{\partial_\nu}]\in {\cal O}_X^{A\!z}$.
   \end{itemize}
  The four components
   $m_{(0)}$, $m_{(1)}$, $m_{(2)}$, $m_{(12)}\in {\cal O}_X^{A\!z} $
   can be arbitrary (i.e.\ not subject to any constraints),
   with $m_{(0)}$ and $m_{(12)}$ determining $\widehat{m}_{(\even)}$ and
           $m_{(1)}$ and $m_{(2)}$ determining $\widehat{m}_{(\odd)}$.
      
 Similarly,  a $\widehat{D}$-antichiral section $\widehat{m}$ of $\widehat{\cal O}_X^{A\!z}$
  is determined by four of its components,
    $m_{(0)}$, $m_{(\dot{1})}$, $m_{(\dot{2})}$, and $m_{(\dot{1}\dot{2})}$,
  in the following form
  {\small
   \begin{eqnarray*}
    \widehat{m} & =
	  & m_{(0)}\,
	      +\, \sum_{\dot{\delta}}m_{(\dot{\delta})}\bar{\theta}^{\dot{\delta}}\,
	      +\, m_{(\dot{1}\dot{2})}\bar{\theta}^{\dot{1}}\bar{\theta}^{\dot{2}} \\
     &&
     +\, [b_{1(2)}, m_{(0)}]\,\theta^1\theta^2\,
	 +\, \sum_{\gamma, \dot{\delta}}
	            \mbox{\Large $($}
				  [b_{\gamma(\dot{\delta})}, m_{(0)}]
				 - \sqrt{-1}\,\mbox{\normalsize $\sum$}_\mu
				     \sigma^\mu_{\gamma\dot{\delta}} D_{\partial_\mu} m_{(0)}
				\mbox{\Large $)$}\, \theta^\gamma\bar{\theta}^{\dot{\delta}}\,       \\
     && +\, \sum_{\dot{\delta}} [b_{1(2)}, m_{(\dot{\delta})}]\,
	               \theta^1\theta^2\bar{\theta}^{\dot{\delta}}\\					
     && +\, \sum_{\gamma}		
                \mbox{\Large $($}
                 -  [b_{\gamma(\dot{2})}, m_{(\dot{1})}]
				 + [b_{\gamma(\dot{1})}, m_{(\dot{2})}]
				 + \sqrt{-1}\,\mbox{\normalsize $\sum$}_\mu
				     \sigma^\mu_{\gamma\dot{2}}D_{\partial_\mu}m_{(\dot{1})}
				 - \sqrt{-1}\,\mbox{\normalsize $\sum$}_\mu	
				     \sigma^\mu_{\gamma\dot{1}}D_{\partial_\mu}m_{(\dot{2})}
                \mbox{\Large $)$}\,
				\theta^\gamma\bar{\theta}^{\dot{1}}\bar{\theta}^{\dot{2}}\, \\
     &&	+\, \mbox{\Large $($}
	 	      -  [b_{2(1)}, m_{(\dot{1}\dot{2})}]
			  -  [b_{2(1\dot{1}\dot{2})} , m_{(0)}]
			  -  [b_{2(\dot{2})}, [b_{1(\dot{1})}, m_{(0)}]]
			  + [b_{2(\dot{1})}  , [b_{1(\dot{2})}, m_{(0)}]]  \\[.6ex]
       && \hspace{1.7em}
	          + \sqrt{-1}\,\mbox{\normalsize $\sum$}_\mu			
	               \sigma^\mu_{2\dot{2}}[D_{\partial_\mu}b_{1(\dot{1})}, m_{(0)}]
			  - \sqrt{-1}\,\mbox{\normalsize $\sum$}_\mu			
	               \sigma^\mu_{2\dot{1}}[D_{\partial_\mu}b_{1(\dot{2})}, m_{(0)}]
			  +\sqrt{-1}\,\mbox{\normalsize $\sum$}_{\mu}
                   \sigma^\mu_{1\dot{1}}[b_{2(\dot{2})} ,  D_{\partial_\mu}m_{(0)}]  \\[.6ex]
       && \hspace{1.7em}
              -  \sqrt{-1}\,\mbox{\normalsize $\sum$}_{\mu}
                   \sigma^\mu_{2\dot{1}}[b_{1(\dot{2})} ,  D_{\partial_\mu}m_{(0)}]
              -  \sqrt{-1}\,\mbox{\normalsize $\sum$}_{\mu}
                   \sigma^\mu_{1\dot{2}}[b_{2(\dot{1})} ,  D_{\partial_\mu}m_{(0)}]
             + \sqrt{-1}\,\mbox{\normalsize $\sum$}_{\mu}
                   \sigma^\mu_{2\dot{2}}[b_{1(\dot{1})} ,  D_{\partial_\mu}m_{(0)}] 	\\[.6ex]
       && \hspace{1.7em}
	          -  \sqrt{-1}\,\mbox{\normalsize $\sum$}_{\mu}
                   \sigma^\mu_{2\dot{2}}[a_{\mu(1\dot{1})} ,  m_{(0)}]
              + \sqrt{-1}\,\mbox{\normalsize $\sum$}_{\mu}
                   \sigma^\mu_{2\dot{1}}[a_{\mu(1\dot{2})} ,  m_{(0)}] \\[.6ex] 				   
        &&	\hspace{1.7em}
       	      -  [F^{\nabla}_{03}- \sqrt{-1}F^{\nabla}_{12}, m_{(0)}]
			  + \square^D m_{(0)}
                  \mbox{\Large $)$}\,
				  \theta^1\theta^2\bar{\theta}^{\dot{1}}\bar{\theta}^{\dot{2}}\,.	
   \end{eqnarray*}}Here, 
 the four components
   $m_{(0)}, m_{(\dot{1})}, m_{(\dot{2})}, m_{(\dot{1}\dot{2})}
      \in \widehat{\cal O}_X^{A\!z} $
   can be arbitrary,
   with $m_{(0)}$ and $m_{(\dot{1}\dot{2})}$ determining $\widehat{m}_{(\even)}$
      and $m_{(\dot{1})}$ and $m_{(\dot{2})}$ determining $\widehat{m}_{(\odd)}$.
\end{proposition}

\medskip

\begin{proof}
 To be specific, consider  $\widehat{D}$-antichiral sections.
 For
  {\small
  \begin{eqnarray*}
     \widehat{m}
	  & = &  m_{(0)}\,
	               +\, \sum_{\gamma}m_{(\gamma)}\theta^\gamma\,
		    +\, \sum_{\dot{\delta}} m_{(\dot{\delta})}\bar{\theta}^{\dot{\delta}}\,
			+\, m_{(12)}\theta^1\theta^2\,
			+\, \sum_{\gamma,\dot{\delta}}
			      m_{(\gamma\dot{\delta})}\theta^\gamma\bar{\theta}^{\dot{\delta}}\,
		    +\, m_{(\dot{1}\dot{2})}\bar{\theta}^{\dot{1}}\bar{\theta}^{\dot{2}}\,\\
     &&	\hspace{2.4em}
	      +\, \sum_{\dot{\delta}}
			       m_{(12\dot{\delta})}\theta^1\theta^2\bar{\theta}^{\dot{\delta}}\,
			+\, \sum_\gamma
			       m_{(\gamma\dot{1}\dot{2})}
				   \theta^\gamma\bar{\theta}^{\dot{1}}\bar{\theta}^{\dot{2}}\,
			+\, m_{(12\dot{1}\dot{2})}
			      \theta^1\theta^2\bar{\theta}^{\dot{1}}\bar{\theta}^{\dot{2}}\;\;\;	
	 \in\; \widehat{\cal O}_X^{A\!z}\,,
  \end{eqnarray*}}the 
   $\widehat{D}$-antichiral conditions
   $$
     \widehat{D}_{e_{1^{\prime}}}\widehat{m}_{(\even)}\;
     =\; \widehat{D}_{e_{1^{\prime}}}\widehat{m}_{(\odd)}\;
     =\; \widehat{D}_{e_{2^{\prime}}}\widehat{m}_{(\even)}\;
     =\; \widehat{D}_{e_{2^{\prime}}}\widehat{m}_{(\odd)}\; =\; 0
   $$
   together  give $8\times 4= 32$ equations from Lemma~3.2.1 on the $16$ components
   $m_{(0)}$, $m_{(1)}$, $m_{(2)}$, $m_{(\dot{1})}$, $m_{(\dot{2})}$,
   $m_{(12)}$,
   $m_{(1\dot{1})}$, $m_{(1\dot{2})}$, $m_{(2\dot{1})}$, $m_{(2\dot{2})}$,
   $m_{(\dot{1}\dot{2})}$,
   $m_{(12\dot{1})}$, $m_{(12\dot{2})}$, $m_{(1\dot{1}\dot{2})}$,
   $m_{(2\dot{1}\dot{2})}$, $m_{(12\dot{1}\dot{2})}$
   of $\widehat{m}$.\;
  $12$ of them
  {\small
  $$
   \begin{array}{c}
    m_{(1)}\;=\; 0\,,\hspace{3em}
    m_{(2)}\;=\; 0\,,\hspace{3em}
    m_{(12)} - [b_{1(2)}, m_{(0)}]\; =\; 0\,, \\[1.6ex]
    m_{(\gamma\dot{\delta})}
	   -  [b_{\gamma(\dot{\delta})}, m_{(0)}]
	   + \sqrt{-1}\,\mbox{\normalsize $\sum$}_\mu
				     \sigma^\mu_{\gamma\dot{\delta}} D_{\partial_\mu} m_{(0)}\; =\; 0\,, \\[2ex]
    m_{(12\dot{\delta})}
	   -  [b_{1(2)}, m_{(\dot{\delta})}] + [b_{1(\dot{\delta})}, m_{(2)}]
	   - \sqrt{-1}\,\sum_\mu\sigma^\mu_{1\dot{\delta}}D_{\partial_\mu}m_{(2)}\;
	   =\; 0\,,  \\[2ex]
    m_{(\gamma\dot{1}\dot{2})}\; 	 	
	   + [b_{\gamma(\dot{2})}, m_{(\dot{1})}]
	    - [b_{\gamma(\dot{1})}, m_{(\dot{2})}]
		- \sqrt{-1}\,\mbox{\normalsize $\sum$}_\mu
				     \sigma^\mu_{\gamma\dot{2}}D_{\partial_\mu}m_{(\dot{1})}
	   + \sqrt{-1}\,\mbox{\normalsize $\sum$}_\mu	
				     \sigma^\mu_{\gamma\dot{1}}D_{\partial_\mu}m_{(\dot{2})}\;
		=\; 0\,,   \\[2ex]
    - m_{(12\dot{1}\dot{2})}\;
	   -  [b_{2(1\dot{1}\dot{2})}, m_{(0)}]
	   -  [b_{2(\dot{2})}, m_{(1\dot{1})}]
	   + [b_{2(\dot{1})}, m_{(1\dot{2})}]
	   -  [b_{2(1)}, m_{(\dot{1}\dot{2})}]
	   + \sqrt{-1}\,\sum_\mu \sigma^\mu_{2\dot{2}}D_{\partial_\mu}m_{(1\dot{1})}
	                 \hspace{2em}\\[1.2ex]
        \hspace{6em}					
	   -  \sqrt{-1}\,\sum_\mu \sigma^\mu_{2\dot{1}}D_{\partial_\mu}m_{(1\dot{2})}
	   -  \sqrt{-1}\,\sum_\mu\sigma^\mu_{2\dot{2}}[a_{\mu(1\dot{1})} , m_{(0)}]
	   + \sqrt{-1}\,\sum_\mu\sigma^\mu_{2\dot{1}}[a_{\mu(1\dot{2})} , m_{(0)}] \;
     	=\; 0
   \end{array}
  $$}are 
  used to solve
    $m_{(1)}$, $m_{(2)}$,
   $m_{(12)}$,
   $m_{(1\dot{1})}$, $m_{(1\dot{2})}$, $m_{(2\dot{1})}$, $m_{(2\dot{2})}$,
   $m_{(12\dot{1})}$, $m_{(12\dot{2})}$, $m_{(1\dot{1}\dot{2})}$,
   $m_{(2\dot{1}\dot{2})}$, $m_{(12\dot{1}\dot{2})}$
   in terms of $m_{(0)}$, $m_{(\dot{1})}$, $m_{\dot{2}}$, $m_{(\dot{1}\dot{2})}$.
 The computation is straightforward and the result is given in the statement of the proposition.
 The remaining $20$ equations give a system of formal constraints on
   $m_{(1)}$, $m_{(\dot{1})}$, $m_{(\dot{2})}$, $m_{(\dot{1}\dot{2})}$.
 We need to show that this system of constraints are in fact redundant.
  
 This is the place we have to use the fact that $\widehat{D}$ is the induced connection on
    $\widehat{\cal O}_X^{A\!z}:= \Endsheaf_{\widehat{\cal O}_X}(\widehat{\cal E})$
	from a SUSY-rep compatible simple hybrid connection  $\widehat{\nabla}$ on $\widehat{\cal E}$
	that is associated to a vector superfield in Wess-Zumino gauge
	 $$
	   V\;=\; \sum_{\alpha, \dot{\beta}}V_{(\alpha\dot{\beta})}
	                 \theta^\alpha\bar{\theta}^{\dot{\beta}}\,
				+\, \sum_{\dot{\beta}}	 V_{(12\dot{\beta})}
				          \theta^1\theta^2\bar{\theta}^{\dot{\beta}}\,
			    +\, \sum_{\alpha}V_{(\alpha\dot{1}\dot{2})}
				          \theta^\alpha\bar{\theta}^{\dot{1}}\bar{\theta}^{\dot{2}}\,
			    +\, V_{(12\dot{1}\dot{2})}
				          \theta^1\theta^2\bar{\theta}^{\dot{1}}\bar{\theta}^{\dot{2}}
					 \;\; \in\;\widehat{\cal O}_X^{A\!z}
	 $$
	 on $\widehat{X}$.	
 Being so, Lemma~3.1.6 says that
  $$									
    \widehat{\cal O}_X^{A\!z,\scriptsizewidehatDach}\;
	   =\;  \mbox{\Large $($}
	               (e^{-V}\widehat{\cal O}_X^{A\!z, \scriptsizedach}\, e^V)
                    \cap \widehat{\cal O}_X^{A\!z, (\even)}	
             \mbox{\Large $)$}\,
		    \oplus\,
		     \mbox{\Large $($}
	               (e^{-V}\widehat{\cal O}_X^{A\!z, \scriptsizedach}\, \,\!^\varsigma\!e^V)
                    \cap \widehat{\cal O}_X^{A\!z, (\odd)}	
             \mbox{\Large $)$}\,.
  $$			
  I.e.\ $\widehat{m}$ is $\widehat{D}$-antichiral if and only if
  there exist $d$-antichiral sections
   $\widehat{n}_1,\, \widehat{n}_2 \in\; \widehat{\cal O}_X^{A\!z, \scriptsizedach}$
   such that
   $$
     \widehat{m}\;
	   =\; (e^{-V}\widehat{n}_1 e^V)_{(\even)}
	            + (e^{-V}\widehat{n}_2 \,\,\!^\varsigma\!e^V)_{(\odd)}\,.
   $$
  From the vanishing of the components
    $V_{(0)}$, $V_{(\dot{1})}$, $V_{(\dot{2})}$,
	$V_{(\dot{1}\dot{2})}$ of $V$,
   a direct computation shows that
   $$
    \widehat{m}_{(0)}\;=\; \widehat{n}_{1,(0)}\,, \hspace{2em}
	\widehat{m}_{(\dot{1}\dot{2})}\;
	    =\; \widehat{n}_{1,(\dot{1}\dot{2})}\,, \hspace{2em}
	\widehat{m}_{(\dot{1})}\;=\; \widehat{n}_{2,(\dot{1})}\,, \hspace{2em}
	\widehat{m}_{(\dot{2})}\;=\; \widehat{n}_{2,(\dot{2})}\,.
   $$
 Since
   $\widehat{n}_{1, (0)}$, $\widehat{n}_{1, (\dot{1}\dot{2})}$,
   $\widehat{n}_{2, (\dot{1})}$,
   $\widehat{n}_{2, (\dot{\dot{2}})}\in {\cal O}_X^{A\!z}$
   can be arbitrary,
  one concludes that
   $\widehat{m}_{(0)}$,
   $\widehat{m}_{(\dot{1}\dot{2})}$,
   $\widehat{m}_{(\dot{1})}$,
   $\widehat{m}_{(\dot{2})}$ must be allowed to be arbitrary as well.
  It follows that
   the system of 20 constraint equations on
     $\widehat{m}_{(0)}$, $\widehat{m}_{(\dot{1}\dot{2})}$,
     $\widehat{m}_{(\dot{1})}$, and $\widehat{m}_{(\dot{2})}$
	 must be redundantly satisfied.
       
  Similar arguments justify the statement for $\widehat{D}$-chiral sections in the proposition.
  
  This completes the proof.
  
\end{proof}

%
%

\bigskip

\section{$\widehat{D}$-chiral map from
         $(\widehat{X}^{\!A\!z},\widehat{\cal E};\widehat{\nabla})$ to a complex manifold}

With all the preparations in the previous sections, we are finally ready
  to introduce the notion of\; `{\it $\widehat{D}$-chiral maps}'\;
  from $(\widehat{X}^{\!A\!z},\widehat{\cal E};\widehat{\nabla})$ to a complex manifold.
This is what we will take to model fermionic D3-branes moving in a target space(-time), cf.\ Sec.~5.
We proceed in three steps.

\bigskip
 
\begin{flushleft}
{\bf Step 1: Smooth map from $(\widehat{X}^{\!A\!z},\widehat{\cal E})$ to a real manifold}
\end{flushleft}
The notion of
  smooth maps from Azumaya/matrix supermanifolds with a fundamental module
  to  a real super $C^{\infty}$-manifold and
 basic facts concerning such maps
 are given in [L-Y9] (D(11.4.1)).
They apply to smooth maps
  from a $d=4$, $N=1$ Azumaya/matrix superspace $(\widehat{X}, \widehat{\cal E})$	
  to a real manifold $Y$
 as a special case.
Some re-statements are given below.
 
\bigskip

\begin{sdefinition} {\bf [smooth map to real manifold]}\; {\rm
 Let $Y$ be a smooth manifold, whose structure sheaf of smooth functions is denoted by ${\cal O}_Y$.
 A {\it smooth map} (or synonymously {\it $C^\infty$-map})
  $\widehat{\varphi}: (\widehat{X}^{\!A\!z}, \widehat{\cal E})\rightarrow Y$
  is defined contravariantly by a ring-homomorphism
  $$
    \widehat{\varphi}^\sharp\; :\; C^{\infty}(Y)\; \longrightarrow\;
    C^{\infty}(\widehat{X}^{\!A\!z})
	  := C^{\infty}(\End_{\widehat{\Bbb C}}(\widehat{E}))
  $$
  over ${\Bbb R}\hookrightarrow {\Bbb C}$.
 Equivalently, in terms of structure sheaves, $\widehat{\varphi}$ is defined contravariantly by
  an equivalence class of gluing systems of ring-homomorhisms (cf.\ [L-Y1] (D(1)))
  $$
     {\cal O}_Y\;  \longrightarrow\;
	    \widehat{\cal O}_X^{A\!z}:= \Endsheaf_{\widehat{\cal O}_X}\!(\widehat{\cal E})
  $$
  over ${\Bbb R}\hookrightarrow {\Bbb C}$.
  Such equivalence class is also denoted by $\widehat{\varphi}^\sharp$.
}\end{sdefinition}

\bigskip

As a consequence of
  the Malgrange Division Theorem ([Mal]; see also [Br] and more references in [L-Y9] (D(11.4.1))
  for germs of smooth functions,
 $\widehat{\varphi}^\sharp$  is compatible with the $C^\infty$-ring structure on ${\cal O}_Y$
  and the partial $C^\infty$-ring structure on $\widehat{\cal O}_X^{A\!z}$
  (cf.\ [L-Y9: Theorem 2.1.5] (D(11.4.1))):

\bigskip

\begin{stheorem} {\bf [compatibility with partial $C^{\infty}$-ring structure]}
 For any $f_1,\,\cdots\,, f_l\in {\cal O}_Y$ and $h\in C^{\infty}({\Bbb R}^l)$,
  the collection $\{\widehat{\varphi}^\sharp(f_1),\,\cdots\,,\, \widehat{\varphi}^\sharp(f_l)\}$
   lies in the weak $C^\infty$-hull of $\widehat{\cal O}_X^{A\!z}$ with
 $$
   \widehat{\varphi}^\sharp (h(f_1,\,\cdots\,,\, h_l))\;
    =\;  h(\widehat{\varphi}^\sharp(f_1),\,\cdots\,,\, \widehat{\varphi}^\sharp(f_l))\,.
 $$
\end{stheorem}
  
\bigskip

A special case of [L-Y9: Theorem 2.1.8] (D(11.4.1)) says that

\bigskip

\begin{stheorem} {\bf [smooth map to ${\Bbb R}^n$]}\;
 Let $Y={\Bbb R}^n$ as a smooth manifold, with coordinate functions $(y^1,\,\cdots\,,\, y^n)$.
 Then
  any specification
   $$
    \widehat{\eta}\;:\;
      y^i\; \longmapsto\;
	  \widehat{m}_i  \in C^{\infty}(\End_{\widehat{\Bbb C}}(\widehat{E}))\,,
	   \hspace{1em}i=1,\,\ldots\,,\, n\,,
   $$
  such that the collection $\{\widehat{m}_1,\,\cdots\,,\, \widehat{m}_n\}$
   lies in the weak $C^\infty$-hull of $C^{\infty}(\End_{\widehat{\Bbb C}}(\widehat{E}))$
  extends to a ring-homomorphism
   $$
     \widehat{\varphi}^\sharp_{\widehat{\eta}}\;  :\; C^\infty	(Y)\; \longrightarrow\;
	 C^\infty(\End_{\widehat{\Bbb C}}(\widehat{E}))
   $$
   over ${\Bbb R}\hookrightarrow {\Bbb C}$ and, hence,
  defines a smooth map
   $$
     \widehat{\varphi}_{\widehat{\eta}}\;:\;
	  (\widehat{X}^{\!A\!z}, \widehat{\cal E})\;\longrightarrow\; Y\,.
   $$
\end{stheorem}
 
\medskip

\begin{sremark} $[${\it components of smooth map}$]$\; {\rm
 The built-in isomorphism\\
  $\Endsheaf_{\widehat{\cal O}_X}(\widehat{\cal E})
     \simeq \Endsheaf_{{\cal O}_X^{\,\Bbb C}}({\cal E})
	                  \otimes_{{\cal O}_X^{\,\Bbb C}}\!\widehat{\cal O}_X$
  induces an expression of
  $\widehat{\varphi}^\sharp:
      C^\infty(Y)\rightarrow C^\infty(\End_{\widehat{\Bbb C}}(\widehat{E}))$
    in terms of  {\it components} in the expansion in $(\theta,\bar{\theta})$:
 \begin{eqnarray*}
  \lefteqn{
   \widehat{\varphi}^\sharp\;
     =\;  \widehat{\varphi}_{(0)}\,
	        +\, \sum_\alpha \widehat{\varphi}^\sharp_{(\alpha)}\theta^\alpha\,
			+\, \sum_{\dot{\beta}}
			         \widehat{\varphi}^\sharp_{(\dot{\beta})}\bar{\theta}^{\dot{\beta}}\,
			+\, \widehat{\varphi}^\sharp_{(12)}\theta^1\theta^2\,
			+\, \sum_{\alpha,\dot{\beta}} \widehat{\varphi}^\sharp_{(\alpha\dot{\beta})}
					       \theta^\alpha \bar{\theta}^{\dot{\beta}}\,
            +\, \widehat{\varphi}^\sharp_{(\dot{1}\dot{2})}
                     \bar{\theta}^{\dot{1}}\bar{\theta}^{\dot{2}}\, }\\
     && \hspace{4em}					
            +\, \sum_{\dot{\beta}} \widehat{\varphi}^\sharp_{(12\dot{\beta})}
					   \theta^1\theta^2\bar{\theta}^{\dot{\beta}}\,
            +\, \sum_\alpha \widehat{\varphi}^\sharp_{(\alpha\dot{1}\dot{2})}
			           \theta^\alpha\bar{\theta}^{\dot{1}}\bar{\theta}^{\dot{2}}\,
			+\,\widehat{\varphi}^\sharp_{(12\dot{1}\dot{2})}	
			           \theta^1\theta^2\bar{\theta}^{\dot{1}}\bar{\theta}^{\dot{2}}\,,
 \end{eqnarray*}
 where the coefficients $\widehat{\varphi}^\sharp_{(\mbox{\tiny $\bullet$})}$
   are all $C^\infty(\End_{\Bbb C}(E))$-valued.
 That $\widehat{\varphi}^\sharp$ is a ring-homomorphism implies in particular the following:
   \begin{itemize}
    \item[(i)]
     {\it $\varphi^\sharp := \widehat{\varphi}^\sharp_{(0)}:
            C^\infty(Y)\rightarrow C^\infty(\End_{\Bbb C}(E))$
     is a ring-homomorphism over ${\Bbb R}\hookrightarrow{\Bbb C}$} and, hence,
	 defines a smooth map $\varphi: (X^{\!A\!z}, {\cal E})\rightarrow Y$
	 that makes the following diagram commute
	 $$
	   \xymatrix{
	      (\widehat{X}^{\!A\!z}\!, \widehat{\cal E})  \ar[rr]^-{\widehat{\varphi}}  && Y \\
		    (X^{\!A\!z}\!, {\cal E})\rule{0ex}{1.2em} \ar@{^{(}->}[u] \ar[rru]_-{\varphi}
	   }\,\raisebox{-2em}{.}	
	 $$
	  	
    \item[(ii)]
     {\it $\widehat{\varphi}_{(\alpha)}$, $\widehat{\varphi}_{(\dot{\beta})}:
            C^\infty(Y)\rightarrow C^\infty(\End_{\Bbb C}(E))$
	  satisfy the identities
	  $$
	   \begin{array}{rcl}
	    \widehat{\varphi}^\sharp_{(\alpha)}(f_1f_2)   & =
		  &  \widehat{\varphi}^\sharp_{(\alpha)}(f_1)\cdot \varphi^\sharp(f_2)\,
		       +\, \varphi^\sharp(f_1)\cdot \widehat{\varphi}^\sharp_{(\alpha)}(f_2)\,, \\[1.2ex]
	    \widehat{\varphi}^\sharp_{(\dot{\beta})}(f_1f_2)   & =
          &  \widehat{\varphi}^\sharp_{(\dot{\beta})}(f_1)\cdot \varphi^\sharp(f_2)\,
		       +\, \varphi^\sharp(f_1)\cdot \widehat{\varphi}^\sharp_{(\dot{\beta})}(f_2)\,,
       \end{array}			
	  $$
	  for all $f_1$, $f_2\in C^\infty(Y)$ and, hence,
	 are ${\cal O}_X^{A\!z}$-valued derivations of $C^\infty(Y)$.}
    They are thus sections of
	  $\varphi^\ast {\cal T}_\ast Y
	     := {\cal O}_X^{A\!z}\otimes_{\varphi^\sharp, {\cal O}_Y}\!{\cal T}_\ast Y$
	 (cf.\ [L-Y8: Sec.\ 3.2] (D(13.3))).
	 The tuple
	    $(\widehat{\varphi}^\sharp_{(1)}, \widehat{\varphi}^\sharp_{(2)},
		      \widehat{\varphi}^\sharp_{\dot{1}},\widehat{\varphi}^\sharp_{\dot{2}})$
		can be thought of as defining the {\it mappino} --- the supersymmetry partner --- of the map $\varphi$.
	
	\item[(iii)]
	 {\it $\widehat{\varphi}^\sharp_{(12)}$,
	   $\widehat{\varphi}^\sharp_{(\alpha\dot{\beta})}$,
       $\widehat{\varphi}^\sharp_{(\dot{1}\dot{2})}$,
       $\widehat{\varphi}^\sharp_{(12\dot{\beta})}$
	   $\widehat{\varphi}^\sharp_{(\alpha\dot{1}\dot{2})}$,
	   $\widehat{\varphi}^\sharp_{(12\dot{1}\dot{2})}:
           C^\infty(Y)\rightarrow C^\infty(\End_{\Bbb C}(E))$
	  are ${\cal O}_X^{A\!z}$-valued higher derivations of $C^\infty(Y)$.}
	 After imposing the $\widehat{D}$-Chirality/$\widehat{D}$-Antichiral Condition
	  on $\widehat{\varphi}$ they become either nondynamical or secondary in the end;
	 cf.\ Proposition~3.2.2,
            theme
			  `{\sl Step 3: $\widehat{D}$-chiral/$\widehat{D}$-antichiral map from
                   $(\widehat{X}^{\!A\!z}, \widehat{\cal E}; \widehat{\nabla})$ to a complex manifold}'
              of the current section, 	 and
		    Example~5.2.
   \end{itemize}

  %
  %
  %
  %
  %
  %
}\end{sremark}

\bigskip

\begin{flushleft}
{\bf Step 2: Smooth map from $(\widehat{X}^{\!A\!z},\widehat{\cal E})$ to a complex manifold}
\end{flushleft}
For a complex manifold $Y$,
denote by ${\cal O}_Y^{C^\infty}$
 (resp.\ ${\cal O}_Y^{C^\infty\!, {\Bbb C}}$,
                ${\cal O}_Y^{\hol}$,
				${\cal O}_Y^{\ahol}$)
 the structure sheaf of smooth functions
  (resp.\
       the complexification
          ${\cal O}_Y^{C^\infty}\!\otimes_{\Bbb R}{\Bbb C}$ of ${\cal O}_Y^{C^\infty}$\!,
      the structure sheaf of holomorphic functions,
	  the structure sheaf of antiholomorphic functions)
 on $Y$.

\bigskip

\begin{sdefinition}  {\bf [smooth map to complex manifold]}\; {\rm
 Let $Y$ be a complex manifold.
 A {\it smooth map} (or synonymously $C^{\infty}$-map)
   from $(\widehat{X}^{\!A\!z}, \widehat{\cal E})$ to $Y$ is by definition
  a smooth map $\widehat{\varphi}$ from $(\widehat{X}^{\!A\!z}, \widehat{\cal E})$
  to the smooth manifold underlying $Y$, defined contravariantly by
  an equivalence class of gluing system of ring-homomorphisms
  $\widehat{\varphi}^\sharp: {\cal O}_Y^{C^{\infty}}\rightarrow
     \widehat{\cal O}_X^{A\!z}$ over ${\Bbb R}\hookrightarrow {\Bbb C}$
  or, equivalently, by a ring-homomorphism
   $\widehat{\varphi}^{\sharp}: C^{\infty}(Y)
      \rightarrow C^{\infty}(\End_{\widehat{\Bbb C}}(\widehat{E}))$
 	 over ${\Bbb R}\hookrightarrow {\Bbb C}$.
 In this case, $\widehat{\varphi}^\sharp$ extends canonically to an equivalence class of gluing systems
   of ${\Bbb C}$-algebra-homomorphisms
   $\widehat{\cal O}_Y^{C^\infty\!, {\Bbb C}}\rightarrow \widehat{\cal O}_X^{A\!z}$
   by the assignment
   $$
     f+\sqrt{-1}\, g\;
	  \longmapsto\;  \widehat{\varphi}^\sharp (f) +\sqrt{-1}\,\widehat{\varphi}^\sharp(g)\,.
   $$
 Denote this extension still by $\widehat{\varphi}^\sharp$. 	
}\end{sdefinition}

\bigskip

Before proceeding further, we digress to discuss the complex form of the $C^{\infty}$-ring structure
 of the function ring of a complex coordinate chart of the complex manifold $Y$.
Let $V\subset Y$ be a local chart, with complex coordiantes
 $(z^1,\,\cdots\,,\, z^n)=(y^1+\sqrt{-1}y^2,\,\cdots\,,\, y^{2n-1}+\sqrt{-1}y^{2n})$.
Then, as generators of  the $C^{\infty}$-ring $C^\infty(V)$,
 an $f\in C^{\infty}(V)$ can be expressed as $f(y^1, \,\cdots\,,\, y^{2n})$.
In the complex form, one denotes $f$ also as
 $$
  f\; =\; f(z^1,\,\cdots\,, z^n, \bar{z}^1,\,\cdots\,, \bar{z}^n) \,.
 $$
In other words, we set the convention that
 {\small
 $$
    h(z^1,\,\cdots\,, z^n, \bar{z}^1,\,\cdots\,, \bar{z}^n) \;
	:=\; h \mbox{\Large $($}
	          \frac{1}{2}\,(z^1+\bar{z}^1),  \frac{1}{2\sqrt{-1}}\,(z^1-\bar{z}^1),\,
			 \cdots\,,\,
			 \frac{1}{2}(z^n+\bar{z}^n),  \frac{1}{2\sqrt{-1}}(z^n-\bar{z}^n)
	          \mbox{\Large $)$}
 $$}for
 $h\in C^{\infty}({\Bbb C}^n)= C^{\infty}({\Bbb R}^{2n})$
 applying to complex-conjugate-paired complex-valued functions
 $(z^1,\,\cdots\,,\, z^n, \bar{z}^1,\,\cdots\,,\, \bar{z}^n)$.
Define the complex-valued derivations on $C^\infty(V)$
 {\small
 $$
   \frac{\partial}{\partial z^i}\;
    :=\; \frac{1}{2}\, \mbox{\Large $($}
	                      \frac{\partial}{\partial y^{2i-1}}\,-\, \sqrt{-1}\, \frac{\partial}{\partial y^{2i}}
                                            \mbox{\Large $)$}\,,\;\;\;
   \frac{\partial}{\partial \bar{z}^i}\;
    :=\; \frac{1}{2}\, \mbox{\Large $($}
	                      \frac{\partial}{\partial y^{2i-1}}\,+\, \sqrt{-1}\, \frac{\partial}{\partial y^{2i}}
                                            \mbox{\Large $)$}\,.									
 $$}Then, one has the following result from basic analysis:
 
\bigskip

\begin{slemma} {\bf [complex form of Taylor's formula]}\;
 Denote coordinates on $V$ collectively by
  $\boldy:= (y^1,\,\cdots\,,\, y^{2n})$,
  $\boldz:= (z^1,\,\cdots\,, z^n) $, and $\bar{\boldz}:= (\bar{z}^1,\,\cdots\,,\, \bar{z}^n)$.
 Then,
   for
    $f\in C^\infty(V)$ and
    $q\in V$ of coordinates $\boldy$,  and
	$\bolda := (a^1,\,\cdots\,,\, a^{2n})\in {\Bbb R}^{2n}$ such that
        points $q_t$ of real coordinates $\boldy_t:=\boldy+t\cdot \bolda$ are contained in $V$ for all $t\in [0,1]$,
   the Taylor's formula
   $$
     f(\boldy+\bolda)\;
	   =\;  \sum_{d=0}^l
              \sum_{|\scriptsizeboldd|=d}	
			  \frac{1}{\boldd !}\,
	          \frac{\partial^{\,d} f}{\partial \boldy^{\scriptsizeboldd}}(\boldy)\,
			  \bolda^{\scriptsizeboldd}\;
			 +\;  \sum_{|\scriptsizeboldd|=l+1}	
			  \frac{1}{\boldd !}\,
	          \frac{\partial^{\,l+1} f}{\partial \boldy^{\scriptsizeboldd}}(\boldy_{t_0})\,
			  \bolda^{\scriptsizeboldd}\;
   $$
  for some $t_0\in [0,1]$  depending on $\bolda$
  has the following complex form
   \begin{eqnarray*}
    \lefteqn{
      f(\boldz+\boldu, \bar{\boldz}+\bar{\boldu})\;
	   =\;  \sum_{d=0}^l\,
              \sum_{|\scriptsizeboldd_1|+|\scriptsizeboldd_2|=d}	
			  \frac{1}{\boldd_1 ! \boldd_2 !}\,
	          \frac{\partial^{\,d} f}{\partial \boldz^{\scriptsizeboldd_1}
			                                          \partial \bar{\boldz}^{\scriptsizeboldd_2}}(\boldz,\bar{\boldz})\,
			  \boldu^{\scriptsizeboldd_1}\bar{\boldu}^{\scriptsizeboldd_2}\;   }\\
     &&\hspace{7em}			
			 +\;  \sum_{|\scriptsizeboldd|=l+1}	
			  \frac{1}{\boldd_1 !\boldd_2 !}\,
	          \frac{\partial^{\,l+1} f}
			      {\partial \boldz^{\scriptsizeboldd_1}\partial \bar{\boldz}^{\scriptsizeboldd_2}}
				  (\boldz_{t_0}, \bar{\boldz}_{t_0})\,
			  \boldu^{\scriptsizeboldd_1}\bar{\boldu}^{\scriptsizeboldd_2}
   \end{eqnarray*}
   for some $t_0\in [0,1]$  depending on $\boldu$.
  Here,
    $\boldd:= (d_1,\,\cdots\,,\, d_{2n})$ with $d_i\in {\Bbb Z}_{\ge 0}$,\\
	$|\boldd|:= d_1+\,\cdots\,+d_{2n}$,
	$\boldd !:= d_1!\,\cdots\,d_{2n}!$ with $0!:=1$,
	$\partial^{\,d}/{\partial\boldy}^{\scriptsizeboldd}
	   :=  (\partial/\partial y^1)^{d_1}
	           \cdots  (\partial/\partial y^{2n})^{d_{2n}} $ for $|\boldd|=d$,
    $\bolda^{\scriptsizeboldd}:= (a^1)^{d_1}\,\cdots\,(a^{2n})^{d_{2n}}$			   			   
  and similarly
    $\boldu:= (u^1,\,\cdots\,,\, u^n)\in {\Bbb C}^n$	
	such that
        points $q_t$ of complex coordinates $\boldz_t:=\boldz+t\cdot \boldu$
		are contained in $V$ for all $t\in [0,1]$, 	
	$\boldd_i=(d_{i,1},\,\cdots\,,\, d_{i, n})$, $i=1,2$, with $d_{i,j}\in {\Bbb Z}_{\ge 0}$,
    $|d_i|:= d_{i,1}+\,\cdots\,+ d_{i,n}$,	
    $\boldd_i !:= d_{i,1}!\,\cdots\,d_{i,n}!$, \\	
    $\partial^{\,d}/(\partial\boldz^{\scriptsizeboldd_1}\partial\bar{\boldz}^{\scriptsizeboldd_2})
	   :=  (\partial/\partial z^1)^{d_{1,1}} \cdots  (\partial/\partial z^n)^{d_{1,n}}
	        (\partial/\partial \bar{z}^1)^{d_{2,1}} \cdots  (\partial/\partial \bar{z}^n)^{d_{2,n}}
	 $ for $|\boldd_1|+|\boldd_2|=d$, \\
    $\boldu^{\scriptsizeboldd}:= (u^1)^{d_{1,1}}\,\cdots\,(u^{1,n})^{d_{1,n}}$,
	$\bar{\boldu}^{\scriptsizeboldd}
	    := (\bar{u}^1)^{d_{2,1}}\,\cdots\,(\bar{u}^{2,n})^{d_{2,n}}$.
\end{slemma}
 
\bigskip

The above complex form to express real-valued functions on a complex manifold and their Taylor's expansions
  motivates a complex form of the partial $C^{\infty}$-ring structure on
  $\widehat{\cal O}_X^{A\!z}$ as follows.
  
\bigskip

\begin{sdefinition} {\bf [$C^\infty$-operable in complex form]}\; {\rm
 A tuple $(\widehat{m}_1,\,\cdots\,,\widehat{m}_{2k})$ of elements in
  $\widehat{\cal O}_X^{A\!z}$ is said to be {\it $C^\infty$-operable in the complex form}
  if its entries can be paired up to
   $\widehat{m}_1$, $\widehat{m}_{k+1}$; $\cdots$;
   $\widehat{m}_{k}$, $\widehat{m}_{2k}$
  so that the collection
   $$
      \mbox{\Large $\{$}\,
	    \frac{1}{2}\,(\widehat{m}_1+\widehat{m}_{k+1}),
        \frac{1}{2\sqrt{-1}}\,(\widehat{m}_1-\widehat{m}_{k+1}),\,
        \cdots,\, 	
        \frac{1}{2}\,(\widehat{m}_k+\widehat{m}_{2k}),
        \frac{1}{2\sqrt{-1}}\,(\widehat{m}_k-\widehat{m}_{2k})
	  \,\mbox{\Large $\}$}
   $$
   lies in the weak $C^\infty$-hull of $\widehat{\cal O}_X^{A\!z}$,
   i.e.\
   for any $p\in X$ so that $\widehat{m}_1(p),\,\cdots\,,\, \widehat{m}_{2k}(p)$
    are defined,
   \begin{itemize}
    \item[(1)]
	 $\widehat{m}_1(p),\,\cdots\,,\, \widehat{m}_{2k}(p)$ commute with each other,
	  	
    \item[(2)]
	 the eigenvalues of
	 $\frac{1}{2}(\widehat{m}_1(p)+\widehat{m}_{k+1}(p))$,
     $\frac{1}{2\sqrt{-1}}\,(\widehat{m}_1(p)-\widehat{m}_{k+1}(p))$,
     $\cdots$, 	
     $\frac{1}{2}(\widehat{m}_k(p)+\widehat{m}_{2k}(p))$,
     $\frac{1}{2\sqrt{-1}}\,(\widehat{m}_k(p)-\widehat{m}_{2k}(p))$
	 are all real.     	
   \end{itemize}
}\end{sdefinition}

\medskip

\begin{sdefinition} {\bf [$C^\infty$-operation in complex form]}\; {\rm
 Let
  $h\in C^\infty({\Bbb C}^k)$  and
  $(\widehat{m}_1,\,\cdots\,,\widehat{m}_{2k}  )$
    be a tuple of elements in $\widehat{\cal O}_X^{A\!z}$ that is $C^\infty$-operable in the complex form.
 Then define
   \begin{eqnarray*}
    \lefteqn{
     h(\widehat{m}_1,\,\cdots,\, \widehat{m}_k,
           \widehat{m}_{k+1}, \,\cdots\,,\widehat{m}_{2k}  )\;  }\\
    &&
	 :=\;  h
     \mbox{\Large $($}\,
	    \frac{1}{2}\,(\widehat{m}_1+\widehat{m}_{k+1}),
        \frac{1}{2\sqrt{-1}}\,(\widehat{m}_1-\widehat{m}_{k+1}),\,
        \cdots,\, 	
        \frac{1}{2}\,(\widehat{m}_k+\widehat{m}_{2k}),
        \frac{1}{2\sqrt{-1}}\,(\widehat{m}_k-\widehat{m}_{2k})
	  \,\mbox{\Large $)$}\,.
   \end{eqnarray*}
}\end{sdefinition}

\bigskip

The following lemma is immediate:

\bigskip

\begin{slemma} {\bf [shift by commuting nilpotent element]}\;
 If the tuple $(\widehat{m}_1,\,\cdots\,,\widehat{m}_{2k}  )$
  of elements in $\widehat{\cal O}_X^{A\!z}$
  is $C^\infty$-operable in the complex form,
 then so is the tuple
  $(\widehat{m}_1+\widehat{n}_1,\,\cdots\,,\widehat{m}_{2k}+\widehat{n}_{2k})$,
 where $\widehat{n}_1,\,\cdots\,,\widehat{n}_{2k}  \in \widehat{\cal O}_X^{A\!z}$
  are nilpotent and commute with each other and with $\widehat{m}_1,\,\cdots\,,\widehat{m}_{2k} $.
 Furthermore, assuming that $\widehat{n}_1^l=\,\cdots\,= \widehat{n}_{2k}^l=0$
   for some $l\in {\Bbb Z}_{\ge 0}$,
 then for $h\in C^\infty({\Bbb C}^k)$, one has
  $$
    h(\widehat{\boldm}+ \widehat{\boldn}, \widehat{\boldm}^\prime+\widehat{\boldn}^\prime)\;
	   =\;  \sum_{d=0}^{2kl}\,
              \sum_{|\scriptsizeboldd_1|+|\scriptsizeboldd_2|=d}	
			  \frac{1}{\boldd_1 ! \boldd_2 !}\,
	          \frac{\partial^{\,d} h}{\partial \boldz^{\scriptsizeboldd_1}
			                                          \partial \bar{\boldz}^{\scriptsizeboldd_2}}
					(\widehat{\boldm},\widehat{\boldm}^\prime)\,
			  \cdot \widehat{\boldn}^{\scriptsizeboldd_1}
			             {\widehat{\boldn}^\prime}\,\!^{\scriptsizeboldd_2}\,.
  $$
 Here,
  $\boldz:=(z^1,\,\cdots\,,\, z^k)$ the complex coordiantes of ${\Bbb C}^k$,
  $\widehat{\boldm}:= (\widehat{m}_1,\,\cdots\,,\, \widehat{m}_k)$,\\
  $\widehat{\boldm}^\prime:= (\widehat{m}_{k+1},\,\cdots\,,\, \widehat{m}_{2k})$,
  $\widehat{\boldn}:= (\widehat{n}_1,\,\cdots\,,\, \widehat{n}_k)$,
  $\widehat{\boldn}^\prime:= (\widehat{n}_{k+1},\,\cdots\,,\, \widehat{n}_{2k})$,
  $\boldd_1:=(d_{1,1},\,\cdots\,,\, d_{1,k})$,
  $\boldd_2:=(d_{2,1},\,\cdots\,,\, d_{2,k})$,
  $\widehat{\boldn}^{\scriptsizeboldd_1}
     :=  (\widehat{n}_1)^{d_{1,1}}\cdots (\widehat{n}_k)^{d_{1,k}}$,
   and
  ${\widehat{\boldn}^\prime}\,\!^{\scriptsizeboldd_2}
     :=  (\widehat{n}_{k+1})^{d_{2,1}}\cdots (\widehat{n}_{2k})^{d_{2,k}}$.
\end{slemma}

\bigskip

In terms of $C^\infty$-operable tuples in the complex form,
Theorem~4.2 has a partial rephrasing for $Y$ a complex manifold:

\bigskip

\begin{stheorem} {\bf [smooth map to complex manifold]}
 Let
   $Y$ be a complex manifold and
   $\widehat{\varphi}:(\widehat{X}^{\!A\!z}, \widehat{\cal E})\rightarrow Y$
     be a smooth map defined contravariantly by an equivalence class
	$\widehat{\varphi}^\sharp: {\cal O}_Y^{C^\infty}\rightarrow \widehat{\cal O}_X^{A\!z}$
	of gluing systems of ring-homomorphisms over ${\Bbb R}\hookrightarrow {\Bbb C}$.
 Let $f_1,\,\cdots\,,\, f_k\in {\cal O}_Y^{\hol}$ be local holomorphic functions on $Y$ and
        $\bar{f_1},\,\cdots\,,\, \bar{f_k}\in {\cal O}_Y^{\ahol}$ their complex conjugates.
Then\\
  $(\widehat{\varphi}^\sharp(f_1),\,\cdots\,,\, \widehat{\varphi}^\sharp(f_k),\,
        \widehat{\varphi}^\sharp(\bar{f_1}),\,\cdots\,,\,
				                                                                               \widehat{\varphi}^\sharp(\bar{f_k}))$
  is a $C^\infty$-operable tuple.
 Furthermore, for $h\in C^\infty({\Bbb C}^k)$,
  $$
   h \mbox{\large $($}
        \widehat{\varphi}^\sharp(f_1),\,\cdots\,,\, \widehat{\varphi}^\sharp(f_k),\,
        \widehat{\varphi}^\sharp(\bar{f_1}),\,\cdots\,,\,
				                                                                               \widehat{\varphi}^\sharp(\bar{f_k})
	   \mbox{\large $)$}\;
    =\; \widehat{\varphi}^\sharp
	       \mbox{\large $($}
	           h(f_1,\,\cdots\,,\, f_k,\, \bar{f_1},\,\cdots\,,\, \bar{f_k})
           \mbox{\large $)$}\,.
  $$
\end{stheorem}

\bigskip
	
\begin{flushleft}
{\bf Step 3: $\widehat{D}$-chiral/$\widehat{D}$-antichiral map from
  $(\widehat{X}^{\!A\!z}, \widehat{\cal E}; \widehat{\nabla})$ to a complex manifold} 	
\end{flushleft}
\begin{sdefinition}
{\bf [$\widehat{D}$-chiral map \& $\widehat{D}$-antichiral map to complex manifold]}\; {\rm
 Let $Y$ be a complex manifold.
 A smooth map
   $\widehat{\varphi}:(\widehat{X}^{\!A\!z}, \widehat{\cal E}; \widehat{\nabla})
     \rightarrow Y$, defined contravariantly by
   $\widehat{\varphi}^\sharp: {\cal O}_Y^{C^\infty}\rightarrow \widehat{\cal O}_X^{A\!z}$,  	
  is called {\it $\widehat{D}$-chiral} (resp.\ {\it $\widehat{D}$-antichiral})
  if the induced equivalence class of gluing systems of ${\Bbb C}$-algebra-homomorphisms
   $\widehat{\varphi}^\sharp: {\cal O}_Y^{C^\infty\!, {\Bbb C}}
      \rightarrow \widehat{\cal O}_X^{A\!z}$
   sends
      ${\cal O}_Y^{\hol}$ to $\widehat{\cal O}_X^{A\!z, \scriptsizewidehatDch}$  and
	  ${\cal O}_Y^{\ahol}$ to $\widehat{\cal O}_X^{A\!z, \scriptsizewidehatDach}$
   (resp.\
       ${\cal O}_Y^{\hol}$ to $\widehat{\cal O}_X^{A\!z, \scriptsizewidehatDach}$  and
	   ${\cal O}_Y^{\ahol}$ to $\widehat{\cal O}_X^{A\!z, \scriptsizewidehatDch}$).
}\end{sdefinition}

%
%

\bigskip

\section{$\widehat{D}$-chiral maps from
         $(\widehat{X}^{\!A\!z},\widehat{\cal E}; \widehat{\nabla})$ to a K\"{a}hler manifold
		 and the $N=1$ Super D3-Brane Theory in Ramond-Neveu-Schwarz formulation}

We conclude the current work with some highlights, test computations, and open ends on
 how $\widehat{D}$-chiral maps from a $d=4$, $N=1$ Azumaya/matrix superspace
   with a fundamental module with a connection
 gives a construction of
 {\it Super D3-Brane Theory in Ramond-Neveu-Schwarz formulation}.
Details and beyond are the focus of separate works.\footnote{This
                                                           section means to give readers a taste of {\it Super D3-Brane Theory
														    in Ramond-Neveu-Schwarz formulation}.
								                           Some physicists' well-accepted facts/rules are taken for granted here
														    without further explanations or justifications.}

As a preliminary, readers are referred to [L-Y4: Sec.\ 5.1] (D(11.2)) for a detailed explanation
 of the following statement:
 \begin{itemize}
  \item[\Large $\cdot$]  [{\sl Ramond-Neveu-Schwarz (RNS) fermionic string }]\hspace{1.6em}
 {\it A Ramond-Neveu-Schwarz (RNS) fermionic string moving in a Minkowski space-time
       ${\Bbb M}^{(d-1)+1}$ as studied in {\rm [Gr-S-W: Chap.\ 4]}
      can be described by a map $\widehat{f}: \widehat{\Sigma}\rightarrow {\Bbb M}^{(d-1)+1}$
	  from a $2$-dimensional superspace to ${\Bbb M}^{(d-1)+1}$
	  in the sense of Grothendieck's Algebraic Geometry.}
 \end{itemize}
Replacing the world-sheet $\widehat{\Sigma}$ of the fermionic string
 by the world-volume
     $\widehat{X}$ (a superspace in the case of simple D-branes)   or
     $\widehat{X}^{\!A\!z}$  (an Azumaya/matrix superspace in the case of stacked D-branes),
 by the same sense of the above statement
 we call a dynamical fermionic D-brane moving in a space(-time) $Y$
  that is modelled,
    with the additional fermionic-string-induced D-brany structure of Chan-Paton bundle
	       with a connection on $\widehat{X}$ or $\widehat{X}^{\!A\!z}$,
  by a map from $\widehat{X}$ or $\widehat{X}^{\!A\!z}$ to a general target-space(-time) $Y$
  a {\it D-brane in the Ramond-Neveu-Schwarz formulation}.

\bigskip

\begin{flushleft}
{\bf Fermionic D3-branes and $\widehat{D}$-chiral maps $\widehat{\varphi}$ from
	   $(\widehat{X}^{\!A\!z},\widehat{\cal E}; \widehat{\nabla})$ to a K\"{a}hler manifold}
\end{flushleft}
Beginning from the string-theory side,
excitations of RNS fermionic (oriented) open strings with end-points on a D-brane create
 the spetrum of fields on the D-brane world-volume.
 
\medskip

\centerline{\includegraphics[width=0.80\textwidth]{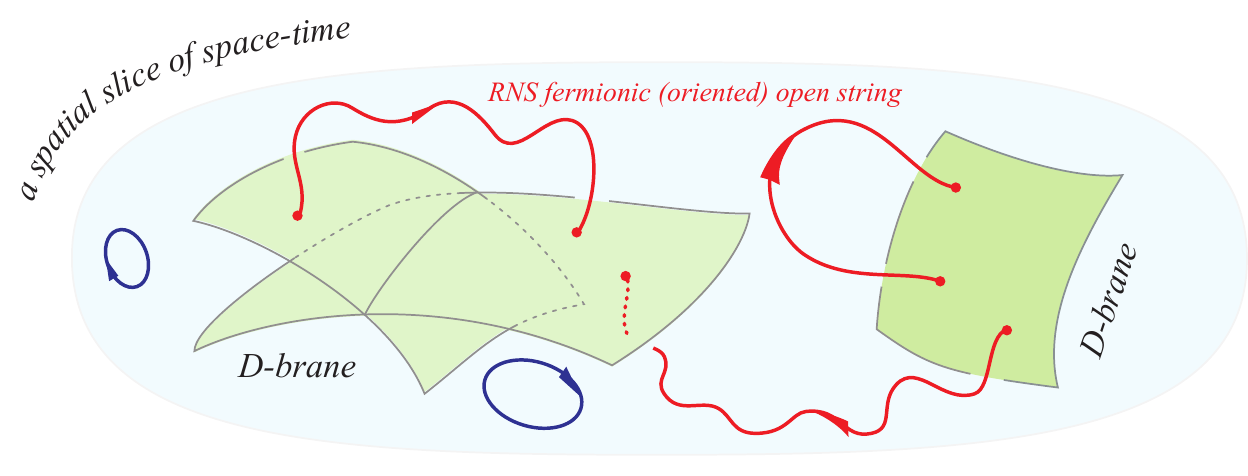}}

\medskip
 
\noindent
After a choice of the Gliozzi-Scherk-Olive (GSO) projection on the spectrum,
 what remains forms a collection of multiplets under the world-volume supersymmetry.
In particular, there are now fermionic fields on the D-brane world-volume.
The massless spectrum consists of massless scalar multiplets and a  massless vector multiplet.
The former collectively describe how the fermionic D-branes fluctuates in the target-space-time
 while the latter gives a (super) connection on the Chan-Paton bundle on the fermionic D-brane world-volume.
For the RNS open string to govern also the dynamics of these massless multiplets
  in a way that is supersymmetric with respect to the D-brane world-volume supersymmetry,
 some appropriate constraint on the geometry of the target space-time is required.
 
%
%
%
%

When a collection of fermionic D-branes coincide,
the rank of the Chan-Paton bundle on the common D-brane world-volume
 increases to the multiplicity of the coincidence and
 the multiplets in the spectrum on the common fermionic D-brane world-volume are enhanced
 to endomorphism/matrix-valued.

\medskip
 
\centerline{\includegraphics[width=0.80\textwidth]{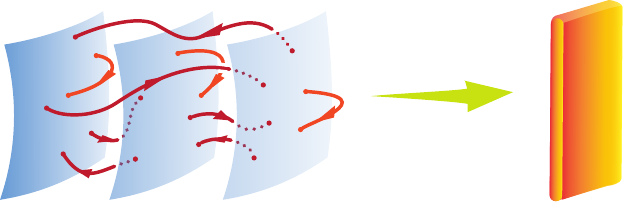}}

\bigskip
  
%
%
%
%

A careful re-examination of the above stringy picture from the aspect of Grothendieck's Algebraic Geometry,
   combined with mathematical naturality,
 gives rise to the following mathematical objects
 in the general case of fermionic coincident/stacked D-branes:\footnote{Readers are referred to
                                                                         [L-Y1] (D(1)) for the very careful explanation
 																		  in the case of bosonic D-branes in the realm of Algebraic Geometry.
																		For the current case, Super $C^\infty$-Algebraic Geometry is involved
																		  but the reasoning is completely the same.
																		See also [Liu].}
 \begin{itemize}
  \item[\LARGE $\cdot$] [{\it meaning of the mass scalar multiplets collectively}]
   \begin{itemize}
    \item[(i)]
    The function-ring of the world-volume of the fermionic D-brane is enhanced to\\
      $C^\infty(\End_{\widehat{\Bbb C}}(\widehat{E}))$,
	  where $\widehat{E}$ is the Chan-Patan bundle on the fermionic D-brane world-volume.
	
    \item[(ii)]	
	The map $\widehat{\varphi}$ that describes how the fermionic stacked D-branes move
 	 around in the target space(-time) $Y$
	  is defined contravariantly via a ring-homomorphism\\
	  $\widehat{\varphi}: C^\infty(Y)
	      \rightarrow C^\infty(\End_{\widehat{\Bbb C}(\widehat{E})})$
	 that satisfies some SUSY-Rep Compatible Condition.
   \end{itemize}	
   
  \item[\LARGE $\cdot$]  [{\it meaning of the vector multiplet}]
   \begin{itemize}
    \item[]
    There is a connection on $\widehat{E}$ that is defined via the vector multiplet.
   \end{itemize}
 \end{itemize}
In this way a supersymmetric D-brane moving in a target space(-time) $Y$ is modelled by
 a SUSY-rep compatible map $\widehat{\varphi}$
 from an Azumaya/matrix superspace $\widehat{X}^{\!A\!z}$
 with a fundamental module with a SUSY-rep compatible connection $(\widehat{\cal E},\widehat{\nabla})$
 to $Y$ with compatible geometry:

\medskip

\centerline{\includegraphics[width=0.80\textwidth]{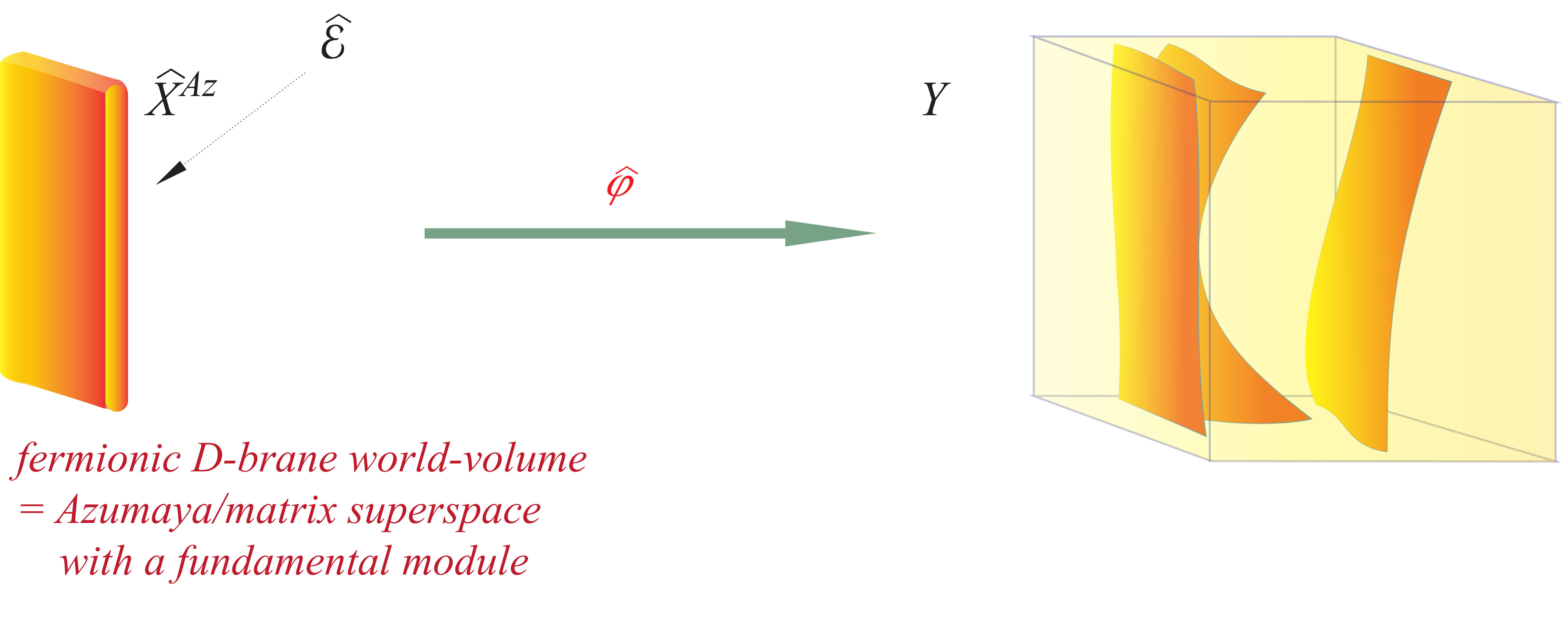}}

\bigskip

In particular, fermionic D3-branes moving in $Y$ with $d=4$, $N=1$ world-volume supersymmetry
 are modelled by
 $\widehat{D}$-chiral maps $\widehat{\varphi}$ from a $d=4$, $N=1$ Azumaya/matrix superspace
 with a fundamental module with a SUSY-rep compatible connection
 $(\widehat{X}^{\!A\!z}, \widehat{\cal E};\widehat{\nabla})$ to $Y$ with $Y$ K\"{a}hler.

\bigskip

\begin{flushleft}
{\bf The standard supersymmetry-invariant action functional for $\widehat{\nabla}$}
\end{flushleft}
(Cf.\ e.g., [G-G-R-S: Sec.\ 4.2] and [We-B: Chap.\ VII].)
%
%
%
%
%
%
%
Let $\widehat{\nabla}$ be the simple hybrid connection on $\widehat{\cal E}$
 associated to a vector superfield  in the Wess-Zumino gauge
{\small
 $$
    V\ :=\;  \sum_{\gamma, \dot{\delta}}
	                V_{(\gamma\dot{\delta})}\theta^\gamma\bar{\dot{\delta}}\,
       	     +\, \sum_{\dot{\delta}}V_{(12\dot{\delta})}		
	                \theta^1\theta^2\bar{\theta}^{\dot{\gamma}}\,
			 +\, \sum_{\gamma} V_{(\gamma\dot{1}\dot{2})}
			        \theta^\gamma\bar{\theta}^{\dot{1}}\bar{\theta}^{\dot{2}}\,
			 +\, V_{(12\dot{1}\dot{2})}
			        \theta^1\theta^2\bar{\theta}^{\dot{1}}\bar{\theta}^{\dot{2}}\;\;
	 \in\; \widehat{\cal O}_X^{A\!z}\,.
 $$}Then
one needs to construct a $d$-chiral section $W_{\alpha}$, $\alpha=1, 2$,
  of $\widehat{\cal O}_X^{A\!z}$ from $V$ that encodes all the curvature information of
  $\widehat{\nabla}$.
Presumably, $W_1$ and $W_2$  should be obtained from the computation of the curvature tensor
 of $\widehat{\nabla}$.
Once having such $W_1$ and $W_2$,
 a general fact on the construction of supersymmetry-invariant action functionals (e.g.\ [Bi: Sec.\ 4.3])
 says that, up to boundary terms,
 the following action functional for $\widehat{\nabla}$
  $$
   S_{\mbox{\tiny SYM}}(\widehat{\nabla})\;
     :=\;  \frac{1}{g_{\mbox{\tiny\it gauge}}^{\,2}}\,
	           \Real \int_{\widehat{X}} \Tr (W_1W_2) \, d^{\,4}x d\theta^2 d\theta^1\;
	  =\;  \frac{1}{g_{\mbox{\tiny\it gauge}}^{\,2}}\,
	           \Real \int_X \Tr(W_1W_2)_{(12)}\,d^{\,4}x
  $$
 is invariant under the $d=4$, $N=1$ supersymmetry and
 gives the analogue of the {\it super-Yang-Mills action functional} for the connection $\widehat{\nabla}$.
Here, $g_{\mbox{\tiny\it gauge}}$ is the gauge coupling constant.

\bigskip

\begin{sexample} {\bf  [test computation of $S_{\mbox{\tiny\rm SYM}}(\widehat{\nabla})$]}\;
{\rm
 As suggested by, e.g., [G-G-R-S: Sec.\ 4.2] and [We-B: Chap.\ VII] with some necessary adaptation,
  one may consider the following sections of $\widehat{\cal O}_X^{A\!z}$:
 $$
   W_{\alpha}\;=\;
    e_{1^{\prime\prime}}   e_{2^{\prime\prime}}
	  \mbox{\large $($}(e_{\alpha^\prime}e^{-V})e^V \mbox{\large $)$}\,,
	  \hspace{2em}
	  \;\; \alpha = 1, 2.
 $$
 Since
  $\{e_{1^{\prime\prime}}, e_{2^{\prime\prime}}\}
     =\{e_{1^{\prime\prime}}, e_{1^{\prime\prime}}\}
	 =\{e_{2^{\prime\prime}}, e_{2^{\prime\prime}}\}
     =0$,
 $$
  e_{\beta^{\prime\prime}}W_\alpha\;=\; 0\,,
     \hspace{2em}
     \mbox{for $\alpha=1, 2$, $\beta^{\prime\prime}=1^{\prime\prime}, 2^{\prime\prime}$}\,.
 $$
 I.e.\ both $W_1$ and $W_2$ are $d$-chiral sections of $\widehat{\cal O}_X^{A\!z}$.
 The very format of $W_1$ and $W_2$ implies that they do contain some curvature information
  of $\widehat{\nabla}$.
 A general fact on the construction of supersymmetry-invariant action functionals (e.g.\ [Bi: Sec.\ 4.3])
 now says that the following action functional
  $$
   S(\widehat{\nabla})\;
     :=\;  \frac{-1}{16\, g_{\mbox{\tiny\it gauge}}^{\,2}}\,
	           \Real \int_{\widehat{X}} \Tr (W_1W_2) \, d^{\,4}x d\theta^2 d\theta^1\;
	  =\;  \frac{-1}{16\, g_{\mbox{\tiny\it gauge}}^{\,2}}\,
	           \Real \int_X \Tr(W_1W_2)_{(12)}\,d^{\,4}x
  $$
  gives a supersymmetry-invariant action functional for $\widehat{\nabla}$.
 Here, $g_{\mbox{\tiny\it gauge}}$ is the gauge coupling constant and
  the factor $-1/16$ is added as a normalization factor from a hindsight from the explicit computation below.

 Written out  explicitly,
 {\small
 \begin{eqnarray*}
   W_\alpha & = &
     V_{(\alpha\dot{1}\dot{2})}\,
	 +\, \sum_{\gamma}
	       \mbox{\Large $($}
		       (1-\delta_{\alpha\gamma})(-1)^\gamma\, V_{(12\dot{1}\dot{2})}
			 + \delta_{\alpha\gamma}\cdot [V_{(\gamma\dot{2})}, V_{(\alpha\dot{1})}]
			                                                                               \\[-2ex]
    && \hspace{7em}			
             + (1-\delta_{\alpha\gamma})
			    ( \mbox{\large $\frac{1}{2}$}\, [V_{(2\dot{2})}, V_{(1\dot{1})}] 				
                     + \mbox{\large $\frac{1}{2}$}\, [V_{(1\dot{2})}, V_{(2\dot{1})}])
			 + \sqrt{-1}\, \mbox{\normalsize $\sum$}_\mu
			      \sigma^\mu_{\gamma\dot{2}} \partial_\mu V_{(\alpha\dot{1})}\\[-.2ex]
   && \hspace{8em}				
			 -  \sqrt{-1}\,\mbox{\normalsize $\sum$}_\mu
			      \sigma^\mu_{\gamma\dot{1}} \partial_\mu V_{(\alpha\dot{2})}
			 + \sqrt{-1}\,\mbox{\normalsize $\sum$}_\mu
			      \sigma^\mu_{\alpha\dot{2}} \partial_\mu V_{(\gamma\dot{1})}
			 - \sqrt{-1}\,\mbox{\normalsize $\sum$}_\mu
			      \sigma^\mu_{\alpha\dot{1}} \partial_\mu V_{(\gamma\dot{2})}
		   \mbox{\Large $)$}\, \theta^\gamma\, \\[.6ex]
   && +\, \mbox{\Large $($}
                 [V_{(\alpha\dot{1})}, V_{(12\dot{2})}]
			 + [V_{(12\dot{1})}, V_{(\alpha\dot{2})}]
			 + 2\sqrt{-1}\, \mbox{\normalsize $\sum$}_\mu
			     \sigma^\mu_{\alpha\dot{1}}\partial_\mu V_{(12\dot{2})}
			 - 2\sqrt{-1}\, \mbox{\normalsize $\sum$}_\mu
			     \sigma^\mu_{\alpha\dot{2}}\partial_\mu V_{(12\dot{1})}
               \mbox{\Large $)$}\,\theta^1\theta^2   \\[.6ex]
  &&  +\, (\,\mbox{terms of $\bar{\theta}$-degree $\ge 1$}\,)
 \end{eqnarray*}
  }

 \noindent
 and

 {\small
 \begin{eqnarray*}
    S(\widehat{\nabla})
	 & = &   \frac{-1}{16\, g_{\mbox{\tiny\it gauge}}^{\,2}}\,
	           \Real \int_X \Tr
			    \mbox{\LARGE $($}
                 V_{(1\dot{1}\dot{2})}
				 \cdot \mbox{\Large $($}
				     [V_{(2\dot{1})}, V_{(12\dot{2})}]
				  + [V_{(12\dot{1})}, V_{(2\dot{2})}]   \\[-1ex]
     && \hspace{13em}				
				  + 2\sqrt{-1}\,\mbox{\normalsize $\sum$}_\mu
				        \sigma^\mu_{2\dot{1}} \partial_\mu V_{(12\dot{2})}
				  -  2\sqrt{-1}\,\mbox{\normalsize $\sum$}_\mu
				        \sigma^\mu_{2\dot{2}} \partial_\mu V_{(12\dot{1})}
				           \mbox{\Large $)$}\,				   \\[1ex]
     && +\, \mbox{\Large $($}
                     [V_{(1\dot{2})}, V_{(1\dot{1})}]
				  + 2\sqrt{-1}\,\mbox{\normalsize $\sum$}_\mu	
				       \sigma^\mu_{1\dot{2}} \partial_\mu V_{(1\dot{1})}
				  -  2\sqrt{-1}\,\mbox{\normalsize $\sum$}_\mu
				       \sigma^\mu_{1\dot{1}} \partial_\mu V_{(1\dot{2})}
                  \mbox{\Large $)$} \\
     && \hspace{2em}				
                  \cdot
                  \mbox{\Large $($}
				     [V_{(2\dot{2})}, V_{(2\dot{1})}]
				  + 2\sqrt{-1}\,\mbox{\normalsize $\sum$}_\mu
				      \sigma^\mu_{2\dot{2}} \partial_\mu V_{(2\dot{1})}
				  -  2\sqrt{-1}\, \mbox{\normalsize $\sum$}_\mu
				      \sigma^\mu_{2\dot{1}} \partial_\mu V_{(2\dot{2})}
                  \mbox{\Large $)$}			\\	
     && -\, \mbox{\Large $($}
	               \mbox{\large $\frac{1}{2}$}\, [V_{(2\dot{2})}, V_{(1\dot{1})}] 				
                 + \mbox{\large $\frac{1}{2}$}\, [V_{(1\dot{2})}, V_{(2\dot{1})}]
				 + V_{(12\dot{1}\dot{2})}
				 + \sqrt{-1}\,\mbox{\normalsize $\sum$}_\mu
				      \sigma^\mu_{2\dot{2}} \partial_\mu V_{(1\dot{1})} \\
       && \hspace{6em}					
				 -  \sqrt{-1}\,\mbox{\normalsize $\sum$}_\mu
				      \sigma^\mu_{2\dot{1}} \partial_\mu V_{(1\dot{2})}
				 + \sqrt{-1}\, \mbox{\normalsize $\sum$}_\mu
				      \sigma^\mu_{1\dot{2}} \partial_\mu V_{(2\dot{1})}
				 -  \sqrt{-1}\, \mbox{\normalsize $\sum$}_\mu
				      \sigma^\mu_{1\dot{1}} \partial_\mu V_{(2\dot{2})}
                 \mbox{\Large $)$}	 \\
    &&  \hspace{2em}
            \cdot 	\mbox{\Large $($}
			      \mbox{\large $\frac{1}{2}$}\, [V_{(2\dot{2})}, V_{(1\dot{1})}] 				
                 + \mbox{\large $\frac{1}{2}$}\, [V_{(1\dot{2})}, V_{(2\dot{1})}]				
				 -  V_{(12\dot{1}\dot{2})}
				 + \sqrt{-1}\,\mbox{\normalsize $\sum$}_\mu
				      \sigma^\mu_{2\dot{2}} \partial_\mu V_{(1\dot{1})} \\
       && \hspace{6em}					
				 -  \sqrt{-1}\,\mbox{\normalsize $\sum$}_\mu
				      \sigma^\mu_{2\dot{1}} \partial_\mu V_{(1\dot{2})}
				 + \sqrt{-1}\, \mbox{\normalsize $\sum$}_\mu
				      \sigma^\mu_{1\dot{2}} \partial_\mu V_{(2\dot{1})}
				 -  \sqrt{-1}\, \mbox{\normalsize $\sum$}_\mu
				      \sigma^\mu_{1\dot{1}} \partial_\mu V_{(2\dot{2})}
                 \mbox{\Large $)$}	 \\		
    && +\, \mbox{\Large $($}
				     [V_{(1\dot{1})}, V_{(12\dot{2})}]
				  + [V_{(12\dot{1})}, V_{(1\dot{2})}]   \\
       && \hspace{6em}				
				  + 2\sqrt{-1}\,\mbox{\normalsize $\sum$}_\mu
				        \sigma^\mu_{1\dot{1}} \partial_\mu V_{(12\dot{2})}
				  -  2\sqrt{-1}\,\mbox{\normalsize $\sum$}_\mu
				        \sigma^\mu_{1\dot{2}} \partial_\mu V_{(12\dot{1})}
				           \mbox{\Large $)$}
				\cdot V_{(2\dot{1}\dot{2})}
				\mbox{\LARGE $)$}\, d^{\,4}x\,.
 \end{eqnarray*}}

 An expression of such explicitness allows one to examine some further detail
 to realize that $S(\widehat{\nabla})$ as defined is not an extension of the usual Yang-Mills action functional.
 One needs to derive the appropriate $W_1$ and $W_2$ from the foundations in Sec.\ 2.2 and Sec.\ 2.3.
}\end{sexample}

\bigskip

\begin{flushleft}
{\bf Given $\widehat{\nabla}$, the standard supersymmetry-invariant action functional for $\widehat{\varphi}$\,:\\
         Zumino meeting Polchinski \& Grothendieck}
\end{flushleft}
Let
 $$
  \widehat{\varphi}\;:\; (\widehat{X}^{\!A\!z}, \widehat{\cal E}; \widehat{\nabla})\;
    \longrightarrow\; Y
 $$
 be a $\widehat{D}$-chiral map
 from a $d=4$, $N=1$ Azumaya/matrix superspace with a fundamental module with a connection
 to a K\"{a}hler manifold $Y$, defined contravariantly by an equivalence class of gluing systems
 of ring-homomorphisms
 $$
  \widehat{\varphi}^\sharp\;:\; {\cal O}_Y^{\,C^\infty}\; \longrightarrow\; \widehat{\cal O}_X^{A\!z}\,.
 $$
Assume that $\widehat{\varphi}(\widehat{X}^{\!A\!z})$ is contained in an open set of $Y$
 on which the K\"{a}hler metric admits a K\"{a}hler potential $h$.
Then,
   guided by the construction [Zu] of Bruno Zumino and
   as a consequence of a general fact in supersymmetry (e.g.\ [Bi: Sec.\ 4.3]),
 the following action functional for $\widehat{\varphi}$
 $$
   S^{\widehat{\nabla}}(\widehat{\varphi})\;
	 :=\;  \Real  T_3\, \int_{\widehat{X}}\Tr \widehat{\varphi}^{\sharp}(h)\,
	          d^{\,4}x\, d\bar{\theta}^{\dot{2}}d\bar{\theta}^{\dot{1}}
			                      d\theta^2 d\theta^1\;
	 =\;  \Real  T_3\, \int_{\widehat{X}}
	         \Tr (\widehat{\varphi}^{\sharp}(h))_{(12\dot{1}\dot{2})}\,
	          d^{\,4}x
 $$
 is invariant under the $d=4$, $N=1$ supersymmetry on $\widehat{X}$,
 up to boundary terms.
Here $T_3$ is a constant for fermionic D3-brane tension.
Since $\widehat{\varphi}$ is $\widehat{D}$-chiral,
 a shift of $h$ by a holomorphic or an antiholomorphic function gives rise only to boundary terms
 of $S^{\nabla}(\widehat{\varphi})$.
Thus, for $\widehat{\varphi}$ $\widehat{D}$-chiral and up to boundary terms,
 $S^{\nabla}(\widehat{\varphi})$
 depends only on the K\"{a}hler metric on $Y$.
As a lesson learned from
        [L-Y6] (D(13.1)), [L-Y7] (D(13.2.1)) and [L-Y8] (D(13.3)),
  to extract information from $\widehat{\varphi}^{\sharp}(h)$,
    one needs to impose an appropriate {\it admissible condition}
	on the pair $(\widehat{\nabla}, \widehat{\varphi})$ in addition to
	the requirement that
	  $\widehat{\varphi}$ be $\widehat{D}$-chiral and $\widehat{\nabla}$ be SUSY-rep compatible.
Such admissible condition reflects the physics requirement that
 the gauge field $\widehat{\nabla}$ on $\widehat{\cal E}$ be massless from the aspect of
 super open strings.
 
\bigskip
 
\begin{sexample}  {\bf [test computation of $S^{\widehat{\nabla}}(\widehat{\varphi})$]}\;
{\rm
 Let
   $Y={\Bbb C}^n$ as a K\"{a}hler manifold, with complex coordinate functions
	  $(z^1, \,\cdots \,, z^n)= (y^1+\sqrt{-1}y^1,\,\cdots \,, y^{2n-1}+\sqrt{-1}y^{2n})$,   and   
   $$
     \widehat{\varphi}\; :\; (\widehat{X}^{\!A\!z}, \widehat{\cal E}; \widehat{\nabla})\;
      \longrightarrow\; 	 Y
   $$
    be a $\widehat{D}$-chiral map defined contravariantly by a ring-homomorphism
   $$
     \widehat{\varphi}^{\sharp}\;:\;   C^{\infty}(Y)\;
	    \longrightarrow \; C^{\infty}(\End_{\widehat{\Bbb C}}(\widehat{E}))
   $$
    over ${\Bbb R}\hookrightarrow {\Bbb C}$ that is specified by
   $$
     y^i \;\longmapsto\;  \widehat{m}^i, \hspace{1em}i=1,\,\ldots\,,\, 2n\,,
   $$
   such that
	 \begin{itemize}
	  \item[\LARGE $\cdot$]
       $\{\widehat{m}^1,\,\cdots \,, \, \widehat{m}^{2n}\}$ lies in the weak $C^{\infty}$-hull
	   of $C^{\infty}(\End_{\widehat{\Bbb C}}(\widehat{E}))$,
	
	  \item[\LARGE $\cdot$]
	   $\widehat{\varphi}^{\sharp}(z^i)
	      := \widehat{\varphi}^\sharp(y^{2i-1})+\sqrt{-1}\,\widehat{\varphi}^\sharp(y^{2i})$
		 are $\widehat{D}$-chiral and\\
       $\widehat{\varphi}^{\sharp}(\bar{z}^i)
	      := \widehat{\varphi}^\sharp(y^{2i-1})- \sqrt{-1}\,\widehat{\varphi}^\sharp(y^{2i})$
		 are $\widehat{D}$-antichiral, for $i=1,\,\cdots\,, n$.
	 \end{itemize}
  (Cf.\ Theorem~4.3.)
         %
		 %
  Recall the normal form of $\widehat{D}$-chiral sections and $\widehat{D}$-antichiral sections of
	   $\widehat{\cal O}_X^{A\!z}$  in Proposition~3.2.2,
    with a conversion of notations\footnote{Some
	                                                                           spinor notation convention in this example follows
																		       [We-B: Appendices A \& B] of Wess and Bagger. 	
	                                                                         } 
	for an easy comparison with [We-B: Chap.\ XXII] of Wess and Bagger:\;
	($\partial_\mu := \partial/\partial x^\mu$, $\mu=0,1,2,3$,\, for formulae below)
	{\small
	\begin{eqnarray*}
	 \Phi^i & := & \widehat{\varphi}^{\sharp}(z^i)\\
	 & \:= &  A^i\, +\, \sqrt{2}\,\theta\chi^i\,
	                +\, \sqrt{-1}\,\sum_\mu\theta\sigma^\mu\bar{\theta} D_{\partial_\mu}A^i\,
				    +\, \theta\theta F^i\,
					-\, \mbox{\Large $\frac{\sqrt{-1}}{\sqrt{2}}$}\,
					      \sum_\mu \theta\theta D_{\partial_\mu}\chi^i \sigma^\mu\bar{\theta}\,
	                +\, \mbox{\Large $\frac{1}{4}$}\,
					     \theta\theta\bar{\theta}\bar{\theta}\,\square^D\!A^i \\
	 && +\, \mbox{(unlike terms depending on $(A^i, \chi^i_\alpha, F^i)$ and $\widehat{D}$)}	
	\end{eqnarray*}}
	and
	{\small
	\begin{eqnarray*}
	 \Phi^{+ i} & := & \widehat{\varphi}^{\sharp}(\bar{z}^i)\\
	 & \:= &  A^{\ast i}\, +\, \sqrt{2}\,\bar{\theta}\bar{\chi}^i\,
	                -\, \sqrt{-1}\,\sum_\mu\theta\sigma^\mu\bar{\theta} D_{\partial_\mu}A^{\ast i}\,
				    +\, \bar{\theta}\bar{\theta} F^{\ast i}\,
					+\, \mbox{\Large $\frac{\sqrt{-1}}{\sqrt{2}}$}\,
					      \sum_\mu \bar{\theta}\bar{\theta}\theta\sigma^\mu D_{\partial_\mu}\bar{\chi}^i
	                +\, \mbox{\Large $\frac{1}{4}$}\,
					     \theta\theta\bar{\theta}\bar{\theta}\,\square^D\!A^{\ast i} \\
	 && +\, \mbox{(unlike terms depending on
	               $(A^{\ast i}, \bar{\chi}^i_{\dot{\beta}}, F^{\ast i})$ and $\widehat{D}$)}\,,		 
	\end{eqnarray*}}
	 for $i=1,\,\ldots\,,\, n$.
	Here,
	 \begin{itemize}
	  \item[\LARGE $\cdot$]
        terms in the normal form that have no counterparts in [We-B: Chap.\ XXII] are omitted;
		
	  \item[\LARGE $\cdot$]	
	   $\chi^i=(\chi^i_1, \chi^i_2)$,
	   $\bar{\chi}^i=(\bar{\chi}^i_{\dot{1}}, \bar{\chi}^i_{\dot{2}})$;
	   $A^i$,  $\chi^i_1$, $\chi^i_2$ , $F^i$,
	   $A^{\ast i}$, $\bar{\chi}^i_{\dot{1}}$, $\bar{\chi}^i_{\dot{2}}$, $F^{\ast i}
	    \in  {\cal O}_X^{A\!z}$;
	
	  \item[\LARGE $\cdot$]
	   $D$ is the connection on ${\cal O}_X^{A\!z}$ from the restriction of
	   $\widehat{D}$ on $\widehat{\cal O}_X^{A\!z}$.
     \end{itemize}	
  Note that, by construction,
    the $2n$ sections $\Phi^1,\,\cdots\,,\, \Phi^n,\Phi^{\ast 1},\,\cdots\,,\, \Phi^{\ast n}$
      of $\widehat{\cal O}_X^{A\!z}$ commute with each other and
    the $2n$ sections $A^1,\,\cdots\,,\, A^n, A^{\ast 1},\,\cdots\,,\, A^{\ast n}$
      of ${\cal O}_X^{A\!z}$ commute with each other.
	
  Let $h\in C^{\infty}(Y)$ be a K\"{a}hler potential for the K\"{a}hler metric on $Y$.
  Assume in addition an {\it admissible condition} that
    \begin{itemize}
	 \item[] \hspace{-1.7em}($\ast\,\!_{\mbox{\rm \scriptsize extreme}}$)\hspace{1em}
	  \begin{itemize}
	   \item[(i)]
	    $\Phi^1-A^1,\,,\cdots\,,\, \Phi^n-A^n,\,
		  \Phi^{\ast 1}-A^{\ast 1},\,\cdots\,,\, \Phi^{\ast n}-A^{\ast n}$
         commute with each other.		
		
	   \item[(ii)]
	    The total collection
	    $$
		 \begin{array}{l}
		   A^i\,,\;  \chi^i_1\,,\;   \chi^i_2\,,\;  F^i\,,\;
	       A^{\ast i}\,,\;   \bar{\chi}^i_{\dot{1}}\,,\;  \bar{\chi}^i_{\dot{2}}\,,\;  F^{\ast i}\,,\;
		   D_\mu A^i\,,\;  D_\mu \chi^i_1\,,\;   D_\mu\chi^i_2\,,\;
	       D_\mu A^{\ast i}\,,\;
		   D_\mu\bar{\chi}^i_{\dot{1}}\,,\;  D_\mu\bar{\chi}^i_{\dot{2}}\,, \\[1.2ex]
		   i=1,\,\ldots\,, n\,,\; \mu=0,1,2,3\,,		
         \end{array}		
		$$
	    of sections of ${\cal O}_X^{A\!z}$	commute with each other.	
	   (Here, $D_{\mu}:= D_{\partial_\mu}$.)
	  \end{itemize}
	\end{itemize}
 Then the expansion in Lemma~4.9 applies to give:\;  (cf.\ [We-B: Eq.\ (22.9)])
   %
  {\small
   \begin{eqnarray*}
    4\cdot(\widehat{\varphi}^\sharp(h))_{(12\dot{1}\dot{2})}
	 & = &
	     -\, \sum_{i,\bar{j}; \mu} g_{i\bar{j}}(D_\mu A^i)(D^\mu\!A^{\ast j})\,
		 -\, \sqrt{-1}\,\sum_{i\bar{j}; \mu}
		       g_{i\bar{j}}\bar{\chi}^j \bar{\sigma}^\mu D_\mu \chi^i\,   \\
     &&
        -\, \sqrt{-1}\,\sum_{i,j,k; l, \mu}	
		    g_{l\bar{k}}\Gamma_{ij}^l \bar{\chi}^k \bar{\sigma}^\mu \chi^i D_\mu A^j		
		+\, \mbox{\Large $\frac{1}{4}$}\,
		       \sum_{i\bar{j}, k, \bar{l}}g_{i\bar{j}, k\bar{l}}
			   \chi^i\chi^k\bar{\chi}^j\bar{\chi}^l			   \\ 	
	 &&
	    +\, \sum_{i,\bar{j}}g_{i\bar{j}}F^i F^{\ast j}\,
	    -\,  \mbox{\Large $\frac{1}{2}$}\,
		      \sum_{i,\bar{j}, \bar{k}; \bar{l}}
		        g_{i\bar{l}}\Gamma^{\bar{l}}_{\bar{j}\bar{k}}
				F^i\bar{\chi}^j\bar{\chi}^k\,
        -\,  \mbox{\Large $\frac{1}{2}$}\,
              \sum_{\bar{i}, j, k; l}		
			    g_{l\bar{i}}\Gamma^l_{jk} F^{\ast i}\chi^j\chi^k\,   \\[.6ex]	
    && +\, (\mbox{other terms unlike the previous seven})\,.
   \end{eqnarray*}}
   where
     $D^\mu$ is the raising of $D_\mu$ by the Minkowski metric on $X$;
     $(g_{i\bar{j}})$ is the K\"{a}hler metric on $Y$,
      $g_{i\bar{j}, \mbox{\tiny $\bullet$}}$	its derivatives and
	 $\Gamma^{\mbox{\tiny $\bullet^\prime$}}_{\mbox{\tiny$\bullet\,\bullet^{\prime\prime}$}}$
	   its Christoffel symbols,
   all in terms of the complex coordinate functions
     $(z^1,\,\cdots\,, z^n, \bar{z}^1,\,\cdots\,, \bar{z}^n)$ on $Y$  and
     evaluated at $(A^1,\,\cdots\,,\, A^n, A^{\ast 1},\,\cdots\,,\, A^{\ast n})$.

 \noindent\hspace{40.8em}$\square$	
}\end{sexample}

\bigskip

The above test-example shows that
   the action functional\\
     $\;S^{\widehat{\nabla}}(\widehat{\varphi})\;
	       :=\;  \Real  \int_{\widehat{X}}\Tr \widehat{\varphi}^{\sharp}(h)\,
	                d^{\,4}x\, d\bar{\theta}^{\dot{2}}d\bar{\theta}^{\dot{1}}
			                      d\theta^2 d\theta^1\;$
   is indeed a generalization of the construction [Zu] of Bruno Zumino.
The first term  \;$-\,\sum_{i,\bar{j}; \mu} g_{i\bar{j}}(D_\mu A^i)(D^\mu\!A^{\ast j})$\;
 of \;$4\cdot(\widehat{\varphi}^\sharp(h))_{(12\dot{1}\dot{2})}$\;
 in the example justifies $S^{\widehat{\nabla}}(\widehat{\varphi})$ thus defined
 as a supersymmetric extension of the term\\
 $\;\frac{1}{2}\,T_3\,\Real \int_X \Tr \langle D\varphi, D\varphi \rangle d^{\,4}x\;$
 in the standard action functional for bosonic, metrically flat D3-branes, cf.\ [L-Y8: Sec.\ 4] (D(13.3)).

\bigskip

\begin{flushleft}
{\bf (Fundamental) $N=1$  Super (Stacked) D3-Brane Theory in the RNS formulation}
\end{flushleft}
Having the SUSY-rep compatible pairs $(\widehat{\nabla}, \widehat{\varphi})$  from Sec.\ 2 -- Sec.\ 4
 that describe fermionic D3-branes moving in a K\"{a}hler target space-time $Y$,
once a supersymmetric action functional $S(\widehat{\nabla}, \widehat{\varphi})$
 is also properly constructed,
one then has the same starting point as that for Superstring Theory in the Ramond-Neveu-Schwarz formulation:
	
 \bigskip

 \centerline{
{\footnotesize\it
 \begin{tabular}{|l||l|}\hline
   \hspace{3em}{\bf RNS-superstring theory}\rule{0ex}{1.2em}
       & \hspace{3em}{\bf super D3-brane theory in RNS formulation}\\[.6ex] \hline\hline
     $\begin{array}{l}
	    \mbox{fundamental objects}: \\
 		 \hspace{3em}
		 \mbox{open or closed fermionic string}\\
		 \hspace{6em}\includegraphics[width=0.18\textwidth]{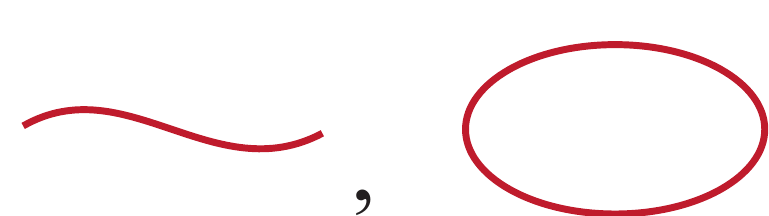}
	   \end{array}$\hspace{.6em}
       &  $\begin{array}{l}
	           \mbox{fundamental objects}:\rule{0ex}{1.2em}\\
			      \hspace{1em}\mbox{\scriptsize Azumaya/matrix}\\
				  \hspace{3em}\mbox{\scriptsize $3$-dimensional superspace}\\
 				  \hspace{1em}\mbox{\scriptsize with a fundamental module}\\
				  \hspace{3em}\mbox{\scriptsize with a connection}\\[1ex]
			   \end{array}$
	           \raisebox{-6ex}{\includegraphics[width=0.25\textwidth]{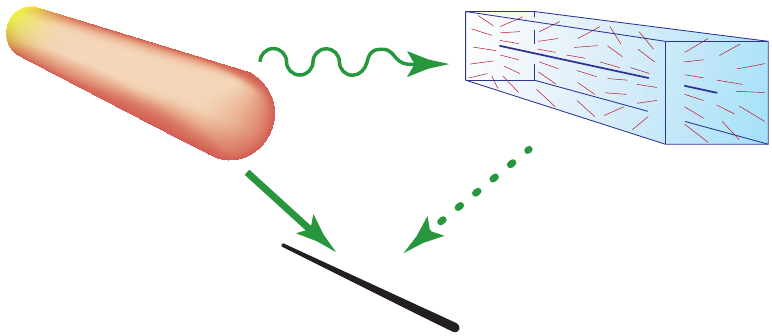}}
			   \hspace{1ex}
			   \\[3.8ex] \hline	   	
     $\begin{array}{l}
	    \mbox{fermionic string world-sheet}: \\
 		 \hspace{3em}
		 \mbox{$2$-dimensional superspace $\,\widehat{\Sigma}$}
	   \end{array}$
       &  $\begin{array}{l}
	           \mbox{fermionic D3-brane world-volume}:\rule{0ex}{1.2em}\\
			      \hspace{3em}\mbox{Azumaya/matrix $4$-dimensional superspace}\\
 				  \hspace{3em}\mbox{with a fundamental module with a connection}\\
				  \hspace{4em}
				    \mbox{$(\widehat{X}^{\!A\!z}, \widehat{\cal E}; \widehat{\nabla})$
					              with $\widehat{\nabla}$ SUSY-rep compatible}\\[1ex]
			   \end{array}$	  				                                                                      \\[3.8ex] \hline
	 $\begin{array}{l}
	    \mbox{fermionic string moving in space-time $Y$}:\rule{0ex}{1.2em} \\
 		 \hspace{3em}
		 \mbox{differentiable map $\widehat{f}: \widehat{\Sigma} \rightarrow Y$}
	   \end{array}$
       &  $\begin{array}{l}\rule{0ex}{1.2em}
	           \mbox{fermionic D3-brane moving in (K\"{a}hler) space-time $Y$}:\\
			      \hspace{3em}
	              \mbox{$($admissible$)$ $\widehat{D}$-chiral map
				       $\widehat{\varphi}:
					     (\widehat{X}^{\!A\!z}, \widehat{\cal E}; \widehat{\nabla})\rightarrow Y$}
			   \end{array}$	  				                                                                             \\[2.4ex] \hline		
       $\begin{array}{l}\rule{0ex}{1.4em}
	        \mbox{action functional $S(\widehat{f})$ for maps $\widehat{f}$ that is}\\
			\mbox{invariant under world-sheet supersymmetry} \\[1ex]
	      \end{array}$
	       & 
		      $\begin{array}{l}\rule{0ex}{1.2em}
			       \mbox{action functional
				                  $S(\widehat{\nabla}, \widehat{\varphi})
				                    = S_{\mbox{\tiny\rm SYM}}(\widehat{\nabla})
									   + S^{\widehat{\nabla}}(\widehat{\varphi})$
								 for pairs $(\widehat{\nabla}, \widehat{\varphi})$}\\
				   \mbox{that is invariant under the world-volume supersymmetry} \\[1ex]
				\end{array}$                                             \\[1ex] \hline		
         \rule{0ex}{1.5em}\hspace{8em}$\cdots\cdots$   &   \hspace {14em}  {\rm ???}
		                             \\[1ex] \hline				
 \end{tabular}
  }
 }
 
\bigskip
\bigskip

\noindent
Challenges remain ahead to understand Super D3-Brane Theory in such a format
 and its generalization to one with extended supersymmetries, central charges, and BPS states.

\bigskip

Finally we remark that
 while the current work focuses on super D3-branes to make all the statements specific,
similar $C^{\infty}$-algebrogeometric foundations to supersymmetry apply to other dimensions  and
 the corresponding notion of SUSY-rep compatible smooth maps
  from the related Azumaya/matrix superspaces in other dimensions to a target space(-time)
  gives then a description of super D$p$-branes for other $p$'s in the Ramond-Neveu-Schwarz formulation.
There are many pieces yet to be understood in this study/subject in the making.

\bigskip
\bigskip
\bigskip
\noindent
\hspace{6em} \includegraphics[width=0.84\textwidth]{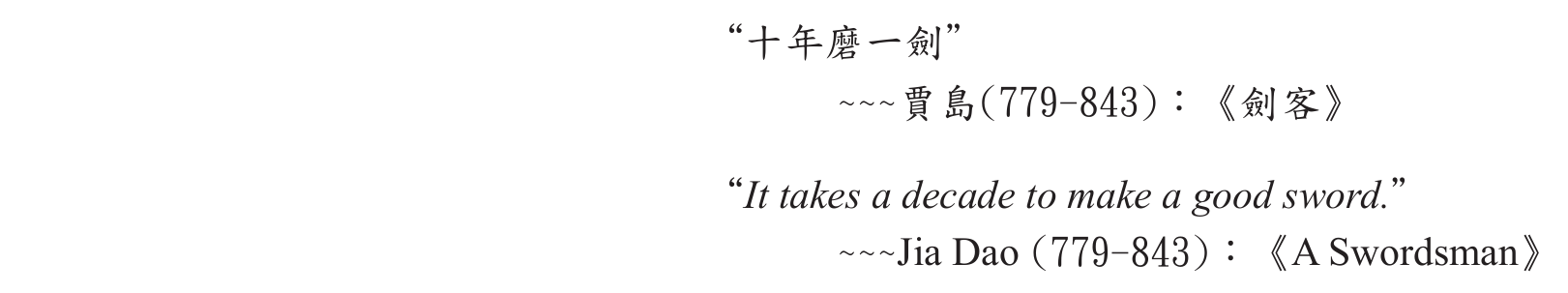}

%
%
%

\newpage

\begin{flushleft}
{\large\bf Appendix\;\; Basic moves for the multiplication of two superfields}
\end{flushleft}
In the block-matrix form of a superfield $f\in C^\infty(\widehat{X})$ on the superspace $\widehat{X}$
 
 {\footnotesize
  \begin{eqnarray*}
   f & = &  f_{(0)}
	  + \sum_\alpha f_{(\alpha)}\theta^\alpha
	  + \sum_{\dot{\beta}} f_{(\dot{\beta})} \bar{\theta}^{\dot{\beta}}
	  + f_{(12)}\theta^1\theta^2
	  + \sum_{\alpha, \dot{\beta}}f_{(\alpha\dot{\beta})}
	        \theta^\alpha \bar{\theta}^{\dot{\beta}}
	  + f_{(\dot{1}\dot{2})}\bar{\theta}^{\dot{1}}\bar{\theta}^{\dot{2}}\\
    && \hspace{2em} 	  	
	  + \sum_{\dot{\beta}} f_{(12\dot{\beta})}	
                      \theta^1\theta^2 \bar{\theta}^{\dot{\beta}}	
	  + \sum_\alpha f_{(\alpha\dot{1}\dot{2})}
	                  \theta^\alpha\bar{\theta}^{\dot{1}}\bar{\theta}^{\dot{2}}
      + f_{(12\dot{1}\dot{2})}					
	                 \theta^1\theta^2\bar{\theta}^{\dot{1}}\bar{\theta}^{\dot{2}}\\
   & = &  f_{(0)}
	  + \sum_\alpha  \theta^\alpha f_{(\alpha)}
	  + \sum_{\dot{\beta}}  \bar{\theta}^{\dot{\beta}} f_{(\dot{\beta})}
	  + \theta^1\theta^2 f_{(12)}
	  + \sum_{\alpha, \dot{\beta}} \theta^\alpha \bar{\theta}^{\dot{\beta}}
	        f_{(\alpha\dot{\beta})}
	  + f_{(\dot{1}\dot{2})}\bar{\theta}^{\dot{1}}\bar{\theta}^{\dot{2}}\\
    && \hspace{2em} 	  	
	  + \sum_{\dot{\beta}} f_{(12\dot{\beta})}	
                      \theta^1\theta^2 \bar{\theta}^{\dot{\beta}}	
	  + \sum_\alpha f_{(\alpha\dot{1}\dot{2})}
	                  \theta^\alpha\bar{\theta}^{\dot{1}}\bar{\theta}^{\dot{2}}
      + f_{(12\dot{1}\dot{2})}					
	                 \theta^1\theta^2\bar{\theta}^{\dot{1}}\bar{\theta}^{\dot{2}}\\					 
   &\:=:&
     \left[
	  \begin{array} {c|cc|c}
	   f_{(0)}  &  f_{(\dot{1})} & f_{(\dot{2})}  & f_{(\dot{1}\dot{2})}
	                                                                 \\[.6ex] \hline  \rule{0ex}{1em}
	   f_{(1)}  &  f_{(1\dot{1})}& f_{(1\dot{2})}& f_{(1\dot{1}\dot{2})}
	                                                                 \\[.6ex]
	   f_{(2)}  & f_{(2\dot{1})} & f_{(2\dot{2})}& f_{(2\dot{1}\dot{2})}
	                                                                 \\[.6ex] \hline  \rule{0ex}{1em}
	   f_{(12)}& f_{(12\dot{1})} & f_{(12\dot{2})}	& f_{(12\dot{1}\dot{2})}
	  \end{array}
     \right], 	
  \end{eqnarray*}} 
  
 \noindent
 the application of a derivation on a superfield or
 the multiplication of two superfields can be decomposed into a combination of {\it basic moves}:
 
 {\footnotesize
  \begin{eqnarray*}
    \frac{\partial}{\partial\theta^1}f\; =\;
     \left[
	  \begin{array} {c|cc|c}
	   f_{(1)}  &  f_{(1\dot{1})} & f_{(1\dot{2})}
	                            & f_{(1\dot{1}\dot{2})}   \\[.6ex] \hline  \rule{0ex}{1em}
	   0   &  0  &  0  &  0   \\[.6ex]
	   f_{(12)}  & f_{(12\dot{1})} & f_{(12\dot{2})}& f_{(12\dot{1}\dot{2})}
	                                                                 \\[.6ex] \hline  \rule{0ex}{1em}
	   0  &  0  &  0  &  0
	  \end{array}
     \right]\,,
   &&
    \frac{\partial}{\partial\theta^2}f\;=\;
	 \left[
	  \begin{array} {c|cc|c}
	   f_{(2)}  &  f_{(2\dot{1})} & f_{(2\dot{2})}  & f_{(2\dot{1}\dot{2})}
	                                                                 \\[.6ex] \hline  \rule{0ex}{1em}
	   -\,f_{(12)}  &  -\,f_{(12\dot{1})}& -\, f_{(12\dot{2})}
	                                  &  -\,f_{(12\dot{1}\dot{2})}   \\[.6ex]
	   0  &  0  &  0  &  0     \\[.6ex] \hline  \rule{0ex}{1em}
	   0  &  0  &  0  &  0
	  \end{array}
     \right]\,,	
   \end{eqnarray*}
 
 \begin{eqnarray*}
  \frac{\partial}{\partial\bar{\theta}^{\dot{1}}} f \; =\;
     \left[
	  \begin{array} {c|cc|c}
	   f_{(\dot{1})}         &  0  &  f_{(\dot{1}\dot{2})}        &  0
	                                                                 \\[.6ex] \hline  \rule{0ex}{1em}
	   -\, f_{(1\dot{1})}  &  0  & -\, f_{(1\dot{1}\dot{2})}  &  0
	                                                                 \\[.6ex]
	   -\, f_{(2\dot{1})}  &  0  & -\, f_{(2\dot{1}\dot{2})}  &  0
	                                                                 \\[.6ex] \hline  \rule{0ex}{1em}
	   f_{(12\dot{1})}    &  0   & f_{(12\dot{1}\dot{2})}    & 0
	  \end{array}
     \right]\,,
   &&
    \frac{\partial}{\partial\bar{\theta}^{\dot{2}}} f \; =\;
     \left[
	  \begin{array} {c|cc|c}
	   f_{(\dot{2})}         &  -\, f_{(\dot{1}\dot{2})}        &  0  &  0
	                                                                 \\[.6ex] \hline  \rule{0ex}{1em}
	   -\, f_{(1\dot{2})}  &   f_{(1\dot{1}\dot{2})}          &  0  &  0
	                                                                 \\[.6ex]
	   -\, f_{(2\dot{2})}  &   f_{(2\dot{1}\dot{2})}          &  0  &  0
	                                                                 \\[.6ex] \hline  \rule{0ex}{1em}
	   f_{(12\dot{2})}    &  -\, f_{(12\dot{1}\dot{2})}    &  0  &  0
	  \end{array}
     \right]\,;
 \end{eqnarray*}

 \begin{eqnarray*}
  \theta^1 f\;=\;
    \left[
	  \begin{array} {c|cc|c}
	   0  &  0  &  0  &  0    \\[.6ex] \hline  \rule{0ex}{1em}
	   f_{(0)}  &  f_{(\dot{1})} & f_{(\dot{2})}  & f_{(\dot{1}\dot{2})}
	                                                                 \\[.6ex]
	   0  &  0  &  0  &  0  \\[.6ex] \hline  \rule{0ex}{1em}
	   f_{(2)}& f_{(2\dot{1})} & f_{(2\dot{2})}	& f_{(2\dot{1}\dot{2})}
	  \end{array}
     \right]\,,
  &&
   \theta^2 f\;=\;
    \left[
	  \begin{array} {c|cc|c}
	   0  &  0  &  0  &  0    \\[.6ex] \hline  \rule{0ex}{1em}
	   0  &  0  &  0  &  0  \\[.6ex]
	   f_{(0)}  &  f_{(\dot{1})} & f_{(\dot{2})}  & f_{(\dot{1}\dot{2})}
	                                                                 \\[.6ex] \hline  \rule{0ex}{1em}
	   -\,f_{(1)}& -\,f_{(1\dot{1})} & -\,f_{(1\dot{2})}	& -\,f_{(1\dot{1}\dot{2})}
	  \end{array}
     \right]\,,
 \end{eqnarray*}

 \begin{eqnarray*}
  \bar{\theta}^{\dot{1}} f\; =\;
    \left[
	  \begin{array} {c|cc|c}
	   0  &  f_{(0)}   & 0 & f_{(\dot{2})}
	                                                                 \\[.6ex] \hline  \rule{0ex}{1em}
	   0  & -\,f_{(1)}& 0 & -\,f_{(1\dot{2})}
	                                                                 \\[.6ex]
	   0  & -\,f_{(2)}& 0 & -\,f_{(2\dot{2})}
	                                                                 \\[.6ex] \hline  \rule{0ex}{1em}
	   0  & f_{(12)}  & 0 & f_{(12\dot{2})}
	  \end{array}
    \right]\,,
  &&
   \bar{\theta}^{\dot{2}} f\; =\;
    \left[
	  \begin{array} {c|cc|c}
	   0  & 0 &  f_{(0)}    & -\,f_{(\dot{1})}
	                                                                 \\[.6ex] \hline  \rule{0ex}{1em}
	   0  & 0 & -\,f_{(1)} &  f_{(1\dot{1})}
	                                                                 \\[.6ex]
	   0  & 0 & -\,f_{(2)} &  f_{(2\dot{1})}
	                                                                 \\[.6ex] \hline  \rule{0ex}{1em}
	   0  & 0 & f_{(12)}   & -\,f_{(12\dot{1})}
	  \end{array}
    \right]\,,
 \end{eqnarray*}

 \begin{eqnarray*}
  \theta^1\theta^2 f\; =\;
    \left[
	  \begin{array} {c|cc|c}
	   0 &  0 & 0 & 0  \\[.6ex] \hline  \rule{0ex}{1em}
	   0 &  0 & 0 & 0  \\[.6ex]
	   0 &  0 & 0 & 0  \\[.6ex] \hline  \rule{0ex}{1em}
	   f_{(0)}& f_{(\dot{1})} & f_{(\dot{2})}	& f_{(\dot{1}\dot{2})}
	  \end{array}
     \right]\,,
  &&
   \bar{\theta}^{\dot{1}}\bar{\theta}^{\dot{2}} f\; =\;
    \left[
	  \begin{array} {c|cc|c}
	   0 &  0 & 0 & f_{(0)}  \\[.6ex] \hline  \rule{0ex}{1em}
	   0 &  0 & 0 & f_{(1)}  \\[.6ex]
	   0 &  0 & 0 & f_{(2)}  \\[.6ex] \hline  \rule{0ex}{1em}
	   0 &  0 & 0 & f_{(12)}
	  \end{array}
    \right]\,,
 \end{eqnarray*}

 \begin{eqnarray*}
  \theta^1\bar{\theta}^{\dot{1}} f\; =\;
    \left[
	  \begin{array} {c|cc|c}
	   0 &  0 & 0 & 0  \\[.6ex] \hline  \rule{0ex}{1em}
	   0 &  f_{(0)}     & 0  & f_{(\dot{2})}  \\[.6ex]
	   0 &  0 & 0 & 0  \\[.6ex] \hline  \rule{0ex}{1em}
	   0 & -\,f_{(2)}  & 0  & -\,f_{(\dot{2}\dot{2})}
	  \end{array}
    \right]\,,
  &&
   \theta^1\bar{\theta}^{\dot{2}} f\; =\;
    \left[
	  \begin{array} {c|cc|c}
	   0 &  0 & 0 & 0  \\[.6ex] \hline  \rule{0ex}{1em}
	   0 &  0 & f_{(0)}      & -\,f_{(\dot{1})}  \\[.6ex]
	   0 &  0 & 0 & 0  \\[.6ex] \hline  \rule{0ex}{1em}
	   0 &  0 & -\,f_{(2)}  & f_{(\dot{2}\dot{1})}
	  \end{array}
    \right]\,,
 \end{eqnarray*}
 
 \begin{eqnarray*}
  \theta^2\bar{\theta}^{\dot{1}} f\; =\;
    \left[
	  \begin{array} {c|cc|c}
	   0 &  0 & 0 & 0  \\[.6ex] \hline  \rule{0ex}{1em}
	   0 &  0 & 0 & 0  \\[.6ex]
	   0 &  f_{(0)} & 0 & f_{(\dot{2})}  \\[.6ex] \hline  \rule{0ex}{1em}
	   0 & f_{(1)}      & 0 & f_{(1\dot{2})}
	  \end{array}
     \right]\,,
  &&
   \theta^2\bar{\theta}^{\dot{2}} f\; =\;
    \left[
	  \begin{array} {c|cc|c}
	   0 &  0 & 0 & 0  \\[.6ex] \hline  \rule{0ex}{1em}
	   0 &  0 & 0 & 0  \\[.6ex]
	   0 &  0 & f_{(0)} & -\,f_{(\dot{1})}  \\[.6ex] \hline  \rule{0ex}{1em}
	   0 &  0 & f_{(1)} & -\,f_{(1\dot{1})}
	  \end{array}
     \right]\,,
 \end{eqnarray*}
  
 \begin{eqnarray*}
  \theta^1\theta^2\bar{\theta}^{\dot{1}} f\; =\;
    \left[
	  \begin{array} {c|cc|c}
	   0 &  0 & 0 & 0  \\[.6ex] \hline  \rule{0ex}{1em}
	   0 &  0 & 0 & 0  \\[.6ex]
	   0 &  0 & 0 & 0  \\[.6ex] \hline  \rule{0ex}{1em}
	   0 & f_{(0)} & 0 & f_{(\dot{2})}
	  \end{array}
    \right]\,,
  &&
   \theta^1\theta^2\bar{\theta}^{\dot{2}} f\; =\;
    \left[
	  \begin{array} {c|cc|c}
	   0 &  0 & 0 & 0  \\[.6ex] \hline  \rule{0ex}{1em}
	   0 &  0 & 0 & 0  \\[.6ex]
	   0 &  0 & 0 & 0  \\[.6ex] \hline  \rule{0ex}{1em}
	   0 &  0 & f_{(0)}  & -\,f_{(\dot{1})}
	  \end{array}
    \right]\,,
 \end{eqnarray*}
 
 \begin{eqnarray*}
  \theta^1\bar{\theta}^{\dot{1}}\bar{\theta}^{\dot{2}} f\; =\;
    \left[
	  \begin{array} {c|cc|c}
	   0 &  0 & 0 & 0  \\[.6ex] \hline  \rule{0ex}{1em}
	   0 &  0 & 0 & f_{(0)} \\[.6ex]
	   0 &  0 & 0 & 0 \\[.6ex] \hline  \rule{0ex}{1em}
	   0 &  0 & 0 & f_{(2)}
	  \end{array}
     \right]\,,
  &&
   \theta^2\bar{\theta}^{\dot{1}}\bar{\theta}^{\dot{2}} f\; =\;
    \left[
	  \begin{array} {c|cc|c}
	   0 &  0 & 0 & 0  \\[.6ex] \hline  \rule{0ex}{1em}
	   0 &  0 & 0 & 0  \\[.6ex]
	   0 &  0 & 0 & f_{(0)} \\[.6ex] \hline  \rule{0ex}{1em}
	   0 &  0 & 0 & -\,f_{(1)}
	  \end{array}
    \right]\,,
 \end{eqnarray*}
   
 $$
  \theta^1\theta^2\bar{\theta}^{\dot{1}}\bar{\theta}^{\dot{2}} f\; =\;
    \left[
	  \begin{array} {c|cc|c}
	   0 &  0  & 0  & 0  \\[.6ex] \hline  \rule{0ex}{1em}
	   0 &  0  & 0  & 0  \\[.6ex]
	   0 &  0  & 0  & 0  \\[.6ex] \hline  \rule{0ex}{1em}
	   0 &  0  & 0  & f_{(0)}
	  \end{array}
     \right]\,;
 $$
 
 \begin{eqnarray*}
  f \theta^1 \;=\;
    \left[
	  \begin{array} {c|cc|c}
	   0  &  0  &  0  &  0    \\[.6ex] \hline  \rule{0ex}{1em}
	   f_{(0)}  &  -\,f_{(\dot{1})} & -\,f_{(\dot{2})}  & f_{(\dot{1}\dot{2})}
	                                                                 \\[.6ex]
	   0  &  0  &  0  &  0  \\[.6ex] \hline  \rule{0ex}{1em}
	   -\,f_{(2)}& f_{(2\dot{1})} & f_{(2\dot{2})}	& -\,f_{(2\dot{1}\dot{2})}
	  \end{array}
     \right]\,,
  &&
   f \theta^2\; =\;
    \left[
	  \begin{array} {c|cc|c}
	   0  &  0  &  0  &  0    \\[.6ex] \hline  \rule{0ex}{1em}
	   0  &  0  &  0  &  0  \\[.6ex]
	   f_{(0)}  &  -\,f_{(\dot{1})} & -\,f_{(\dot{2})}  & f_{(\dot{1}\dot{2})}
	                                                                 \\[.6ex] \hline  \rule{0ex}{1em}
	   f_{(1)}& -\,f_{(1\dot{1})} & -\,f_{(1\dot{2})}	& f_{(1\dot{1}\dot{2})}
	  \end{array}
     \right]\,,
 \end{eqnarray*}

 \begin{eqnarray*}
  f \bar{\theta}^{\dot{1}}\; =\;
    \left[
	  \begin{array} {c|cc|c}
	   0 & f_{(0)}  &  0  & -\,f_{(\dot{2})}
	                                                                 \\[.6ex] \hline  \rule{0ex}{1em}
	   0 & f_{(1)}  &  0  & -\,f_{(1\dot{2})}
	                                                                 \\[.6ex]
	   0 & f_{(2)}  &  0  & -\,f_{(2\dot{2})}
	                                                                 \\[.6ex] \hline  \rule{0ex}{1em}
	   0 & f_{(12)}&  0  & -\,f_{(12\dot{2})}	
	  \end{array}
     \right]\,,
  &&
   f \bar{\theta}^{\dot{2}}\; =\;
    \left[
	  \begin{array} {c|cc|c}
	   0 & 0 & f_{(0)}  & f_{(\dot{1})}
	                                                                 \\[.6ex] \hline  \rule{0ex}{1em}
	   0 & 0 & f_{(1)}  & f_{(1\dot{1})}
	                                                                 \\[.6ex]
	   0 & 0 & f_{(2)}  & f_{(2\dot{1})}
	                                                                 \\[.6ex] \hline  \rule{0ex}{1em}
	   0 & 0 & f_{(12)}& f_{(12\dot{1})}	
	  \end{array}
    \right]\,,
 \end{eqnarray*}

 \begin{eqnarray*}
  f \theta^1\theta^2 \; =\;  \theta^1\theta^2 f\,,
    && f \bar{\theta}^{\dot{1}}\bar{\theta}^{\dot{2}}\;
          =\;  \bar{\theta}^{\dot{1}}\bar{\theta}^{\dot{2}} f\,,\\[1ex]
  f \theta^1\bar{\theta}^{\dot{1}} \; =\;   \theta^1\bar{\theta}^{\dot{1}} f\,,
    && f \theta^1\bar{\theta}^{\dot{2}} \; =\; \theta^1\bar{\theta}^{\dot{2}} f\,, \\[1ex]
  f \theta^2\bar{\theta}^{\dot{1}} \; =\;    \theta^2\bar{\theta}^{\dot{1}} f\,,
    &&  f \theta^2\bar{\theta}^{\dot{2}}\; =\; \theta^2\bar{\theta}^{\dot{2}} f\,,
 \end{eqnarray*}

 \begin{eqnarray*}
  f \theta^1\theta^2\bar{\theta}^{\dot{1}} \; =\;
    \left[
	  \begin{array} {c|cc|c}
	   0 &  0 & 0 & 0  \\[.6ex] \hline  \rule{0ex}{1em}
	   0 &  0 & 0 & 0  \\[.6ex]
	   0 &  0 & 0 & 0  \\[.6ex] \hline  \rule{0ex}{1em}
	   0 & f_{(0)} & 0 & -\,f_{(\dot{2})}
	  \end{array}
    \right]\,,
  &&
   f \theta^1\theta^2\bar{\theta}^{\dot{2}} \; =\;
    \left[
	  \begin{array} {c|cc|c}
	   0 &  0 & 0 & 0  \\[.6ex] \hline  \rule{0ex}{1em}
	   0 &  0 & 0 & 0  \\[.6ex]
	   0 &  0 & 0 & 0  \\[.6ex] \hline  \rule{0ex}{1em}
	   0 &  0 & f_{(0)}  & f_{(\dot{1})}
	  \end{array}
    \right]\,,
 \end{eqnarray*}
 
 \begin{eqnarray*}
  f \theta^1\bar{\theta}^{\dot{1}}\bar{\theta}^{\dot{2}} \; =\;
    \left[
	  \begin{array} {c|cc|c}
	   0 &  0 & 0 & 0  \\[.6ex] \hline  \rule{0ex}{1em}
	   0 &  0 & 0 & f_{(0)} \\[.6ex]
	   0 &  0 & 0 & 0 \\[.6ex] \hline  \rule{0ex}{1em}
	   0 &  0 & 0 & -\,f_{(2)}
	  \end{array}
     \right]\,,
  &&
   f \theta^2\bar{\theta}^{\dot{1}}\bar{\theta}^{\dot{2}} \; =\;
    \left[
	  \begin{array} {c|cc|c}
	   0 &  0 & 0 & 0  \\[.6ex] \hline  \rule{0ex}{1em}
	   0 &  0 & 0 & 0  \\[.6ex]
	   0 &  0 & 0 & f_{(0)} \\[.6ex] \hline  \rule{0ex}{1em}
	   0 &  0 & 0 & f_{(1)}
	  \end{array}
    \right]\,,
 \end{eqnarray*}
   
 $$
  f \theta^1\theta^2\bar{\theta}^{\dot{1}}\bar{\theta}^{\dot{2}}\;
   =\;   \theta^1\theta^2\bar{\theta}^{\dot{1}}\bar{\theta}^{\dot{2}} f\,.
 $$

} 

\newpage
\baselineskip 13pt
{\footnotesize

\vspace{1em}

\noindent
chienhao.liu@gmail.com, 
chienliu@cmsa.fas.harvard.edu; \\
yau@math.harvard.edu

}


\begin{thebibliography}{AAaaaa}
%
\marginpar{\raggedright\tiny $\bullet$
        References\\
		$\hspace{1.6ex}$ for D(14.1).
		}		
%
\bibitem[Ar]{} P.\ Argyres,
 {\sl Introduction to supersymmetry},
  Physics 661 lecture notes, Cornell University, fall, 2000.

\bibitem[AG-F1]{} L.\ Alvarez-Gaum\'{e} and D.Z.\ Freedman, 	
 {\it Geometrical structure and ultraviolet finiteness in the supersymmetric sigma model},
 {\sl Commun.\ Math.\ Phys.}\ {\bf 80} (1981), 443--451.

\bibitem[AG-F2]{} --------, 	
 {\it Potentials for the supersymmetric nonlinear sigma model},
 {\sl Commun.\ Math.\ Phys.}\  {\bf 91} (1983), 87--101.

\bibitem[A-P]{} R.L.\ Arnowitt and P.\ Nath,
 {\it Superconnections in extended supergravity},
 {\sl Phys.\ Rev.\ Lett.}\ {\bf 44} (1980), 223--226.

\bibitem[A-V]{} I. Ya Aref'eva and I.V.\ Volovich,
 {\it Reconstruction of superconnection from physical fields in the $N=4$ supersymmetric Yang-Mills theory},
 {\sl Class.\ Quantum Grav.}\ {\bf 3} (1986), 617--623.

\bibitem[A-V-S]{} V.P.\ Akulov, D.V.\ Volkov, and V.A.\ Soroka,
 {\it Gauge fields on superspaces with different holonomy groups},
 {\sl JETP Lett.}\ {\bf 22} (1975), 187--188.
 
\bibitem[Ba]{} J.A.\ Bagger,
 {\it Supersymmetric sigma models},
 in {\sl Supersymmetry} (Bonn, 1984), 45--87,
 NATO Adv.\ Sci.\ Inst.\ Ser.\ B Phys.\ 125, Plenum 1985.

\bibitem[Bi]{}  A.\ Bilal,
 {\it Introduction to supersymmetry},
  arXiv:hep-th/0101055.
 
\bibitem[Bl]{} D.\ Bleecker,
 {\sl Gauge theory and variational principles},
 Addison-Wesley, 1981.

\bibitem[Br]{} Th.$\:$Br\"{o}cker,
 {\sl Differentiable germs and catastrophes},
 translated from the German, last chapter and bibliography by L. Lander.
 London Math.\ Soc.\ Lect.\ Note Ser.\ 17.
 Cambridge Univ.\ Press, 1975.

\bibitem[B-B-S]{} K.\ Becker, M.\ Becker, and J.H.\ Schwarz,
 {\sl String theory and M-theory -- a modern introduction},
  Cambridge Univ.\ Press, 2007.
 
\bibitem[B-DV-H]{} L.\ Brink, P.\ Di Vecchia, P.S.\ Howe,
 {\it A locally supersymmetric and reparametrization invariant action for the spinning string},
 {\sl Phys.\ Lett.}\ {\bf  65B} (1976), 471--474.
 
\bibitem[B-M-V]{} J.V.\ Beltr\'{a}n, J.\ Monterde, and J.A.\ Vallejo,
 {\it Quillen superconnections and connections on supermanifolds},
 {\sl J.\ Geom.\ Phys.}\ {\bf 86} (2014), 180--193.
 (arXiv.1305.3677 [math.DG])

\bibitem[B-S-S]{} L.\ Brink, J.H.\ Schwarz, J.\ Scherk,
 {\it Supersymmetric Yang-Mills theories},
 {\sl Nucl.\ Phys.}\ {\bf B121} (1977), 77--92.

\bibitem[B-W]{} J.\ Bagger and E.\ Witten,
 {\it The gauge invariant supersymmetric nonlinear sigma model},
 {\sl Phys.\ Lett.}\ {\bf 118B} (1982), 103--106.
  
\bibitem[Ch]{} S.\ Chandrasekhar,
 {\sl The mathematical theory of black holes},
 Oxford Univ.\ Press, 1983.

\bibitem[CB]{} Y.\ Choquet-Bruhat,
 {\sl Graded bundles and supermanifolds},
  Mono.\ Text.\ Phys.\ Sci.\ Lect.\ Notes 12,
Bibliopolis, 1989.
 
\bibitem[C-C-F]{} C.\ Carmeli, L.\ Caston, and R.\ Fioreci,
 {\sl Mathematical foundations of supersymmetry},
 European Math.\ Soc., 2011.

\bibitem[C-G-P]{} T.\ Covolo, J.\ Grabowski, and N.\ Poncin,
 {\it ${\Bbb Z}_2^n$-supergeometry I: manifolds and morphisms},
 arXiv:1408.2755 [math.DG].

\bibitem[C-J-S-F-G-vN]{} E.\ Cremmer, B.\ Julia, J.\ Scherk, S.\ Ferrara, L.\ Girardello, P.\ van Nieuwenhuizen,
 {\it Spontaneous symmetry breaking and Higgs effect in supergravity without cosmological constant},
 {\sl Nucl.\ Phys.}\ {\bf B147} (1979), 105--131.	
 
\bibitem[C-K-P1]{}T.\ Covolo, S.\ Kwok, and N.\ Poncin,
 {\it Differential calculus on ${\Bbb Z}_2^n$-supermanifolds},
  arXiv:1608.00949 [math.DG].

\bibitem[C-K-P2]{} --------,
 {\it The Frobenius theorem for ${\Bbb Z}_2^n$-supermanifolds},
  arXiv:1608.00961 [math.DG].

\bibitem[C-T]{} C.G.\ Callan, Jr.\ and L.\ Thorlacius,
 {\it Sigma models and string theory},
 in {\sl  Particles, strings and supernovae (TASI 88)},
  A.\ Jevicki and C.-I.\ Tan eds.,  795--878,
  World Scientific, 1989.

\bibitem[DeW]{} B.S.$\:$DeWitt,
 {\sl Supermanifolds}, 2nd ed.,
 Cambridge Mono.\ Math.\ Phys., Cambridge Univ.\ Press, 1992.
    
\bibitem[Dr]{} N.\ Dragon,
 {\it Torsion and curvature in extended supergravity},
 {\sl Z.\ Phys.}\ {\bf C2} (1979), 29--32.
 
\bibitem[Du]{} F.\ Dumitrescu,
 {\it Superconnections and parallel transport},
 Ph.D.\ thesis, University of Notre Dame, 2006.

\bibitem[D-F1]{} P.\ Deligne and D.S.\ Freed, ,
 {\it Supersolutions},
  in {\sl Quantum fields and strings: a course for mathematicians}, vol.\ 1,
  P.\ Deligne, P.\ Etingof, D.S.\ Freed, L.C.\ Jeffrey, D.\ Kazhdan, J.W.\ Morgan, and E.\ Witten eds., 227--355,
  American Math.\ Soc., 1999.

\bibitem[D-F2]{} --------,
 {\it Sign manifesto},
  in {\sl Quantum fields and strings: a course for mathematicians}, vol.\ 1,
  P.\ Deligne, P.\ Etingof, D.S.\ Freed, L.C.\ Jeffrey, D.\ Kazhdan, J.W.\ Morgan, and E.\ Witten eds., 357--363,
  American Math.\ Soc., 1999.

\bibitem[DV-F]{} P.\ Di Vecchia and S.\ Ferrara,
 {\it Classical solutions in two-dimensional supersymmetric field theories},
 {\sl Nucl.\ Phys.}\ {\bf B130} (1977), 93--104.

\bibitem[D-K]{} S.K.\ Donaldson and P.B.\ Kronheimer,
 {\sl The geometry of four manifolds},
 Oxford Univ.\ Press, 1990.
 
\bibitem[D-M]{} P.\ Deligne and J.W.\ Morgan,
 {\it Notes on supersymmetry (following Joseph Bernstein)},
  in {\sl Quantum fields and strings: a course for mathematicians}, vol.\ 1,
  P.\ Deligne, P.\ Etingof, D.S.\ Freed, L.C.\ Jeffrey, D.~Kazhdan, J.W.\ Morgan, and E.\ Witten eds., 41--97,
  American Math.\ Soc., 1999.
 
\bibitem[DV-M-K]{} M.\ Dubois-Violette, J.\ Madore, and R.\ Kerner,
 {\it Super matrix geometry},
 {\sl Class.\ Quantum Grav.}\ {\bf 8} (1991), 1077--1089.

\bibitem[Ei]{} D.\ Eisenbud,
 {\sl Commutative algebra -- with a view toward algebraic geometry},
 GTM 150, Springer, 1994.
 
\bibitem[Ev]{} J.M.\ Evans,
 {\it Supersymmetric Yang-Mills theory and division algebras},
 {\sl Nucl.\ Phys.}\ {\bf B298} (1988), 92--108.
 
\bibitem[E-H]{} D.~Eisenbud and J.~Harris,
 {\sl The geometry of schemes},
 GTM~197, Springer, 2000.

\bibitem[Freed]{} D.S.\ Freed,
 {\sl Five lectures on supersymmetry},
 Amer.\ Math.\ Soc., 1999.
 
\bibitem[Freund] {} P.G.O.\ Freund,
 {\sl Introduction to supersymmetry},
  Cambridge Univ.\ Press, 1986.

\bibitem[F-F-vN-B-G-S]{} S.\ Ferrara, D.Z.\ Freedman, P.\ van Nieuwenhuizen,
       P.\ Breitenlohner, F.\ Gliozzi, and J.\ Scherk,
 {\it Scalar multiplet coupled to supergravity},
{\sl Phys.\ Rev.}\ {\bf D15} (1977), 1013--1018.
  
\bibitem [F-vN]{} S.\ Ferrara  and P.\ van Nieuwenhuizen,
 {\it Tensor calculus for supergravity},
 {\sl Phys.\ Lett.}\ {\bf 76B} (1978), 404--408.
  
\bibitem[F-T]{} D.Z.\ Freedman and P.K.\ Townsend,
 {\it Antisymmetric tensor gauge theories and nonlinear sigma models},
 {\sl Nucl.\ Phys.}\ {\bf B177} (1981), 282--296.

\bibitem[F-W-Z]{} S.\ Ferrara, J.\ Wess, and B.\ Zumino,
 {\it Supergauge multiplets and superfields},,
 {\sl Phys.\ Lett.}\ {\bf 51B} (1974), 239--241.
 
\bibitem[F-Z]{} S.\ Ferrara and B.\ Zumino,
 {\it Supergauge invariant Yang-Mills theories},
 {\sl Nucl.\ Phys.}\ {\bf B79} (1974), 413--421.

\bibitem[Gi]{} F.$\:$Gieres,
 {\sl Geometry of supersymmetric gauge theories -- including an introduction to BRS differential algebras and anomalies},
 Lect.\ Notes Phys.\ 302, Springer, 1988.
 
\bibitem[G-G-R-S]{} S.J.\ Gates, Jr., M.T.\ Grisaru, M.\ Roc\u{c}ek, and W.\ Siegel,
 {\sl Superspace -- one thousand and one lessons in supersymmetry}, Frontiers Phys.\ Lect.\ Notes Ser.\ 58,
  Benjamin/Cummings Publ.\ Co., Inc., 1983.

\bibitem[G-H]{}  P.\ Griffiths and J.\ Harris,
 {\sl Principles of algebraic geometry},
 John Wiley \& Sons, 1978.
     
\bibitem[G-S]{} S.J.\ Gates, Jr., and W.\ Siegel,
 {\it Superfield supergravity},
 {\sl Nucl.\ Phys.}\ {\bf B147} (1979), 77--104.

\bibitem[Ga-S-W]{} S.J.\ Gates, Jr., K.\ Stelle, and P.\ West,
 {\it Algebraic origins of superspace constraints in supergravity},
 {\sl Nucl.\ Phys.}\ {\bf B169} (1980), 347--364.

\bibitem[Gr-S-W]{} M.B.\ Green, J.H.\ Schwarz, and E.\ Witten,
 {\sl Superstring theory}, {\sl vol.\ 1}$\,$: {\sl Introduction};
 {\sl vol.\ 2}$\,$: {\sl Loop amplitudes, anomalies, and phenomenology},
 Cambridge Univ.\ Press, 1987.

\bibitem[Gri-S-W]{} R.\ Grimm, M.\ Sohnius and J.\ Wess,
 {\it Extended supersymmetry and gauge theories},
 {\sl Nucl.\ Phys.}\ {\bf B133} (1978), 275--284.

\bibitem[G-W-Z]{} R.\ Grimm, J.\ Wess, and B.\ Zumino,
 {\it A complete solution of the Bianchi identities in superspace},
 {\sl Nucl.\ Phys.}\ {\bf B152} (1979), 255--265.

\bibitem[Han]{} F.\ Hanisch,
 {\sl Variational problems on supermanifolds},
 Universit\"{a}t Potsdam thesis, 2011.
 
\bibitem[Har]{} R.\ Hartshorne,
 {\sl Algebraic geometry},
 GTM 52, Springer, 1977.

\bibitem[Hi]{} N.J.\ Hicks,
 {\sl Notes on differential geometry},
  Van Nostrand Math. Studies 3,  D. Van Nostrand Co., Inc., 1965.
   
\bibitem[H-K]{} K.\ Hamada and M.\ Takao,
 {\it Supersymmetric Yang-Mills theory in two-dimensional superspcace and super Kac-Moody algebra},
 {\sl Phys.\ Lett.} {\bf B210} (1988), 120--124.

\bibitem[H-P]{} P.S.\ Howe and G.\ Papadopoulos,
 {\it Further remarks on the geometry of two-dimensional non-linear $\sigma$-models},
 {\sl Class.\ Quantum Grav.}\ {\bf 5} (1988), 1647--1661.
 
\bibitem[Ha-S]{} J.\ Harnard and S.\ Shnider,
 {\it Constraints and field equations for ten-dimensional super Yang-Mills theory},
 {\sl Commun.\ Math.\ Phys.}\ {\bf 106} (1986), 183--199.

\bibitem[Hi-S]{} C.D.\ Hill and S.R.\ Simanca,
 {\it The super complex Frobenius theorem},
 {\sl Ann.\ Polonici Math.}\ {\bf 55} (1991), 139--155.
 
\bibitem[H-V]{}  K.\ Hori and C.\ Vafa,
 {\it Mirror symmetry},
 arXiv:hep-th/0002222.

\bibitem[H-W]{} P.-M.~Ho and Y.-S.~Wu,
 {\it Noncommutative geometry and D-branes},
 {\sl Phys.\ Lett.}\ {\bf B398}  (1997), 52--60.
 (arXiv:hep-th/9611233)
 
\bibitem[Ja]{} N.\ Jacobson,
 {\sl Basic algebra I}, W.H.\ Freeman \& Co., 1974.

\bibitem[Joh]{} C.V.~Johnson,
 {\sl D-branes},
 Cambridge Univ.\ Press, 2003.
  
\bibitem[Joy]{} D.\ Joyce,
 {\it Algebraic geometry over $C^{\infty}$-rings},
 arXiv:1001.0023 [math.AG].

\bibitem[Ko]{} S.$\:$Kobayashi,
 {\sl Differential geometry of complex vector bundles},
 Publ.\ Math.\ Soc.\ Japan 15, Princeton Univ.\ Press, 1987.
 
\bibitem[K-N]{} S.$\:$Kobayashi and K.$\:$Nomizu,
 {\sl Foundations of differential geometry}, vol.\:I \& vol.$\:$II,
 Interscience Publ., John Wiley \& Sons, 1963 and 1969.

\bibitem[Lin]{} U.\ Lindstr\"{o}m,
 {\it Supersymmetric nonlinear sigma model geometry},
  arXiv:1207.1241 [hep-th].
  
\bibitem[Liu]{} C.-H.\ Liu,
 {\it Azumaya noncommutative geometry and D-branes
      - an origin of the master nature of D-branes},
 lecture given at the workshop
  {\sl Noncommutative algebraic geometry and D-branes},
  December 12 -- 16, 2011,
  organized by Charlie Beil, Michael Douglas, and Peng Gao,
  at Simons Center for Geometry and Physics,
  Stony Brook University, Stony Brook, NY;
 arXiv:1112.4317 [math.AG].
  
\bibitem[L-M-S]{} F.\ Lizzi, N.E.\ Mavromatos, and R.J.\ Szabo,
 {\it Matrix sigma models for multi-D-brane dynamics},
 {\sl Mod.\ Phys.\ Lett.}\ {\bf A13} (1998), 829--842.
  (arXiv:hep-th/9711012)
   
\bibitem[L-Y1]{} C.-H.\ Liu and S.-T.\ Yau,
  {\it Azumaya-type noncommutative spaces and morphism therefrom:
       Polchinski's D-branes in string theory from Grothendieck's
       viewpoint},
  arXiv:0709.1515 [math.AG]. (D(1))

\bibitem[L-L-S-Y]{} S.~Li, C.-H.~Liu, R.~Song, S.-T.~Yau,
 {\it Morphisms from Azumaya prestable curves with a fundamental module
       to a projective variety:
      Topological D-strings as a master object for curves},
  arXiv:0809.2121 [math.AG]. (D(2))
  
\bibitem[L-Y2] {} C.-H.\ Liu and S.-T.\ Yau,
 {\it D-branes and Azumaya noncommutative geometry},
 arXiv:1003.1178 [math.SG]. (D(6))
 
\bibitem[L-Y3]{} --------,
  {\it D-branes and Azumaya/matrix noncommutative differential geometry,
 I: D-branes as fundamental objects in string theory  and differentiable maps
    from Azumaya/matrix manifolds with a fundamental module to real manifolds},
 arXiv:1406.0929 [math.DG]. (D(11.1))
 
\bibitem[L-Y4]{} --------,
 {\it D-branes and Azumaya/matrix noncommutative differential geometry,
 II: Azumaya/matrix supermanifolds and differentiable maps therefrom
      - with a view toward dynamical fermionic D-branes in string theory},
  arXiv:1412.0771 [hep-th]. (D(11.2))
 
\bibitem[L-Y5]{} --------,
 {\it Further studies on the notion of differentiable maps from Azumaya/matrix manifolds, I. The smooth case},
  arXiv:1508.02347 [math.DG]. (D(11.3.1))
    
\bibitem[L-Y6]{} --------,
 {\it Dynamics of D-branes I.
          The non-Abelian Dirac-Born-Infeld action, its first variation, and the equations of motion for D-branes
		  --- with remarks on the non-Abelian Chern-Simons/Wess-Zumino term},
  arXiv:1606.08529 [hep-th]. (D(13.1))

\bibitem[L-Y7]{} --------,
 {\it More on the admissible condition on differentiable maps $\varphi: (X^{\!A\!z},E;\nabla)\rightarrow Y$
         in  the construction of the non-Abelian Dirac-Born-Infeld action $S_{DBI}(\varphi,\nabla)$},
 arXiv:1611.09439 [hep-th]. (D(13.2.1))		
 
\bibitem[L-Y8]{} --------, 
 {\it Dynamics of D-branes II. The standard action
         - an analogue of the Polyakov action for (fundamental, stacked) D-branes}, 
 arXiv:1704.03237 [hep-th]. (D(13.3))		

\bibitem[L-Y9]{} --------,
 {\it Further studies of the notion of differentiable maps from Azumaya/matrix supermanifolds, 
          I. The smooth case: Ramond-Neveu-Schwarz and Green-Schwarz meeting Grothendieck},
  arXiv:1709.08927 [math.DG]. (D(11.4.1))
 
\bibitem[L-Y10]{} --------,
 manuscript in preparation.

\bibitem[Mal]{} B.\ Malgrange,
 {\sl Ideals of differentiable functions},
  Oxford Univ.\ Press, 1966.
 
\bibitem[Man]{} Y.I.\ Manin,
 {\sl Gauge field theory and complex geometry},
  Springer, 1988.

  
\bibitem[M-Q]{} V.\ Mathai and D.\ Quillen, 
 {\it Superconnections, Thom classes, and equivariant differential forms},
 {\sl Topology}, {\bf 25} (1986), 85--110.

\bibitem[Ni]{} H.\ Nishimura, 
 {\it Synthetic theory of superconnections}, 
 {\sl Int.\ J.\ Theo.\ Phys.}\ {\bf 39} (2000), 297--320. 

\bibitem[O'R]{} L.\ O'Raifeartaigh, 
 {\sl Lecture notes on supersymmetry}, 
  Dublin Inst.\ Adv.\ Studies, 1975.
  
\bibitem[Po1]{} J.\ Polchinski,
 {\it Lectures on D-branes},
 in ``{\sl Fields, strings, and duality}", TASI 1996 Summer School,
 Boulder, Colorado, C.\ Efthimiou and B.\ Greene eds.,
 World Scientific, 1997.
 (arXiv:hep-th/9611050)

\bibitem[Po2]{} --------,
 {\sl String theory},
 vol.\ I$\,$: {\sl An introduction to the bosonic string};
 vol.\ II$\,$: {\sl Superstring theory and beyond},
 Cambridge Univ.\ Press, 1998.

\bibitem[Qu]{} D.\ Quillen, 
 {\it Superconnections and the Chern character}, 
 {\sl Topology}, {\bf 24} (1985), 89--95. 

\bibitem[Ro\v{c}]{} M.\ Ro\v{c}ek, 
 {\it Introduction to supersymmetry}, TASI lectures, preprint, 1993.
 
\bibitem[Roe]{} G.\ Roepstorff, 
 {\it Superconnections and the Higgs field}, 
 {\sl J.\ Math.\ Phys.}\ {\bf 40} (1999), 2698--2715.
  (arXiv:hep-th/9801040)

\bibitem[R-L1]{} M.\ Ro\v{c}ek and U.\ Lindstr\"{o}m, 
 {\it Components of superspace}, 
 {\sl Phys.\ Lett.}\ {\bf 79B} (1978), 217--218. 
 
\bibitem[R-L2] {} --------,
 {\it More components of superspace},
 {\sl Phys.\ Lett.}\ {\bf 83B} (1979), 179--184.
  
\bibitem[R-V]{} G.\ Roepstorff and Ch.\ Vehns,
 {\it Generalized Dirac operators and superconnections}, 
  arXiv:math-ph/9911006.

\bibitem[St]{} M.J.\ Strassler, 
 {\it An unorthodox introduction to supersymmetric gauge theory}, 
   lectures given at TASI 2001, arXiv:hep-th/0309149.
  
\bibitem[Sz1]{} R.J.\ Szabo, 
 {\it Superconnections, anomalies, and non-BPS brane charges}, 
  {\sl J.\ Geom.\ Phys.}\ {\bf 43} (2002), 241--292.
  (arXiv:hep-th/0108043)

\bibitem[Sz2]{} R.J.\ Szabo, 
 {\sl An introduction to string theory and D-brane dynamics}, 
  Imperial College Press, 2004. 
    
\bibitem[S-S1]{} A.\ Salam and J.A.\ Strathdee, 
 {\it Super-gauge transformations}, 
 {\sl Nucl.\ Phys.}\ {\bf B76} (1974), 477--482.  

\bibitem[S-S2]{} --------, 
 {\it On superfields and fermi-bose symmetry}, 
 {\sl Phys.\ Rev.}\ {\bf D11} (1975), 1521--1535.

\bibitem[S-S3]{} --------,
 {\it Supersymmetry and superfields}, 
 {\sl Fort.\ Phys.}\ {\bf 26} (1978), 57--142.
 
\bibitem[S-W]{} S.\ Shnider and R.O.\ Wells, Jr.,
 {\sl Supermanifolds, super twister spaces and super Yang-Mills fields},
   S\'{e}minaire Math.\ Sup\'{e}r.\ 106. Press.\ Univ.\ Montr\'{e}al, 1989.

\bibitem[Te]{} J.\ Terning, 
 {\sl Modern supersymmetry -- Dynamics and duality}, 
  Oxford Univ.\ Press, 2006.
   
\bibitem[Wa]{} F.W.\ Warner,
 {\sl Foundations of differentiable manifolds and Lie groups}, 
 Scott Foresmann \& Company, 1971.  

\bibitem[Wei]{} S.\ Weinberg,
 {\sl The quantum theory of fields}, vol.\ III, {\sl Supersymmetry},
  Cambridge Univ.\ Press, 2000.

\bibitem[Wess]{} J.\ Wess, 
 {\it Supersymmetry -- supergravity},  
  in {\sl Topics in quantum field theory and gauge theories}, 
  Proceedings, Salamanca (1977), 
  J.A.\ de Azc\'{a}rraga ed., 81--125,  Lect.\ Notes Phys.\ 77, Springer, 1978.
  
\bibitem[West1]{} P.\ West, 
 {\sl Introduction to supersymmetry and supergravity},
  extended 2nd ed., 
 World Scientific, 1990.
  
\bibitem[West2]{} --------, 
 {\sl Introduction to strings and branes}, 
  Cambridge Univ.\ Press, 2012.

\bibitem[Westra]{} D.B.\ Westra, 
 {\sl Superrings and supergroups}, 
  Ph.D.\ thesis, Universit\"{a}t Wien, October 2009.
  
\bibitem[Wi1]{} E.\ Witten, 
 {\it Introduction to supersymmetry},
  in {\sl The unity of the fundamental interactions}, 
  A.\ Zichichi ed., 305--371,
 Plenum Press, 1983.
    
\bibitem[Wi2]{} --------,
 {\it Topological sigma models},
 {\sl Commun.\ Math.\ Phys.}\ {\bf 118} (1988), 411--449.  

\bibitem[Wi3]{} --------,
 {\it Mirror manifolds and topological field theory},
  arXiv:hep-th/9112056.  

\bibitem[Wi4]{} --------,
 {\it Phases of $N=2$ theories in two dimensions}, 
 {\sl Nucl.\ Phys.}\ {\bf B403} (1993), 159--222.   
 (arXiv:hep-th/9301042)   

\bibitem[Wi5]{} --------,
 {\it Bound states of strings and $p$-branes},
 {\sl Nucl.\ Phys.}\ {\bf B460} (1996), 335--350.
 (arXiv:hep-th/9510135)

\bibitem[Wi6]{} --------,
 {\it Notes on supermanifolds and integration},
 arXiv:1209.2199 [hep-th].  
 
\bibitem[Wi7]{} --------,
 {\it Notes on super Riemann surfaces and their moduli},
 arXiv:1209.2459 [hep-th]. 

\bibitem[Wu]{} S.\ Wu, 
 {\it Mathai-Quillen formalism},
  arXiv:hep-th/0505003.
   
\bibitem[We-B]{} J.\ Wess and J.\ Bagger,
 {\sl Supersymmetry and supergravity}, 2nd ed.,
  Princeton Univ.\ Press, 1992.

\bibitem[Wi-B]{} E.\ Witten and J.\ Bagger,
 {\it Quantization of Newton's constant in certain supergravity theories},
 {\sl Phys.\ Lett.}\ {\bf 115B} (1982), 202--206.
  
\bibitem[W-Z1]{} J.\ Wess and B.\ Zumino, 
 {\it Supergauge transformations in four-dimensions}, 
 {\sl Nucl.\ Phys.}\ {\bf B70} (1974), 39--50.  
  
\bibitem[W-Z2]{} --------,
 {\it A Lagrangian model invariant under supergauge transformations},
 {\sl Phys.\ Lett.}\ {\bf 49B} (1974), 52--54.

\bibitem[W-Z3]{} --------,
 {\it Supergauge invariant extension of quantum electrodynamics},
 {\sl Nucl.\ Phys.}\ {\bf B78} (1974), 1--13. 
  
\bibitem[W-Z4]{} --------,
 {\it Superspace formulation of supergravity},
 {\sl Phys.\ Lett.}\ {\bf 66B} (1977), 361--364.

\bibitem[W-Z5]{} --------,
 {\it The component formalism follows from the superspace formulation of supergravity},
 {\sl Phys.\ Lett.}\ {\bf 79B} (1978), 394--398.
 
\bibitem[Zu]{} B.\ Zumino, 
 {\it Supersymmetry and K\"{a}hler manifolds},
 {\sl Phys.\ Lett.}\ {\bf 87B} (1979), 203--206. 
\end{thebibliography}
\end{document}